\documentclass[pdflatex,sn-mathphys-num]{sn-jnl}% Math and 

\usepackage[utf8]{inputenc}
\usepackage{mathtools}
\usepackage{graphicx}%
\usepackage{multirow}%
\usepackage{amsmath,amssymb,amsfonts}%
\usepackage{amsthm}%
\usepackage{parskip}
\usepackage{imakeidx}
\usepackage{hyperref}
\usepackage{mathrsfs}%
\usepackage[title]{appendix}%
\usepackage{xcolor}%
\usepackage{textcomp}%
\usepackage{manyfoot}%
\usepackage{booktabs}%
\usepackage{algorithm}%
\usepackage{algorithmicx}%
\usepackage{algpseudocode}%
\usepackage{listings}%
\DeclareMathOperator{\dist}{dist}
\usepackage{xparse, etoolbox}
\DeclarePairedDelimiterX\innerp[2]{\langle}{\rangle}{#1,#2}

\DeclareMathOperator{\supp}{supp}

\theoremstyle{thmstyleone}%
\newtheorem{theorem}{Theorem}[section]%  meant for continuous numbers
\newtheorem{proposition}[theorem]{Proposition}% 
\newtheorem{corollary}[theorem]{Corollary}
\newtheorem{lemma}[theorem]{Lemma}

\theoremstyle{thmstyletwo}%
\newtheorem{remark}[theorem]{Remark}%

\theoremstyle{thmstylethree}%
\newtheorem{definition}{Definition}[section]%
\newtheorem*{notation}{Notation}

\begin{document}

\title[Article Title]{On blow-up trees for the harmonic map heat flow from $B^2$ to $S^2$}

\author*{\fnm{Dylan} \sur{Samuelian}}\email{dylan.samuelian@epfl.ch}

\affil*{\orgdiv{SB MATH PDE}, \orgname{EPFL}, \orgaddress{\street{Station 8}, \city{Lausanne}, \postcode{1015}, \country{Switzerland}}}


\abstract{We consider finite-time and $k$-equivariant solutions to the harmonic map heat flow from $B^2$ to $S^2$ under general time-dependent boundary data and prove that the bubble tree decomposition contains only one bubble. The method relies on the Maximum and Comparison Principle. We also exhibit solutions blowing up in infinite time for any $k \geq 1$.}


\maketitle

\tableofcontents

\section{Introduction}
We are interested in the behaviour of finite-time and finite-energy solutions to the harmonic map heat flow (\ref{heat map flow}) from the 2D-ball to the 2-sphere
\begin{align}
    v_t &= \Delta v + |\nabla v|^2 v, \quad (x,t) \in B^2 \times  (0,+\infty) \label{heat map flow} \\
    v(x,t) &= v_0(x,t), \quad (x,t) \in \partial B^2 \times [0,+\infty) \cup B^2 \times \{0\}, \notag
\end{align}
where $v: B^2 \rightarrow S^2 \subset \mathbb R^3$, $\nabla v = (Jv)^T$ is the transposed Jacobian matrix, the energy is defined as
\begin{equation}\label{energy of v}
    E(v(\cdot,t)) = E(v(t)) = \frac{1}{2}\int_{B^2} |\nabla v(\cdot,t)|^2 dx 
\end{equation}
and the boundary data is the restriction of a function $v_0 \in C^{\infty}(\overline{B^2} \times [0,+\infty); S^2)$ which is smooth and $k$-equivariant, meaning it has the form:
\begin{equation}
    v_0(re^{i\theta},t) = \left( e^{ik\theta} \sin h_0(r,t), \cos h_0(r,t) \right) \in \mathbb C \times \mathbb R \simeq \mathbb R^{3},  \label{k-equivariance}
\end{equation}
with $h_0(r,t) \in C^{0}([0,1] \times [0,+\infty),\mathbb R)$ being the \textit{inclination coordinate} of $v_0$ (which can be chosen, without loss generality, smooth and such that $h(r,t) = o(r^k)$ for fixed time, see Lemma \ref{k-equivariance, smoothness of h}). It follows that $v(x,t)$ is also $k$-equivariant (Proposition \ref{prop:equivalent fomrulation heat map flow}) with an inclination coordinate $h(r,t)$ solving the following one-dimensional nonlinear heat equation \begin{align}
h_t &= h_{rr} + \frac{h_r}{r} - k^2 \frac{\sin(2h)}{2r^2}, \quad (r,t) \in (0,1) \times (0,+\infty), \label{heat map formulation for h(r,t)} \\
h(r,t) &= h_0(r,t), \quad (r,t) \in \{0,1\} \times [0,+\infty) \cup (0,1) \times \{0\}. \notag
\end{align}

The study of the Harmonic Map Heat Flow goes back to the foundational work of by Eells and Sampson (\cite{Eells1964}). The equation was first introduced as a mean of finding harmonic maps which are homotopic to some given smooth map $v_0 \in C^{\infty}(M,N)$ inbetween two compact manifolds without boundary, $N$ having non-positive curvature (the analog work with boundaries is due to Hamilton, \cite{Hamilton1975}). It turns out that finding such harmonic maps is related to the limiting behaviour of solutions to the Harmonic Map Heat Flow with initial condition $v_0$ near global time. Struwe (\cite{struwe2008variational}) was able to remove the non-negative curvature assumption on $N$ by assuming that $N$ is a 2D-surface instead (with partial results in higher dimensions, \cite{Struwe1988OnTE}). Moreover, he described the formation of both finite-time and global-time singularities in terms of concentration of energy (Lemma \ref{lemma:no concentration of energy away origin, general case} and Theorem \ref{thm:smoothness around all but finitely many points}) and he was able to show that one can extract an harmonic map from $S^2$ to $N$ near such singularities (the analog work with boundaries and a boundary data $v_0(x,t) = v_0(x)$ is due to Chang, \cite{ChangStruweAnalog}). The description of the blow-up event was further improved by Qing (\cite{Qing1995OnSO}), who showed that near a singularity, the solution decomposes asymptotically as a sum of rescaled harmonic maps and a radiation term (Theorem \ref{thm: full bubble decomposition}). Such a rescaled harmonic map is called a \textit{bubble} of energy because the sum of the energy of all the bubbles accounts for all the energy loss which occurs at the blow-up event (see (\ref{pythagorean decomposition of energy})). Further exposition on the topic of harmonic maps can be found in the textbook \cite{analysisofharmonicmaps}. 

Finite time blow-up was shown to exist in the $1$-equivariant setting by Chang, Ding and Ye (\cite{Chang1992FinitetimeBO}) and infinite-time blow-up was exhibited by Gustafson, Nakanishi and Tsai for $k  = 2$ (\cite{stability}) and Chang, Ding for $k = 1$ (\cite{ChangDing2011}). However, the precise form of the decomposition (How many bubbles ? Is the decomposition unique along all possible time sequences ? Is the decomposition continuous in time ?) remained unknown for a long time. In the $1$-equivariant setting on $B^2$, Van der Hout (\cite{VANDERHOUT2003}) was able to prove that only single-bubbling can occur. In the $\mathbb R^2$ case instead of $B^2$, it has been recently proved by Jendrej and Lawrie (\cite{jendrej}) that the bubbling happens continuously in time (in the $k$-equivariant setting) using dispersive PDEs techniques. In the general non-symmetrical case, one only has a partial result, due to Jendrej, Lawrie and Schlag (\cite{Jendrej_Lawrie_Schlag_2025}): along any sequence of times, one can find a subsequence along which bubbling happens. However, they only treat the case $v_0(x,t) = v_0(x)$ and their methods rely on the monotonicity of the energy, which is not available in our setting. Even more recently, it has been shown by Kim and Merle (\cite{kim2025classificationglobaldynamicsenergycritical}) that there is no blow-up at all when $k \geq 3$, still in the $\mathbb R^2$ case with $v_0(x,t) = v_0(x)$. This was already hinted by a previous stability result from Gustafson, Nakanishi and Tsai (\cite{stability}) for $k \geq 3$. It is possible that no blow-up occurs for the ball when $k \geq 3$, but in any case, our rigidity result covers the unknown $k = 2$ case.

The aim of this paper is to provide a self-contained proof that if $v(x,t)$ blows up in finite time, then there can be only one bubble in the multi-bubble decomposition (\ref{multi bubble decomposition along a sequence of times}), following the intersection-comparison argument from Van der Hout (\cite{VANDERHOUT2003}) and Matano (\cite{matano}). This rigidity result could be used as a starting point to show nonexistence of finite time blow-up for $k \geq 2$. 

Our main contribution to the theory is the extension of the single bubble result from \cite{VANDERHOUT2003} for $k$-equivariant boundary data $v_0(x,t)$ from $k = 1$ to $k \geq 1$. The novelty relies on the proof of Lemma \ref{lemma:geometry of h_1 < h < h_2}, which is stated in \cite{VANDERHOUT2003} without proof and is actually non-trivial. This lemma studies the topology of level sets $h_1 < h < h_2$ for solutions $h, h_i$ of (\ref{heat map formulation for h(r,t)}) and shows that these level sets can be connected to the parabolic boundary of the square. The main difficulty follows from the fact that the boundary of a connected component $\Gamma$ of such a level set is not necessarily connected and both $h = h_1$ or $h = h_2$ could happen. To circumvent this problem, we rely on the notion of accessible boundary points from \cite{newman1964elements}, which are dense in the boundary. Even though the boundary is not connected, most boundary points can be reached from inside the domain. Hence, they can be connected together through some path inside the domain and this is the key element used to prove Lemma \ref{lemma:geometry of h_1 < h < h_2}.

We also extend the global existence result from Chang and Ding (\cite{ChangDing2011}) from $k = 1$ with a time-independent boundary data $v_0(x,t) = v_0(x)$ to $k \geq 1$ with a more general time-dependent boundary data $v_0(x,t)$ (Proposition \ref{prop:criterion for global solution}). We also show the existence of a solution which blows-up at infinity for any $k \geq 1$ (Proposition \ref{prop:blow up infinity global solution}), as a generalization of the result of \cite{ChangDing2011} from $k = 1$ to $k \geq 1$. The rigidity theorem from Van der Hout for $k = 1$ (\cite{VANDERHOUT2003}) relies on these results from Chang and Ding (\cite{ChangDing2011}), but they were not shown to be true in the new context of a time-dependent boundary data. Hence, we fill the gap and show that they remain true in our setting.

In addition, we strive to give a complete and modern presentation of classical $L^p$, parabolic Sobolev and Schauder estimates for which the main reference is the book from Ladyzhenskaia, Solonnikov, and Ural’tseva (\cite{ladyzhenskaia1968linear}) where some proofs (mainly Theorem \ref{thm:parabolic sobolev embedding}) are omitted. The standard references (e.g. \cite{krylov} for Schauder estimates) on that matter also assume some a priori regularity (e.g. a priori $C^{2+\alpha}$ regularity for Schauder estimates) of the solution, which we weaken so that, as expected, we obtain the higher regularity for ``free'' when dealing with parabolic equations. We also prove that, under the finite energy assumption, the concentration of energy argument from Struwe (\cite{struwe2008variational}), which is essentially local but was done on a manifold without boundary, works in the interior of the ball (which has a boundary) without imposing any boundary condition and with the absence of monotonicity for the energy. This is the content of Theorem \ref{thm:smoothness around all but finitely many points} which is to contrast with the analog work from Chang (\cite{ChangStruweAnalog}), where he assumes a time-independent boundary data $v_0(x,t) = v_0(x)$. We then prove that under the $k$-equivariant boundary data assumptions (which implies finiteness of the energy), there is no singularity at the boundary and interior singularities can only occur at the origin (Theorem \ref{thm:smoothness around all but finitely many points}, Proposition \ref{prop:smoothness at the boundary}). The work from Qing (\cite{Qing1995OnSO}), which is also local, finally allows getting a decomposition in terms of multi-bubbles along appropriate sequences (Theorem \ref{thm: full bubble decomposition}). No reference for the bubble-tree decomposition in the context of a time-dependent boundary data $v_0(x,t)$ could be found, so we have decided to include it in this work as it requires slightly adapting Struwe's work.

The generalization of the arguments from Van der Hout (\cite{VANDERHOUT2003}) and Chang, Ding (\cite{ChangDing2011}) holds because the Comparison Principle (Theorem \ref{comparison principle}), the intersection-number argument (Theorem \ref{lemma:intersection argument with chains mod 4}) and the different characterizations of blow-up (e.g. Proposition \ref{prop:criterion for global solution}) all generalize from $k = 1$ to $k \geq 1$. This is mostly due to the fact that the $k^2$ factor in the equation (\ref{heat map formulation for h(r,t)}) for $h(r,t)$ does not change the sign of the nonlinearity, gives a better decay $h(r,t) = r^k \tilde{h}(r,t)$ near $r = 0$ and that the family of stationary solutions given by the arctangent function still exists when $k \geq 1$.

The main theorems are the following:

\begin{theorem}[Nonexistence of multi-bubbles]\label{thm: main theorem, nonexistence bubble trees}
   Let $v(x,t)$ solve (\ref{heat map flow}) with smooth and $k$-equivariant boundary data, $k \geq 1$. Assume that $v$ blows-up at time $T < +\infty$. Consider a sequence $T_n \to T^-$, as well as finitely many smooth, non-constant, $k$-equivariant harmonic maps $\{\omega_1, ..., \omega_p\}$ from $\mathbb R^2 \cup \{\infty\} \simeq S^2$ to $S^2$ and positive sequences $(R_n^{(i)})_{n \geq 0}$, $i \in \{1,...,p\}$, all converging to zero, satisfying for all $i \neq j$,
    \begin{align}
        \frac{R_n^{(i)}}{R_n^{(j)}} + \frac{R_n^{(j)}}{R_n^{(i)}}  \to +\infty, \quad n \to +\infty,
    \end{align}
    and
 \begin{align}
     v(R_n^{(i)}x,T_n)  \to \omega_i(x), \quad n \to +\infty,
 \end{align}
    pointwisely almost everywhere on $\mathbb R^2$. Then $p = 1$.
    
    In particular, there is only one bubble (i.e., harmonic map) in the multi-bubble decomposition from Theorem \ref{thm: full bubble decomposition} and Corollary \ref{rmk:recovering the bubbles}.
\end{theorem}

\begin{proof}
    The full multi-bubble decomposition from Theorem \ref{thm: full bubble decomposition} implies the weaker formulation used in the above statement in the $k$-equivariant setting (see Corollary \ref{rmk:recovering the bubbles}).  The proof of Theorem \ref{thm: main theorem, nonexistence bubble trees} can be found in Corollary \ref{cor:end of proof}.
\end{proof}

\begin{theorem}[Global existence criterion]
    Let $v(x,t)$ solve (\ref{heat map flow}) on $[0,T)$ with smooth, $k$-equivariant boundary data, $k \geq 1$. Assume that the inclination coordinate $h_0$ of $v_0$ satisfies $h_0(0,t) = 0$, $|h_0(r,t)| \leq \pi$ on $\{0,1\} \times [0,T) \cup [0,1] \times \{0\}$. 
    
    If $T < +\infty$, then $v(x,t)$ and $h(r,t)$ cannot blow-up at time $T$. In particular, if $|h_0(r,t)| \leq \pi$ on $\{0,1\} \times [0,+\infty) \cup [0,1] \times \{0\}$, then $v(x,t)$ and $h(r,t)$ are global. 
\end{theorem}

\begin{proof}
    See Proposition \ref{prop:criterion for global solution}.
\end{proof}

\begin{remark}
    Chang, Ding and Ye (\cite{Chang1992FinitetimeBO}) have shown in the 1-equivariant setting that if $h_0(r,t) = h_0(r)$, $h_0(0) = 0$ and $|h_0(1)| > \pi$, then the solution $h(r,t)$ blows-up in finite-time.
\end{remark}

\begin{theorem}[Existence of $k$-equivariant solutions blowing-up at infinity]
    Let $k \geq 1$. Assume that $\psi = r^k\tilde{\psi}(r)$, $\tilde{\psi} \in C^{\infty}([0,1])$ with $\partial^{(2n+1)}_r \tilde{\psi}(0) = 0$ for all $n \in \mathbb N_{\geq 0}$, $0 \leq \psi \leq \pi$, $\tau(\psi)(1) = 0$ and $\tau(\psi)(r) \geq 0$ for all $r \in [0,1]$, where
    $$
\tau(\psi) := \psi_{rr} + \frac{1}{r}\psi_r -k^2 \frac{\sin(2\psi)}{2r^2}.
    $$
   Further assume there exists $r^* \in (0,1]$ for which $\psi(r^*) = \pi$. For example, one can choose
$$
\psi(r) = 4 \arctan \left( r^k \right).
$$
   Then the solution $h(r,t) \in C^{\infty}([0,1] \times (0,T))$ of (\ref{heat map formulation for h(r,t)}) with boundary data $h_0(r,t) = \psi(r)$ on $[0,1] \times \{0\} \cup \{0,1\} \times [0,+\infty)$ is global, has finite energy (\ref{Energy of v in terms of h}), satisfies $0 < h < \pi$ and $h_t > 0$ on $(0,1) \times (0,T)$, as well as $h_r > 0$ on $(0,1] \times (0,T)$. Moreover, $h(r,t)$ and the corresponding $k$-equivariant solution
   $$
   v(re^{i\theta},t) = \left( e^{ik \theta} \sin h(r,t), \cos h(r,t) \right)
   $$
   of (\ref{heat map flow})  blow-up at infinity in the sense that
   $$
\limsup \limits_{t \to +\infty}||h_r(\cdot,t)||_{C^1([0,r_0])} = \limsup \limits_{t \to +\infty}||\nabla v(\cdot,t)||_{C^1(B_{r_0}(0))} = +\infty \quad \forall r_0 \in (0,1].
$$
\end{theorem}

\begin{proof}
    See Theorem \ref{thm:existence of global solution} and Proposition \ref{prop:blow up infinity global solution}.
\end{proof}

\section{Outline of the proof}

This paper aims to be self-contained. That is why the third section goes back to the problem of existence and regularity for nonlinear parabolic problems. The aim of this section is, through a fixed point argument and a Schauder iteration, to deduce the regularity that one has for 'free' on a solution to the harmonic map heat flow. It is proved in Corollary \ref{cor: existence of solution for heat map flow} that a solution of (\ref{heat map flow}) has regularity
$$
    v(t,x) \in C^0([0,T_{\max}),C^1(\overline{B^2})) \cap C^{\infty}((0,T_{\max}) \times \overline{B^2}) 
$$
and that $\partial_t v, \partial_t \nabla_x v, \Delta v$ have additional Hölder-regularity at $t = 0$ (Proposition \ref{prop:holder continuity partial_t v}) even though we do not assume any kind of $m$-th order compatibility conditions (such as \cite[Chapter 7.1.4, Theorem 7]{evans10}). This is not trivial for nonlinear equations and those results only rely on the smoothness of the boundary data but not on its symmetry. 

In the fourth section, we add the assumptions of finite energy and symmetry of the boundary data. Under those assumptions, one obtains additional smoothness on $\overline{B^2} \setminus \{0\} \times \{T\}$ (Proposition \ref{prop:smoothness at the boundary}), thanks to the concentration of energy argument from Struwe (\cite{struwe2008variational}), as well as existence and rigidity for the limits $\lim \limits_{x \to 0} v(x,T), \lim \limits_{r \to 0^+} h(r,T)$ (Proposition \ref{value of v and h at origin}). Using the so-called ``barrier argument'', we also refine the description of blow-up. If $h(0,t) = 0$ for all $t \in [0,T)$, we prove that it cannot be that the boundary data $h_0$ is bounded by $\pi$ (Proposition \ref{prop:criterion for global solution}, which generalizes the result from \cite{ChangDing2011} to $k \geq 2$ and time-dependent boundary data) and one has 
$$
    \limsup_{\substack{r \to 0^+ \\ t \to T^-}}|h(r,t)| = \pi.
$$
As a bonus, we also generalize the other result from \cite{ChangDing2011} and show that there exists global solutions which blow-up at infinity for any $k \geq 1$.

In the fifth section, we recall the maximum principle and a comparison principle for the harmonic map heat flow problem. 

In the sixth section, following the lap-number argument from \cite{VANDERHOUT2003} and \cite{matano}, we are concerned on the level sets $h_1 < h_2$, as well as $h_1 < h < h_2$ for solutions $h, h_i$ of (\ref{heat map formulation for h(r,t)}). We show that those levels sets are connected to the parabolic boundary (Lemma \ref{geometry of h_1 < h_2} and Lemma \ref{lemma:geometry of h_1 < h < h_2}). Moreover, at a fixed time, there cannot be too many successive intersections between $h$ and $h_i$ on $[0,1]$ and this intersection number cannot increase over time (Lemma \ref{lemma:intersection argument with chains mod 4}). These topological results play a key role in the single bubble result, as one can obtain barriers from them restricting the development of bubbles.

Armed with this a priori knowledge on a solution which blows-up, we are prepared to prove the single bubble result in the seventh section. This knowledge, together with the comparison principle, imply (up to replacing $h$ by $-h$) that 
    $$
    \limsup_{\substack{r \to 0^+ \\ t \to T^-}} h(r,t) = \pi, \quad \lim \limits_{r \to 0^+} h(r,T) \in \{0,\pi\}.
    $$
    The difficult part is proving that $\lim \limits_{r \to 0^+} h(r,T) = \pi$, which follows by studying chains satisfying the property $P(t,h,h_1 = \pi - 2\arctan(\alpha^kr^k),h_2= \pi,M)$ from Lemma \ref{lemma:intersection argument with chains mod 4} when $t$ and $\alpha$ are large enough. A similar argument shows that the scale $(R_n)_{n \in \mathbb N}$ of any bubble lies on the left of the quantity 
  $$
 r^+(t) := \sup\{r \in [0,1]: h(s,t) \leq \pi \quad \forall s \in [0,r]\}
  $$
  for all $t$ large enough. After proving $\lim \limits_{r \to 0^+} h(r,T) = \pi$, one can assume without loss of generality that $h(r,T) > \pi/2$ on $(0,1]$ and $h(1,t) > \pi/2$ for $t \in [t_0,T]$. Another use of these chains and the Maximum Principle allows to exclude multi-bubbling.

Finally, the appendices build on top of each other and provide the aforementioned modern presentation of classical $L^p$, parabolic Sobolev and Schauder estimates, as well as the concentration of energy argument from Struwe.

\section{Local Existence and regularity} 
We first consider general nonlinear heat equations and prove a local existence and uniqueness result, before deducing existence of solutions for (\ref{heat map flow}). The goal of this section is to obtain a maximum of regularity for free on the solutions of the harmonic map heat flow, which is not trivial because the equation is nonlinear and it is not expected that the boundary data and the nonlinear term satisfy the usual compatibility conditions (e.g. \cite[Chapter 7.1.4, Theorem 7]{evans10}) required to get infinite differentiability.

We start with a classical fixed point argument to deduce the existence of a solution to nonlinear heat equations (Theorem \ref{thm:taylor local existence}). Combined with a Schauder iteration (Theorem \ref{thm:resgularty solution nonlinear pb}), one will deduce smoothness of the solution away from time zero. The fixed point argument is standard but we refine it to give a global existence result and a short-time estimate on the norm of the solution. This short-time estimate is used to prove that if $v$ is the solution to the harmonic map heat flow, then  $\partial_t v$, $\nabla_x \partial_t v$ and $\Delta v$ have some additional Hölder-regularity at time $t = 0$ (Proposition \ref{prop:holder continuity partial_t v}). This regularity is actually not needed for the bubble-tree result but only to show existence of a solution which blows-up at infinity (Proposition \ref{prop:blow up infinity global solution}).

Then, we turn to the particular structure of solutions under the $k$-equivariant symmetry assumption and show that the harmonic map heat flow problem reformulates as a one-dimensional nonlinear heat equation (Proposition \ref{prop:equivalent fomrulation heat map flow}).

\begin{notation}
    By $|D_x^kD_t^l u|$, we consider the Euclidean norm of the vector $(\partial_{x_1}^{k_1} ... \partial_{x_n}^{k_n}\partial_t^l u)_{k_1,...,k_n}$ for any indices $k_1 + ... + k_n = k$.

    For $x,t \in \mathbb R^n \times \mathbb R$ and $r > 0$, $C(x,t;r)$ denotes the open cylinder
    $$
    C(x,t;r) := \{(y,s) \in \mathbb R^n \times \mathbb R: |x-y| < r, t-r^2 < s < t\}.
    $$

    For $\alpha \in (0,1)$ and $u(x,t)$, $[u]_{\alpha,\alpha/2;U}$ denotes the Hölder semi-norm
\begin{align}
        \sup_{(x,t) \neq (y,s) \in U} \frac{|u(x,t)-u(y,s)|}{\left(|x-y| + |s-t|^{1/2}\right)^{\alpha}} &\sim \sup_{(x,t) \neq (y,t) \in U} \frac{|u(x,t)-u(y,t)|}{|x-y|^{\alpha}} \notag \\
    &+ \sup_{(y,t) \neq (y,s) \in U} \frac{|u(y,t)-u(y,s)|}{|s-t|^{\alpha/2}} \label{equivalence Holder norm}
\end{align}
and for $k,l \in \mathbb N_{\geq 0}$, $||u||_{k+\alpha,l+\alpha/2;U}$ denotes the norm
\begin{align}
    ||u||_{k+\alpha,l+\alpha/2;U} &= ||u||_{L^{\infty}_{x,t}(U)}+\sum_{i=1}^k ||D^i_x u||_{L^{\infty}_{x,t}(U)} + \sum_{j=1}^l  ||\partial^j_t u||_{L^{\infty}_{x,t}(U)} \notag \\
&+ [D^k_x u]_{\alpha,\alpha/2;U} + [\partial_t^l u]_{\alpha,\alpha/2;U}. \label{eq:holder norm}
\end{align}
Let $U \subset \mathbb R^n$ be any set (not necessarily open). When we write $u \in C^{k+\alpha,l+\alpha/2}(U)$ for $\alpha \in (0,1)$, we assume the norm (\ref{eq:holder norm}) to be finite. When we write $u \in C^{k,l}(U)$, we do not assume boundedness of $u$ and its derivatives, only continuity on $U$ (up to points $U \cap \partial U$). When we write $C_{loc}^{k+\alpha,l+\alpha/2}(U)$, it is to precise that we only assume (Hölder-)continuity on compact subsets $K \subset U$. Finally, $C^{k+\alpha,l+\alpha/2}_0(U)$ is the space of  $C^{k+\alpha,l+\alpha/2}(U)$ functions which vanishes at the boundary $\partial U$.
\end{notation}

\begin{theorem}[Local Existence of Solutions for nonlinear heat equation]\label{thm:taylor local existence}
    Consider the following general nonlinear heat equation 
\begin{align}
     V_t(x,t) &= \Delta V(x,t) + F(x,t,V(x,t),\nabla_x V(x,t)) , \quad (x,t) \in \Omega \times  (0,+\infty), \label{general nonlinear heat equation} \\
    V(x,t) &= 0, \quad (x,t) \in \partial \Omega \times [0,+\infty), \notag \\
    V(x,0) &= V_0(x), \quad x \in \Omega, \notag
\end{align}
where $\Omega \subset \mathbb R^n$ is a bounded, open, connected set with smooth boundary, $V: \overline{\Omega} \rightarrow \mathbb R^m$, the nonlinearity $F \in C^{0}(\overline{\Omega} \times \mathbb R_{\geq 0} \times \mathbb R^m \times \mathbb R^{n \times m};\mathbb R^m)$ is continuous, locally Lipschitz with respect to $(P,Q)$, i.e., 
    $$
    ||F(x,t,P_1,Q_1) - F(x,t,P_2,Q_2) ||_{L^{\infty}(\Omega)} \leq ||F||_{Lip,K} \cdot \left(|P_1-P_2| + |Q_1-Q_2|\right) 
    $$
    on any compact sets $K \subset \overline{\Omega} \times \mathbb R \times \mathbb R^m \times \mathbb R^{n \times m}$,
    and the initial data is $$
V_0 \in C^1_0(\overline{\Omega}) = \{f \in C^1(\overline{\Omega};\mathbb R^m): f_{|\partial \Omega} = 0\}.
$$
Then there exists a unique local solution
\begin{equation*}
    V(t,x) \in C^0([0,T_{\max}),C^1_0(\overline{\Omega}))
\end{equation*}
solving the fixed point problem
\begin{align}
    V(t) = e^{t\Delta} V_0 + \int_0^t e^{(t-s)\Delta} F(\cdot,s,V(\cdot,s),\nabla_x V(\cdot,s)) ds. \label{integral problem}
\end{align}
In fact, for all $0 < \alpha < 1$, the solution always has additional regularity
$$
V(x,t) \in C^{1+\alpha,\alpha/2}_{loc}(\Omega \times [0,T_{\max}) )
$$
(meaning it is Hölder-continuous on $K \times [0,T']$ for $T' < T$ and compact subsets $K \subset \Omega$) and if $V_0 = 0$, then one has regularity up to the boundary
     $$
     V(x,t) \in C^{1+\alpha,\alpha/2}_{loc}(\overline{\Omega} \times [0,T_{\max}) ).
     $$
      Moreover, if $T_{max} < +\infty$, then  
     $$
     \limsup \limits_{t \to T_{\max}}||V(\cdot,t)||_{C^1(\overline{\Omega})} = +\infty.
     $$
    Finally, if $F$ is globally bounded and $F(x,t,P,Q)$ is globally Lipschitz with respect to $(P,Q)$, i.e., 
    $$
    ||F(x,t,P_1,Q_1) - F(x,t,P_2,Q_2) ||_{L^{\infty}(\Omega)} \leq ||F||_{Lip} \cdot \left(|P_1-P_2| + |Q_1-Q_2|\right) 
    $$
    for some $||F||_{Lip}  < +\infty$, then the solution is global. Moreover, one has the short time estimate 
    $$
     \sup_{0 < t < T_0}||V(\cdot,t)||_{C^1(\overline{\Omega})} \leq 2 \sup \limits_{0 < t < 1}||e^{t \Delta}V_0(\cdot)-V_0(\cdot)||_{C^1(\overline{\Omega})} + 4C_{\frac{1}{2}}||V_0(\cdot)||_{C^1(\overline{\Omega})} + 1,
     $$
    where $2C_{\frac{1}{2}}T_0^{\frac{1}{2}}\max\{||F||_{\infty} + 1 , ||F||_{Lip} + 1\} = 1/2$ and $C_{\frac{1}{2}} = C_{\frac{1}{2}}(\Omega)$ is a constant depending on $\Omega$.
\end{theorem}

\begin{proof}
One can show (\cite[Chapter 8.1]{taylor2023partialvolI} and \cite[Chapter 13.7, (7.52), Proposition 7.4]{taylor2023partial}) that $e^{t\Delta}$ maps $L^{\infty}(\Omega)$ to $C^1_0(\overline{\Omega})$ for $t > 0$, is strongly continuous from $C^0_0(\overline{\Omega})$ to itself and from $C^1_0(\overline{\Omega})$ to itself. Moreover,
\begin{align*}
||e^{t \Delta} f||_{L^{\infty}(\Omega)} &\leq C_{0} ||f||_{L^{\infty}(\Omega)} \quad \forall t > 0, \\
||e^{t \Delta} f||_{C^1(\overline{\Omega})} &\leq C_{\frac{1}{2}} t^{-\frac{1}{2}} ||f||_{L^{\infty}(\Omega)} \quad \forall t > 0.
\end{align*}

Given any $\alpha > 0$, one can find $T_0 > 0$ small enough for which
$$
||e^{t\Delta} V_0(\cdot) - V_0(\cdot) ||_{C^1(\overline{\Omega})} \leq \alpha/2 \quad \forall t \in [0,T_0] 
$$
by continuity of the semigroup and then $K_1, K_2 > 0$ such that
\begin{align*}
    ||F(\cdot,t,V,\nabla V)||_{L^{\infty}(\Omega)} &\leq K_1 \quad \forall t \in [0,T_0], \\
    ||F(\cdot,t,V_1,\nabla V_1) - F(\cdot,t,V_2,\nabla V_2) ||_{L^{\infty}(\Omega)} &\leq K_2 ||V_1(\cdot,t) - V_2(\cdot,t)||_{C^1(\overline{\Omega})} \quad \forall t \in [0,T_0],
\end{align*}
for all
$$
V, V_1, V_2 \in Z(\alpha,T_0) = \{V \in C^0([0,T_0];C^1_0(\overline{\Omega})): V(x,0) = V_0(x), ||V(\cdot,t) - V_0(\cdot)||_{C^1(\overline{\Omega})} \leq \alpha\}
$$
by boundedness of $F$ and the local Lipschitz property around the closure of $\Omega \times [0,T_0] \times B_{\alpha + ||V_0||_{C^1(\overline{\Omega})}}(0) \times B_{\alpha + ||V_0||_{C^1(\overline{\Omega})}}(0)$.

Let $0 < T_1 \leq T_0$ be taken so that $2C_{\frac{1}{2}}T_1^{\frac{1}{2}}K_1 \leq \alpha/2$ and $2C_{\frac{1}{2}}T_1^{\frac{1}{2}}K_2 < 1$. Then the integral operator
$$
V \mapsto e^{t\Delta} V_0 + \int_0^t e^{(t-s)\Delta} F(\cdot,s,V(\cdot,s),\nabla_x V(\cdot,s)) ds
$$
is a contraction from $Z(\alpha,T_1)$ into itself. The Banach Fixed Point Theorem shows existence of a unique local solution $V \in Z(\alpha,T_1)$ to the corresponding integral problem (\ref{integral problem}). 

Let $T_{\max}$ be maximal with $V \in C^0([0,T_{\max}),C^1_0(\overline{\Omega}))$. We observe that $V$ has additional regularity
\begin{equation*}
    V(x,t) \in C^{1+\alpha,\alpha/2}_{loc}(\Omega \times [0,T_{\max}))
\end{equation*}
for all $0 < \alpha < 1$.

Indeed, fix any compact set $K \subset \Omega$ and any $0 < T' < T_{\max}$. We have that $V(s,y)$, $\nabla_x V(s,y)$ are bounded on $[0,T'] \times \overline{\Omega}$ thanks to our initial regularity given by the fixed point theorem. 

By continuity of $F$, $\partial_t V - \Delta V = F(y,s,V(y,s),\nabla_x V(y,s))$ is bounded as well on $[0,T'] \times \overline{\Omega}$. The parabolic Sobolev embedding (Corollary \ref{cor: local parabolic sobolev embedding}) implies that $V \in C^{1+\alpha,\alpha/2}(K \times [0,T'])$ for all $0 < \alpha < 1$. If $V_0 = 0$, one can apply Theorem \ref{thm:boundary parabolic sobolev embedding} instead to get the estimate up to the boundary.

If $T_{max} < +\infty$ is the maximal time of existence of the solution, then one must have 
$$
\limsup \limits_{t \to T_{\max}}||V(\cdot,t)||_{C^1(\overline{\Omega})} = +\infty.
$$
Otherwise, let $\alpha = \sup \limits_{0 < t < T_{\max}}||V(\cdot,t)||_{C^1(\overline{\Omega})} + 2 < +\infty$ and
$$
\tilde{K}_0 = ||F||_{L^{\infty}(\overline{\Omega} \times [0,2T_{\max}]  \times \overline{B_{\alpha}(0)} \times  \overline{B_{\alpha}(0)})}.
$$
Let $T_{\max}/2 < T^* < T_{\max}$. Observe that
\begin{align*}
||e^{t\Delta} V(\cdot,T^*)||_{C^1(\overline{\Omega})} &= ||e^{(T^*+t)\Delta} V_0(\cdot)
 + \int_0^{T^*} e^{(T^*+t-s)\Delta}F(\cdot,s,V(\cdot,s),\nabla_x V(\cdot,s))ds||_{C^1(\overline{\Omega})}  \\
 &\leq C_{\frac{1}{2}}\left(T^* + t\right)^{-\frac{1}{2}} \alpha  + \int_0^{T^*}C_{\frac{1}{2}}(T^* +t-s)^{-\frac{1}{2}} \tilde{K}_0 ds \\
 &\leq C_{\frac{1}{2}}\left(T^*\right)^{-\frac{1}{2}}\alpha + 2 \tilde{K}_0 C_{\frac{1}{2}} \left[\left(T^*+t\right)^{\frac{1}{2}} - t^{\frac{1}{2}} \right] \\
 &\leq C_{\frac{1}{2}}\left(T^*\right)^{-\frac{1}{2}}\alpha + 2 \tilde{K}_0 C_{\frac{1}{2}} \left(T^*\right)^{\frac{1}{2}} \\
 &\leq 2C_{\frac{1}{2}}\left(T_{\max}\right)^{-\frac{1}{2}}\alpha + 2 \tilde{K}_0 C_{\frac{1}{2}} \left(T_{\max}\right)^{\frac{1}{2}}, \quad t \geq 0, \\
 ||e^{t\Delta} V(\cdot,T^*) - V(\cdot,T^*) ||_{C^1(\overline{\Omega})} &\leq 2C_{\frac{1}{2}}\left(T_{\max}\right)^{-\frac{1}{2}}\alpha + 2 \tilde{K}_0 C_{\frac{1}{2}} \left(T_{\max}\right)^{\frac{1}{2}} + \alpha =: 2\beta, \quad t \geq 0.
\end{align*}
Then let 
    \begin{align*}
\tilde{K}_1 &= ||F||_{L^{\infty}(\overline{\Omega} \times [0,2T_{\max}]  \times \overline{B_{\alpha+\beta}(0)} \times  \overline{B_{\alpha+\beta}(0)})}, \\
\tilde{K}_2 &= ||F||_{Lip,\overline{\Omega} \times [0,2T_{\max}]  \times \overline{B_{\alpha+\beta}(0)} \times  \overline{B_{\alpha+\beta}(0)}},
    \end{align*}
    which are independent of $T^*$ and observe that for all $T_{\max}/2 < T^* < T_{\max}$ and $0 < t < T_{\max}$,
\begin{align*}
    ||F(\cdot,t+T^*,W,\nabla W)||_{L^{\infty}(\Omega)} &\leq \tilde{K}_1, \\
    ||F(\cdot,t+T^*,W_1,\nabla W_1) -F(\cdot,t+T^*,W_2,\nabla W_2) ||_{L^{\infty}(\Omega)} &\leq \tilde{K}_2 ||W_1(\cdot,t) - W_2(\cdot,t)||_{C^1(\overline{\Omega})},
\end{align*}
    on
\begin{align*}
    \tilde{Z}&(\beta,T^*,T_{\max}) \\
   & = \{W \in C^0([0,T_{\max}];C^1_0(\overline{\Omega})): W(x,0) = V(x,T^*), ||W(\cdot,t) - V(\cdot,T^*)||_{C^1(\overline{\Omega})} \leq \beta \}.
\end{align*}
Fix $0 < T_2 \leq T_{\max}$ so that $2C_{\frac{1}{2}}T_2^{\frac{1}{2}}\max\{\tilde{K}_1, \tilde{K}_2\} = 1/2$ and pick $T^*$ close enough to $T_{\max}$ so that $T^* + T_2 > T_{\max}$. Then solving the heat equation 
\begin{align*}
     W_t(x,t) &= \Delta W(x,t) + F(x,t+T^*,W(x,t),\nabla_x W(x,t)) , \quad (x,t) \in \Omega \times  (0,+\infty), \\
    W(x,t) &= 0, \quad (x,t) \in \partial \Omega \times [0,+\infty), \notag \\
    W(x,0) &= W(x,T^*), \quad x \in \Omega, \notag
\end{align*}
on $\tilde{Z}(\beta,T^*,T_2)$ allows extending $V$ past time $T_{\max}$, which contradicts its maximality.

Finally, if $F$ is globally bounded and globally Lipschitz, one can set
    \begin{align*}
        K_1 = ||F||_{L^{\infty}}+1, \quad K_2 = ||F||_{Lip}+1.
    \end{align*}
    Then for $2C_{\frac{1}{2}}T_0^{\frac{1}{2}}\max\{K_1, K_2\} = 1/2$ and any initial data $V_0 \in C^1_0$, set 
    $\alpha = 2\sup \limits_{0 < t < T_0}||e^{t \Delta}V_0(\cdot)-V_0(\cdot)||_{C^1(\overline{\Omega})} + 1$. The Banach Fixed Point then applies on $Z(\alpha, T_0)$. In particular, we have a minimal time of existence $T_0$ which is independent of the initial data. Hence, we can use $V(\cdot,T_0)$ as our new initial data to obtain a solution up to $2T_0$ and iterate this procedure. 
\end{proof}

\begin{remark}
    We remark that blow-up of the $C_x^1$-norm at infinity can occur if the solution is global (e.g. the global solution $\overline{h}$ constructed in Proposition \ref{prop:blow up infinity global solution}). 
\end{remark}

\begin{theorem}[Improved Regularity]\label{thm:resgularty solution nonlinear pb}
        Consider the solution $V(t,x)$ of (\ref{general nonlinear heat equation}). If the nonlinearity $F$ is smooth, then $V$ has regularity 
     $$
     V(t,x) \in C^{\infty}((0,T_{\max}) \times \Omega).
     $$
    Moreover, if $V_0 = 0$, then $V \in C^{\infty}((0,T_{\max}) \times \overline{\Omega})$ as well.
\end{theorem}

\begin{proof}
  Fix any cylinder $C(x,t;r) \subset \overline{C(x,t;r)} \subset (0,T_{\max}) \times \Omega$. As $V$ and $\nabla_x V$ are in $C^{\alpha,\alpha/2}_{x,t}(C(x,t;r))$ and $F$ is smooth, $F(y,s,V(y,s),\nabla_x V(y,s))$ is $C^{\alpha,\alpha/2}$-Hölder continuous as well on $C(x,t;r)$. The Schauder estimates (Theorem \ref{interior schauder estimates}) imply that $V(x,t) \in C^{2+\alpha,1+\alpha/2}(C(x,t;r/2))$. As $C(x,t;r)$ was arbitrary, $V(x,t) \in C^{2+\alpha,1+\alpha/2}_{loc}(\Omega \times (0,T_{\max}))$.
 In case $V_0 = 0$, we know that $V$ and $\nabla_x V$ are actually in $C^{\alpha,\alpha/2}_{loc}(\overline{\Omega} \times  [0,T_{\max}))$. The boundary Schauder estimates (Theorem \ref{boundary schauder smooth boundary}) show that $V(x,t) \in C^{2+\alpha,1+\alpha/2}_{x,t}(\Omega_{\varepsilon} \times   [T_1,T_2])$ for any $0 < T_1 < T_2 < T_{\max}$ and some $\varepsilon$-neighborhood $\Omega_{\varepsilon, T_1, T_2}$ of the boundary, hence  $V(x,t) \in C^{2+\alpha,1+\alpha/2}_{loc}(\overline{\Omega} \times  (0,T_{\max}))$.

    Finally, we do the so-called Schauder iteration to deduce higher regularity. Observe that $w = \partial_{x_k} V \in  C^{1+\alpha,\alpha/2}_{x,t}$ solves
    $$
    \partial_t w = \Delta w + \partial_{x_k} \left[ F(x,t,V,\nabla_x V) \right],
    $$
    where the forcing term is Hölder-continuous (either on a cylinder $C(x,t;r)$ or on $\overline{\Omega} \times [T_1,T_2]$ if $V_0 = 0$) by smoothness of $F$ and Hölder-continuity of $V, \nabla_x V$. Schauder estimates imply that $w, \nabla_x V \in  C^{2+\alpha,1+\alpha/2}_{loc}(\Omega \times (0,T_{\max}))$ and $w, \nabla_x V \in  C^{2+\alpha,1+\alpha/2}_{loc}(\overline{\Omega} \times (0,T_{\max}))$ if $V_0 = 0$.
    
    Next, $\tilde{w} = \partial_t V \in  C^{1+\alpha,\alpha/2}$ solves
    $$
    \partial_t \tilde{w} = \Delta \tilde{w} + \partial_{t} \left[ F(x,t,V,\nabla_x V) \right],
    $$
    and we can apply Schauder estimates again to deduce that $\partial_t V \in  C^{2+\alpha,1+\alpha/2}$. We carry on this procedure to obtain smoothness of $V$.
\end{proof}

\begin{corollary}[Existence of solutions for the Harmonic Map Heat flow]\label{cor: existence of solution for heat map flow}
For the harmonic map heat flow with smooth (not necessarily $k$-equivariant) boundary data, one can rewrite (\ref{heat map flow}) as a Dirichlet problem via $v(x,t) = v_0(x,t) + \tilde{v}(x,t)$, where $\tilde{v}(x,t)$ must solve
\begin{align}
     \tilde{v}_t &= \Delta \tilde{v} + F(x,t,\tilde{v},\nabla \tilde{v}) , \quad (x,t) \in B^2 \times  (0,+\infty), \label{heat map flow, dirichlet bound condition} \\
    \tilde{v} &= 0, \quad (x,t) \in \partial B^2 \times [0,+\infty) \cup B^2 \times \{0\}, \notag
\end{align}
and $F$ is a smooth nonlinearity
\begin{equation}
    F(x,t,\tilde{v},\nabla \tilde{v}) = - \partial_t v_0 + \Delta v_0 + |\nabla v_0|^2 (v_0 + \tilde{v}) + |\nabla \tilde{v} |^2 (v_0 + \tilde{v}) + 2 \nabla \tilde{v} \cdot \nabla v_0 (v_0 + \tilde{v}), \label{nonlinearity heat map flow}
\end{equation}
where $\nabla \tilde{v} \cdot \nabla v_0$ designates the dot product of the two (transposed) Jacobian matrices. We get a unique local solution of (\ref{heat map flow}) of regularity 
\begin{equation}
    v(t,x) \in C^0([0,T_{\max}),C^1(\overline{B^2})) \cap C^{\infty}((0,T_{\max}) \times \overline{B^2}), \label{Taylor existence of solution}
\end{equation}
 and if $T_{max} < +\infty$ is the maximal time of the solution, then one must have 
 $$
 \lim \limits_{t \to T_{\max}}||v(\cdot,t)||_{C^1(\overline{B^2})} = +\infty.
 $$
 One also has $v(x,t) \in C^{1+\alpha,\alpha/2}_{loc}(\overline{B^2} \times [0,T_{\max}))$ for any $0 < \alpha < 1$.
\end{corollary}

 \begin{proposition}[Hölder-Continuity of $\partial_t v$, $\Delta v$ at time zero]\label{prop:holder continuity partial_t v}
    Let
$$
     v(t,x) \in C^0([0,T_{\max}),C^1(\overline{B^2})) \cap C^{\infty}((0,T_{\max}) \times \overline{B^2}) \cap C_{loc}^{\alpha/2,1+\alpha}([0,T_{\max}) \times \overline{B^2})
$$
solve (\ref{heat map flow}) with smooth boundary data. Then $\partial_t v(x,t) \in C_{loc}^{1+\alpha,\alpha/2}(\overline{B^2} \times [0,T_{\max}))$, as well as $\Delta v(x,t) \in C_{loc}^{\alpha,\alpha/2}(\overline{B^2} \times [0,T_{\max}))$.
\end{proposition}

\begin{proof}
We take inspiration from \cite[Chapter V, Theorem 5.1]{ladyzhenskaia1968linear} by studying difference quotients. 

Write $T_{\max} = T$, $v = v_0 + \tilde{v}$ as in (\ref{heat map flow, dirichlet bound condition}). The nonlinearity can be rewritten as
$$
F(x,t,\tilde{v},\nabla \tilde{v}) = A(x,t) + B(x,t) \tilde{v} + \sum_{i,j} C_{i,j}(x,t) \partial_{x_i} \tilde{v}_j(x,t) \tilde{v} + \sum_{i,j}D_{i,j}(x,t) \left( \partial_{x_i} \tilde{v}_j(x,t) \right)^2 \tilde{v},
$$
where the coefficients are $C^{\infty}(\overline{B^2} \times [0,+\infty))$. Let $0 < h < T/4$ be fixed and $0 < t < T/2$. Consider the difference quotient $\tilde{v}^h(x,t) = h^{-1}[\tilde{v}(x,t+h)-\tilde{v}(x,t)]$. The goal is to bound the norm of this difference uniformly in $h$ in some parabolic Hölder space.

Note that 
$$A^h(x,t) = h^{-1} \left( A(x,t+h) - A(x,t) \right) = \int_0^1 (\partial_s A)(x,t+sh)ds \in C^{0}(\overline{B^2}\times [0,T/2])$$ with norm independent of $0 < h < T/4$. Similarly, one observes that
\begin{align*}
    &h^{-1} \left[ D_{i,j}(x,t+h) \left( \partial_{x_i} \tilde{v}_j(x,t+h) \right)^2 \tilde{v}(x,t+h) -  D_{i,j}(x,t) \left( \partial_{x_i} \tilde{v}_j(x,t) \right)^2 \tilde{v}(x,t)  \right] \\
    &=  h^{-1} \left[ D_{i,j}(x,t+h) \left( \partial_{x_i} \tilde{v}_j(x,t+h) \right)^2 \tilde{v}(x,t+h) - D_{i,j}(x,t+h) \left( \partial_{x_i} \tilde{v}_j(x,t+h) \right)^2 \tilde{v}(x,t) \right. \\
    &\phantom{=}  \left. +  D_{i,j}(x,t+h) \left( \partial_{x_i} \tilde{v}_j(x,t+h) \right)^2 \tilde{v}(x,t)  - D_{i,j}(x,t) \left( \partial_{x_i} \tilde{v}_j(x,t+h) \right)^2 \tilde{v}(x,t) \right. \\
    &\phantom{=} \left. + D_{i,j}(x,t)\left( \partial_{x_i} \tilde{v}_j(x,t+h) \right)^2 \tilde{v}(x,t) - D_{i,j}(x,t) \left( \partial_{x_i} \tilde{v}_j(x,t) \right)^2 \tilde{v}(x,t)  \right] \\
    &= D_{i,j}(x,t+h) \left( \partial_{x_i} \tilde{v}_j(x,t+h) \right)^2 \tilde{v}^h(x,t) + D_{i,j}^h(x,t) \left( \partial_{x_i} \tilde{v}_j(x,t+h) \right)^2 \tilde{v}(x,t) \\
    &\phantom{=}+ D_{i,j}(x,t)\partial_{x_i} \tilde{v}_j^h(x,t+h)  \left( \partial_{x_i} \tilde{v}_j(x,t+h) +\partial_{x_i} \tilde{v}_j(x,t) \right) \tilde{v}(x,t) \\
    &= D_{i,j,0}(x,t) + D_{i,j,1}(x,t) \tilde{v}^h(x,t) + D_{i,j,2}(x,t)  \partial_{x_i} \tilde{v}_j^h(x,t),
\end{align*}
where the coefficients are $C^0(\overline{B^2} \times [0,T/2])$ with norm independent of $0 < h < T/4$ given the regularity $v \in C^0([0,T), C^1(\overline{B^2}))$. We proceed similarly for the difference involving the coefficients $B$ and $C$. In other words, $\tilde{v}^h(x,t)$ is the unique local solution to a linear parabolic equation
\begin{align*}
    \partial_t \tilde{v}^h &=   \Delta  \tilde{v}^h + G_{h,0}(x,t) \tilde{v}^h(x,t) + \sum_{i,j} G_{h,i,j}(x,t) \partial_{x_i} \tilde{v}^h_j(x,t), \quad x \in B^2, 0 < t < T/2, \\
    \tilde{v}^h(x,t) &= 0, \quad (x,t) \in B^2 \times \{0\} \cup \partial B^2 \times [0,T/2],
\end{align*}
where $G_{h,0}(t,x), G_{h,i,j}(t,x) \in C^0([0,T/2] \times \overline{B^2})$ are coefficients (which we can extend as constant for $t \geq T/2$) with norm independent of $0 < h < T/4$.

It follows from Theorem \ref{thm:taylor local existence} that there is a small time $0 <T_0 < T/2$, independent of $h$, for which $v^h \in C^0([0,T_0];C^1(\overline{B^2}))$ with
$$
||\tilde{v}^h(\cdot,t)||_{C^1(\overline{B^2})} \leq 1, \quad \forall 0 < t < T_0,
$$
and the parabolic Sobolev embedding (Theorem \ref{thm:boundary parabolic sobolev embedding}) then implies
$$
||\tilde{v}^h||_{C^{1+2\gamma,\gamma}(\overline{B^2} \times [0,T_0])} \leq C(v_0,\gamma)
$$
for any $0 < \gamma < 1/2$, for some constant $C(v_0,\gamma)$ which is independent of $h$. Take any sequence $h_k \to 0^+$. One has $\lim \limits_{k \to +\infty}\tilde{v}^{h_k}(x,t) \to \partial_t \tilde{v}(x,t)$ on $\overline{B^2} \times (0,T_0]$ using our a priori smoothness away from $t = 0$. Moreover, by Arzela-Ascoli, the convergence is uniform on $\overline{B^2} \times [0,T_0]$, meaning that the limit $\partial_t \tilde{v}(x,t)$ extends continuously to $t = 0$. Similarly, $\partial_t \nabla_x \tilde{v}$ extends continuously at $t = 0$. Moreover, the limits $\partial_t v$ and $\partial_t \nabla_x v$ must be $C^{2\gamma, \gamma}$ Hölder-continuous as well. 

As $v \in C_{loc}^{1+\alpha,\alpha/2}([0,T_{\max}) \times \overline{B^2})$, it follows from the equation that $\Delta v \in  C_{loc}^{\alpha,\alpha/2}([0,T_{\max}) \times \overline{B^2})$.
\end{proof}

It turns out that if $v_0(x,t)$ is $k$-equivariant, then so is $v(x,t)$. First, we need a lemma about the structure of $k$-equivariant functions.

\begin{lemma}[Smoothness of inclination coordinate] \label{k-equivariance, smoothness of h}
     Let $I$ be any kind of interval and $R > 0$. If $v_0 \in C^{\infty}(\overline{B_R(0)} \times I, S^2)$ is smooth and $k$-equivariant in the sense that
     \begin{align}
         v_0(re^{i\theta},t) = \left( e^{ik\theta} \sin h_0(r,t), \cos h_0(r,t) \right) \in S^2 \label{k equivariant form of v_0}
     \end{align}
     for some function $h_0: [0,R] \times I \rightarrow \mathbb R$. Then one can always find some $h_0$ for which (\ref{k equivariant form of v_0}) holds and $h_0(r,t) = r^k \tilde{h}_0(r,t)$, where $\tilde{h}_0 \in C^{\infty}([0,R] \times I,\mathbb R)$ and $\partial_r^{(2n+1)}\tilde{h}_0(0,t) = 0$ for all $n \in \mathbb N_{\geq 0}$.

     Conversely, if $h_0$ satisfies the above properties, then $v_0$ defined as in (\ref{k equivariant form of v_0}) is smooth, $v_0 \in C^{\infty}(\overline{B_R(0)} \times I, S^2)$.
\end{lemma}

\begin{remark}
    In the direction where $v_0$ is smooth and $k$-equivariant, no smoothness is a priori assumed from $h_0(r,t)$. The lemma shows that $h_0$ can be replaced by a smooth representative.
    
    The representative is not unique. Letting $\theta = 0$, one observes that two continuous inclination coordinates for $v_0$ must differ by a multiple of $\pi$. Hence, we will always choose the one with $h_0(0,0) = 0$.
\end{remark}

\begin{proof}

First, assume that $v_0$ is smooth and $k$-equivariant.

Using a rescaling of the space variable, one can assume that $B_R(0) = B_1(0) = B^2$. Write $v_0 = (v_{0,1}, v_{0,2}, v_{0,3}) \in S^2$.

    It suffices to verify that $F(r,t) = \sin h_0(r,t)$ is smooth, $F(r,t) = r^k \tilde{F}(r,t)$ with $\tilde{F}$ smooth, as well as $\partial_r^{(2n+1)}\tilde{F}_0(0,t) = 0$ for all $n \in \mathbb N$.  

    If this is proved, then it is a classical result from topology that one can lift the smooth map $\tilde{v}_0(r,t) = (\sin h_0(r,t), \cos h_0(r,t)) = (v_{0,1}(re^{i0},t), v_{0,3}(re^{i0},t)) \in S^1$ to a continuous map $h_0(r,t): [0,R] \times I \rightarrow \mathbb R$ satisfying $\tilde{v}_0 = \pi \circ h_0$, where $\pi(\theta) = (\sin \theta, \cos \theta)$. Then $h_0$ is smooth around any points where $F(r,t) = \sin h_0(r,t) \neq \pm 1$ by the inverse function theorem, and the expansion of $h_0$ at $(0,t)$ follows from the ones of $F$ and $\arcsin(\cdot)$. At points where $F(r,t) = \pm 1$, one has $\cos h_0(r,t) = 0$ and the smoothness of $h_0(r,t)$ is obtained from the smoothness of $G(r,t) = \cos h_0(r,t) = v_{0,3}$.

    Since 
    $$
\underbrace{(v_{0,1}(x,y,t),v_{0,2}(x,y,t))}_{=: \phi(x,y,t)} = F(r,t)e^{ik\theta} = F\left(\sqrt{x^2+y^2},t\right) \frac{(x,y)}{\sqrt{x^2+y^2}}
    $$
    is smooth on $\overline{B^2} \times I$, setting $x = r$, $y = 0$ proves the smoothness of $F(r,t)$ on $[0,R] \times I$. Moreover, considering $\phi(1/n,0,t)$ and $\phi(0,1/n,t)$ with $n \to +\infty$ shows that $F(0,t) = 0$. In the following, we fix $t \in I$ and omit it from the notation $\phi(x,y,t)$ and $F(r,t)$. 
    
    Now, we prove that $F(r) = \mathcal{O}(r^k)$ near $r = 0$. Indeed,
    $$
    \Delta_{(x,y)}^j \phi = \left( \partial_{rr} + \frac{1}{r}\partial_r - \frac{1}{r^2}\partial_{\theta\theta} \right)^j \phi = e^{ik\theta} \left( \partial_{rr} + \frac{1}{r}\partial_r - \frac{k^2}{r^2}\cdot  \right)^j F 
    $$
    is smooth on $\mathbb R^2$ for all $j \in \mathbb N$. In particular, the radial part 
    $$
    \left( \partial_{rr} + \frac{1}{r}\partial_r - \frac{k^2}{r^2}\cdot  \right)^j F(r)
    $$
    is smooth on $[0,1]$ and zero at $r = 0$.
    
    On a sum $\sum_{n=n_0}^{N}a_nr^{n-i}$, the operator $\left( \partial_{rr} + \frac{1}{r}\partial_r - \frac{k^2}{r^2}\cdot  \right)^j$  acts as
    $$
\left( \partial_{rr} + \frac{1}{r}\partial_r - \frac{k^2}{r^2}\cdot  \right)^j \sum_{n=n_0}^{N}a_nr^{n-i} = \sum_{n=n_0}^{N}\prod_{l=0}^{j-1}\left[(n-i-2l)^2 - k^2\right] a_n r^{n-i-2j}.
    $$    
    Write $F(r)$ with a one-sided Taylor expansion at $r = 0$, $F(r) = \sum_{n=1}^Na_nr^n + \mathcal{O}(r^{N+1})$ where $N > 2k$ is fixed and the remainder is given by the integral formula (from which it is seen that $\partial_{r}^i \mathcal{O}(r^{N+1}) = \mathcal{O}(r^{N+1-i})$ for $i \leq N$). 

    Assume $a_0, ..., a_{2j} = 0$ for $j \geq 0$ and $2j < k-1$. We prove that $a_{2j+1} = 0$ and $a_{2j+2} = 0$ as well if $2j+2 < k$. In other words, we prove by induction that $a_0 = a_1 = ... = a_{k-1} = 0$. Since
\begin{align*}
    \left( \partial_{rr} + \frac{1}{r}\partial_r - \frac{k^2}{r^2}\cdot  \right)^{j+1} F(r) &= \sum_{n=1}^{N} \prod_{l=0}^{j}\left[(n-2l)^2 - k^2\right] a_n r^{n-2(j+1)} + \mathcal{O}(r^{k+1}) \\
    &= \sum_{n=2j+1}^{N} \prod_{l=0}^{j}\left[(n-2l)^2 - k^2\right] a_n r^{n-2(j+1)} + \mathcal{O}(r^{k+1})
\end{align*}
is smooth on $[0,1]$, the first term 
$$
 \prod_{l=0}^{j}\left[(2j+1-2l)^2 - k^2\right] a_{2j+1} r^{-1}
$$
is smooth as well on $[0,1]$. Since $1 \leq 2j+1-2l < k$, the only possibility is that $a_{2j+1} = 0$. Moreover, if $2j+2 < k$, then 
\begin{align*}
    \left( \partial_{rr} + \frac{1}{r}\partial_r - \frac{k^2}{r^2}\cdot  \right)^{j+1} F(r) 
\end{align*}
is zero at $r = 0$. Hence,
$$
 \prod_{l=0}^{j}\left[(2j+2-2l)^2 - k^2\right] a_{2j+2}r^0 = 0,
$$
meaning that $a_{2j+2} = 0$ as well.

We have proved that $F(r) = r^k \tilde{F}(r)$, where $\tilde{F}(r)$ is smooth. We show that $\tilde{F}(r)$ must be even. Fix $J > 2k$ and $N > 2J$. Write $F(r) = r^k \sum_{n=0}^N a_nr^n + \mathcal{O}(r^{N+1})$. We prove that $a_n = 0$ for any odd $n \leq J$. Indeed,
\begin{align*}
    \left( \partial_{rr} + \frac{1}{r}\partial_r - \frac{k^2}{r^2}\cdot  \right)^{J+1} F(r) &= \sum_{n=0}^{N} \prod_{l=0}^{J}\left[(n+k-2l)^2 - k^2\right] a_n r^{n+k-2(J+1)} + \mathcal{O}(r^{J+1})
\end{align*}
is smooth on $[0,1]$. If $n \leq J$ is odd, then $k+n-2(J+1) \leq -1$,
$$
\prod_{l=0}^{J}\left[(n+k-2l)^2 - k^2\right] a_n r^{n+k-2(J+1)}
$$
is smooth on $[0,1]$ and $\prod_{l=0}^{J}\left[(n+k-2l)^2 - k^2\right] \neq 0$, meaning that $a_n = 0$.

Conversely, assume that $h_0 = r^k \tilde{h}_0$ is smooth with $\partial_{r}^{(2n+1)}\tilde{h}_0(0,t) = 0$ for $n \geq 0$. Let 
$$
v_0(z,t) = \left( \frac{z^k}{|z|^k} \sin h_0(|z|,t), \cos h_0(|z|,t) \right), \quad z = (x,y) \in \overline{B_R(0)} \subset \mathbb C,
$$
which is clearly smooth away from $z = 0$. It suffices to verify the $\mathbb R^2$-smoothness of the radial functions
$$
z \mapsto \frac{\sin h_0(|z|,t)}{|z|^k}, \quad z \mapsto \cos h_0(|z|,t)
$$
near $z = 0$, which is the case if and only if the odd (one-sided) $r$-derivatives of
$$
\frac{\sin h_0(r,t)}{r^k} = \frac{\sin \left( r^k\tilde{h}_0(r,t) \right)}{r^k} = \frac{\sin \left( r^k\tilde{h}_0(r,t) \right)}{r^k \tilde{h}_0(r,t)}  \tilde{h}_0(r,t) , \quad \cos h_0(r,t) = \cos \left( r^k\tilde{h}_0(r,t) \right)
$$
vanish at zero. Extend $\tilde{h}_0$ for $r < 0$ by even reflection $\tilde{h}_0(-r,t) = \tilde{h}(r,t)$ for $r < 0$. Then $\tilde{h}_0 \in C^{\infty}([-R,R] \times I,\mathbb R)$ is even with respect to $r$ and it also follows that 
$$
\frac{\sin \left( r^k\tilde{h}_0(r,t) \right)}{r^k \tilde{h}_0(r,t)}  \tilde{h}_0(r,t) , \quad  \cos \left( r^k\tilde{h}_0(r,t) \right)
$$
extend as $C^{\infty}([-R,R] \times I,\mathbb R)$ and even functions, which concludes the proof.
\end{proof}

\begin{proposition}[Solutions of harmonic map heat flow with $k$-equivariant boundary data] \label{prop:equivalent fomrulation heat map flow}
Let $h_0(r,t) = r^k \tilde{h}_0(r,t)$, where $\tilde{h}_0 \in C^{\infty}([0,1] \times [0,+\infty),\mathbb R)$ and $\partial_r^{(2n+1)}\tilde{h}_0(0,t) = 0$ for all $n \in \mathbb N_{\geq 0}$. There exists a unique solution
$$
h(r,t) = r^k \tilde{h} + h_0, \quad \tilde{h}(t,r) \in C^0([0,T),C^1([0,1])) \cap C^{\infty}((0,T) \times [0,1]),
$$
which solves 
\begin{align}
h_t &= h_{rr} + \frac{h_r}{r} - k^2 \frac{\sin(2h)}{2r^2}, \quad (r,t) \in (0,1) \times (0,+\infty), \tag{\ref{heat map formulation for h(r,t)}} \\
h(r,t) &= h_0(r,t), \quad (r,t) \in \{0,1\} \times [0,+\infty) \cup (0,1) \times \{0\}. \notag
\end{align}
Moreover, 
$$
     v(t,x) \in C^0([0,T),C^1(\overline{B^2})) \cap C^{\infty}((0,T) \times \overline{B^2}) 
    $$
defined as 
\begin{equation} \label{k-equivariant form of v(x,t)}
    v(re^{i\theta},t) = \left( e^{ik\theta} \sin h(r,t), \cos h(r,t) \right)
\end{equation}
solves (\ref{heat map flow}) with smooth and $k$-equivariant data
\begin{equation*}
    v_0(re^{i\theta},t) = \left( e^{ik\theta} \sin h_0(r,t), \cos h_0(r,t) \right).
\end{equation*}
    Conversely, solutions 
    $$
     v(t,x) \in C^0([0,T),C^1(\overline{B^2})) \cap C^{\infty}((0,T) \times \overline{B^2}) 
    $$
    to equation (\ref{heat map flow}) with smooth and $k$-equivariant boundary data $v_0$ are exactly of the form:
\begin{equation*} 
    v(re^{i\theta},t) = \left( e^{ik\theta} \sin h(r,t), \cos h(r,t) \right),
\end{equation*}
where 
$$
h(t,r) \in C^0([0,T),C^1([0,1])) \cap C^{\infty}((0,T) \times [0,1])
$$
solves (\ref{heat map formulation for h(r,t)}) with boundary data $h_0$, where $h_0$ is the inclination coordinate of the boundary data $v_0$, chosen of the form (see Lemma \ref{k-equivariance, smoothness of h}) $h_0(r,t) = r^k \tilde{h}_0(r,t)$, where $\tilde{h}_0 \in C^{\infty}([0,1] \times [0,+\infty),\mathbb R)$ and $\partial_r^{(2n+1)}\tilde{h}_0(0,t) = 0$ for all $n \in \mathbb N_{\geq 0}$.

In particular, $v$ and $h$ have the same maximal time of existence $T_{\max} \leq +\infty$.

Finally, if $T \leq +\infty$ and $m \in \mathbb N$ is such that $||h_0||_{L^{\infty}_{r,t}(\Sigma_T)} \leq m \pi$ on the parabolic boundary $\Sigma_T = \{0,1\} \times [0,T] \cup [0,1] \times \{0\}$, then $|h(r,t)| \leq m\pi$ on $[0,1] \times (0,T)$ as well.
\end{proposition}

\begin{remark}\label{remark:uniqueness of inclination coordinate}
    It is possible to solve equation (\ref{prop:equivalent fomrulation heat map flow}) even if the boundary data $h_0(r,t)$ is not of the form $h_0 = r^k \tilde{h}_0$.

     Indeed, if $h$ solves (\ref{prop:equivalent fomrulation heat map flow}) (without caring about the boundary data), then so does $h + m \pi, m \in \mathbb Z$. Hence, the equation can be solved with boundary data $h_0 = r^k \tilde{h}_0 + m\pi$ as well. This is expected as continuous inclination coordinates are unique up to a multiple of $\pi$. Hence, the multiple of $\pi$ does not change $v$ or $v_0$.
     
     However, the constant solutions $\pi/2 + m \pi$, $m \in \mathbb Z$, also solve (\ref{prop:equivalent fomrulation heat map flow}) with a boundary data which is not of the previous form $h_0 = r^k \tilde{h}_0 + m\pi$. This is consistent as $\pi/2$ corresponds to a solution
     $$
    v(x,t) = (e^{ik\theta},0) = \left( \frac{x}{|x|}, 0 \right)
     $$
     of the harmonic map heat flow for which $v(x,0) = v_0(x,0)$ is not smooth at the origin.
\end{remark}

\begin{proof}
Looking for a solution of the form $h = r^k \tilde{h} + h_0$, $\tilde{h}(r,t)$ must solve
\begin{align*}
    \tilde{h}_t &= \tilde{h}_{rr} + \frac{2k+1}{r}\tilde{h}_r - G(r^2,\tilde{h}) + H_0, \quad (r,t) \in (0,1) \times [0,+\infty), \\
    \tilde{h}(r,t) &= 0, \quad (r,t) \in \{0,1\} \times [0,+\infty) \cup (0,1) \times \{0\},
\end{align*}
where, writing $h_0 = r^k \tilde{h}_0$ as in Lemma \ref{k-equivariance, smoothness of h},
\begin{align*}
    G(r^2,\tilde{h}) &= r^{-k} \left( \frac{\sin(2h) - 2h - \sin(2h_0) + 2h_0}{2r^2} \right) = r^{-k-2} \sum_{n=1}^{\infty} \frac{(2r^k\tilde{h} + 2h_0)^{2n+1}-(2h_0)^{2n+1}}{2(2n+1)!} \\
    &= r^{-k-2} \sum_{n=1}^{\infty} \frac{2r^k\tilde{h} \sum_{j=0}^{2n} (2r^k \tilde{h})^j (2h_0)^{2n-j}}{2(2n+1)!} = \sum_{n=1}^{\infty} \frac{4^n \tilde{h} (r^2)^{kn -1} \sum_{j=0}^{2n}\tilde{h}^j \tilde{h}_0^{2n-j}}{(2n+1)!}
\end{align*}
is a smooth nonlinearity and 
$$
H_0(r,t) =\left[  -h_{0,t} + h_{0,rr} + \frac{h_{0,r}}{r} - k^2 \frac{\sin(2h_0)}{2r^2} + \frac{h_0}{r^2}\right] r^{-k} 
$$
is smooth as well given the regularity of $h_0$ obtained in Lemma \ref{k-equivariance, smoothness of h}. The PDE for $\tilde{h}(r,t)$ is a radial heat equation in dimension $2k+2$ with a smooth nonlinearity. It follows from Theorem \ref{thm:taylor local existence} that there exists a unique smooth solution
$$
\tilde{h}(t,x) \in C^0([0,T),C^1(\overline{B^2})) \cap C^{\infty}((0,T) \times \overline{B^2})
$$ for (\ref{heat map formulation for h(r,t)}). Moreover, if $T < +\infty$, then 
$$\lim \limits_{t \to T}||\tilde{h}(\cdot,t)||_{C^1(\overline{B^2})} = +\infty.$$

Such a solution must be radial. This is the case because $h$ is obtained as the limit of a Banach Fixed Point Iteration for the problem (\ref{integral problem}). Initializing the iteration with the radial initial data, one must check that at each iteration, the approximate solution is still radial. This is the case since the nonlinearity is radial and the heat semi-group maps radial functions to radial functions (one can solve the Dirichlet problem for the heat equation on $B^2$ in radial coordinates using separation of variables and Sturm-Liouville theory). Hence,
$$
\tilde{h}(t,r) \in C^0([0,T),C^1([0,1])) \cap C^{\infty}((0,T) \times [0,1])
$$
is radial and we get the same regularity and blow-up criterion for $h$.

If $|h_0| \leq m\pi$ on the parabolic boundary of $(0,1) \times (0,T)$ by smoothness of the boundary data, then the Comparison Principle (Theorem \ref{comparison principle}) shows that $|h(r,t)| \leq m \pi$ is bounded on $[0,1] \times (0,T')$ for any $0 < T' < T$ since $m\pi$ is always a finite energy solution.

If $h$ solves (\ref{heat map formulation for h(r,t)}) on $[0,T)$, letting
$$
v(re^{i\theta},t) = \left( e^{ik\theta} \sin h(r,t), \cos h(r,t) \right),
$$
one verifies directly that $v$ is smooth (Proposition \ref{k-equivariance, smoothness of h}) and solves the harmonic map heat flow problem (\ref{heat map flow}) on $[0,T)$. 

In the converse direction, one starts with a solution $v$ of (\ref{heat map flow}) on $[0,T)$. Solving (\ref{heat map formulation for h(r,t)}) with the inclination coordinate $h_0$ of the data $v_0$ yields an inclination coordinate $h$ on $[0,1] \times [0,T_{h,\max})$ which we use to form another solution $\tilde{v}$ of (\ref{heat map flow}) defined up to time $T_{h,\max}$. If $T_{h,\max} < T \leq +\infty$, then $\tilde{v}$ must blow-up at $T_{h,\max}$. Indeed, as seen above, the Comparison Principle (Theorem \ref{comparison principle}) shows that $h(r,t)$ is bounded on $[0,1] \times (0,T_{h,\max})$ by a large multiple $m \pi$. Hence, only the derivative $h_r$ blows-up at time $T_{h,\max}$. Yet,
\begin{equation} 
    |\nabla \tilde{v}|^2 = h_r^2 + \frac{k^2}{r^2} \sin(h)^2, \label{nabla v in terms of h}
\end{equation}
meaning that $\nabla \tilde{v}$ must also blow-up at time $T_{h,\max}$ as well. By uniqueness of the problem (\ref{heat map flow}), $\tilde{v} = v$ up to $T_{h,\max} < T$, meaning that $v$ blows-up before time $T$, a contradiction. Hence, $T \leq T_{h,\max}$, $h$ exists up to time $T$ and $v = \tilde{v}$ on $[0,T)$ by uniqueness.
\end{proof}

\begin{proposition}[Continuity of $h_t$ and $h_{rr}$]\label{prop:continuity of h_t}
Let
$$
     v(t,x) \in C^0([0,T),C^1(\overline{B^2})) \cap C^{\infty}((0,T) \times \overline{B^2}) 
$$
solve (\ref{heat map flow}) with smooth and $k$-equivariant boundary data. Let 
$$
h(t,r) \in C^0([0,T),C^1([0,1])) \cap C^{\infty}((0,T) \times [0,1])
$$
be the corresponding inclination coordinate (\ref{heat map formulation for h(r,t)}). Then $h_t(t,r) \in C^0([0,T) \times [0,1])$ and $h_{rr} \in C^0([0,T) \times (0,1])$.
\end{proposition}

\begin{proof}
    Continuity of $h$ and $h_r$ on $[0,T) \times [0,1]$ is already known. Observe that $v_t(t,re^{i 0}) \in C^0([0,T) \times [0,1])$ by Proposition \ref{prop:holder continuity partial_t v} and
    $$
    v_t(t,re^{i 0}) = \left( \cos(h) h_t,\cos(h) h_t, - \sin(h) h_t \right).
    $$
    As $\cos(h), \sin(h) \in C^0([0,T) \times [0,1])$, continuity of $h_t$ follows. Similarly, one has the regularity $\Delta v(t,re^{i 0}) \in C^0([0,T) \times [0,1])$ and
    $$
   \Delta v(t,re^{i 0}) = \begin{pmatrix}
       \cos(h) h_{rr} - \sin(h) h_r^2 + \cos(h) h_r r^{-1} - k^2 \sin(h)r^{-2} \\
       \cos(h) h_{rr} - \sin(h) h_r^2 + \cos(h) h_r r^{-1} - k^2 \sin(h)r^{-2} \\
       - \sin(h) h_{rr} - \cos(h) h_r^2 - \sin(h) h_r r^{-1}  
   \end{pmatrix},
    $$
    from which continuity away from $r = 0$ follows.
\end{proof}

We finish with an estimate of the blow-up rate of $h$.

\begin{proposition}[Blow-Up of $k$-equivariant solutions] \label{blow up point is at zero k equivariant}
Let $T < +\infty$ and assume
$$
     v(t,x) \in C^0([0,T),C^1(\overline{B^2})) \cap C^{\infty}((0,T) \times \overline{B^2}) 
$$
solve (\ref{heat map flow}) with smooth and $k$-equivariant boundary. Let 
$$
h(t,r) \in C^0([0,T),C^1([0,1])) \cap C^{\infty}((0,T) \times [0,1])
$$
be the corresponding inclination coordinate (\ref{heat map formulation for h(r,t)}). For all $n \in \{0,1,2\}$, there exists $C = C(k^2,n,h) > 0$ for which
\begin{align*}
    |\partial_r^{(n)} h(\lambda,t)| \leq ||h_{r}||_{L^{\infty}_{r,t}([\lambda,1] \times [T/2,T])} &\leq \frac{\tilde{C}(k^2,n,h)}{\lambda^n} \quad \forall \lambda \in (0,1], t \in [T/2,T].
\end{align*}
In particular, if $v$ blows-up at time $T$, then the blow-up of $h_r$ happens exactly at $r = 0$, i.e., 
$$
\limsup_{\substack{r \to 0^+ \\ t \to T^-}} |h_r(r,t)| = +\infty.
$$
\end{proposition}

\begin{proof}
The inclination coordinates solves a 2D radial heat equation with smooth and bounded forcing term on $[1/2,1] \times [T/2,T]$. The estimate follows from the boundary Schauder estimates (Theorem \ref{boundary schauder smooth boundary}) on that region. Hence, it suffices to prove the statement for all $\lambda$ small enough. Let $\lambda \in (0,1)$ and 
$$
h_{\lambda}(r,t) = h( \lambda r,\lambda^2 t) \in C^{\infty}( (0,\lambda^{-1}) \times (0,\lambda^{-2}T) ).
$$
It follows from the Comparison Principle (Proposition \ref{prop:equivalent fomrulation heat map flow}) that $h$, hence $h_{\lambda}$, is bounded by a large multiple $m\pi$.

Moreover, $h_{\lambda}$ solves the radial 2D nonlinear heat equation
$$
h_t = h_{rr} + \frac{h_r}{r} - k^2 \frac{\sin(2h)}{2r^2}, \quad (r,t) \in  (1/2,\lambda^{-1}) \times (0,\lambda^{-2}T)
$$
with a bounded forcing term. For any open cylinder 
$$
C(x,t;1)  = B_1(x) \times (t-1,t) \subset \{(x,t) \in \mathbb R^2 \times \mathbb R_{\geq 0}:   (|x|,t) \in (3/2,\lambda^{-1}-1) \times (0,\lambda^{-2}T)\},
$$ 
the Schauder estimates (Theorem \ref{interior schauder estimates}, Remark \ref{weakening of interior Schauder estimates}) show the existence of $\gamma \in (0,1)$ and a constant $C = C(n)$ for which 
\begin{align*}
    ||h_{\lambda}||_{C^{2+2\gamma,1+\gamma}_{x,t}(C(x,t;1/8) )} \leq C \left( ||h_{\lambda}||_{L^p(C(x,t;1))} + \bigg| \bigg|  k^2 \frac{\sin(2h_{\lambda})}{2r^2} \bigg| \bigg|_{L^p(C(x,t;1))} \right) \leq C(n,m\pi,k^2)
\end{align*}
on the cylinder $C(x,t;1/8) = B_{1/8}(x) \times (t-1/64,t)$. We deduce that
 \begin{align*}
          &||(|\partial_t h_{\lambda}| + |\partial_r h_{\lambda}| + |\partial_{rr} h_{\lambda}| )||_{L^{\infty}_{r,t}((5/2,\lambda^{-1}-2) \times (1,\lambda^{-2}T))} \\
          &\lesssim ||(|\partial_t h_{\lambda}| + |\nabla_x h_{\lambda}| + |D^2_xh_{\lambda}| )||_{L^{\infty}_{x,t}(\{5/2 \leq |x| \leq \lambda^{-1}-2\} \times (1,\lambda^{-2}T))} \\
          &\leq C(n, m\pi, k^2),
     \end{align*}
which proves the desired bounds on $h_r$, $h_{rr}$ on $(r,t) \in (5\lambda/2,1-2\lambda) \times (\lambda, T)$ for all $\lambda$ small enough, which is enough to conclude the proof.
\end{proof}

\section{Finite energy solutions and further blow-up characterizations}

In this section, we restrict to solutions of (\ref{heat map flow}) on $[0,T)$ with finite energy and $k$-equivariant symmetry.

Under these assumptions, one obtains additional regularity of solutions at $t = T$ (Proposition \ref{prop:smoothness at the boundary}) as well as existence and rigidity for the value of the limit $\lim \limits_{r \to 0^+} h(r,T)$ (Proposition \ref{value of v and h at origin}). As seen as a corollary of the fixed point argument in Theorem \ref{thm:taylor local existence}, the usual blow-up criterion is the blow-up of the $C^1$-norm of the solution. In this section, we refine the description of blow-up using the so-called ``barrier argument''. Indeed, a blow-up event implies that the limit 
$$
\lim \limits_{\substack{r \to 0^+ \\ t \to T^{-}}} h(r,t)
$$
does not exist, as one can extract at least one non-constant harmonic map along sequences $r_n \to 0^+, t_n \to T^{-}$, as shown in the blow-up description from Struwe (Theorem \ref{thm:smoothness around all but finitely many points}). The barrier argument consists of finding an upper bound $|h(r,t)| \leq g(r,t)$ which is continuous at $(r,t) = (0,T)$ with $g(0,T) = 0$, which leads to a contradiction. This majorant is usually a stationary solution $2\arctan(\alpha^kr^k)$ and it suffices to show the upper bound on the parabolic boundary of $[0,1] \times [0,T]$ thanks to the Comparison Principle (Theorem \ref{comparison principle}).

Under the aforementioned assumptions, if $h(0,t) = 0$ for all $t \in [0,T)$, we prove that it cannot be that the boundary data $h_0$ is bounded by $\pi$ (Proposition \ref{prop:criterion for global solution}, which generalizes the result from \cite{ChangDing2011} to $k \geq 2$ and time-dependent boundary data) and that 
$$
    \limsup_{\substack{r \to 0^+ \\ t \to T^-}}|h(r,t)| = \pi.
$$
Moreover, we also generalize the other result from \cite{ChangDing2011} and show that there are global solutions which blow-up at infinity for any $k \geq 2$.
 
 \begin{definition}[Finite energy solutions]
One says that $v$ has finite energy on $(0,T)$ if 
\begin{align*}
   \sup_{0 < t < T} E(v(\cdot,t)) < + \infty.
\end{align*}
From (\ref{nabla v in terms of h}), it follows that
\begin{align}
    E(v(\cdot,t)) :=  E(v(t)) := \frac{1}{2}\int_{B^2} |\nabla v|^2 dx = \pi \int_0^1 \left( h_r^2 + \frac{k^2}{r^2} \sin(h)^2 \right) r dr = E(h(\cdot,t)). \label{Energy of v in terms of h}
\end{align}
\end{definition}

\begin{remark}[Equivariant solutions have finite energy]\label{k-equivariance implies finite energy}
    If $v_0(x,t) = v_0(x)$ is a time-independent boundary data and $v_0$ has finite energy, then the energy of the solution $v$ is non-increasing in time. Hence, $v$ has finite energy.

    In the time-dependent case, we lose monotonicity. However, if $v_0(x,t)$ is $k$-equivariant, then $v(x,t)$ has finite energy as well even though it might not be monotone in time. Indeed, if $v(x,t)$ is defined on $[0,T)$, then
\begin{align*}
    \frac{d}{dt}E(h(\cdot,t)) &= \pi \int_0^1 \left( 2 h_r h_{rt} + 2 \frac{k^2}{r^2} \sin(h) \cos(h) h_t \right) rdr  \\
     &= 2\pi [h_rh_t r]_{r=0}^{r=1} - 2\pi \int_0^1 h_t \left(- h_{rr}r -  h_r  +  \frac{k^2}{r} \sin(h) \cos(h) \right) dr  \\
     &= 2\pi h_r(1,t)h_t(1,t) - 2\pi \int_0^1 h_t(r,t)^2 rdr \\
     &= 2\pi h_r(1,t)\partial_t h_0(1,t) - 2\pi \int_0^1 h_t(r,t)^2 rdr \\
     &\leq 2\pi h_r(1,t)\partial_t h_0(1,t), \quad t \in [0,T),
\end{align*}
 and we know from Schauder estimates (see Proposition \ref{blow up point is at zero k equivariant}) that $h_r(1,t)$ is bounded on $[0,T)$ under the $k$-equivariance symmetry. Hence,
 $$
\sup_{t \in [0,T)} E(h(\cdot,t)) \leq E(h_0) + C(k,h,T)
 $$
 after integrating with respect to time.
\end{remark}

 First, observe that the symmetry assumption implies some further smoothness at $t = T$. In particular, from now on, it makes sense to consider and write $v(T,x)$ and $h(T,r)$ away from zero.

 {
\begin{proposition}[Smoothness at $t = T$] \label{prop:smoothness at the boundary, earlier statement}
Let
$$
     v(t,x) \in C^0([0,T),C^1(\overline{B^2})) \cap C^{\infty}((0,T) \times \overline{B^2}) 
$$
solve (\ref{heat map flow}) with smooth and $k$-equivariant boundary data. Let 
$$
h(t,r) \in C^0([0,T),C^1([0,1])) \cap C^{\infty}((0,T) \times [0,1])
$$
be the corresponding inclination coordinate (\ref{heat map formulation for h(r,t)}). Assume that $v$ blows-up at time $T < +\infty$.

Then 
$$
v(t,x) \in C^{\infty}((0,T] \times \overline{B^2} \setminus \{0\} ), \quad h(t,r) \in  C^{\infty}((0,T] \times (0,1] ).
$$
\end{proposition}

\begin{proof}
    See Proposition \ref{prop:smoothness at the boundary} in Appendix D. This is a consequence of the energy concentration argument from Struwe (\cite{struwe2008variational}).
\end{proof}

Even though we have an estimate of the blow-up rate, it turns out that the limit $\lim \limits_{r \to 0^+} h(r,T)$ exists and is a multiple of $\pi$. 

\begin{proposition}[Value at the origin]\label{value of v and h at origin}
    Let $v(x,t)$ solve (\ref{heat map flow}) with smooth and $k$-equivariant boundary data. Assume that $v$ blows-up at time $T < +\infty$. Then for all $t \in [0,T]$, one has
    $$
    \lim_{r \to 0^+} h(r,t) = k(t) \pi, \quad k(t) \in \mathbb Z.
    $$
    In particular, $k(t) = c$ is constant on $[0,T)$.
\end{proposition}

\begin{remark}
    It was already known from Proposition \ref{prop:equivalent fomrulation heat map flow} and the uniqueness of a continuous inclination coordinate (up to a multiple of $\pi$) that $h(0,t) = m \pi$ is constant on $[0,T)$. The interesting part is that the limit $\lim \limits_{r \to 0^+} h(r,T)$ also exists and is a multiple of $\pi$. Moreover, there might be a jump at time $T$, i.e.,  $k(T) \neq c$.
\end{remark}

\begin{proof}
    We note that $h(r,t)$ is continuous on $[0,1] \times [0,T) \cap (0,1] \times [0,T]$ (Proposition \ref{prop:equivalent fomrulation heat map flow} and Proposition \ref{prop:smoothness at the boundary, earlier statement} for the additional regularity at $t = T$) and, for fixed $t \in [0,T]$, belongs to $L^2([0,1],rdr)$ since the energy is finite. The existence of the limit and its value follow from the following two lemmas. The fact that $k(t)$ is constant on $[0,T)$ follows from the continuity of $h$ on $[0,1] \times [0,T)$.
\end{proof}

\begin{lemma}\label{value at r = 0 is a multiple of pi}
Assume that $\theta(r) \in L^2([0,R],rdr)$ with a weak derivative $\theta'(r) \in L^2([0,R],rdr)$. Moreover, assume there exists $k \in \mathbb Z$ for which 
$$
\frac{\theta-k\pi}{r} \in L^2([0,R],rdr)
$$
as well. Then $\theta \in C^0([0,R])$ and $\lim \limits_{r \to 0}\theta(r) = k \pi$.
\end{lemma}

\begin{proof}
    For any compact interval $I \subset (0,R]$, one has $\theta, \theta' \in L^2(I)$. Hence $\theta \in W^{1,2}(I) \subset C^{0}(I)$.

    Then $(\theta-k\pi)^2 \in W^{1,2}(I)$ as well and $\partial_r \left[ (\theta-k\pi)^2 \right] = 2(\theta-k\pi) \theta'$ using the product rule.

    Since
    $$
2(\theta-k\pi) \theta' = \frac{\theta-k\pi}{r}r^{1/2} \cdot \theta' r^{1/2} \in L^1([0,R]),
    $$
    one actually has $(\theta-k\pi)^2 \in W^{1,1}([0,R]) \subset C^0([0,R])$ and $(\theta-k\pi)^2$ is continuous at $r = 0$.

    Finally, since
    $$
    \int_0^R \frac{(\theta(r)-k\pi)^2}{r} dr < +\infty, 
    $$
    one must have $\lim \limits_{r \to 0} (\theta - k\pi)^2 = 0$, which concludes the proof.
\end{proof}

\begin{lemma}\label{value at r = 0 is a multiple of pi, ver.2}
    Assume that $\theta(r) \in L^2([0,R],rdr)$ with a weak derivative $\theta'(r) \in L^2([0,R],rdr)$. Then 
    $$
    \exists k \in \mathbb Z: \frac{\theta-k\pi}{r} \in L^2([0,R],rdr) \iff \frac{\sin \theta}{r} \in  L^2([0,R],rdr).
    $$
\end{lemma}

\begin{proof}
    The direction $\Rightarrow$ follows from the inequality $|\sin(\theta)| = |\sin(\theta-k\pi)| \leq |\theta-k\pi|$.
    
    We prove $\Leftarrow$. For any compact interval $I \subset (0,R]$, $\theta \in W^{1,2}(I) \subset C^0(I)$. Applying the chain rule, $\cos \theta \in W^{1,2}(I)$ with 
    $$
    \partial_r \cos(\theta) = -\sin(\theta) \theta' = -\frac{\sin \theta}{r}r^{1/2} \cdot \theta' r^{1/2} \in L^1([0,R]),
    $$
    meaning that $\cos \theta \in W^{1,1}([0,R]) \subset C^0([0,R])$. Hence, $\sin(\theta)^2 = 1 - \cos(\theta)^2 \in C^0([0,R])$ and $\lim \limits_{r \to 0} \sin(\theta(r))^2 = 0$ because
    $$
    \int_0^R \frac{\sin(\theta)^2}{r} dr < +\infty. 
    $$
    Note that $\sin(\theta) \in W^{1,2}(I) \subset C^0(I)$ as well. Hence, we have proved that $\sin(\theta)$ is continuous at zero with $\lim \limits_{r \to 0} \sin(\theta(r)) = 0$. One checks, using an argument by contradiction and the Intermediate Value property of $\theta$ away from $r = 0$, that this implies $\lim \limits_{r \to 0} \theta(r) = k \pi$ for some $k \in \mathbb Z$.

    Finally, $|\sin(\theta(r))| = |\sin(\theta(r)-k\pi)| \gtrsim |\theta(r) - k\pi|$ near $r = 0$ which implies the desired integrability.
\end{proof}

Next, we present a theorem which, in a more general setting, states that solutions to nonlinear parabolic problems with analytic nonlinearity are real-analytic with respect to the space variable. This result is attributed to Friedman (\cite[Theorem 2]{friedman}), but looking at the paper, the proof is omitted. In any case, we only need a weaker statement for our purpose, taking inspiration from Friedman's techniques for the elliptic case.

\begin{theorem}[Space-analyticity]\label{thm:space analyticity}
Let $v(t,x)$ solve (\ref{heat map flow}) on $(0,T)$, $T  < +\infty$, with smooth, $k$-equivariant boundary data. Let
$$
h(t,r) \in C^0([0,T),C^1([0,1])) \cap C^{\infty}((0,T] \times (0,1]) 
$$
be the corresponding inclination coordinate (\ref{heat map formulation for h(r,t)}) chosen with $h(t,0) = h_0(t,0) = 0$.

Then $r \mapsto h(r,t)$ is real-analytic at $r_0 \in (0,1]$ for any fixed $t \in (0,T]$.
\end{theorem}

\begin{proof}
After parabolic rescaling, one can assume that $r_0 = 1$.

    Let $H = h - h(t,1)$. Then that $H(t,x) \in C^0([0,T);C^1(\overline{B^2}))$ (the regularity on $\overline{B^2}$ can be found in Proposition \ref{prop:equivalent fomrulation heat map flow} and its proof) satisfies a 2D nonlinear heat equation
    \begin{align*}
        H_t &= \Delta H  - k^2\frac{\sin(2H + 2h(t,1))}{2|x|^2} - h_t(t,1), \quad x \in \overline{B^2} \setminus \{0\} \times (0,T), \\
        H &= 0, \quad (x,t) \in \overline{B^2} \times \{0\} \cup \partial{B^2} \times (0,T).
    \end{align*}
   Fix a $\varepsilon$-neighborhood $\Omega_{\varepsilon}$ of the boundary $\partial B^2$ on which Theorem \ref{thm:boundary parabolic sobolev embedding} applies and let $G = \overline{\Omega_{\varepsilon}}$. 

    Suppose by induction that one has
    \begin{align*}
        ||H||_{L^{\infty}(G \times (0,T))} + ||\nabla H||_{L^{\infty}(D \times (0,T))} \leq H_0, \\
        ||D^q_x H||_{L^{\infty}(G \times (0,T))} \leq H_0 H_1^{q}q!, \quad 2 \leq q \leq n,
    \end{align*}
    for $n \geq 2$. Letting 
    $$
X(t,x,h) = - k^2\frac{\sin(2H + 2h(t,1))}{2r^2} - h_t(t,1),
    $$
    one has
  \begin{align*}
      |\partial^n_{x_1} X| &\leq k^2 \sum_{i=0}^n \binom{n}{i} (n-i)!C(G)^{n-i} \sum_{\sum_{j=1}^i jm_j = i} \frac{i!}{m_1! \cdot ... \cdot m_i!} \prod_{j=1}^i \left( \frac{2 |\partial^{j}_{x_1} H|}{j!} \right)^{m_j}  \\
     &\leq n!  k^2  \sum_{i=0}^n  C(G)^{n-i}\sum_{\sum_{j=1}^i jm_j = i} \frac{1}{m_1! \cdot ... \cdot m_i!} \prod_{j=1}^i \left( \frac{2 H_0 H_1^{j} j!}{j!} \right)^{m_j}  \\ 
     &\leq n!    k^2 H_1^n  \sum_{i=0}^n  C(G)^{n-i} \sum_{\sum_{j=1}^i jm_j = i} \prod_{j=1}^i \frac{(2H_0)^{m_j}}{m_j!}
  \end{align*}
 on $G \times (0,T)$ for $n \geq 2$ using the product rule and Faa di bruno's formula. The constant $C(G)$ essentially comes from the $r^{-2-n+i}$ factor when differentiating $X$. When $i = 0$, we use the convention that
 $$
\sum_{\sum_{j=1}^i jm_j = i} \frac{i!}{m_1! \cdot ... \cdot m_i!} \prod_{j=1}^i \left( \frac{2 |\partial^{j}_{x_1} H|}{j!} \right)^{m_j}  = 1.
 $$
 
 There exists $B_0 = B_0(H_0) > 0$  for which
 $$
\frac{(2H_0)^{j}}{j!} \leq B_0 \quad \forall j \in \mathbb N_{\geq 0}.
 $$
 Hence,
   \begin{align*}
      |\partial^n_{x_1} X| &\leq n!  H_1^n k^2 \sum_{i=0}^n C(G)^{n-i} B_0^i \sum_{\sum_{j=1}^i jm_j = i}1  \leq n!  k^2 H_1^n \sum_{i=0}^n   C(G)^{n-i}B_0^i \binom{2i-1}{i-1} \\
      &\leq n!   k^2 H_1^n  B_1 \sum_{i=0}^n   C(G)^{n-i} B_0^i 4^{i} \leq  n! k^2 H_1^n  B_1 C(G)^n \frac{ C(G)^{-(n+1)} B_0^{n+1} 4^{n+1} - 1}{4C(G)^{-1} B_0-1},
   \end{align*}
  where the binomial coefficient is treated as $1$ when $i = 0$ and $B_1$ is a universal constant obtained, for example, via Stirling's formula. Taking $B_0 = B_0(H_0,G)$ larger if needed, one can assume that $4C(G)^{-1}B_0 \gg 1$ and
  $$
|\partial^n_{x_1} X| \leq 2k^2 B_1 4^n H_1^n B_0^{n} n!
  $$
  on $G \times (0,T)$ for $n \geq 2$.
  
Then Theorem \ref{thm:boundary parabolic sobolev embedding} yields that
    $$
    ||\partial^n_{x_1} h||_{L^{\infty}(G \times (0,T))} \leq 2C(G)k^2 B_1 4^n H_1^n B_0^{n} n!.
    $$
    Setting
    $$
    H_0 =||h||_{L^{\infty}(G \times (0,T))} + ||\nabla h||_{L^{\infty}(G \times (0,T))}  +  ||D^2_x h||_{L^{\infty}(G \times (0,T))} + 1,
    $$
    as well as $H_1 =8 k^2C(G) B_1 B_0(H_0,G)+ 1$ as constants initializing our induction, one deduces
    $$
    ||\partial^n_{x_1} h||_{L^{\infty}(K \times (0,T))} \leq H_0H_1^{n+1}(n+1)!
    $$
    for all $n \geq 2$ and similarly for $\partial^n_{x_2}h$. As $h$ is smooth on $G \times (0,T]$, we obtain that $h(x,t)$ is real-analytic with respect to $x \in G$ for any fixed $t \in (0,T]$. In particular, $h(r,t)$ is real-analytic at $r = 1$ for any fixed $t \in (0,T]$.
\end{proof}

We come now to the so-called barrier argument. First, we show that one can construct a barrier in some subset of the parabolic boundary. 

\begin{lemma}\label{lemma:semi-barrier with global solution}
     Let $v(x,t)$ solve (\ref{heat map flow}) on $[0,T)$, $0 < T \leq +\infty$, with smooth, $k$-equivariant boundary data. Assume that the inclination coordinate $h$ is chosen so that $h(0,t) = 0$. Then for all $t_0 \in (0,T)$, there exists $\alpha_0 > 0$ and $r_0 \in (0,1)$ for which
      \begin{equation}
       \theta_{\alpha}(r) \geq |h(r,t_0)|, \quad \alpha \geq \alpha_0, \quad r \in [0,r_0], \label{eq:comparison between theta_alpha and any solution}
   \end{equation}
   where $\theta_{\alpha}(r) = 2\arctan\left( (\alpha r)^k \right)$.
\end{lemma}

\begin{proof}
    Fix an arbitrary $t_0 \in (0,T)$. As $v(x,t)$ is smooth and $k$-equivariant on $\overline{B^2} \times \{t_0\}$, one has
    $$
    h(0,t_0) = \partial_r h(0,t_0) = ... = \partial_r^{(k-1)} h(0,t_0) = 0
    $$
    by Lemma \ref{k-equivariance, smoothness of h}. Similarly, one computes
    $$
    \theta_{\alpha} (0) = \partial_r \theta_{\alpha} (0) = ... = \partial_r^{(k-1)} \theta_{\alpha} (0) = 0,
    $$
    while $\partial_r^{(k)} \theta_{\alpha}(0) = 2k!\alpha^k$. Fix $\alpha_0 \gg 1$ for which $\partial_r^{(k)} \theta_{\alpha_0}(0) > |\partial_r^{(k)} h(0,t_0)|$. By continuity, this inequality remains valid on $[0,r_0]$ for some $r_0 > 0$. In particular, for $r \in [0,r_0]$,
\begin{align*}
        \theta_{\alpha_0}(r) &= \int_0^r \int_0^{s_{k-1}} ... \int_0^{s_1} \partial_r^{(k)} \theta_{\alpha_0}(s_0) ds_0 ds_1 ... ds_{k-1} \\
        &\geq \int_0^r \int_0^{s_{k-1}} ... \int_0^{s_1} \partial_r^{(k)}h(s_0,t_0) ds_0 ds_1 ... ds_{k-1} = h(r,t_0),
\end{align*}
    and similarly, $\theta_{\alpha_0}(r) \geq -h(r,t_0)$.
    
    Moreover, if $\alpha \geq \alpha_0$, then
    $$
\theta_{\alpha}(r) \geq  \theta_{\alpha_0}(r) \geq |h(r,t_0)|, \quad r \in [0,r_0],
    $$
    by monotonicity of the $\arctan$ function. 
\end{proof}

\begin{proposition}[Further Property of Blow-up]\label{prop:further properties of blow-up of h}
    Let $v(x,t)$ solve (\ref{heat map flow}) with smooth and $k$-equivariant boundary data. Assume that $v$ blows-up at time $T < +\infty$. Then one has
    $$
    \limsup_{\substack{r \to 0^+ \\ t \to T^-}}|h(r,t)-h(0,t)| \geq \pi.
    $$
\end{proposition}

\begin{proof}
     Replacing $h$ and $h_0$ by $h - m\pi$ and $h_0 - m\pi$, where $h(0,t) = m\pi$ for $t \in [0,T)$ (Proposition \ref{value of v and h at origin}), one can assume that $h(0,t) = 0$ for $t \in [0,T)$ and one needs to prove 
     $$
    \limsup_{\substack{r \to 0^+ \\ t \to T^-}}|h(r,t)| \geq \pi.
    $$
    By the Comparison Principle (Proposition \ref{prop:equivalent fomrulation heat map flow}, Theorem \ref{comparison principle}), $h$ is bounded by a large multiple of $\pi$. Hence, the limsup exists. 

    Assume for a contradiction that 
    $$
 \limsup_{\substack{r \to 0^+ \\ t \to T^-}}|h(r,t)| = M < \pi.
    $$
    We use a barrier argument using a stationary solution to obtain a contradiction. There exists $0 < \delta < \pi - M$ and a neighborhood $(0,r_0) \times (t_0,T)$ with $r_0 < 1$, $t_0 > 0$, on which $|h(r,t)| < M + \delta$. Consider the stationary solution $\theta_{\alpha}(r) = 2 \arctan(\alpha^k r^k)$ of (\ref{heat map formulation for h(r,t)}). Up to choosing a smaller $r_0$, there is some $\alpha_0 > 0$ for which
    $$
\theta_{\alpha}(r) \geq  \theta_{\alpha_0}(r) \geq |h(r,t_0)|, \quad \alpha \geq \alpha_0, r \in [0,r_0],
    $$
    by Lemma \ref{lemma:semi-barrier with global solution}.
    
    Fix $\alpha \geq \alpha_0$ so that $\theta_{\alpha}(r_0) > M$. One has found a solution $\theta_{\alpha}$ satisfying $h(0,t) = 0 = \theta_{\alpha}(0)$ on $\{0\} \times (t_0,T)$, $|h(r_0,t)| \leq M < \theta_{\alpha}(r_0)$ on $\{r_0\} \times (t_0,T)$ and $|h(r,t_0)| \leq \theta_{\alpha}(r)$ on $[0,r_0] \times \{t_0\}$.
    
    The Comparison Principle (Theorem \ref{comparison principle}) applied with $h, \theta_{\alpha}$ and $-h, \theta_{\alpha}$ on $[0,r_0] \times [t_0,T']$ for any $T' < T$ shows that $|h(r,t)| \leq \theta_{\alpha}(r)$ on $[0,r_0] \times [t_0,T)$. In particular, $\lim \limits_{r \to 0^+} |h(r,t)| = 0$ uniformly with respect to $t \in [t_0,T)$. 
     As
    $$
      v(re^{i\theta},t) = \left( e^{ik\theta} \sin h(r,t), \cos h(r,t) \right),
    $$
    one deduces that $\lim \limits_{x \to 0} v(x,t) = (0,0,1)$ uniformly with respect to $t \in [t_0,T)$ as well.
    
    If $v(x,t)$ blows-up at time $T$, blows-up happens at $(0,T)$ (Proposition \ref{prop:smoothness at the boundary, earlier statement}). Then $v_m(x) = v(x_m + R_mx, T_m + R_m^2s)$ converges strongly in $H^{1,2}_{loc}(\mathbb R^2, S^2)$ to a non-constant harmonic map along some appropriate sequences $x_m \to 0, (x_m) \subset B^2, R_m \to 0^+, T_m \to T^-$ and some fixed $s \leq 0$ (Theorem \ref{thm:smoothness around all but finitely many points}). Up to taking a subsequence, the convergence holds pointwise almost everywhere for $x \in \mathbb R^2$. But for fixed $x \in \mathbb R^2$, $x_m + R_mx \to 0$ and since $T_m + R_m^2s \to T^-$, the uniform convergence $\lim \limits_{x \to 0} v(x,t) = (0,0,1)$ implies that $\lim \limits_{m \to +\infty} v(x_m + R_mx, T_m + R_m^2s) = (0,0,1)$ is a constant function of $x$, which is a contradiction.
\end{proof}

\begin{proposition}[Criterion for existence of global solution]\label{prop:criterion for global solution}
    Let $v(x,t)$ solve (\ref{heat map flow}) on $[0,T)$ with smooth, $k$-equivariant boundary data.  Assume that the initial inclination coordinate satisfies $h_0(0,t) = 0$, $|h_0(r,t)| \leq \pi$ on $\{0,1\} \times [0,T) \cup [0,1] \times \{0\}$. If $T < +\infty$, then $v(x,t)$ and $h(r,t)$ cannot blow-up at time $T$. In particular, if  $|h_0(r,t)| \leq \pi$ on $\{0,1\} \times [0,+\infty) \cup [0,1] \times \{0\}$, then $v(x,t)$ and $h(r,t)$ are global. 
\end{proposition}

\begin{proof}
This is a generalization of the global solution constructed in \cite{ChangDing2011}'s paper for $k > 1$. The proof can be made shorter thanks to the Comparison Principle (Theorem \ref{comparison principle}).

    Assume for a contradiction that $T < +\infty$ and the solution blows-up at time $T$. We use a barrier argument, as in Proposition \ref{prop:further properties of blow-up of h}.
    
    It follows from the Comparison Principle (Theorem \ref{comparison principle}) that $|h(r,t)| \leq \pi$ on $[0,1] \times [0,T)$. Write
    $$
    (h-\pi)_t = (h-\pi)_{rr} +\frac{(h-\pi)_r}{r} -k^2 (h-\pi) \frac{\sin(2h)-\sin(2\pi)}{2r^2(h-\pi)}
    $$
    Assume for a contradiction that $h(r_0,t_0) = \pi$ at some point $(r_0,t_0) \in (0,1) \times (0,T]$ ($t_0 = T$ is allowed as $h$ is smooth on $(0,1] \times (0,T]$ by Proposition \ref{prop:smoothness at the boundary, earlier statement}). On any fixed region $E_{r_1,t_1} = (r_1,1) \times (t_1,t_0]$, $0 < r_1 < r_0$, $0 < t_1 < t_0 \leq T$, the coefficient in front of $h-\pi$ in the above PDE is bounded (in particular, from above). The Maximum Principle (Theorem \ref{maximum principle}) applied on $H = h-\pi$ implies that $h(r,t) = \pi$ is constant on $[0,1] \times [0,t_0]$, which is a contradiction to $h(0,t) = 0$. Similarly, if $h(r_0,t_0) = -\pi$, applying the Maximum Principle shows that $h = -\pi$ is constant on $[0,1] \times [0,t_0]$.  Hence, $|h(r,t)| < \pi$ on $(0,1) \times (0,T]$.

Fix an arbitrary $t_0 \in (0,T)$. It follows from Lemma \ref{lemma:semi-barrier with global solution} that
    $$
\theta_{\alpha}(r) \geq  \theta_{\alpha_0}(r) \geq |h(r,t_0)|, \quad \alpha \geq \alpha_0, r \in [0,r_0]
    $$
    for some $\alpha_0 > 0, r_0 \in (0,1)$. Moreover,
    $$
    \sup_{t \in [t_0,T]} |h(r_0,t)| = M < \pi
    $$
    as a consequence of the Comparison Principle. Fixing $\alpha \geq \alpha_0$ so that $\theta_{\alpha}(r_0) > M$, one can conclude as in the proof of Proposition \ref{prop:further properties of blow-up of h}.
%     Let $\overline{h}$ be the global solution constructed in Theorem \ref{thm:existence of global solution}. Fix an arbitrary $t_0 \in (0,T)$. By Lemma \ref{lemma:semi-barrier with global solution}, there is $t_1 > 0$ and $\tilde{r}_0 \in (0,1)$ for which 
%    \begin{equation*}
%        \overline{h}(r,t) \geq |h(r,t_0)|, \quad t \geq t_1, r \in [0,\tilde{r}_0] 
%    \end{equation*}
% As $\sup_{t \in [t_0,T]} |h(\tilde{r}_0,t)| = m < \pi$ and $\overline{h}(\tilde{r}_0,t)$ is increasing with respect to $t$ with $\overline{h}(\tilde{r}_0,+\infty) = \pi$, there is $t_2 \geq t_1$ large enough with
% $$
% \overline{h}(\tilde{r}_0,t) \geq m \geq \sup_{s \in [t_0,T]} |h(\tilde{r}_0,s)| \quad t \geq t_2
% $$
% Let $\tilde{h}(r,t) = \overline{h}(r,t + (t_2-t_0))$. Then $\tilde{h} \geq |h|$ on $[0,\tilde{r}_0] \times \{t_0\} \cup \{\tilde{r}_0\} \times [t_0,T)$, $\tilde{h} = |h| = 0$ on $\{0\} \times [t_0,T)$. It follows from the Comparison Principle (Theorem \ref{comparison principle}) that $\tilde{h} \geq |h|$ on $[0,\tilde{r}_0] \times [t_0,T)$. In particular, $\lim \limits_{r \to 0^+} |h(r,t)| = 0$ uniformly with respect to $t \in [t_0,T)$ as $\tilde{h}$ is continuous on $[0,\tilde{r}_0] \times [t_0,T]$. This prevents the convergence of $v(x,t)$ to a non-constant harmonic map along appropriate sequences as shown in the end of the proof of Proposition \ref{prop:further properties of blow-up of h}.
\end{proof}

\begin{corollary}[An even more precise blow-up description, \cite{vanderhout_scale}]\label{cor:comparison of h with itself and exact value of limsup}
 Let $v(x,t)$ solve (\ref{heat map flow}) with smooth and $k$-equivariant boundary data on $[0,T)$ with $T < +\infty$ (which may or may not be the maximal existence time). Assume that $h$, $h_0$ are chosen so that $h(0,t) = 0$ on $[0,T)$. There exists $\tau_0 = \tau_0(v,T)$ such that for all $0 < \tau \leq \tau_0$, $t \in [\tau,T)$, one has
    $$
    h(r,t-\tau) - \pi \leq h(r,t) \leq h(r,t-\tau) + \pi, \quad r \in [0,1].
    $$
In particular, if $T = T_{\max} < +\infty$, for any sequence $(r_n,t_n) \to (0^+,T^-)$, one has
$$
-\pi \leq \liminf_{n \to +\infty} h(r_n,t_n) \leq \limsup_{n\to +\infty} h(r_n,t_n) \leq \pi,
$$
hence
$$
    \limsup_{\substack{r \to 0^+ \\ t \to T^-}}|h(r,t)| = \pi.
$$
\end{corollary}

\begin{proof}
    As $h(r,t) \in C^{0}([0,1] \times [0,T'])$ is uniformly continuous for any $0 < T' < T$, there is $\tau_0 > 0$ small enough for which $h(r,\tau) \leq h(r,0) + \pi$ for all $r \in [0,1]$, $0 < \tau \leq \tau_0$. Similarly, $h(1,t) = h_0(1,t) \leq h_0(1,t-\tau) + \pi = h(1,t-\tau) + \pi$ for all $\tau \leq \tau_0 \leq t \leq T$ by uniform continuity of the boundary data. Finally, $0 = h(0,t) \leq h(0,t-\tau) + \pi = \pi$ for all $\tau \leq \tau_0 \leq t \leq T$. For $0 < \tau \leq \tau_0$ fixed, the Comparison Principle (Theorem \ref{comparison principle}) implies that 
    $$
    h(r,t-\tau) - \pi \leq h(r,t) \leq h(r,t-\tau) + \pi, \quad r \in [0,1], t \in [\tau,T),
    $$
    where we used a similar argument with $-\pi$ instead of $\pi$. The limsup value follows from Proposition \ref{prop:further properties of blow-up of h}.
\end{proof}

\begin{theorem}[Existence of a global solution, \cite{ChangDing2011}]\label{thm:existence of global solution}
    Assume that $\psi = r^k\tilde{\psi}$, $\tilde{\psi} \in C^{\infty}([0,1])$ with $\partial^{(2n+1)}_r \tilde{\psi}(0) = 0$ for all $n \in \mathbb N_{\geq 0}$, $0 \leq \psi \leq \pi$, and
    $$
\tau(\psi) := \psi_{rr} + \frac{1}{r}\psi_r -k^2 \frac{\sin(2\psi)}{2r^2} \geq 0, \quad \tau(\psi)(1) = 0.
    $$
   Then the solution $h(r,t) \in C^{\infty}([0,1] \times (0,T))$ of (\ref{heat map formulation for h(r,t)}) with boundary data $h_0(r,t) = \psi(r)$ on $[0,1] \times \{0\} \cup \{0,1\} \times [0,+\infty)$ is global, i.e.,  $T = +\infty$,  has finite energy (\ref{Energy of v in terms of h}), satisfies $0 < h < \pi$ and $h_t > 0$ on $(0,1) \times (0,T)$, as well as $h_r > 0$ on $(0,1] \times (0,T)$.
\end{theorem}

\begin{remark}
    One can choose $\psi(r) = 4 \arctan \left( (\alpha r)^k \right)$, $\alpha \in (0,1]$, which satisfies
    $$
\tau(\psi) = -k^2 \frac{\sin(2\psi)}{2r^2}  + 2 k^2 \frac{\sin(\psi)}{2r^2} = k^2r^{-2} \sin(\psi) (1-\cos(\psi)) \geq 0,
    $$
    and all the other required properties. If one wants $\psi(1) = \pi$ (as in Proposition \ref{prop:blow up infinity global solution}) to have a blow-up at infinity, then one needs to take $\alpha = 1$.
\end{remark}

\begin{proof}
    This is a generalization of \cite{ChangDing2011}'s result for $k > 1$, but the proof is essentially the same because the sign of the nonlinearity does not change and we still have access to a family of stationary solutions given by the arctan function. Let $\psi$ be as in the theorem.
    
    Let $h(r,t) \in C^{\infty}([0,1] \times (0,T))$ solve (\ref{heat map formulation for h(r,t)}) with boundary data $h_0(r,t) = \psi(r)$ on $[0,1] \times \{0\} \cup \{0,1\} \times [0,+\infty)$ with maximal time of existence $T =  +\infty$ by Proposition \ref{prop:criterion for global solution}. 
    
    Then $h, h_r, h_t$ are well-defined and continuous on $[0,1] \times [0,T)$, $h_{rr}$ is well-defined and continuous on $(0,1] \times [0,T)$ (Proposition \ref{prop:continuity of h_t}). Differentiating $h(0,t) = \psi(0)$, $h(1,t) = \psi(1)$ with respect to $t$ yields $h_t(0,t) = h_t(1,t) = 0$ on $[0,T)$ and (\ref{heat map formulation for h(r,t)}) at $t = 0$, $r > 0$ yields $h_t(r,0) = \tau(\psi)(r) \geq 0$ on $(0,1]$, $h_t(1,0) = \tau(\psi)(1) = 0$.
    
    Moreover $u = h_t$ solves a 2D radial heat equation
    $$
    u_t = \Delta u - k^2 \frac{2\cos(2h)}{2|x|^2} u, \quad (r,t) \in  (0,1) \times (0,T).
    $$
    Near $r = 0$, $\cos(2h) \approx 1$ and away from $r = 0$, the coefficient of $u$ is bounded. Hence, it is bounded from above. The Maximum Principle (Theorem \ref{maximum principle}) applied on $E_{t_1} = B^2 \times (0,t_1]$, $t_1 < T$, with $v = ue^{-ct}$ for $c$ sufficiently large shows $v$ cannot achieve a minimum $m \leq 0$ inside $E_{t_1}$. Hence, either the minimum is achieved at the parabolic boundary, on which case it must be $m = 0$, and $h_t > 0$ on $E_{t_1}$, either there is a positive minimum $m > 0$ achieved inside $E_{t_1}$. The latter is not possible as points $h_t = 0$ at the parabolic boundary would be a smaller minimum. In other words, $h_t > 0$ on $(0,1) \times (0,T)$.

    Writing
    $$
    -k^2 \frac{\sin(2h)}{2r^2} =  -k^2 \frac{\sin(2h)}{2hr^2} h,
    $$
    one has $\sin(2h)/2h \approx 1$ near $r = 0$ and away from zero, the coefficient of $h$ is bounded. As for $h_t$, the Maximum Principle implies $0 < h < M \leq \pi$ on on $(0,1) \times (0,T)$, where $0$ and $0 < M \leq \pi$ are respectively the minimum and maximum of $h$ achieved on the parabolic boundary. Hopf's Lemma (\cite[Chapter 3, Theorem 6 and Theorem 7]{protter2012maximum}) also implies that $h_r(0,t) > 0$ on $(0,T)$ and then $h_r(r,t) > 0$ on $[0,\delta_t]$ for some $\delta_t > 0$.

     % Since $h_r$ is continuous on $[0,1] \times [0,T)$, $h(0,t) = 0$ on $[0,T)$, $h(r,t) > 0$ on $(0,1) \times (0,T)$, it follows from the mean value theorem that for each $t \in [0,T)$, there is $\delta_t > 0$ for which $h_r(r_0,t) > 0$ for some $r_0 = r_{0,t} \in (0,\delta_t)$. If $\delta_t > 0$ is small enough so that $0 < h(r,t) < \pi/2$ on $(0,\delta_t)$, then $h_r(r,t) > 0$ on the whole $(0,\delta_t)$. Otherwise, one can find $r_1 = r_{1,t} \in (0,\delta)$ with $h_r(r_1,t) = 0$ using the intermediate value theorem, meaning that $h(r,t)$ has a local maximum with respect to $r$ at $r_1$. Hence, $h_r(r_1,t) = 0$, $h_{rr}(r_1,t) \leq 0$. As $h_t(r_1,t) > 0$, $\sin(2h(r_1,t)) > 0$, this contradicts (\ref{heat map formulation for h(r,t)}).

     Actually, $h_r(r,t) > 0$ on $(0,1] \times (0,T)$. Otherwise, there is $t \in (0,T)$ and $r^* \in (0,1]$ for which $h_r(r,t) > 0$ on $(0,r^*)$, $h_r(r^*,t) = 0$. Multiplying (\ref{heat map formulation for h(r,t)}) by $r^2h_r$, one finds
     \begin{align}
         r^2h_rh_t = \frac{1}{2}\partial_r \left[r^2 h_r^2 + k^2 \frac{\cos(2h)}{2} \right], \quad (r,t) \in (0,1) \times (0,T), \notag \\
         2\int_0^r s^2 h_r(s,t)h_t(s,t)ds = r^2 h_r^2 + k^2 \frac{\cos(2h)-1}{2} = r^2 h_r^2 - k^2 \sin(h)^2, \quad (0,1) \times (0,T). \label{eq:sacks-uhlenbeck identity}
     \end{align}
     At $(r^*,t)$, (\ref{eq:sacks-uhlenbeck identity}) has strictly positive left-hand side with a non-positive right-hand side, which is a contradiction.

As $h$ does not depend on time at the boundary $r = 1$, one checks that
    $$
\frac{d}{dt} E(h(\cdot,t)) = -\int_0^1 |h_t|^2 r dr, \quad t > 0,
    $$
    which means that the energy is decreasing over time. Hence, $h$ is a finite-energy solution on $[0,1] \times [0,T)$.
    
    % If $h$ blows-up at a finite time $T < +\infty$, blows-up happens at $(0,T)$ (Proposition \ref{prop:further properties of blow-up of h}). Hence, along some appropriate monotonic sequences $R_m \to 0^+, T_m \to T^-$ and some fixed $s \leq 0$ (Theorem \ref{thm:smoothness around all but finitely many points}), $h(R_m R, T_m + R_m^2s)$ converges uniformly on all compact sets of $[0,+\infty)$ to $H(r) = 2 \arctan \left( (\alpha r)^k \right)$ for some $\alpha > 0$.

    % For fixed time, $h$ is increasing with respect to $r$. For any $r \in (0,1]$, $R \in (0,+\infty)$ fixed, there is $m$ large enough so that $r > R_m R$, meaning that
    % $$
    % 0 \leq h(R_m R, T_m + R_m^2s) \leq h(r, T_m + R_m^2s) \leq \pi
    % $$
    % Taking the limit $m \to +\infty$ yields
    % $$
    % 0 \leq 2 \arctan \left( (\alpha R)^k \right) \leq \lim \limits_{m \to +\infty} h(r, T_m + R_m^2s) \leq \pi
    % $$
    % where $\lim \limits_{m \to +\infty} h(r, T_m + R_m^2s)$ exists as $h$ is smooth on $(0,1] \times [0,T]$. Letting $R \to +\infty$ shows that
    % $$
    % h(r,T) = \pi, \quad r \in (0,1]
    % $$
    % meaning that $h$ achieves its maximum at time $T$. As $h$ is smooth on $(0,1] \times (0,T]$ (Proposition \ref{prop:smoothness at the boundary}), the Maximum Principle applied on $E_T = (r_0,1) \times (0,T]$ for any $r_0 \in (0,1)$ shows that $h$ is constant on $E_T$, a contradiction.
\end{proof}

\begin{proposition}[Blow-Up at infinity of global solution, \cite{ChangDing2011}]\label{prop:blow up infinity global solution}
Let $\overline{h}$ be the global solution constructed in Theorem \ref{thm:existence of global solution} for the $k$-equivariant problem, $k \geq 1$. Assume there is $r^* \in (0,1]$ for which $\psi(r^*) = \pi$. Then $\overline{h}$ blows-up at $T = +\infty$ in the sense that for any $m \in \mathbb N$, there exists sequences $T_n \to +\infty$, $R_n \to 0^+$, $l = l(m)$, $\alpha > 0$ for which 
    $$
\lim \limits_{n \to +\infty } ||\overline{h}_r(\cdot,T_n)||_{C^0([0,r_0])} = +\infty \quad \forall r_0 \in (0,1],
$$
and $\overline{h}_n(r) = \overline{h}(R_nr, R_n^2t + T_n): [0,R_n^{-1}2^{-l}] \times [-1,0] \to [0,\pi]$ converges in $C^{m}_{loc}(\mathbb R \times [-1,0])$ to the stationary map
$$
\overline{H}_{\infty}(r) = 2 \arctan \left( (\alpha r)^k \right).
$$
Moreover, for any $r \in (0,1]$,
$$
\lim \limits_{t \to +\infty} \overline{h}(r,t) = \pi.
$$
\end{proposition}

\begin{proof}
     As $\overline{h}$ does not depend on time at the boundary $r = 1$, one checks that
    $$
\frac{d}{dt} E(\overline{h}(\cdot,t)) = -\int_0^1 |\overline{h}_t|^2 r dr, \quad t > 0.
    $$
    If $\overline{v}$ is the $k$-equivariant solution of the heat-map flow corresponding to $\overline{h}$ (Proposition \ref{prop:equivalent fomrulation heat map flow}), one deduces the energy inequality
  \begin{align}
        E(\overline{h}(\cdot,t)) = E(\overline{v}(\cdot,t)) &\leq E(\overline{v}(\cdot,0)) < +\infty, \label{eq: energy inequality, bounded energy} \\
        \int_0^{+\infty} \int_{0}^1 |\partial_t \overline{h}|^2r dr dt = \int_0^{+\infty} \int_{B^2} |\partial_t \overline{v}|^2 dx dt &\leq E(\overline{v}(\cdot,0)) < +\infty, \label{eq: energy inequality, bounded partial_t v L^2 norm}
  \end{align}
    (see also \cite[Chapter 3, Lemma 5.8]{struwe2008variational}). The energy inequality implies
    \begin{equation}
        \lim \limits_{n \to +\infty} \int_{0}^1 |\partial_t \overline{h}(\cdot,T_n)|^2 rdr = \lim \limits_{n \to +\infty} \int_{B^2} |\partial_t \overline{v}(\cdot,T_n)|^2 dx = 0  \label{eq: decay of partial_t overline v along T_n}
    \end{equation}
    along some sequence $T_n \to +\infty$. It follows that along any such sequence, there is a blow-up of $\overline{h}_r$ at $r = 0$ in the sense that
    $$
\lim \limits_{n \to +\infty } ||\overline{h}_r(\cdot,T_n)||_{C^0([0,r_0])} = +\infty \quad \forall r_0 \in (0,1].
$$
    Otherwise, $\overline{h}_r$ is bounded on $[0,r_0]$ along the sequence $T_n \to +\infty$. Along a subsequence, $\overline{v}(x,T_n)$ must then converge weakly in $H^{1,2}(\overline{B}^2, S^2)$ to a harmonic map $\overline{v}_{\infty}$ which must be smooth on $B^2$ (\cite[Theorem 3.6]{sacks-uhlenbeck}). In fact, the convergence is uniform on $\overline{B^2}$ up to taking another subsequence. Away from zero, this follows from the fact that $h(r,T_n)$ converges weakly in $H^{1,2}([r_0,1])\hookrightarrow C^{0,1/4}([r_0,1]) \subset \subset C^0([r_0,1])$, hence uniformly, to a $H^{1,2}([r_0,1])$-map. Moreover, on $[0,r_0]$, $h(r,T_n)$ converges uniformly along a subsequence to a $C^0([0,r_0])$-map by Arzela-Ascoli as $h_r$ is bounded. Finally, the harmonic map satisfies $\overline{v}_{\infty}(\partial B_{r^*}(0)) = \{(0,0,-1)\}$, meaning that $\overline{v}_{\infty}$ must be constant (\cite[Theorem 3.2]{lemaire}), which contradicts $\overline{v}_{\infty}(0) = (0,0,1)$. 

    Let
    $$
    \theta_T = \max_{t \in [0,T]} ||\partial_r \overline{h}(\cdot,t)||_{C^0([0,1])}\to  +\infty, \quad T \to +\infty.
    $$
    One can find $T_n \to +\infty$, $r_n \in [0,1]$, such that
    $$
     R_n^{-1} := \theta_{T_n} = ||\partial_r \overline{h}(\cdot,T_n)||_{C^0([0,1])} = |\partial_r \overline{h}(r_n,T_n)| \to +\infty.
    $$
    Consider the sequence $\overline{V}_n(x,t) = \overline{v}(R_nx, R_n^2t + T_n)$, $t \in [-1,0]$, where we take $n$ large enough so that $T_n - R_n^2 > T_{n-1}$, $T_0 - R_0^2 > 0$. It follows from (\ref{eq: energy inequality, bounded partial_t v L^2 norm}) that
    \begin{equation}
        \lim \limits_{n \to +\infty} \int_{-1}^{0} \int_{B_{R_n^{-1}}(0)} |\partial_t \overline{V}_n|^2 dxdt = \lim \limits_{n \to +\infty} \int_{T_n-R_n^2}^{T_n} \int_{B^2} |\partial_t \overline{v}|^2 dxdt = 0. \label{eq: convergence of partial_t overline v^2 to zero in L^2}
    \end{equation}
   Recall now from (\ref{nabla v in terms of h}) that
\begin{equation*} 
    |\partial_r \overline{h}|^2 \leq |\nabla_x \overline{v}|^2 \leq  |\partial_r \overline{h}|^2 + \frac{k^2}{r^2}\sin(\overline{h})^2, \quad (r = |x|,t) \in (0,1] \times [0,+\infty),
\end{equation*}
where 
$$
\sin(\overline{h}(r,t)) = r \cos(\overline{h}(c_{r,t},t)) \partial_r \overline{h}(c_{r,t},t) \leq r||\partial_r \overline{h}(\cdot,t)||_{C^0([0,1])}.
$$
 Hence,
\begin{align}
       ||\nabla_x V_n||_{C^0\left(B_{R_n^{-1}}(0) \times [-1,0] \right)} &\lesssim_k 1, \label{eq: uniform bounded C1 norm form v_n} \\
    ||(\partial_t - \Delta) \overline{V}_n ||_{C^0\left(B_{R_n^{-1}}(0) \times [-1,0] \right)} &= ||\left( |\nabla_x \overline{V}_n|^2 \overline{V}_n \right) ||_{C^0\left(B_{R_n^{-1}}(0) \times [-1,0] \right)} \lesssim_k 1,  \notag
\end{align}
as $|\tilde{V}_n| = 1$. It follows from Parabolic Sobolev Embedding (Corollary \ref{cor: local parabolic sobolev embedding}) and a Schauder iteration as in the proof of Theorem \ref{thm:resgularty solution nonlinear pb} that
$$
\overline{V}_n, \nabla_x \overline{V}_n, ..., D^{m}_x \overline{V}_n \in C^{2+2\gamma,1+\gamma} \left(B_{R_n^{-1} 2^{-l}}(0) \times [-1,0] \right)
$$
for some $l = l(m) \in \mathbb N$ with norm uniformly bounded with respect to $n$ (for $m$ fixed).

 By Arzela-Ascoli, up to taking a subsequence, the sequence of rescaled solutions $\overline{W}_n = \overline{v}(R_n r, R_n^2t + T_n): [0,R_n^{-1}2^{-l}] \times [-1,0] \rightarrow S^2$ converges in $C^{k+2,k+1}(K \times [-1,0])$ for any compact subset $K \subset \mathbb R^2$ to some harmonic map $\overline{V}_{\infty} \in C^{k+2}(\mathbb R^2)$ which must be stationary by (\ref{eq: convergence of partial_t overline v^2 to zero in L^2}) and smooth (\cite[Theorem 3.6]{sacks-uhlenbeck}, Theorem 3.6). 

Similarly, the sequence of rescaled solutions $\overline{H}_n = \overline{h}(R_n r, R_n^2t + T_n): [0,R_n^{-1}2^{-l}] \times [-1,0] \rightarrow [0,\pi]$ converges uniformly on all compact sets $K \times [-1,0] \subset [0,+\infty) \times [-1,0]$ to a map $\overline{H}_{\infty} \in C^0([0,+\infty))$ solving
\begin{align}
     0 = H_{rr} + \frac{H_r}{r} - k^2 \frac{\sin(2H)}{2r^2}, \quad r \in (0,+\infty), \label{time independent ode for H}
\end{align}
with $H_{\infty}(0) = 0$. But Lemma \ref{k-equivariance, smoothness of h} shows that 
$$
\overline{V}_{\infty}(re^{i\theta}) = \left(e^{ik \theta} \sin H(r), \cos H(r) \right) 
$$
for some lifting $H \in C^{\infty}([0,+\infty))$. Hence, $\overline{H}_{\infty}(r)$ and $H(r)$ must differ by a fixed multiple of $2\pi$ and $\overline{H}_{\infty}$ is smooth as well. It follows that $\overline{H}_{\infty} \in C^{\infty}([0,+\infty);[0,\pi])$ is a smooth solution for the ODE (\ref{time independent ode for H}) with initial condition $H_{\infty}(0) = 0$. 

Those solutions $V_{\infty}$, $H_{\infty}$ must be non-constant. Indeed, one has
\begin{equation*}
    R_n^{-1} = |\partial_r \overline{h}(r_n,T_n)| \leq |\nabla_x \overline{v}(r_n,T_n)|
\end{equation*}
as a consequence from (\ref{nabla v in terms of h}) and
$$
R_n^{-1} = |\partial_r \overline{h}(r_n,T_n)| \lesssim r_n^{-1}
$$
by Proposition \ref{blow up point is at zero k equivariant}. Along a subsequence, $r_nR_n^{-1} \to r^* \in \mathbb R$ converges. Hence, it holds that
$$
  |\nabla_x \overline{W}_n(R_n^{-1}r_n,T_n)| \geq |R_n \nabla_x \overline{v}(r_n,T_n)| \geq 1 \quad \forall n \in \mathbb N,
$$
meaning that $|\nabla_x \overline{V}_{\infty}(r^*)| \geq 1$ is not identically zero. Hence,  $\overline{H}_{\infty}$ is non-constant as well and exactly of the form:
 $$
\overline{H}_{\infty}(r) = 2\arctan \left( (\alpha r)^k \right), \quad \alpha > 0,
 $$
 (see (\ref{all possible form for limiting H}) for a short argument based on the comparison principle). Finally, the convergence the convergence $\overline{H}_n \to \overline{H}_{\infty}$ holds in $C^{k+2,k+1}(K \times [-1,0])$ (and not only in $C^0(K \times [-1,0])$) for any compact subset $K \subset [0,+\infty)$ by locally expressing $\overline{H}_n$ in terms of $\overline{V}_n$ and the arccos, arcsin functions (e.g. see the proof of Theorem \ref{thm:smoothness around all but finitely many points}).

Finally, the limit $\overline{h}(r,+\infty)$ exists as $\overline{h}$ is increasing for fixed $r \in (0,1]$ and bounded by $0 \leq \overline{h} \leq  \pi$.  If $R \in (0,+\infty)$ is also fixed, there is $m$ large enough so that $r > R_m R$, meaning that
    $$
    0 \leq \overline{h} (R_m R, T_m + R_m^2s) \leq \overline{h}(r, T_m + R_m^2s) \leq \pi
    $$
    as $\overline{h}$ is increasing for fixed time. Taking the limit $m \to +\infty$ yields
    $$
    0 \leq 2 \arctan \left( (\alpha R)^k \right) \leq \lim \limits_{m \to +\infty} h(r, T_m + R_m^2s) \leq \pi,
    $$
    and letting $R \to +\infty$ shows that $\overline{h}(r,+\infty) = \pi$.
\end{proof}

\section{Comparison Principle, Maximum Principle}
We state two fundamental results for the study of parabolic PDEs, which are mostly independent from the previous sections and which will be used extensively in the following sections. The Comparison Principle (Theorem \ref{comparison principle}) is stated only for our specific nonlinear problem, but the proof is standard and can be applied to a wide class of nonlinear heat equations.

\begin{theorem}[Maximum Principle] \label{maximum principle}
    Let $E \subset \mathbb R^n \times \mathbb R$ be a non-empty, open and connected set. Let $E_{t_1} = \{(x,t) \in E: t \leq t_1\}$. Consider a parabolic operator
    $$
    L = \sum_{i=1}^n b^i(x,t) D_{x_i} + \sum_{i,j=1}^n a^{i,j}(x,t) D_{x_i}D_{x_j} - \partial_t,
    $$
    where $a_{i,j},b_i \in L^{\infty}(E_{t_1})$ and
    \begin{align*}
    \sum_{i,j=1}^n a^{i,j}(x,t)\xi_i\xi_j &\geq \mu |\xi|^2, \quad (x,t) \in E_{t_1}, \xi \in \mathbb R^n
    \end{align*}
    for some $\mu > 0$.

    Assume that $u \in C^{2,1}_{x,t}(E_{t_1})$ (meaning that $u$ is continuously differentiable with respect to $t$, twice-continuously differentiable with respect to $x$ and the derivatives extend continuously at $t = t_1$) and $L[u] \geq 0$. Assume that $M = \sup_{E_{t_1}} u(x,t) < +\infty$. 
    
    If $u(x_1,t_1) = M$ at some point $(x_1,t_1) \in E_{t_1}$, then $u = M$ at every point $(x,t) \in E_{t_1}$ which can be attained from $(x_1,t_1)$ by a finite sequence of horizontal line segments 
    $$H(p_0,p_1,T) = \{(x,T) \in E_{t_1}: u(p_0,T) = M, x = p_0 + h(p_1-p_0), 0 \leq h \leq 1\}$$ 
    and vertical line segments going backwards in time
    $$V(x_0,T_0,T_1) = \{(x_0,t) \in E_{t_1}: u(x_0,T_1) = M, T_0 \leq t \leq T_1\}$$ 
    that stay completely in $E_{t_1}$ (i.e., they are connected segments in the topological sense).

    The theorem remains true if $(L+c(x,t))[u] \geq 0$ with $c(x,t) \leq 0$ measurable and $M \geq 0$, as well as if $(L+c(x,t))[u] \geq 0$ with any $c(x,t)$ bounded from above and $M = 0$.
\end{theorem}

\begin{proof}
    See \cite[Chapter 3, Theorem 5, Theorem 7]{protter2012maximum} and the remark that follows it.
\end{proof}

\begin{theorem}[Comparison Principle for sub- and super-solutions, \cite{vanderhoutcomparison}]\label{comparison principle}
    Let $0 < r_0,T < +\infty$, $0 \leq t_0 < +\infty$, $R_T = (0,r_0) \times (t_0,t_0+T)$ and denote by $\Sigma_T = \partial R_T \setminus [0,r_0] \times \{t_0+T\}$ its parabolic boundary. 
    
    Assume that $\theta, \psi \in C^0(\overline{R_T}) \cap W^{1,2}_{loc}(R_T)$ (meaning that $\theta_r, \psi_r, \theta_t, \psi_t$ are defined in a weak sense and square-integrable on compact subsets of $R_T$) are respectively subsolutions and supersolutions of (\ref{heat map formulation for h(r,t)}), i.e., 
    $$
    \int_{R_T} r \theta_t \phi drdt \leq -\int_{R_T} r \left( \phi_r \theta_r + k \frac{\sin 2\theta}{2r^2} \phi \right) drdt
    $$
    and
    $$
    \int_{R_T} r \psi_t \phi drdt \geq -\int_{R_T} r \left( \phi_r \psi_r + k \frac{\sin 2\psi}{2r^2} \phi \right) drdt
    $$
    for any $\phi \in W^{1,2}(R_T)$ with compact support in $R_T$. 
    
    Suppose that either $\sup \limits_{t} E(\theta(\cdot,t)) < +\infty$ or $\sup \limits_{t} E(\psi(\cdot,t)) < +\infty$. 
    
    If $\theta \leq \psi$ on $\Sigma_T$, then $\theta \leq \psi$ on $\overline{R_T}$.
\end{theorem}

\begin{remark}
     This is a standard argument, which would also work when replacing 
    $$
    N(h,r) = k \frac{\sin 2 h}{2r^2}
    $$ 
    by any nonlinearity $N(h,r)$ for which one can prove 
    $$
    -(N(\theta(r),r) - N(\psi(r),r)) \cdot \chi_{\{\theta(r)-\psi(r) \geq 0\}}(r) \lesssim|\theta(r)-\psi(r)|, \quad r \in (0,r_0)
    $$
    given the assumptions on $\theta, \psi$. For example, take $N$ continuous on $\mathbb R \times [0,r_0]$ and Lipschitz with respect to $h$.
\end{remark}

\begin{proof}
   After parabolic rescaling and shifting time to zero, one can assume that $R_T = (0,1) \times (0,T)$. The idea is to prove 
\begin{align}
            \int_0^1 r(1-r) \left[ (\theta-\psi)^+ \right]^2(r,\tau) dr \lesssim - \int_0^{\tau}\int_0^1 r(1-r)  (\theta-\psi)^+ (N(\theta,r)-N(\psi,r))drdt \label{gronwall pre-step, comparison proof}
\end{align}
for any $0 < \tau < T$ and then use Grönwall's inequality. Intuitively, this amounts to choosing $\phi = (1-r) (\theta-\psi)^+ \chi_{[0,\tau]}$ as a test function, but it does not have compact support, so additional cut-offs are required. See the proof in the Appendix of \cite{vanderhoutcomparison} for more details. The $k$-factor in front of the nonlinearity plays no significant role in the argument. 
\end{proof}
%\printindex

\section{On the geometry of level sets $h_1 < h_2$}

In this section, following the lap-number argument from \cite{VANDERHOUT2003} and \cite{matano}, we are concerned on the level sets $h_1 < h_2$, as well as $h_1 < h < h_2$. We show that under some mild assumptions, those levels sets are connected to the parabolic boundary (Lemma \ref{geometry of h_1 < h_2} and Lemma \ref{lemma:geometry of h_1 < h < h_2}). Moreover, at a fixed time, there cannot be too many successive intersections between $h$ and $h_i$ on $[0,1]$ and this intersection number cannot increase over time (Lemma \ref{lemma:intersection argument with chains mod 4}. These results play a key role in the bubble-tree result.

Even though the results are formulated in greater generality, we will mostly be concerned with the cases where $h_i = \pi/2, \pi$ or $h_i$ is a stationary solution given in terms of the arctan function.

We note that the proof of Lemma \ref{lemma:geometry of h_1 < h < h_2} is novel. The original paper from \cite{VANDERHOUT2003} omits this proof but one cannot repeat the argument from Lemma \ref{geometry of h_1 < h_2}. Indeed, repeating the argument by contradiction from Lemma \ref{geometry of h_1 < h_2} shows that the maximum of $h-h_1$, resp. $h_2-h$, is achieved on $\overline{\Gamma} \setminus \Gamma$. However, one cannot conclude as $h = h_1$ or $h = h_2$ on $\overline{\Gamma} \setminus \Gamma$ and both can happen at different points because this set is not necessarily connected. We propose a more subtle argument based on the notion of accessible boundary points from \cite{newman1964elements}. These points, which are dense in the boundary, allow to circumvent the aforementioned problem. Even though the boundary is not connected, most boundary points can be reached from inside the domain and are connected together.

\begin{lemma}[\cite{VANDERHOUT2003}, Lemma 2.1] \label{geometry of h_1 < h_2}
Let $k \in \mathbb N$ and let $h_i(r,t) \in C^0([0,1] \times [0,T)) \cap C^{2,1}((0,1) \times (0,T))$ (meaning that $u$ is continuously differentiable with respect to $t$, twice-continuously differentiable with respect to $r$ and the derivatives are continuous at the included boundary) , $i = 1,2$, be two classical solutions for the nonlinear problem
\begin{align*}
    h_{i,t} &= h_{i,rr}+\frac{h_{i,r}}{r} - k^2 \cdot \frac{\sin(2h_i)}{2r^2}, \quad 0 < r < 1, \quad 0 < t < T,
\end{align*}
with a common time of existence $T \leq +\infty$. Let $0 \leq t_1 < t_2 < T$ and $0 \leq R_1 < R_2 \leq 1$. Let $A$ be a (path)-connected component of $\{(r,t) \in [R_1,R_2] \times [t_1,t_2]: h_1 > h_2\}$. Then
$$
A \cap \left[ \{r = R_1\} \cup \{r = R_2\} \cup \{t = t_1\} \right] \neq \emptyset,
$$
where, if $R_1 = 0$, one needs to further assume either $h_1(0,t) \neq h_2(0,t)$ for all $t \in [0,T)$ or $h_1(0,t) = h_2(0,t) = m\pi \in \pi \mathbb Z$ for all $t \in [0,T)$.
\end{lemma}

\begin{notation}
    In the following, we write $\overline{S}$, $\partial_{\mathbb R^2} S$, $\mathrm{int}_{\mathbb R^2}(S)$ for the closure, boundary and interior of $S \subset [R_1,R_2] \times [t_1,t_2]$ with respect to $\mathbb R^2$ topology. One should note that the closure $\overline{S}$ with respect the $\mathbb R^2$ topology or the closed square topology is the same, but the boundary might differ. 
\end{notation}

\begin{remark}
    In our Euclidean setting, connectedness, path-connectedness and arc-connectedness (i.e., existence of injective paths between any two points) are all equivalent.
\end{remark}

\begin{proof}
We follow the proof in \cite[Lemma 2.1]{VANDERHOUT2003}. Assume for a contradiction that this is not the case.

     First, observe that $A$ is open in the closed square topology. Moreover, $h_1 = h_2$ on $\overline{A} \setminus A$ (i.e., the boundary with respect to the closed square topology). If we had $h_1 > h_2$ at a point $(r,t) \in \overline{A} \setminus A$, then there would be an open (in the closed square topology) path-connected neighborhood $U \subset [R_1,R_2] \times [t_1,t_2]$ of $(r,t)$ on which $h_1 > h_2$. Since $(r,t)$ belongs to the boundary of $A$ (in the closed square topology), such a neighborhood must intersect $A$ by definition, meaning we can find a path from $A$ to $(r,t) \notin A$ on which $h_1 > h_2$. This contradicts the maximality of the connected component $A$.

    We consider $H = (h_1-h_2)$. Then
    $$
     H_{t} = H_{rr}+\frac{H_{r}}{r} - k^2 \cdot H \cdot \frac{\sin(2h_1)-\sin(2h_2)}{2r^2(h_1-h_2)}, \quad (r,t) \in A,
    $$
    which is a radial heat equation in 2D. If $R_1 > 0$, the coefficients in front of $H$ is bounded. If $R_1 = 0$ and $h_1(0,t) \neq h_2(0,t)$, then $\mathrm{dist}(\overline{A},\{r = 0\}) > 0$. Otherwise, there would be $x \in \left( \overline{A} \setminus A \right)\cap \{r = 0\}$ which is not possible because $h_1 = h_2$ on $\overline{A} \setminus A$, while $h_1 \neq h_2$ at $r = 0$. In that case, the coefficient in front of $H$ is bounded as well. If $R_1 = 0$ and $\mathrm{dist}(\overline{A},\{r = 0\}) > 0$, no matter $h_1$ and $h_2$, the situation is the same. Finally, if $R_1 = 0$ and $h_1(0,t) = h_2(0,t) = m\pi$ for all $t \in [0,T)$ while $\mathrm{dist}(\overline{A},\{r = 0\}) = 0$, then replacing $h_i$ by $h_i - m\pi$, we can assume that $m = 0$. Near $r = 0$, given $t_1, t_2$, there is $r_0$ for which $|2h_i(r,t)| \leq \pi/2$ for $(r,t) \in [0,r_0] \times [t_1,t_2]$ by uniform continuity. On $\overline{A} \cap  [0,r_0] \times [t_1,t_2]$, one has
    $$
    \frac{\sin(2h_1)-\sin(2h_2)}{2r^2(h_1-h_2)}  \geq 0,
    $$
    as $h_1 \geq h_2$ on $\overline{A}$ and $\sin(x)$ is increasing on $[-\pi/2,\pi/2]$. Away from $r_0$, the coefficient is bounded. 

    All in all, the coefficient in front of $H$ is bounded from above. Hence, one can apply the maximum principle (Theorem \ref{maximum principle}) on $e^{-ct}H$ for $c \geq 0$ large enough so that the coefficient in front of $e^{-ct}H$ is bounded from above by zero.
    
    Let $M = \sup_A e^{-ct}H$ be achieved at a point $(r_0,t_0) \in \overline{A}$. As $A$ is non-empty, $M > 0$. We prove that the maximum is also achieved on $\overline{A} \setminus A$, meaning that $M \leq 0$ on $A$: a contradiction.

    If $M$ is achieved on $\overline{A} \setminus A$, we are done. Moreover, $M$ cannot be achieved on the empty set
    $$
    A \cap \left[ \{r = R_1\} \cup \{r = R_2\} \cup \{t = t_1\} \right] = \emptyset.
    $$
    Hence, we can assume that $t_1 < t_0 \leq t_2$ and $R_1 < r_0 < R_2$. we will prove that a maximum cannot occur at such point. Since $A$ is open with respect to the closed square topology, one can find a lower half-ball
    $$
    \{(r-r_0)^2 + (t-t_0)^2 < \varepsilon, t \leq t_2\} \subset A \cap \{t_1 < t \leq t_2, R_1 < r < R_2\}
    $$
    for $\varepsilon > 0$ small enough. On this half-ball, the maximum principle (Theorem \ref{maximum principle}) applies and $e^{-ct}H = M$ is constant. Hence, we can assume that $t_1 < t_0 < t_2$ without loss of generality. Now, the same theorem implies that $e^{-ct}H = M$ for any point of $A$ which can be connected to $(r_0, t_0)$ using a horizontal segment which stays inside $A$. Take such a point $(r,t_0)$ with $r \in [R_1,R_2]$ minimal. This point must be in $\overline{A}$ (because it is the limit of a sequence in $A$) and it is connected to the point $(r_0,t_0)$ with an horizontal segment staying completely in $A$ except for the endpoint $(r,t_0)$. Since $A$ is a connected component and $e^{-ct}H = M > 0$ at $(r,t_0)$, one must have $(r,t_0) \in A$. And it cannot be that $r > R_1$ since $A$ is open (in the square topology) but $r$ is minimal. As $A \cap \{r = R_1\} = \emptyset$, we reach a contradiction.
\end{proof}

\begin{lemma}
    Let $A \subset [0,R] \times [t_1,t_2]$ be open (with respect to the subspace topology). Then 
    $$\mathrm{int}_{\mathbb R^2}(A) = A \setminus \partial_{\mathbb R^2} ([0,R] \times [t_1,t_2]), \quad A \subset \overline{\mathrm{int}_{\mathbb R^2}(A)}, \quad \overline{A} = \overline{\mathrm{int}_{\mathbb R^2}(A)}, \quad \partial_{\mathbb R^2} \mathrm{int}_{\mathbb R^2}(A) = \partial_{\mathbb R^2} A
    $$
    Moreover, if $A$ is connected, then $\mathrm{int}_{\mathbb R^2}(A)$ is connected as well.
\end{lemma}

\begin{proof}
     This is an elementary topological argument.
\end{proof}

\begin{definition}[Accessible points]
Let $U \subset \mathbb R^2$ be an open and connected domain. A boundary point $x \in \partial U$ is called \textit{accessible} if there exists a simple curve $\gamma: [0,1] \rightarrow \overline{U}$ for which $\gamma(0) = x$ and $\gamma((0,1]) \subset U$.
\end{definition}

\begin{lemma}[\cite{newman1964elements}, Chapter 6.4]
    The subset of accessible points $A \subset \partial U$ is dense in $\partial U$.
\end{lemma}

\begin{proof}
    Let $x \in \partial U$ and $\varepsilon > 0$ be fixed. By definition, $B(x,\varepsilon) \cap U \neq \emptyset$ and there is some $y \neq x$ in $B(x,\varepsilon) \cap U$. Consider the segment $\{y + s(x-y): s \in [0,1]\} \subset B(x,\varepsilon)$ going from $y$ to $x$. The first point $y + s(x-y)$ which intersects the boundary $\partial U$ is accessible.
\end{proof}

\begin{lemma}\label{lemma:geometry of h_1 < h < h_2}
     Let $0 \leq t_1 < t_2 < T$ and $0 < R \leq 1$. Let $(h, h_1)$ and $(h,h_2)$ be two pairs of solutions for which Lemma \ref{geometry of h_1 < h_2} applies on each pair on the square $[0,R] \times [t_1,t_2]$. If $\Gamma$ is a non-empty connected component of $\{(r,t) \in [0,R] \times [t_1,t_2]: h_1 < h < h_2\}$ and if
    $$
    \max_{(x,t) \in \overline{\Gamma}} h_1 < C < \min_{(x,t) \in \overline{\Gamma}} h_2
    $$
    for some $C \in \mathbb R$, then 
    $$
    \Gamma \cap \left[ \{r = 0\} \cup \{r = R\} \cup \{t = t_1\} \right] \neq \emptyset.
    $$
    No matter $\Gamma$, one can always choose $h$ with $h(0,t) = 0$ for all $t \in [0,T)$ satisfying the regularity from Lemma \ref{geometry of h_1 < h_2}, as well as $h_1 = \pi/2$ or $h_1 = \chi_{\alpha}(r) = \pi - 2 \arctan(\alpha^k r^k)$, $\alpha > 0$, together with $h_2 = \pi$. 
    % If $R < 1$, one can also choose $h_1$ the global solution from Theorem \ref{thm:existence of global solution} with $h_2 = \pi$.
\end{lemma}

\begin{proof}
This corresponds to \cite[Corollary 2.2]{VANDERHOUT2003}, in which the proof is omitted. We believe that the previous proof cannot be repeated and we need a more subtle topological argument in this case using the notion of accessible points.

First, observe that either $h = h_1$ or $h = h_2$ on $\overline{\Gamma} \setminus \Gamma$ (as in Lemma \ref{geometry of h_1 < h_2}).

In the case $h_1 = \chi_{\alpha}$ and $h_2 = \pi$, observe that $\mathrm{dist}(\overline{\Gamma}, \{r = 0\}) > 0$ (because $h \neq \chi_{\alpha}$ and $h \neq \pi$ on $\{r = 0\}$), meaning that 
$$
\chi_{\alpha}(r) \leq \max_{(r,t) \in \overline{\Gamma}} \chi_{\alpha}(r) < C < \min_{(r,t) \in \overline{\Gamma}} \pi = \pi, \quad (r,t) \in \overline{\Gamma}
$$
on $\overline{\Gamma}$ for some constant $C(\Gamma)$. The same holds with $h_1 = \pi/2$.

Now, suppose for a contradiction that the statement of the lemma is not true. We observe that if $h = h_1$ (resp. $h = h_2$) on the whole $\overline{\Gamma} \setminus \Gamma$, then $\Gamma$ is actually a connected component of $\{h >  h_1\}$ (resp. $\{h < h_2 \}$) and is connected to the desired boundary of the square. Hence, such a case does not happen.

% For $h_1 = \overline{h}$ the global solution from Theorem \ref{thm:existence of global solution}, one has $h_1(0,t) = 0$, $h_1(r,t) < \pi$ for all $(r,t) \in [0,1) \times [t_1,t_2]$ because $h$ is increasing with $h(r,+\infty) = \pi$ for $r \in (0,1]$ (Proposition \ref{prop:blow up infinity global solution}). Then
% $$
% \overline{h}(r,t) \leq \max_{(r,t) \in \overline{\Gamma}} \overline{h} < C < \min_{(r,t) \in \overline{\Gamma}} \pi = \pi, \quad (r,t) \in \overline{\Gamma}
% $$
% on $\overline{\Gamma} \subset [0,R] \times [t_1,t_2]$ as well if $R < 1$. 

Our previous observation allows us to assume that there is a point $(r_0,s_0)$ in $ \partial_{\mathbb R^2}\Gamma$ with $h = h_1$ and another point $(r_1,s_1)$ with $h = h_2$. Slightly perturbating these points, we can find accessible boundary points of $\mathrm{int}_{\mathbb R^2}(\Gamma)$ satisfying $h(r_0,s_0) < C < h(r_1,s_1)$. These two points can be connected through a continuous path (it does not matter whether it is injective) lying entirely in $\mathrm{int}_{\mathbb R^2}(\Gamma) \subset \{t < t_2\}$ except for those endpoints. By continuity, there exists a point $(r,t)$ on that curve for which $h(r,t) = C$ and $t < t_2$. Now, take any point $(r^*,t^*) \in \overline{\mathrm{int}_{\mathbb R^2}(\Gamma)}$ which achieves the infimum
$$
\inf\{t_1 \leq t < t_2: (r,t) \in \mathrm{int}_{\mathbb R^2}(\Gamma), h(r,t) = C\}.
$$
Then $(r^*,t^*) \notin \partial_{\mathbb R^2}  \mathrm{int}_{\mathbb R^2}(\Gamma) = \partial_{\mathbb R^2} \Gamma$. Otherwise, slightly perturbating $(r^*,t^*)$ leads to an accessible boundary point $(r,t) \in \partial_{\mathbb R^2} \Gamma$ with $h_1 < h(r,t) \approx C < h_2$, hence in $\partial_{\mathbb R^2}  \Gamma\cap \Gamma$. This means that either $r = R$ or $t = t_1$ (which contradicts our absurd hypothesis), either $0 < r < R$ and $t > t_1$, which is not possible by maximality of the (path-)connected component $\Gamma$.

Hence, $(r^*,t^*) \in  \mathrm{int}_{\mathbb R^2}(\Gamma)$. In particular, $t_2 > t^* > t_1$ and any point 
$$
 \mathrm{int}_{\mathbb R^2}(\Gamma) \cap \{t_1 \leq t < t^*\}
$$
satisfies $h(r,t) \neq C$. Take the connected component of the intersection between the open, connected domains 
$$
 \mathrm{int}_{\mathbb R^2}(\Gamma) \cap (0,R) \times (t_1,t^*),
$$
which contains $(r^*,t^*-\varepsilon)$ for any $\varepsilon > 0$ small enough.

Such a connected component is part of a larger connected component $(r^*,t^*) \in \tilde{\Gamma} \subset \Gamma$ of $\{(r,t) \in [0,R] \times [t_1,t^*]: h_1< h < h_2\}$. If 
$$
\tilde{\Gamma} \cap \left[ \{r = 0\} \cup \{r = R\} \cup \{t = t_1\} \right] \neq \emptyset,
$$
then $(r^*,t^*)$ can be connected to the desired boundary: a contradiction. Otherwise, reiterating the above argument with $\tilde{\Gamma}$ instead of $\Gamma$, we find a point $(r,t)$ in $\mathrm{int}_{\mathbb R^2}(\tilde{\Gamma}) \subset \mathrm{int}_{\mathbb R^2}(\Gamma)$ with $h(r,t) = C$ and $t_1 < t < t^*$, which contradicts the minimality of $t^*$.
\end{proof}

\begin{lemma}[\cite{VANDERHOUT2003}, Lemma 2.3]\label{lemma:intersection argument with chains mod 4}
     Let $h,h_1,h_2$ be defined on $[0,T)$ and satisfy Lemma \ref{lemma:geometry of h_1 < h < h_2} with $R = 1, t_1 = 0$, any $0 < t_2 < T$ and any $\Gamma$. Further assume that $h$ has finite energy (\ref{Energy of v in terms of h}), $h(0,t) < h_1(0,t)$, $h(1,t) > h_1(1,t)$ on $[0,T)$ and $h_1(r,t) \leq h_2(r,t)$ for all $r \in (0,1]$ and $t \in [0,T)$.

     There exists $M(h) \in \mathbb N_{\geq 1}$ such that for all $t \in [0,T)$, any sequence $0 \leq r_1 < r_2 < ... < r_N \leq 1$ with the property
     $$
     P(t,N)
     = \begin{cases}
         N = 1 \mod 4; \\
         h(r_{4j+1},t) < h_1(r_{4j+1},t); \quad h_1(r_{4j+2},t) < h(r_{4j+2},t) < h_2(r_{4j+2},t); \\
         h_2(r_{4j+3},t) < h(r_{4j+3},t); \quad  h_1(r_{4j},t) < h(r_{4j},t) < h_2(r_{4j},t);
     \end{cases}
    $$
    one has $N \leq M$. Moreover, if $M(t,h,h_{1},h_2)$ denotes the maximal length of such sequence, then $N$ is non-decreasing with respect to $t$. 
    
    If $h_1(r,t) = h_1(r,t,\alpha)$ depends on a parameter $\alpha > 0$ for which $h_1(r,t,\alpha_1) \geq h_1(r,t,\alpha_2)$ whenever $\alpha_1 \leq \alpha_2$, if the hypotheses of the Lemma for $h,h_1,h_2$ are satisfied for any fixed $\alpha > 0$ with a strict inequality $h_1(r,t,\alpha) < h_2(r,t)$ when $r \in (0,1]$, $t \in [0,T)$, then $M(t,h,h_1(\cdot,\alpha),h_2)$ is non-decreasing with respect to $\alpha$ as well. 
    
    Finally, a sequence satisfying $P(t,1)$ always exists.
\end{lemma}

\begin{proof}
We follow the proof in \cite[Lemma 2.3]{VANDERHOUT2003}. As $h(0,t) < h_1(0,t)$, a trivial sequence satisfying $P(t,1)$ always exists. As $h(1,t) > h_1(1,t)$, there is also 
    $$
    0 < r_1^*(t) = \sup \{ r \in [0,1]: h(s,t) < h_1(s,t) \ \forall s \in [0,r] \} < 1
    $$
for which $h(r_1^*(t),t) = h_1(r_1^*(t),t)$, $h(r,t) < h_1(r,t)$ on $[0,r_1^*(t))$. Hence, any sequence $r_1,...,r_N$ with $N > 1$ satisfying $P(t,N)$ must have $r_2 > r_1^*(t)$, as well as $r_N < 1$.

For $t = 0$, there exists $C = C(h) > 0$ for which
$$
 E(r_{j},h(\cdot,0)) = \pi \int_0^{r_{j}} \left( h_{r}^2 + \frac{k^2}{r^2} \sin(h)^2 \right) r dr \geq jC \quad \forall j \geq 2,
$$
if $r_1, ..., r_N$ satisfies $P(t,N)$ with $N \geq 2$. Otherwise, for all $C > 0$, there would be some $j \geq 2$ with
$$
 E(r_1^*(0),h(\cdot,0)) \leq E(r_{j},h(\cdot,0)) = \pi \int_0^{r_{j}} \left( h_{r}^2 + \frac{k^2}{r^2} \sin(h)^2 \right)  r dr \leq jC.
$$
Letting $C \to 0$, we find $E(r_1^*(0),h(\cdot,0)) = 0$ which implies that $h(r,0)$ is constant on $[0,r_1^*(0)]$. By hypothesis, $h(0,0) < h_1(0,0)$ but $h(r_1^*(0),0) = h_1(r_1^*(0),0)$, which is a contradiction. In particular, 
$$
N \leq C(h)^{-1} \sup_{t \in [0,T]}E(h(\cdot,t)) = M(h)
$$
for any such sequence.

Finally, we show that if $r_1,...,r_N$ satisfies $P(t_2,N)$ for $t_2 \in (0,T)$, then one can always find a sequence satisfying $P(t_1,N)$ with $0 \leq t_1 < t_2$, which proves that the maximal length is non-decreasing and bounded by the one for time $t = 0$. As $h(0,t) < h_1(0,t)$  and $h, h_1$ are uniformly continuous on $[0,1] \times [t_1,t_2]$, we may always replace $r_1$ in the chain by a sufficiently small value $r_1 \in (0,r_2)$ which satisfies $h(r_1,t) < h_1(r_1,t)$ for all $t \in [t_1,t_2]$.

We define first an ordering on the boundary $\{r = 0\} \cup \{r = 1\} \cup \{t = t_2\}$ of $[0,1] \times [t_2,t_1]$. We say that $(r_1,s_1) < (r_2,s_2)$ whenever $(r_2,s_2)$ 'lies on the right' of $(r_1,s_1)$ along the parabolic boundary. More precisely, $(r_1,s_1) < (r_2,s_2)$ when $r_1 < r_2$, $s_1 = s_2$ or $r_1 = r_2 = 0$, $s_1 > s_2$ or $r_1 = r_2 = 1$, $s_1 < s_2$.

By Lemma \ref{lemma:geometry of h_1 < h < h_2} applied on the square $[0,1] \times [t_1,t_2]$, for each $(r_i,t_2)$, there is an injective path $\Sigma_i$ connecting $(r_i,t_2)$ to $\xi_i \in \{r = 0\} \cup \{r = 1\} \cup \{t = t_1\}$ while preserving the strict inequality between $h$ and $h_1,h_2$ associated to $(r_i,t_2)$. Given our hypotheses on $r_1$, $\Sigma_1$ can always be chosen as a vertical path from $(r_1,t_2)$ to $(r_1,t_1)$. The important point being that $\Sigma_1$ does not intersect $\{r = 0\}$ and $\{t = t_2\}$ (except for the initial point $(r_1,t_2)$).

\begin{figure}
    \centering
    \includegraphics{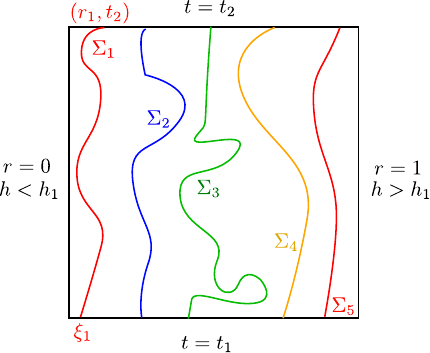}
    \caption{\label{fig:figure1} Sequence of five non-intersecting paths going from $(r_i,t_2)$ to $\xi_i \in \{t = t_1\}$}
    \label{fig:sigma_paths}
\end{figure}

Without loss of generality, one can assume that the others paths $\Sigma_i$, $i \neq 1$, do not intersect $\{t = t_2\}$ except for the initial point $(r_i,t_2)$. Indeed, the strict inequality between $h$ and $h_1,h_2$ is still satisfied at $(r_i,t)$ for all $t$ sufficiently close to $t_2$. We may connect $(r_i,t_2)$ to $(r_i,t)$ through a vertical line and then $(r_i,t)$ to the parabolic boundary through Lemma $\ref{lemma:geometry of h_1 < h < h_2}$ applied on $[0,1] \times [t_1,t]$. By choosing a subpath if needed, one can assume that $\xi_i$ is the first intersection of $\Sigma_i$ with the parabolic boundary of $[0,1] \times [t_1,t]$. In particular, $\Sigma_i \subset (0,1) \times (t_1,t_2)$ except for the endpoints $(r_i,t_2)$ and $\xi_i$. These properties are also satisfied by $\Sigma_1$.

Observe next that $(r_{4j + 1},t_2)$ can only be connected to $\{r = 0\} \cup \{t = t_1\}$ while $(r_{4j+2},t_2)$, $(r_{4j+3},t_2)$ and $(r_{4j},t_2)$ can only be connected to $\{t = t_1\} \cup \{r = 1\}$. We also know that $(r_1,t_2)$ is connected to $\xi_1 \in \{t = t_1\}$.

Clearly, two consecutive paths $\Sigma_{i}, \Sigma_{i+1}$ associated to $(r_i,t_2)$, $(r_{i+1},t_2)$ must be non-intersecting. If $C_i$ is the interior of the square enclosed by $\Sigma_i$ and the left-oriented path going from $(r_i,t_2)$ to $\xi_i$ through $\partial([0,1] \times [t_1,t_2])$, it follows from the Jordan Curve Theorem that either $\Sigma_{i+1} \subset C_i$ or $\Sigma_{i+1} \subset (\overline{C_i})^c$, which are two disjoint, open path-connected sets. As $(r_{i+1},t_2)$ lies on the right of $(r_i,t_2)$, it can only be the latter. In particular, $\xi_{i+1}$ must be in $(\overline{C_i})^c$ as well, meaning it lies on the right of $\xi_i$. By induction, $\xi_1 < \xi_2 < ... < \xi_N$. As $\xi_1 \in \{t = t_1\}$ and $\xi_N \in \{r = 0\} \cup \{t = t_1\}$ lies on the right of $\xi_1$, we must have $\xi_N \in \{t = t_1\}$. It then follows from the ordering that $\xi_i \in \{t = t_1\}$ for all $i = 1, ..., N$ and $\xi_1 < \xi_2 < ... < \xi_N$ satisfies $P(t_1,N)$ by construction.

If $h_1$ depends on a parameter $\alpha> 0$, then the maximal number $N(t,h,h_1(\cdot,\alpha),h_2)$ is well-defined for fixed $\alpha > 0$. Let $r_1,...,r_N$ satisfy $P(t,N,\alpha_2)$ (i.e., $P(t,N)$ with $h$, $h_1 = h_1(\cdot,\alpha_2)$, $h_2$) and we find a $s_1, ..., s_N$ sequence satisfying $P(t,N,\alpha_2)$ whenever $0 < \alpha_1 < \alpha_2$. Without loss of generality, $r_i > 0$ for all $i$.

As our initialization, we set $s_{4j+1} = r_{4j+1}$, $s_{4j+3} = r_{4j+3}$ for all $j$. Since
$$
h(s_{4j+1},t) < h_1(s_{4j+1},t, \alpha_2) \leq  h_1(s_{4j+1},t, \alpha_1), \ 
 h_1(s_{4j+3},t, \alpha_1) < h_2(s_{4j+3},t) < h(s_{4j+3},t),
$$
there is $s^* \in (s_{4j+1},s_{4j+3})$ maximal with $h(s^*,t) = h_1(s^*,t,\alpha_1)$, $h(s,t) > h_1(s^*,t,\alpha_1)$ on $(s^*,s_{4j+3})$. As $h_2(s^*,t) > h_1(s^*,t,\alpha_1) = h(s^*,t)$, one can set $s_{4j+2} \in (s^*,s_{4j+3})$ close enough to $s^*$ so that $h_2(s_{4j+2},t) > h_1(s_{4j+2},t,\alpha_1)$, $h_2(s_{4j+2},t) >  h(s_{4j+2},t)$ and $h(s_{4j+2},t) > h_1(s_{4j+2},t,\alpha_1)$.

Similarly, since
$$
h_1(s_{4j+3},t, \alpha_1) < h_2(s_{4j+3},t) < h(s_{4j+3},t), \ h(s_{4j+5},t) < h_1(s_{4j+5},t, \alpha_2) \leq  h_1(s_{4j+5},t, \alpha_1),
$$
there is $s^* \in (s_{4j+3},s_{4j+5})$ minimal with $h(s^*,t) = h_1(s^*,t,\alpha_1)$, $h(s,t) > h_1(s^*,t,\alpha_1)$ on $(s_{4j+3}, s^*)$. As $h_2(s^*,t) > h_1(s^*,t,\alpha_1) = h(s^*,t)$, one can set $s_{4j} \in (s_{4j+3},s^*)$ close enough to $s^*$ so that $h_2(s_{4j},t) > h_1(s_{4j},t,\alpha_1)$, $h_2(s_{4j},t) >  h(s_{4j},t)$ and $h(s_{4j},t) > h_1(s_{4j},t,\alpha_1)$. This finishes the proof.
\end{proof}

% \begin{lemma}[\cite{VANDERHOUT2003}, Lemma 2.9]\label{lemma:intersection argument with chains mod 2}
%      Let $h,h_1,h_2$ be as in Lemma \ref{lemma:geometry of h_1 < h < h_2}. Further assume that $h$ has finite energy (\ref{Energy of v in terms of h}), $h(0,t) < h_1(0,t)$, $h(1,t) > h_1(1,t)$ on $[0,T)$ and $h_1(r,t) \leq h_2(r,t)$ for all $r \in (0,1]$ and $t \in [0,T)$.

%      There exists $M(h) \in \mathbb N_{\geq 1}$ such that for all $t \in [0,T)$, any sequence $0 \leq r_1 < r_2 < ... < r_N \leq 1$ with the property
%      $$
%      Q(t,N)
%      = \begin{cases}
%          N = 1 \mod 2; \\
%          h(r_{2j+1},t) < h_1(r_{2j+1},t); \quad h_2(r_{2j},t) < h(r_{2j},t);
%      \end{cases}
%     $$
%     one has $N \leq M$. Moreover, if $M(t,h,h_{1},h_2)$ denotes the maximal length of such sequence, then $N$ is non-decreasing with respect to $t$. Finally, a sequence satisfying $Q(t,1)$ always exists.
% \end{lemma}

% \begin{proof}
% The proof is identical to the one of Lemma \ref{lemma:intersection argument with chains mod 4}.
% \end{proof}

\section{Nonexistence of multi-bubbles}
    Let $v(x,t)$ solve (\ref{heat map flow}) with smooth and $k$-equivariant boundary data, $k \geq 1$. Assume that $v$ blows-up at time $T < +\infty$. In particular, $v$ has finite energy (Remark \ref{k-equivariance implies finite energy}). Let $h$ be the corresponding inclination coordinate chosen so that $h(0,t) = 0$. We give a further description of the blow-up event using the intersection-comparison argument from \cite{VANDERHOUT2003} and prove that there can only be a single bubble in the asymptotic decomposition (\ref{multi bubble decomposition along a sequence of times}). Recall from Proposition \ref{prop:smoothness at the boundary, earlier statement} that $v(x,T)$ and $h(r,T)$ are well-defined except at the origin.

    The starting point is Proposition \ref{prop:criterion for global solution} and Corollary \ref{cor:comparison of h with itself and exact value of limsup}. From those two results, one deduces (up to replacing $h$ by $-h$) that 
    $$
    \limsup_{\substack{r \to 0^+ \\ t \to T^-}} h(r,t) = \pi, \quad \lim \limits_{r \to 0^+} h(r,T) \in \{0,\pi\},
    $$
    and that there are points with $h(r,t) > \pi$ for each $t$.
    
    The difficult part is actually showing that $\lim \limits_{r \to 0^+} h(r,T) = \pi$. Together with Corollary \ref{cor:comparison of h with itself and exact value of limsup}, one deduces Corollary \ref{cor: liminf h(r,t) = 0}, i.e.,
    $$
    \liminf_{\substack{r \to 0^+ \\ t \to T^-}}  h(r,t) = 0.
    $$
    The idea is that if $h(0^+,T) = 0$ (hence $h(r,T) \leq \pi/2$ w.l.o.g.) and if $t_0,\alpha_0$ are large enough so that the maximal length $M(t,h,h_1 = \chi_{\alpha} = \pi - 2\arctan(\alpha^k r^k),h_2 = \pi)$ of a chain $r_1(\alpha,t) < ... < r_M(\alpha,t)$ satisfying $P(t,h,h_{1} = \chi_{\alpha},h_2 = \pi,M)$ from Lemma \ref{lemma:intersection argument with chains mod 4} is minimal for $t \geq t_0, \alpha \geq \alpha_0$, then $M \geq 5$. Intuitively, this is because if we had $M = 1$, then there would be a curve $h(r(t),t) = \chi_{\alpha}(r(t))$ splitting the open square into two regions: $h \leq \chi_{\alpha}$ on the left of the curve and $h \geq \chi_{\alpha}$ on the right. One must have $\lim \limits_{t \to T^-} r(t) = 0$ as otherwise, Proposition \ref{prop:criterion for global solution} applied on a sub-rectangle of the region $h \leq \chi_{\alpha}$ would imply that $h$ is global. Thus $\{t = T\}$ must be on the right of the curve, contradicting $h(0^+,T) = 0$.
    
    Hence, one can let $\alpha^* > \alpha_0$ large enough such that $0 < \chi_{\alpha^*}(r_{M-2}(\alpha_0,t_0)) < h(r_{M-2}(\alpha_0,t_0),t_0) - \pi$ and let
   \begin{align*}
         R_1 &= \inf\{r \in (0,r_{M-2}(\alpha_0,t_0)]: h(s,t_0) - \pi > \chi_{\alpha^*}(s) \quad \forall s \in (r,r_{M-2}(\alpha_0,t_0)]\}, \\
  R_2  &= \sup\{r \in (r_{M-2}(\alpha_0,t_0),1]: h(s,t_0)  - \pi > \chi_{\alpha^*}(s) \quad \forall s \in [r_{M-2}(\alpha_0,t_0),r)\}.
   \end{align*}
   We will prove that $r_{M-2}(\alpha^*,t) \to 0$ as $t \to T^-$. This is because there must be a curve $h(r(t),t) = \chi_{\alpha^*}(r(t))$ with $r_M(\alpha^*,t) \geq r(t) \geq r_{M-2}(\alpha^*,t)$ satisfying $h(r(t)+\varepsilon,t) < \chi_{\alpha^*}(r(t) + \varepsilon)$ for $\varepsilon \leq \varepsilon(t) \ll 1$, as well as $r(t) \to 0$ as $t \to T^-$ since $h(r,T) \leq \pi/2$ on $(0,1]$. 
   
   Fix $t^*$ large enough for which $r_{M-2}(\alpha^*,t^*) < r(t^*) < R_1$. One can also redefine $r_M(\alpha^*,t^*) = r(t^*) + \varepsilon < R_1$, and then $r_{M-1}(\alpha^*,t^*) \in (r_{M-2}(\alpha^*,t^*), r_M(\alpha^*,t^*))$. Next, one constructs paths going from time $t^*$ to $t_0$ as in Lemma \ref{lemma:geometry of h_1 < h < h_2} so that one obtains a new chain $s_i(\alpha^*,t_0)$ at time $t_0$ satisfying $P(t_0,\alpha^*,M)$, while still having the other chain $r_i(\alpha^*,t_0)$ at our disposal. We will observe that $s_M(\alpha^*,t_0)$ cannot exist: $s_M \notin (R_1,R_2)$ by definition, $s_M$ cannot be larger than $R_2$ because of the inequality deduced in Corollary \ref{cor:comparison of h with itself and exact value of limsup} and $s_M$ cannot be smaller than $R_1$ as this would induce chains of length $M+4$ satisfying $P(t_0,\alpha_0,M+4)$ or $P(t_0,\alpha^*,M+4)$, which is a contradiction.

  All of this reduces the possible harmonic maps in Theorem \ref{thm: main theorem, nonexistence bubble trees} to two possibilities: harmonic maps $\omega_i$ having inclination coordinate $2\arctan( \alpha^k r^k)$ or $\pi - 2 \arctan( \alpha^k r^k)$. As one can now assume without loss of generality that $h(r,T) > \pi/2$ on $(0,1]$ and $h(1,t) > \pi/2$ for $t \in [t_0,T]$, these chains allow to find a barrier $h(r,t) \geq \chi_{\alpha_0}(r)$ when $t \geq t_0, 1 \geq r \geq r^{+}(t)$, which implies that the scale $(R_n)_{n \in \mathbb N}$ of any bubble must lie on the left of the quantity 
  $$
 r^+(t) := \sup\{r \in [0,1]: h(s,t) \leq \pi \quad \forall s \in [0,r]\}.
  $$
  For fixed $t \in [t_0,T)$, the Maximum Principle also shows that $r \mapsto h(r,t)$ must first go from $0$ to $\pi/2$ and then from $\pi/2$ to $\pi$ without going below $\pi/2$ again. This prevents $h(r,t)$ from developing two harmonic maps or a harmonic map $\pi - 2 \arctan( \alpha^k r^k)$ on the left of $r^+(t)$, which concludes the proof.

We now prove the bubble-tree result with all the necessary details.

    \begin{lemma}\label{lemma:existence of value larger than pi on parabolic boundary}
        For any $(r_0,t_0) \in (0,1) \times (0,T)$, there is $(r_1,t_1) \in \{r_0\} \times [t_0,T] \cup (0,r_0] \times \{t_0\}$ for which
        $$
        |h(r_1,t_1)| > \pi.
        $$
    \end{lemma}

    \begin{proof}
   Fix any $(r_0,t_0) \in (0,1) \times (0,T)$. After parabolic rescaling, one can assume that $\{0,r_0\} \times [t_0,T] \cup [0,r_0] \times \{t_0\}$ is the new parabolic boundary. The new boundary data is given by any extension (e.g. given by Whitney's extension theorem combined with a smooth cut-off) of 
    $$
    \tilde{h}_0(r,t) = \begin{cases}
        h(r,t) &\text{ if } (t,r) \in   \{r_0\} \times  [t_0,T],  \\
        h(r,t) &\text{ if } (t,r) \in [0,r_0] \times  \{t_0\},  \\
        h_0(r,t) = 0 &\text{ if } (t,r) \in   \{0\} \times [t_0,T],
    \end{cases}
    $$
    which is smooth on the connected set where it is defined (Proposition \ref{prop:smoothness at the boundary, earlier statement}).

   For any such $\tilde{h}_0$, as $\tilde{h}_0(0,t) = 0$, there must be a point with $|\tilde{h}_0(r_1,t_1)| = |h(r_1,t_1)| > \pi$ by Proposition \ref{prop:criterion for global solution}, $(r_1,t_1) \in \{r_0\} \times [t_0,T] \cup (0,r_0] \times \{t_0\}$. Otherwise, $h$ would be global.
   \end{proof}

As a Corollary, replacing $h$ by $-h$ and $v = (v_1,v_2,v_3)$ by $(-v_1,-v_2,v_3)$ if needed, one can assume that there is $(r_n,t_n) \to (0,T)$, $r_n \in (0,1)$, $0 < t_n \leq T$ for which 
\begin{equation}
    h(r_n,t_n) > \pi  \quad \forall n. \label{eq: sequence of h larger than pi}
\end{equation}
Slightly perturbating $t_n$, we may even assume $t_n < T$. In particular,
    \begin{equation}
        \limsup_{\substack{r \to 0^+ \\ t \to T^-}}h(r,t) = \pi \label{eq: limsup value of h}
    \end{equation}
    by Corollary \ref{cor:comparison of h with itself and exact value of limsup}. 

   \begin{lemma}\label{lemma:limit of h(r,T) at zero}
   For all $r_0 \in (0,1)$, there exists $r \in (0,r_0)$ with $h(r_0,T) > 0$. In particular, $\lim \limits_{r \to 0^+} h(r,T) \in \{0,\pi\}$.
   \end{lemma}

   \begin{proof}
   The limit at $0$ exists and is a multiple of $\pi$ by Proposition \ref{value of v and h at origin}.

   As     $$
\limsup_{\substack{r \to 0^+ \\ t \to T^-}}h(r,t) = \pi,
    $$
    the limit must be $\leq \pi$. If the first part of the lemma is proved, then it follows that the limit is either $0$ or $\pi$.

   If the lemma was wrong, there would be $r_0 \in (0,1)$ with $h(r,T) \leq 0$ for all $r \in (0,r_0]$.  By Corollary \ref{cor:comparison of h with itself and exact value of limsup}, there exists $\tau_0 = \tau_0(v,T)$ such that for all $0 < \tau \leq \tau_0$, one has
    $$
    h(r,T-\tau) - \pi \leq h(r,T), \quad r \in (0,1].
    $$    
    Let $(r_n,t_n) \to (0,T)$ with $h(r_n,t_n) > \pi$, $r_n \in (0,1)$, $t_n \in (0,T)$. Write $t_n = T - \tau_n$ and assume that $0 < \tau_n \leq \tau_0$ and $0 < r_n < r_0$ by taking $n$ large enough. Then
    $$
    \pi < h(r_n, T-\tau_n) \leq h(r_n,T) + \pi \leq \pi,
    $$
    which is a contradiction. 
   \end{proof}

   \begin{theorem}\label{thm:limit r to 0 at time T is pi}
     The limit of $h(\cdot,T)$ at zero is exactly  $\lim \limits_{r \to 0^+} h(r,T) = \pi$.
   \end{theorem}

   \begin{proof}
       Assume not. Choose $r_0$ small enough so that $|h(r,T)| < \pi/4$ on $(0,r_0]$, $h(r_0,T) > 0$ by Lemma \ref{lemma:limit of h(r,T) at zero} and $t_0$ large enough so that $0 < h(r_0,t) < \pi/4$ on $[t_0,T]$. After parabolic rescaling and translating the time, one can assume that $(r_0,t_0) = (1,0)$ for simplicity.

    Let $\alpha_0$ and $t_0$ be large enough so that the maximal number of intersections $M(t,h,h_1 = \chi_{\alpha},h_2 = \pi)$ from Lemma \ref{lemma:intersection argument with chains mod 4} is minimal when $t \geq t_0$, $\alpha \geq \alpha_0$, $\chi_{\alpha} = \pi - 2\arctan \left( (\alpha r)^k \right)$.

    By taking $\alpha_0$ even larger, one can assume that $\chi_{\alpha}(1) < h(1,t)$ for all $t \in [t_0,T]$, $\alpha \geq \alpha_0$.

    As there exists $(r_n,t_n)_{n \in \mathbb N} \to (0,T)$, $r_n \in (0,1)$, $t_n \in (0,T)$,  with $h(r_n,t_n) > \pi$, for any $t \in [t_0,T)$, one can fix $t_n$ large enough so that $t_n > t$ and $\varepsilon$ small enough so that $t_n + \varepsilon < T$. Applying Lemma \ref{geometry of h_1 < h_2}, one must have a point $h(r(t),t) > \pi$ on the parabolic boundary of $[0,1] \times [t,t_n + \varepsilon]$. 

    Then $\lim \limits_{t \to T^-} r(t) = 0$. Indeed, if $(t_n) \to T^{-}$ is any sequence and $(t_{n_m})$ is any subsequence, then there is a further subsequence (still denoted with index $n_m$) for which $r(t_{n_m}) \to r^* \in [0,1]$. If $r^* > 0$, then $h(r(t_{n_m}),t_{n_m}) \to h(r^*,T) \in [-\pi/4,\pi/4]$ by continuity on $(0,1] \times [0,T]$, which is a contradiction. Hence, $r^* = 0$.

    In particular, $M(t_0,h,h_1 = \chi_{\alpha_0},h_2 = \pi) \geq 5$. Assume not and let  $t \geq t_0, \alpha \geq \alpha_0$. As $h(r_1,t) = 0 < \chi_{\alpha}(r_1)$ near $r_1 = 0$, $\chi_{\alpha}(r_2) < h(r_2,t) < \pi$ at a point $r_2 < r^-(t)$ close enough to 
    $$
    r^-(t) = \sup\{r \in [0,1]: h(s,t) < \pi \quad \forall s \in [0,r]\} \in (0, r(t)),
    $$
    as well as $h(r_3,t) > \pi$ at $r_3 = r(t)$ and $\chi_{\alpha}(r_4) < h(r_4,t) < \pi$ at a point $r_4 > r^+(t)$ close enough to
        $$
    r^+(t) = \sup\{r \in [r(t),1]: h(s,t) > \pi \quad \forall s \in [r(t),r]\} \in (r(t),1),
    $$
    then one would have $h(r,t) \geq \chi_{\alpha}(r)$ for all $r \geq r(t) > 0$, $t \in [t_0,T)$. If $r \in (0,1]$ is fixed, then there is $t_1(r) \in [t_0,T)$ large enough for which $r \geq r(t) > 0$ for all $[t_1,T)$. Hence, $h(r,t) \geq \chi_{\alpha}(r)$ on $[t_1(r),T)$. Taking the limit $t \to T^{-}$, we find out that $h(r,T) \geq \chi_{\alpha}(r)$ for all $r \in (0,1]$, which contradicts $\lim \limits_{r \to 0} h(r,T) = 0$.

    Hence, $M(t_0,h,h_1 = \chi_{\alpha_0},h_2 = \pi)  \geq 5$. For $\alpha \geq \alpha_0$, $t \in [t_0,T)$, fix chains $r_1(\alpha,t) < ... < r_M(\alpha,t)$ satisfying $P(t,\alpha,M)$ with $M$ minimal. Consider $r_{M-2}(\alpha, t)$ for which $h(r_{M-2}(\alpha,t),t) > \pi$. Arguing as with $r(t)$ above,
    $$\lim \limits_{t \to T^{-}} r_{M-2}(\alpha,t) = 0.$$
    Let 
    $$
    z(\alpha,t) = \sup\{r \in (r_{M-2}(\alpha,t),1]: h(s,t) \geq \chi_{\alpha}(s) \quad \forall s \in [r_{M-2}(\alpha,t),r]\}.
    $$
    Note that $z(\alpha,t) < 1$ since
    $$
    r_{M-2}(\alpha,t) < z(\alpha,t) < r_{M}(\alpha,t) \leq 1.
    $$
    In particular, $h(r,t) < \chi_{\alpha}(r)$ for $(r,t)$ with all $r > z(\alpha,t)$ close enough to $z(\alpha,t)$.

    Then one must have $\lim \limits_{t \to T^-} z(\alpha,t) = 0$ for all $\alpha \geq \alpha_0$. Otherwise, there would be $t_n \to T^{-}$, $t_n > t_0$ and $z^* \in (0,1]$ for which $z(\alpha,t_n) \geq z^*$ for all $n$ so that
    $$
    h(s,t_n) \geq \chi_{\alpha}(s) \quad \forall s \in [r_{M-2}(\alpha,t_n),z^*] \quad \forall n \in \mathbb N.
    $$
    As $\lim \limits_{n \to +\infty} r_{M-2}(\alpha,t_n) = 0$, for $s \in (0,z^*]$ fixed, one deduces $s \in [r_{M-2}(\alpha,t_n), z^*]$ for all $n \geq n_0$ large enough so that
    $$h(s,t_n) \geq \chi_{\alpha}(s) \quad \forall n \geq n_0,$$
    hence $h(s,T) \geq \chi_{\alpha}(s)$ for all $s \in (0,z^*]$ by passing at the limit, which contradicts $\lim \limits_{r \to 0} h(r,T) = 0$ as before.

     It follows from Corollary \ref{cor:comparison of h with itself and exact value of limsup} that there exists $\tau_0 = \tau_0(v,T)$ such that for all $0 < \tau \leq \tau_0$, $t \in [0,T-\tau]$, one has
    \begin{equation}
        h(r,t) - \pi \leq h(r,t+\tau), \quad r \in (0,1]. \label{eq: contradiction at time t^*}
    \end{equation}
    Fix $t_0$ large enough so that $0 < T-t_0 < \tau_0$. Fix $\alpha^* > \alpha_0$ for which $0 < \chi_{\alpha^*}(r_{M-2}(\alpha_0,t_0)) < h(r_{M-2}(\alpha_0,t_0),t_0) - \pi$.

  Let 
    $$
  R_1 = \inf\{r \in (0,r_{M-2}(\alpha_0,t_0)]: h(s,t_0) - \pi > \chi_{\alpha^*}(s) \quad \forall s \in (r,r_{M-2}(\alpha_0,t_0)]\}
  $$
  and
  $$
  R_2  = \sup\{r \in (r_{M-2}(\alpha_0,t_0),1]: h(s,t_0)  - \pi > \chi_{\alpha^*}(s) \quad \forall s \in [r_{M-2}(\alpha_0,t_0),r)\}.
  $$
  One has $R_1 > 0$ so there is $t^* \in (t_0,T)$ large enough for which $0 < z(\alpha^*,t^*)< R_1$. In particular, $z(\alpha^*,t^*) < z(\alpha_0,t_0)$. By definition of $z(\alpha^*,t^*)$, there is $\varepsilon > 0$ for which  $0 < z(\alpha^*,t^*)+ \varepsilon < R_1$ and
  $$
  h(z(\alpha^*,t^*) + \varepsilon,t^*) < \chi_{\alpha^*}(z(\alpha^*,t^*) + \varepsilon).
  $$
  As 
  $$r_M(\alpha^*,t^*) > z(\alpha^*,t^*)  > r_{M-2}(\alpha^*,t^*),$$
  one can redefine $r_M(\alpha^*,t^*)  = z(\alpha^*,t^*) + \varepsilon < R_1$ and find another point $r_{M-1}(\alpha^*,t^*)$ inbetween $(r_{M-2}(\alpha^*,t^*), z(\alpha^*,t^*))$ without altering the length of the chain. As $h(0,t) < \chi_{\alpha^*}(0)$ on $[t_0,t^*]$, there is $r_1 > 0$ for which $h(r,t) < \chi_{\alpha^*}(r)$ on $[t_0,t^*] \times [0,r_1]$. Then we can replace $r_1(\alpha^*,t)$ by the same value $r_1$ for all $t \in [t_0,t^*]$.
  
  Arguing as in the proof of Theorem \ref{lemma:intersection argument with chains mod 4}, there are injective paths  $\Sigma_1, ..., \Sigma_M$ for which $\Sigma_i$ connects $(r_i(\alpha^*,t^*), t^*)$ to $(s_i,t_0)$ and $s_1 < ... < s_M$ satisfies $P(t_0,\alpha^*,M)$ as well.

  Then one must have either $s_M \leq R_1$ or $s_M \geq R_2$. In the former case, $s_M \notin [r_{M-2}(\alpha^*,t_0),z(\alpha^*,t_0)]$ as $h \geq \chi_{\alpha^*}$ on that interval at time $t_0$. One cannot have $z(\alpha^*,t_0) < s_M \leq R_1 \leq r_{M-2}(\alpha_0,t_0)$ either as $h < \chi_{\alpha^*} \leq \chi_{\alpha_0}$ at $(s_M,t_0)$, $h > \pi$ at $s_{M+2} = R_1$, $\chi_{\alpha_0} < h < \pi$ at some $s_{M+1}$ inbetween $(s_M, s_{M+1})$ and at $s_{M+3} = r_{M-1}(\alpha_0,t_0)$, as well as $h < \chi_{\alpha_0}$ at $s_{M+4} = r_M(\alpha_0,t_0)$, leading to a chain satisfying $P(t_0,\alpha_0,M+4)$, contradicting the maximality of $M$. 
  
  Hence, $s_M < r_{M-2}(\alpha^*,t_0) < z(\alpha^*,t_0)$. Then $h < \chi_{\alpha^*}$ at $(s_M,t_0)$, $h > \pi$ at $(s_{M+2},t_0) := (r_{M-2}(\alpha^*,t_0),t_0)$, $\chi_{\alpha*} < h < \pi$ for a point $s_{M+1}$ inbetween $(s_M,r_{M-2}(\alpha^*,t_0))$ and a point $s_{M+3} \in (r_{M-2}(\alpha^*,t_0),  z(\alpha^*,t_0))$, and $h < \chi_{\alpha^*}$ at $(z_{M+4},t_0) := (z(\alpha^*,t_0)+ \varepsilon',t_0)$, $\varepsilon' > 0$ small enough. Hence, we would have found a chain satisfying $P(t_0,\alpha^*,M+4)$, contradicting the maximality of $M$ once again.

  Thus, $s_M \geq R_2$. The path $\Sigma_M$ from $(r_M(\alpha^*,t^*),t^*)$ to $(s_M, t_0)$ must path through a point $(\tilde{r},\tilde{t}) \in (R_1,R_2) \times [t_0,t^*]$. 
 
 Finally, note that $t_0 \in [0,T-\tau]$ for any $0 < \tau \leq T - t_0 < \tau_0$. Then $\tilde{t} = t_0 + \tau$ for some $0 < \tau \leq T-t_0$. Evaluating (\ref{eq: contradiction at time t^*}) at $r = \tilde{r}$, $t = t_0$ and $t+\tau = \tilde{t}$ yields
 $$
 h(\tilde{r},t_0) - \pi \leq h(\tilde{r},\tilde{t}) < \chi_{\alpha^*}(\tilde{r}).
 $$
 As $\tilde{r} \in (R_1,R_2)$, this contradicts the definition of $R_1, R_2$.
    \end{proof}

\begin{corollary}\label{cor: liminf h(r,t) = 0}
The limit inferior at the singular point is exactly 
    $
\liminf \limits_{\substack{r \to 0^+ \\ t \to T^-}}h(r,t) = 0.
    $
\end{corollary}

\begin{proof}
    Assume not. As $h(0,t) = 0$, one must have
    $$
m:= \liminf_{\substack{r \to 0^+ \\ t \to T^-}}h(r,t) < 0,
    $$
    and there exists $(r_n,t_n) \to (0,T)$ with  $\lim \limits_{n \to +\infty} h(r_n,t_n) = m < 0$. Moreover, $r_n > 0$ (as $h(0,t) = 0$) and $t_n < T$ (as $\lim \limits_{r \to 0^+} h(r,T) = \pi$) for all $n$ large enough.
    
    By Corollary \ref{cor:comparison of h with itself and exact value of limsup}, there exists $\tau_0 = \tau_0(v,T)$ such that for all $0 < \tau \leq \tau_0$, one has
    $$
    h(r,T) \leq h(r,T-\tau) + \pi, \quad r \in (0,1].
    $$    
    Write $t_n = T - \tau_n$ and assume that $0 < \tau_n \leq \tau_0$ and $0 < r_n < r_0$ by taking $n$ large enough. Then
    $$
     h(r_n, T) \leq h(r_n,T-\tau_n) + \pi.
    $$
    This is a contradiction since the left-hand side converges to $\pi$, while the right-hand side converges to $m + \pi < \pi$.
\end{proof}

    Next, choose $r_0$ small enough and $t_0$ large enough so that 
    \begin{equation}
            5\pi/4 \geq h(r,t) \geq -\pi/2, \quad (r,t) \in (0,r_0] \times [t_0,T). \label{eq: bounds on h}
    \end{equation}
   Choose $r_0$ smaller if needed so that 
    \begin{equation}
            h(r,T) > \pi/2 \quad r \in (0,r_0] \label{eq: lower bound of h at t = T}
    \end{equation}
    as well and $t_0$ larger so that 
\begin{equation}
        h(r_0,t) > \pi/2 \quad t \in [t_0,T]. \label{eq: lower bound of h at r = 1}
\end{equation}
    
    After parabolic rescaling and translating the time, one can assume that $(r_0,t_0) = (1,0)$ for simplicity. 
    
    Let $t_0 > 0$ and $\alpha_0 > 0$ be large enough so that $\chi_{\alpha_0}(1) < \pi/2$ and the length $\tilde{M}$ of a chain satisfying the property $P(t,h,h_1 = \chi_{\alpha},h_2 = \pi,\tilde{M})$ from Lemma \ref{lemma:intersection argument with chains mod 4} is minimal when $t \geq t_0$, $\alpha \geq \alpha_0$.
    
    Let 
    \begin{align*}
        r^+(t) &= \sup\{r \in [0,1]: h(s,t) \leq \pi \quad \forall s \in [0,r]\}, \\
        r^-(t) &= \sup\{r \in [0,1]: h(s,t) \leq \pi/2 \quad \forall s \in [0,r]\} < 1.
    \end{align*}

\begin{lemma} \label{lemma: lower bound for h on the right of r^+}
One has
         $$
         h(r,t) \geq \chi_{\alpha_0}(1), \quad \forall t \geq t_0, \ r \geq r^+(t).
         $$
\end{lemma}

\begin{proof}
     Assume there is $t^* \geq t_0$ and $r^* \geq r^+(t^*)$ with 
     $$
h(r,t)< \chi_{\alpha_0}(1) = \inf_{r \in [0,1]} \chi_{\alpha_0}(r) < \pi/2.
     $$
    In particular, $r^+(t^*) \leq r^* < 1$ by (\ref{eq: lower bound of h at r = 1}), where the problem has been rescaled so that $r_0 = 1$.
    
    Then at time $t^*$, $h < \chi_{\alpha_0}$ near $r_1 = 0$, $h > \pi > \chi_{\alpha_0}$ near $r_3 \in (r^+(t^*),r^+(t^*) + \delta)$ with $\delta \ll 1$, $r^+(t^*) + \delta < r^*$, $\chi_{\alpha_0} < h < \pi$ for a point $r_2 \in (r_1,r_2)$, $h < \chi_{\alpha_0}$ for a point $r_5 = r^*$ and $\chi_{\alpha_0} < h < \pi$ for some $r_4 \in (r_3,r_5)$. Hence, $\tilde{M}(t,h,h_1 = \chi_{\alpha},h_2 = \pi) \geq 5$ for $t \geq t_0$ and $\alpha \geq \alpha_0$.

    For $t \geq t_0,\alpha \geq \alpha_0$, consider chains $r_1(t,\alpha) < ... < r_{\tilde{M}}(t,\alpha)$ of length $\tilde{M} \geq 5$ satisfying $P(t,h,h_1 = \chi_{\alpha},h_2 = \pi,\tilde{M})$ from Lemma \ref{lemma:intersection argument with chains mod 4}. In particular, $r^+(t) < r_3(t,\alpha) \leq 1$ for all $t\geq t_0$.

    Let also
    $$
    r^+(t) < z(t,\alpha) = \sup\{r \in [r^+(t),1]: h(s,t) \geq \chi_{\alpha}(s) \ \forall s \in [r^+(t),r]\} < r_{5}(t,\alpha) \leq 1.
    $$
    We show that $z(t,\alpha) > r_3(t,\alpha)$. Assume not. Then, there is $\rho_3 \in (r^+(t),r_3(t,\alpha))$ with $h(\rho_3,t) < \chi_{\alpha}$. There also exists $\rho_4 \in (\rho_3, r_3(t,\alpha))$ with $\chi_{\alpha}(\rho_4) < h(\rho_4,t) < \pi$ since $h(r_3(t,\alpha),t) > \pi$. Moreover, there exists $\rho_1 \in (r^+(t),\rho_3)$ close to $r^+(t)$ with $h(\rho_1,t) > \pi$ and $\rho_2 \in (\rho_1,\rho_3)$ with $\chi_{\alpha} < h  < \pi$. Hence, $r_1 < r_2 < \rho_1 < \rho_2 < \rho_3 < \rho_4 < r_3 < ... < r_{\tilde{M}}$ is chain of length $\tilde{M} + 4$, which contradicts the maximality of the chain.
    
    Then $\lim \limits_{t \to T^{-}}z(t,\alpha) = 0$. Otherwise, there is $(t_n) \to T^{-}$ with $z(t_n,\alpha) \to z^*(\alpha) > 0$. As $z^*(t_n,\alpha) < 1$, one must have $h(z(t_n,\alpha),t_n) = \chi_{\alpha}(t_n)$ for all $n$. Then 
    $$
    \chi_{\alpha}(z(t_n,\alpha)) = h(z(t_n,\alpha),t_n) \to h(z^*(\alpha),T) = \chi_{\alpha}(z^*(\alpha))
    $$
    with $h(z^*(\alpha),T) \geq \pi/2$. Hence, $z^*(\alpha) \leq \alpha^{-1}$ meaning that $\lim \limits_{\alpha \to +\infty} z^*(\alpha) = 0$. Yet, observe that $z(t_n,\alpha)$ and $z^*(\alpha)$ are increasing with respect to $\alpha$, which would imply $z^*(\alpha) = 0$ is constant with respect to $\alpha$, a contradiction.

It follows from Corollary \ref{cor:comparison of h with itself and exact value of limsup} that there is $\tau_0 > 0$ such that for all $0 < \tau \leq \tau_0$, one has
    $$
    h(r,T-\tau) - \pi \leq h(r,T), \quad r \in (0,1].
    $$
     Choose $t_0$ initially larger so that $0 < T-t_0 < \tau_0$. Fix $\alpha^* > \alpha_0$ for which $0 < \chi_{\alpha^*}(r_{3}(\alpha_0,t_0)) < h(r_{3}(\alpha_0,t_0),t_0) - \pi$.

  Letting
    \begin{align*}
          R_1 &= \inf\{r \in (0,r_{3}(\alpha_0,t_0)]: h(s,t_0) - \pi > \chi_{\alpha^*}(s) \quad \forall s \in (r,r_{3}(\alpha_0,t_0)]\}, \\
  R_2  &= \sup\{r \in (r_{3}(\alpha_0,t_0),1]: h(s,t_0)  - \pi > \chi_{\alpha^*}(s) \quad \forall s \in [r_{3}(\alpha_0,t_0),r)\},
    \end{align*}
    and arguing as in the end of the proof of Theorem \ref{thm:limit r to 0 at time T is pi}, one deduces a contradiction.
\end{proof}

\begin{corollary}\label{cor: bubbles lie on the left of r^+}
   Assume that there exists $R_n \to 0^+, T_n \to T^-$ for which $h(R_nr,T_n)$ converges to 
     $$
    h(R_nr,T_n) \to h_{\infty}(r) = m_i \pi + \varepsilon_i \cdot 2 \arctan \left( (\alpha_i r)^k \right), \quad m_i \in \mathbb Z, \varepsilon_i \in \{-1,1\}, \alpha_i > 0,
    $$
    uniformly on any compact subsets of $(0,+\infty)$. Then there exists $R_0 \in (0,+\infty)$ and $n_0 \in \mathbb N$ for which $R_0R_n \leq r^+(T_n)$ for all $n \geq n_0$.
\end{corollary}

\begin{proof}
    As 
    $$
\liminf_{\substack{r \to 0^+ \\ t \to T^-}}h(r,t) = 0, \quad \limsup_{\substack{r \to 0^+ \\ t \to T^-}}h(r,t) = \pi,
    $$
    the only possibilities for the limit $h_{\infty}$ are
    $$
 \pi -  2 \arctan\left( (\alpha r)^k \right), \quad 2 \arctan \left( (\alpha r)^k \right)
    $$
    In both cases, as $h_{\infty}((0,+\infty)) = (0,\pi)$, there must be $R_0 \in (0,+\infty)$ for which
    $$
    h(R_nR_0,T_n) \to h_{\infty}(R_0) < \chi_{\alpha_0}(1),
    $$
    where $\alpha_0$ is as in Lemma \ref{lemma: lower bound for h on the right of r^+}. For all $n$ large enough, one must have $T_n \geq t_0$ and $h(R_nR_0,T_n) < \chi_{\alpha_0}(1)$, hence $R_0R_n \leq r^+(T_n)$ by Lemma \ref{lemma: lower bound for h on the right of r^+}.
\end{proof}

% \begin{lemma}
% One has $\liminf \limits_{t \to T^-} r^+(t) = 0$.
% \end{lemma}

% \begin{proof}
%     Otherwise, there would be some $r^* > 0$ with $r^+(t) \geq r^* > 0$ for all $t \in [t_0,T]$. Hence, $h(0,t) = 0$, $h(r,t_0) \leq \pi$ on $[0,r^*]$ and $h(r^*,t) \leq \pi$ for all $t \in [t_0,T]$. The Comparison Principle (Theorem \ref{comparison principle}) implies that $h(r,t) \leq \pi$ on $[0,r^*] \times [t_0,T]$ contradicting the blow-up assumption that there is $(r_n,t_n) \to (0,T)$ with $h(r_n,t_n) > \pi$ (see the consequence following Lemma \ref{lemma:existence of value larger than pi on parabolic boundary}).
% \end{proof}
     
\begin{lemma}\label{lemma:continuity of r^+}
The maps $t \mapsto r^-(t)$, $t \mapsto r^+(t)$ are continuous on $[t_0,T)$.
\end{lemma}

\begin{proof}
The proof is exactly the same for both functions, so we only do the one for $r^{+}(t)$.

    Fix $t_1 \in [t_0,T)$ and we show continuity at $t_1$. Given $\varepsilon > 0$, there must be $r_1 \in [r^+(t_1)-\varepsilon, r^+(t_1)]$ which $h(r_1,t_1) < \pi$ as otherwise, $h(r,t_1) = \pi$ would be constant on $(0,1]$ since $r \mapsto h(r,t_1)$ is real-analytic on $(0,1]$ (Theorem \ref{thm:space analyticity}), which contradicts $h(0,t_1) = 0$. By continuity of $h$, there is $\delta > 0$ small enough for which $h(r_1,t) < \pi$ on $[t_1, t_1 + \delta]$. The Comparison Principle (Theorem \ref{comparison principle}) implies $h(r,t) \leq \pi$ on $[0,r_1] \times [t_1,t_1+\delta]$. As in the proof of Proposition \ref{prop:criterion for global solution}, applying the Maximum Principle (Theorem \ref{maximum principle}), one finds that $h < \pi$ on $(0,r_1) \times (t_1,t_1+\delta)$. Hence, $r^+(t) \geq r_1 \geq r^+(t_1) - \varepsilon$ for all $t \in [t_1,t_1+\delta]$.

If $r^+(t_1) = 1$, we are done as $r^+(t) \leq 1 = r^+(t_1)$ for all $t \in [t_1,t_1+\delta]$. Otherwise, $r^+(t_1) < 1$ and there is $r_2 \in (r^+(t_1), r^+(r_1) + \varepsilon]$ with $h(r_2,t_1) > \pi$. Then there is $\delta > 0$ small enough for which $h(r_2,t) > \pi$ on $[t_1, t_1 + \delta]$ meaning that $r^+(t) \leq r_2 \leq r^+(t_1) + \varepsilon$ on $[t_1, t_1 + \delta]$. This concludes the proof of continuity.
\end{proof}

% Taking $t_0$ large enough, assume now that $r^+(t_0) < 1$.

\begin{lemma}\label{cor:only one bubble in the left of r^+}
    For $t \in [t_{0},T)$, $h < \pi/2$ on $[0,r^{-}(t))$ and $\pi/2 < h < \pi$ on $(r^{-}(t),r^+(t))$. 
\end{lemma}

\begin{proof}
We start by proving that $h(r,t) < \pi/2$ for all $r \in (0, r^{-}(t))$, $t \in [t_0,T)$. Indeed, fix $t_1 \in [t_0,T)$, $0 < r_1 < \min\{r^{-}(t) : t \in [t_0,t_1]\}$ and denote by $E_{r_1,t_1}$ the open, connected, bounded domain enclosed by 
$$
\{r = r_1, t_0 < t < t_1\} \cup \{t \in \{t_0,t_1\}, r_1 \leq r \leq r^{-}(t)\} \cup \{(t,r^{-}(t)) : t_0 < t < t_1\}.
$$
Then $h \leq \pi/2$ on this domain by definition. As in the proof of Proposition \ref{prop:criterion for global solution}, applying the Maximum Principle (Theorem \ref{maximum principle}) to $(h-\pi/2)$ shows that $h < \pi/2$ on $E_{r_1,t_1}$. As $t_1$ is arbitrary and $r_1$ is arbitrarily small, $h(r,t) < \pi/2$ on $(0,r^{-}(t))$ for any $t \in [t_0,T)$.

Similarly, one proves that $h < \pi$ on $(0,r^{+}(t))$ by looking at $h-\pi$ and the domain $E_{r_1,t_1}$ enclosed by 
$$
\{r = r_1, t_0 < t < t_1\} \cup \{t \in \{t_0,t_1\}, r_1 \leq r \leq r^{+}(t)\} \cup \{(t,r^{+}(t)) : t_0 < t < t_1\}.
$$
Finally, one proves that $h > \pi/2$ on $(r^{-}(t),r^{+}(t))$ by looking at $\pi/2-h$ and the domain $E_{t_1}$ enclosed by 
$$
 \{(t,r^{-}(t)) : t_0 < t < t_1\} \cup \{t \in \{t_0,t_1\}, r^{-}(t) \leq r \leq r^{+}(t)\} \cup \{(t,r^{+}(t)) : t_0 < t < t_1\}.
$$
which lies at a positive distance from the line $\{r = 0\}$ (since $h(0,t) = 0$).
\end{proof}

\begin{corollary}\label{cor:end of proof}
Theorem \ref{thm: main theorem, nonexistence bubble trees} holds and $\omega_1$ has inclination coordinate of the form:
$$
\pm 2 \arctan \left( (\alpha r)^k \right) + 2\pi \mathbb Z, \quad \alpha > 0.
$$
\end{corollary}

\begin{remark}
   The sign comes from the fact that, without loss of generality, one replaces $h$, $(v_1,v_2,v_3)$ by $-h$, $(-v_1,-v_2,v_3)$ if needed to get (\ref{eq: sequence of h larger than pi}). Moreover, inclination coordinates are unique up to a multiple of $2\pi$. Up until now, we have always considered inclination coordinates chosen so that $h(0,t) = 0$.
\end{remark}
\begin{proof}
     Suppose for a contradiction that there are at least two harmonic maps (i.e. $p \geq 2$ in the statement of Theorem \ref{thm: main theorem, nonexistence bubble trees}) and order their scales so that $R_n^{(1)} \ll R_n^{(2)} \to 0^+$. Thanks to the Sobolev embedding $H^{1,2}_{loc}((0,+\infty)) \hookrightarrow C^{0,1/4}_{loc}((0,+\infty))$, $h(R_n^{(i)}r,T_n)$ converges uniformly on any compact set to a limit $H_{\infty}^{(i)} \in C^{0}((0,+\infty))$. One must have
\begin{equation*}
    \omega_i(re^{i\theta}) = \left(e^{ik \theta} \sin H_{\infty}^{(i)}(r), \cos H_{\infty}^{(i)}(r) \right)  
\end{equation*}
 by pointwise convergence $v(R_n^{(i)}x,T_n) \to \omega_i(x)$ given by the statement of Theorem \ref{thm: main theorem, nonexistence bubble trees}. But Lemma \ref{k-equivariance, smoothness of h} shows that 
\begin{equation*}
    \omega_i(re^{i\theta}) = \left(e^{ik \theta} \sin H^{(i)}(r), \cos H^{(i)}(r) \right) 
\end{equation*}
for some lifting $H^{(i)} \in C^{\infty}([0,+\infty))$. Hence, $H_{\infty}^{(i)}(r)$ and $H^{(i)}(r)$ must differ by a fixed multiple of $2\pi$ and $H_{\infty}$ is smooth as well. Repeating the ODE argument that lies beneath equation (\ref{eq: H vs H_infty, 1}) in the proof of Theorem \ref{thm:smoothness around all but finitely many points} shows that $H_{\infty}^{(i)}(r)$ is of the form (\ref{inclination coordinate of harmonic map}). In other words, there are some parameters for which
    $$
    h(R_n^{(i)}r,T_n) \to H_{\infty}^{(i)}(r) = m_i \pi + \varepsilon_i \cdot 2 \arctan \left( (\alpha_i r)^k \right), \quad m_i \in \mathbb Z, \varepsilon_i \in \{-1,1\}, \alpha_i > 0,
    $$
    uniformly on compact subsets of $(0,+\infty)$. Replacing $R_n^{(i)}$ by $R_n^{(i)}\alpha_i^{-1}$, one can assume that $\alpha_i = 1$.
    
    As 
    $$
\liminf_{\substack{r \to 0^+ \\ t \to T^-}}h(r,t) = 0, \quad \limsup_{\substack{r \to 0^+ \\ t \to T^-}}h(r,t) = \pi,
    $$
    the only possibilities for the limits $H_{\infty}^{(i)}$ are
    $$
 \pi -  2 \arctan (r^k), \quad 2 \arctan (r^k),
    $$
    where it might a priori be the case that we get the same stationary solution for both scales $R_n^{(1)}$ and $R_n^{(2)}$. 

 % One can also assume without loss of generality, up to slightly modifying $R_n^{(i)}$, that $h(R_n^{(i)},T_n) = \pi/2$ for all $n$. Indeed, observe that $h(R_n^{(i)},T_n) \to \pi/2$ and either $h(R_n^{(i)}/2,T_n) \to a_i < \pi/2$, $h(2R_n^{(i)},T_n) \to b_i > \pi/2$ or $h(R_n^{(i)}/2,T_n) \to a_i > \pi/2$, $h(2R_n^{(i)},T_n) \to b_i < \pi/2$. Hence, there is $\tilde{R}_n^{(i)} \in [R_n^{(i)}/2, 2R_n^{(i)}]$ for which $h(\tilde{R}_n^{(i)},T_n) = \pi/2$ for all $i$.

 %    As $\tilde{R}_n^{(i)} = R_n^{(i)}\kappa^{(i)}_n$ for $\kappa^{(i)}_n \in [1/2,2]$, up to passing to a subsequence, one can assume that $\kappa^{(i)}_n \to \kappa^{(i)} \in [1/2,2]$ converges.

 %    As $h(R_n^{(i)},T_n)$ converges uniformly to $h_{\infty}^{(i)}(r)$ on compact subsets of $(0,+\infty)$, $h(\tilde{R}_n^{(i)}r,T_n)$ converges uniformly to $h_{\infty}^{(i)}(r \kappa^{(i)})$ on compact subsets of $(0,+\infty)$. At $r = 1$, one must have $\pi/2$ as the limit, meaning that $\kappa^{(i)} = 1$. Replacing $R_n^{(i)}$ by $\tilde{R}_n^{(i)}$ proves the claim.
    
   Fix $\pi/2 > \varepsilon > 0$ small enough so that $\varepsilon < \sup_{r \in [0,1]} \chi_{\alpha_0}(r)$, where $\alpha_0$ is as in Lemma \ref{lemma: lower bound for h on the right of r^+}. In case of a limiting map $2\arctan(r^k)$, for $R_n^{(i)}$, one chooses $0 < r_0 < r_1 < +\infty$ so that $2\arctan(r_0^k) < \varepsilon$ and $2\arctan(r_1^k) > \pi-\varepsilon$ and $N(\varepsilon)$ large enough so that $h(r_0R_n^{(i)}, T_n) < \varepsilon$ and $h(r_1R_n^{(i)}, T_n) > \pi - \varepsilon$ for all $n \geq N(\varepsilon)$. We proceed similarly for a limiting map $\pi - 2 \arctan(r^k)$. Hence, for all $n \geq N(\varepsilon)$, $h(r,T_n)$ achieves all the values inbetween $(\varepsilon, \pi-\varepsilon)$ twice, once inside an interval $[r_0R_n^{(1)},r_1R_n^{(1)}]$ and once inside the interval $[r_2R_n^{(2)},r_3R_n^{(2)}]$. As $r_i = r_i(\varepsilon)$ does not depend on $n$, for $n$ large enough, $r_1R_n^{(1)} < r_2 R_n^{(2)}$, i.e.,  the intervals $[r_0R_n^{(1)},r_1R_n^{(1)}]$, $[r_2R_n^{(2)},r_3R_n^{(2)}]$ are disjoint. By Corollary \ref{cor:only one bubble in the left of r^+}, they cannot both lie on the left of $r^{+}(T_n)$, as otherwise, they would both contain $r^{-}(T_n)$, the only possible point for which $h(r,T_n) = \pi/2$ on $(0,r^+(T_n))$. 
   
   Hence, there is at least one interval lying on the right of $r^+(T_n)$. By Lemma \ref{lemma: lower bound for h on the right of r^+}, $h(r,T_n) \geq \chi_{\alpha_0}(1) > \varepsilon$ on $(r^{+}(T_n),1]$, i.e., $h$ cannot achieve all the values $(\varepsilon, \chi_{\alpha_0}(1)) \subset (\varepsilon, \pi-\varepsilon)$ on that side. This is a contradiction.

    Finally, we cannot have a single harmonic map $\omega_1$ with inclination coordinate $\pi - 2 \arctan( r^k)$. Indeed, if $R_n^{(1)}$ and $T_n^{(1)}$ are the corresponding sequences, then 
 we would have $[r_0R_n^{(1)},r_1R_n^{(1)}] \subset [0,r^+(T_n)]$ at $n = N(\varepsilon)$, i.e., the interval must lie on the left of $r^+$ (otherwise it would contradict Lemma \ref{lemma: lower bound for h on the right of r^+} as above), as well as $h(0, T_n) = 0$, $h(r_0R_n^{(1)},T_n) > \pi - \varepsilon > \pi/2$, $h(r_1R_n^{(1)},T_n) = \varepsilon < \pi/2$, contradicting Corollary \ref{cor:only one bubble in the left of r^+}.

\end{proof}

% \begin{proposition}
%     $t \mapsto r^{-}(t)$ is continuous on $[t_0,T]$ with $\lim \limits_{t \to T^-}r^{-}(t) = 0$.
% \end{proposition}

% \begin{proof}
%     Continuity is shown as in Lemma \ref{lemma:continuity of r^+}. If $r^{-}(T_n) \to r^* > 0$ along a sequence $t_n \to T^{-}$, then one would get $h(r^*,T) \leq \pi/2$, contradicting (\ref{eq: lower bound of h at t = T}).
% \end{proof}

% \begin{proposition}
%     The level curve $h(r^{-}(t),t) = \pi/2$ is smooth on $[t_0,T)$.
% \end{proposition}

% \begin{proof}
%     Assume for a contradicting that there is $t_1 \in [t_0,T)$ for which $\nabla_{r,t} h(r^{-}(t_1),t_1) = 0$. Then looking at the equation (\ref{heat map formulation for h(r,t)}), one deduces that $h_{rr}(r^{-}(t_1),t_1) = 0$ as well. 
    
%     Differentiating (\ref{heat map formulation for h(r,t)}) with respect to $r$, one deduces that $h_{r,t}(r^{-}(t_1),t_1) = 0$
% \end{proof}

\renewcommand{\theHsection}{A\arabic{section}}
\numberwithin{equation}{section}
\renewcommand{\theHequation}{\theHsection.\arabic{equation}}
  \appendix

\section{Calderon-Zygmund estimates ($L^p$-estimates) for the heat equation}
In this appendix, we present the so-called Calderon-Zygmund estimates for the linear heat equation. The goal is to estimate $||D^2_x u||_{L^p_{x,t}}$ in terms of $u$ and $u_t - \Delta u$. The global estimate can be proved using the generalized theory of Calderon-Zygmund operators in $\mathbb R^{n+1}$, as in \cite[Chapter 7]{LemariRieusset2002RecentDI}, which we then localize (Theorem \ref{thm:Interior L^p Estimates for the Heat Equation}). Before that, we recall the concept of distributional solution to the heat equation, which will allow us to derive the $L^p$ estimates without assuming any a priori differentiability on the solution. We also study what kind of regularity is needed on a global solution to ensures that it is given by convolution with the Gaussian heat kernel
$$
G(x,t) = \frac{1}{(4\pi t)^{n/2}}e^{-\frac{|x|^2}{4t}}.
$$
The usual criterion (given for example in \cite{evans10}) is that the solution should not grow faster (in sup norm) than some exponential. We prove that boundedness of a mixed norm $L^{\infty}_{t,x}((0,T);L^p(\mathbb R^n))$ (and other related criterions) also ensures the same result (Lemma \ref{lemma:uniqueness heat equation}).

\begin{definition}[Distributional Solution]\label{def:distributional solution}
    Let $\Omega \subset \mathbb R^{n+1}$ be open and $f \in \mathcal{D}'(\Omega)$ be a distribution. We say that $u \in \mathcal{D}'(\Omega)$ is a distributional solution to the linear heat equation
\begin{equation}
u_t - \Delta u = f(x,t), \quad (x,t) \in \Omega, \label{eq:heat equation distribution sense}
\end{equation}
if
$$
\langle u_t - \Delta u - f, \phi \rangle = 0
$$
for any test functions $\phi(x,t) \in C^{\infty}_c(\Omega)$.
\end{definition}

\begin{lemma}[Gaussian Convolution Estimates for Inhomogeneous Problem] \label{lemma:gaussian convolution estimate}
    Let $1 \leq p \leq +\infty, 0 < T < +\infty$, $f \in L^p(\mathbb R^n \times (0,T))$. The convolution with the heat kernel
   \begin{equation}
           u(x,t) = \int_{0}^t \int_{\mathbb R^n} G(y,s)f(x-y,t-s)dyds, \quad x \in \mathbb R^n, t \in [0,T) \label{gaussian convolution solution}
   \end{equation}
    is a well-defined $C^0([0,T);L^p_x(\mathbb R^n))$-function which solves the heat equation (\ref{eq:heat equation distribution sense}) in the sense of distributions. Moreover, there exists $C = C(n)$ such that 
\begin{align}
     ||u||_{L^p(\mathbb R^n \times (0,t))}&\leq C t ||f||_{L^p(\mathbb R^n \times (0,t))}, \label{Gaussian heat kernel convolution, young inequality} \\
     ||\nabla_xu||_{L^p(\mathbb R^n \times (0,t))} &\leq C \sqrt{t} ||f||_{L^p(\mathbb R^n \times (0,t))}, \label{Gaussian heat kernel convolution, young inequality gradient} 
 \end{align}
 for all $t \in (0,T]$, where $\nabla_x u$ is a distributional (hence weak) derivative, and
 \begin{equation}
     ||u(\cdot,t)||_{L^p(\mathbb R^n)} \leq Ct^{1-\frac{1}{p}}||f||_{L^p(\mathbb R^n \times (0,t))} \label{eq: weak continuity of u(t,x), bound}
 \end{equation}
 for almost every $t \in (0,T)$. If $p = +\infty$, one actually has $u(x,t) \in C^0([0,T) \times \mathbb R^n)$.
%  and, if $p \geq n/2 + 1$, one also has
% \begin{align}
%      |u(x,t)| &\lesssim t^{1-\left( \frac{n}{2} + 1 \right) \frac{1}{p}}||f||_{L^p(\mathbb R^n \times (0,T))}, \quad (x,t) \in \mathbb R^n \times (0,T) \label{eq: strong continuity of u(t,x), bound}
%  \end{align}
\end{lemma}

\begin{proof}
Observe that
$$
 u =  [G \cdot 1_{(0,T)}] *_{x,t} [f \cdot 1_{(0,T)}]
 $$
 is well-defined and measurable. It follows from Young's Convolution inequality that
 \begin{align*}
     ||u||_{L^p(\mathbb R^n \times (0,T))}&\leq 
||G||_{L^1(\mathbb R^n \times (0,T))} ||f||_{L^p(\mathbb R^n \times (0,T))}  \leq ||f||_{L^p(\mathbb R^n \times (0,T))}.
\ \end{align*}
One checks that $u(x,t)$ solves the heat equation in the sense of distributions using Fubini's theorem.
 
 As for the weak derivative, the convolution makes sense as a convolution of distributions if $f$ has compact support. Hence, we can differentiate and apply Young's inequality once again to obtain 
 $$
 ||\nabla_x u||_{L^p(\mathbb R^n \times (0,T))} \leq ||\nabla_x G||_{L^1(\mathbb R^n \times (0,T))} ||f||_{L^p(\mathbb R^n \times (0,T))}
$$
since
 $$
 ||\nabla_x G(\cdot,t)||_{L^p(\mathbb R^n)} \sim t^{-(n/2)(1-1/p)-1/2}, \quad 1 \leq p \leq +\infty.
 $$
 If $f$ does not compact support, one can approximate $f$ by a sequence of $L^p(\mathbb R^n \times (0,T))$-functions $(f_n)_{n \geq 1}$ with compact support. The sequence of corresponding solutions $(u_n)_{n \geq 1}$ converge to $u$ in $L^p$, as
 $||u_n-u||_{L^p(\mathbb R^n \times (0,T))} \lesssim T||f_n-f||_{L^p(\mathbb R^n \times (0,T))}$
 thanks to (\ref{Gaussian heat kernel convolution, young inequality}) The sequence of weak derivative is Cauchy in $L^p$ thanks to (\ref{Gaussian heat kernel convolution, young inequality gradient}) and must converge to $\nabla u$ in the sense of distributions, hence in $L^p$ as well.

Next, we prove (\ref{eq: weak continuity of u(t,x), bound}). If $f(x,t) \in C^{\infty}_c(\mathbb R^n \times (0,T))$, one can first apply Young Convolution Inequality with respect to $x$ for fixed $0 < s < t < T$,
 $$
\bigg| \bigg| \int_{\mathbb R^n}G(y,s)f(x-y,t-s)dy \bigg| \bigg|_{L^p(\mathbb R^n)} \leq ||G(\cdot,s)||_{L^{1}(\mathbb R^n)} ||f(\cdot,t-s)||_{L^p(\mathbb R^n)},
 $$
 and then for fixed $0 <t < T$, use Minkowski's inequality,
\begin{align}
   ||u(\cdot,t)||_{L^p(\mathbb R^n)} &\leq  \bigg| \bigg|\int_{0}^{T}   \int_{\mathbb R^n}G(y,s)f(x-y,t-s)1_{(0,T)}(t-s)dy  ds \bigg| \bigg|_{L^p(\mathbb R^n)}\notag  \\
   &\leq \int_{0}^{T} \bigg| \bigg| \int_{\mathbb R^n}G(y,s)f(x-y,t-s)1_{(0,T)}(t-s)dy \bigg| \bigg|_{L^p(\mathbb R^n)} ds \notag \\
    &\leq \int_0^{T} ||f(\cdot,t-s)1_{(0,T)}(t-s)||_{L^p(\mathbb R^n)} ds = \int_0^{t} ||f(\cdot,t-s)||_{L^p(\mathbb R^n)} ds \notag \\
    &\leq t^{1-\frac{1}{p}} ||f||_{L^p(\mathbb R^n \times (0,t))}. \label{eq: heat equation, inhomogeneous problem, pointwise bound for fixed time}
\end{align}
One also finds
\begin{align}
     ||u(\cdot,t_1)-u(\cdot,t_2)||_{L^p(\mathbb R^n)} \leq \int_0^{T} ||f(\cdot,t_1-s)1_{(0,T)}(t_1-s) - f(\cdot,t_2-s)1_{(0,T)}(t_2-s)||_{L^p(\mathbb R^n)} ds. \label{eq: heat equation, inhomogeneous problem, time continuity in L^p}
\end{align}

More generally, if $f(x,t) \in L^p(\mathbb R^n \times (0,T))$, one has a solution $u(x,t) \in L^p(\mathbb R^n \times (0,T))$ as before. In particular, $u(\cdot,t) \in L^p(\mathbb R^n)$ is measurable and $L^p$ for almost every $t \in (0,T)$ by Fubini. One can approximate $f$ in $L^p$ by a sequence of $C^{\infty}_c(\mathbb R^n \times (0,T))$-functions $(f_n)_{n \geq 1}$ and the sequence of corresponding solutions $(u_n)_{n \geq 1}$ converge to $u$ in $L^p$. In other words,
$$
||u_n-u||_{L^p(\mathbb R^n \times (0,T))} = \big|\big| ||u_n-u||_{L^p_x(\mathbb R^n)} \big|\big|_{L^p_t(\mathbb R^n)} \to 0,
$$
meaning that, if $p < +\infty$, along a subsequence (still denoted with index $n$), $||u_n(\cdot,t)|| \to ||u(\cdot,t)||_{L^p_x(\mathbb R^n)}$ for almost every $t \in (0,T)$. Hence, one can pass to the limit in (\ref{eq: heat equation, inhomogeneous problem, pointwise bound for fixed time}) for almost every $t$. 

Similarly, one can pass to the limit in (\ref{eq: heat equation, inhomogeneous problem, time continuity in L^p}) for almost every $0 < t_1, t_2 < T$. As the translation operator in $L^p$ is continuous, one deduces that, given $\varepsilon > 0$, there is a $\delta > 0$ such that for almost every $0 < t_1, t_2 < T$, if $|t_1-t_2| < \delta$, then 
$$
||u(\cdot,t_1)-u(\cdot,t_2)||_{L^p(\mathbb R^n)} \leq \varepsilon,
$$
which proves that $u(x,t)$ has a $C^0((0,T);L^p(\mathbb R^n))$ representative.

If $p = +\infty$, the bound (\ref{eq: weak continuity of u(t,x), bound}) follows from (\ref{Gaussian heat kernel convolution, young inequality}) and continuity follows from Dominated Convergence.

% Finally, for almost all $0 < t_1 < t_2 < T$, for almost every $x \in \mathbb R^n$, the integrals $u(x,t_1), u(x,t_2)$ are finite and
% $$
% u(x,t_2) - u(x,t_1) =  [G \cdot 1_{(0,T)}] *_{x,t} [f \cdot 1_{(0,T)}] = 
% $$

%  Similarly, if $f(x,t) \in L^p(\mathbb R^n \times (0,T))$, then $t \mapsto ||f(\cdot,t)||_{L^p(\mathbb R^n)} \in L^p((0,T))$ and for $0 < s < t < T$ fixed, one has
%  $$
% \left| \int_{\mathbb R^n}G(y,s)f(x-y,t-s)dy \right| \leq ||G(\cdot,s)||_{L^{p'}(\mathbb R^n)} ||f(\cdot,t-s)||_{L^p(\mathbb R^n)} 
%  $$
%  by Young's Convolution Inequality and then
%  \begin{align*}
%      \left| \int_0^t \int_{\mathbb R^n}G(y,s)f(x-y,t-s)dyds \right| &\lesssim \int_0^t s^{-\frac{n}{2}\left(1- \frac{1}{p'} \right)} ||f(\cdot,t-s)||_{L^p(\mathbb R^n)} ds \\
%      &\lesssim ||s^{-\frac{n}{2}\left(1- \frac{1}{p'} \right)}||_{L^{p'}([0,t])}\cdot ||f||_{L^p(\mathbb R^n \times (0,T))} \\
%      &\lesssim\left( t^{1-\frac{n}{2}\left(1- \frac{1}{p'}\right)p'} \right)^{\frac{1}{p'}} ||f||_{L^p(\mathbb R^n \times (0,T))} \\
%      &\lesssim t^{1-\left( \frac{n}{2} + 1 \right) \frac{1}{p}}||f||_{L^p(\mathbb R^n \times (0,T))} 
%  \end{align*}
%  provided that 
%  $$1-\left( \frac{n}{2} + 1 \right) \frac{1}{p} \geq 0 \iff p \geq \frac{n}{2}+1$$
%  which concludes the proof of (\ref{eq: strong continuity of u(t,x), bound}).
\end{proof}

\begin{lemma}[Gaussian Convolution Estimates for Homogeneous Problem] \label{lemma:gaussian convolution estimate, homogeneous problem}
    Let $1 \leq p \leq +\infty$, $g \in L^p(\mathbb R^n)$. The convolution with the heat kernel
   \begin{equation}
           u(x,t) = \int_{\mathbb R^n} G(y,t)g(x-y)dy, \quad x \in \mathbb R^n, t \in [0,+\infty) \label{gaussian convolution solution, homogeneous problem}
   \end{equation}
    is a well-defined $C^0([0,+\infty);L^p_x(\mathbb R^n))$-function which solves the heat equation (\ref{eq:heat equation distribution sense}) with zero forcing term $f = 0$ in the sense of distributions. Moreover, there exists $C = C(n)$ such that 
\begin{align*}
     ||u(\cdot,t)||_{L^p(\mathbb R^n)}&\leq C ||g||_{L^p(\mathbb R^n)},\\
     ||\nabla_x u(\cdot,t)||_{L^p(\mathbb R^n)}&\leq C t^{-\frac{1}{2}} ||g||_{L^p(\mathbb R^n)},
 \end{align*}
 for all $t > 0$, where $\nabla_x u$ is a distributional (hence weak) derivative.
 
Moreover, if $p < +\infty$ or if $g \in C^0_b(\mathbb R^n)$ in case $p = +\infty$, one has
 \begin{equation}
    \lim \limits_{t \to 0} ||u(\cdot,t)-g(\cdot)||_{L^p(\mathbb R^n)} = 0.
\end{equation}
\end{lemma}

\begin{proof}
    The estimates follow from Young Convolution's inequality, as in Lemma \ref{lemma:gaussian convolution estimate}. The continuity and convergence $u(x,t) \to g(x)$ as $t \to 0$ follows from the fact that $G(x,t)$ is an approximate identity.
\end{proof}

\begin{lemma}[Uniqueness for the initial-value problem] \label{lemma:uniqueness heat equation}
    Let $0 < T < +\infty$, $f \in \mathcal{D}'(\mathbb R^n \times (0,T))$, $g \in \mathcal{D}'(\mathbb R^n)$. The initial-value problem
    \begin{align*}
        u_t - \Delta u &= f(x,t), \quad (x,t) \in \mathbb R^n \times (0,T),  \\
        u(x,0) &= g(x), \quad x \in \mathbb R^n, \notag 
    \end{align*}
    has exactly one solution (in the sense of distributions) of class $L^{\infty}_{t,x}((0,T);L^p(\mathbb R^n))$ with $1 < p < +\infty$, of class $L^p(\mathbb R^n \times (0,T))$ with $1 < p < +\infty$ and $u(x,t) = g(x)$ for all $t$ close to $t = 0$, and of class $C^{0}_b(\mathbb R^n \times [0,T))$, where the initial condition is understood as follows
    $$
    \lim \limits_{t \to 0^+} \langle u(\cdot,t), \phi \rangle = \langle g, \phi \rangle \quad \forall \phi \in C^{\infty}_c(\mathbb R^n).
    $$
    If $f \in L^p(\mathbb R^n \times (0,T))$ and $g \in L^p(\mathbb R^n)$, resp. $f \in C^0_b(\mathbb R^n \times (0,T))$ and $g \in C^0_b(\mathbb R^n)$, such solution is given by convolution of $f$ and $g$ with the heat kernel
      \begin{equation}
           u(x,t) = \int_{\mathbb R^n} G(y,t)g(x-y)dy + \int_{0}^t \int_{\mathbb R^n} G(y,s)f(x-y,t-s)dyds, \quad x \in \mathbb R^n, t \in [0,T). \label{gaussian convolution solution, full problem}
   \end{equation}
% and satisfies
% \begin{align}
%     |u(x,t)| \leq ||g||_{L^{\infty}(\mathbb R^n)} + t ||f||_{L^{\infty}(\mathbb R^n \times [0,T))} 
%  \quad \forall t \in [0,T)  \label{l infty pointwise bound, gaussian heat kernel solution}
% \end{align}
\end{lemma}

\begin{proof}
We start with existence. One checks that $u(x,t) \in L^{\infty}_{t,x}((0,T);L^p(\mathbb R^n))$ given by (\ref{gaussian convolution solution, full problem}) solves the heat equation in the sense of distributions using Fubini's Theorem. As for the initial condition, write
$$
u(x,t) = \int_{\mathbb R^n} G(y,t)g(x-y)dy + \int_{0}^t \int_{\mathbb R^n} G(y,s)f(x-y,t-s)dyds = u_0(x,t) + u_1(x,t)
$$
and let $\phi \in C^{\infty}_c(\mathbb R^n)$. As in Lemma \ref{lemma:gaussian convolution estimate}, it follows from Fubini's theorem that $u_1(\cdot,t) \in L^p(\mathbb R^n)$ for almost every $t \in (0,T)$ and
$$
| \langle u_1(\cdot,t), \phi \rangle | \leq ||u_1(\cdot,t)||_{L^p(\mathbb R^n)} ||\phi||_{L^q(\mathbb R^n)} \leq C t^{1-\frac{1}{p}} ||\phi||_{L^q(\mathbb R^n)},
$$
leading to $\lim \limits_{t \to 0^+} | \langle u_1(\cdot,t), \phi \rangle | = 0$. Similarly, $\lim \limits_{t \to 0^+} \langle u_0(\cdot,t)-g, \phi \rangle  = 0$ as $u_0(\cdot,t) \to g$ in $L^p$ (resp. $C^0_b$) when $t \to 0$ (see Lemma \ref{lemma:gaussian convolution estimate, homogeneous problem}).

In the continuous case, after a change of variables,
\begin{equation}
         u(x,t) \sim \int_{\mathbb R^n} e^{-\frac{|y|^2}{4}}g(x-t^{\frac{1}{2}}y)dy + \int_{0}^t \int_{\mathbb R^n} e^{-\frac{|y|^2}{4}}f(x-s^{\frac{1}{2}}y,t-s)dyds, \quad x \in \mathbb R^n, t \in [0,T).  \label{eq:heat convolution, change of variable, continuity}
\end{equation}
     Hence, if $f$, $g$ are continuous and bounded, then one can apply Dominated Convergence to show that $u(t,x) \in C^0_b( [0,T) \times \mathbb R^n)$ as well.

     Finally, if there were two distributional solutions $u, v$ of class $L^p(\mathbb R^n \times (0,T))$, $1 < p < +\infty$, (resp. $C^0_b$) to the initial-value problem, then $w = u-v \in L^p(\mathbb R^n \times (0,T))$ (resp. $C^0_b$) would be a solution to the initial-value problem with zero forcing term and zero initial condition. This solution is $C^{\infty}(\mathbb R^n \times (0,T))$ by hypoellipticity. Now, let $\phi_{\varepsilon}(x) \in C^{\infty}_c(\mathbb R^n)$ be a standard mollifier. Then $w_{\varepsilon} = \phi_{\varepsilon} *_x w \in C^{\infty}(\mathbb R^n \times (0,T))$ is a solution to the linear heat equation which is weakly continuous at zero with $w_{\varepsilon}(x,0) = 0$. Moreover, it is continuous at zero in the classical sense with $w_{\varepsilon}(x,0) = 0$. If $u(x,t) = v(x,t) = g(x)$ for all $t$ close to $0$, then this is verified as $w_{\varepsilon}(x,t) = 0$ for all $t$ close to $0$. Else, if $u$, $v \in L^{\infty}_{t,x}((0,T);L^p(\mathbb R^n))$, then $w_{\varepsilon}$ is bounded (by Young's inequality). Moreover, for fixed $x_0 \in \mathbb R^n$ and $\varepsilon > 0$,
     \begin{align*}
         w_{\varepsilon}(x,t) &= \int_{\mathbb R^n} w(y,t) \phi_{\varepsilon}(x-y)dx = \langle w(\cdot,t), \phi_{\varepsilon}(x-\cdot) \rangle \\
         &= \langle w(\cdot,t), \phi_{\varepsilon}(x_0-\cdot) \rangle  + \langle w(\cdot,t), \phi_{\varepsilon}(x-\cdot) - \phi_{\varepsilon}(x_0-\cdot) \rangle  = A + B.
     \end{align*}
     One has $A \to 0$ as $x \to x_0, t \to 0$ since $\phi_{\varepsilon}(x_0-\cdot) \in C^{\infty}_c(\mathbb R^n)$ is a test function. Similarly, we have the inequality
     \begin{align*}
         |B| \leq ||w_{\varepsilon}||_{L^{\infty}_{x,t}(\mathbb R^n \times (0,T))}  \int_{\mathbb R^n} |\phi_{\varepsilon}(x-y) - \phi_{\varepsilon}(x_0-y)| dy,
     \end{align*}
     which also tends to $0$ as $x \to x_0, t \to 0$ by continuity of the translation operator in $L^1$.
     
     The classical theory for the heat equation (see \cite[Chapter 2.3.1, Theorem 7]{evans10}, which we apply on $\mathbb R^n \times [0,T']$ for any $0 < T' < T$) proves that $w_{\varepsilon} = 0$ is the only solution for all $\varepsilon > 0$. 
     
     As $w \in L^p(\mathbb R^n \times (0,T))$, $w(\cdot,t) \in L^p(\mathbb R^n)$ for almost every $t \in (0,T)$ by Fubini's Theorem. In particular, $w(\cdot,t) \in C^0(\mathbb R^n) \cap L^p(\mathbb R^n)$ for almost every $t$. For such $t \in (0,T)$ fixed, $w_{\varepsilon}(\cdot,t) \to w(\cdot,t)$ in $L^p(\mathbb R^n)$ which implies $w(\cdot,t) = 0$ almost everywhere. We conclude that $w = 0$ everywhere by smoothness. One proceeds in a similar fashion for the $C^0_b$ case.
\end{proof}

\begin{theorem}[Global $L^p$ Estimates for the Heat Equation]\label{thm: global lp estimate}
    Let $1 < p < +\infty$ and $0 < T < +\infty$. Assume that $u \in L^{p}(\mathbb R^n \times (0,T))$ is the solution (\ref{gaussian convolution solution}) given by convolution with the heat kernel to the linear heat equation
    \begin{align*}
    u_t - \Delta u &= f(x,t), \quad (x,t) \in \mathbb R^n \times (0,T),  \\
    u(x,0) &= 0, \quad x \in \mathbb R^n,
    \end{align*}
    where $f \in L^p(\mathbb R^n \times (0,T))$.
    
    There exists $C = C(n,p)$ (independent of $T$) for which
    $$
    ||D^2_xu||_{L^p_{x,t}(\mathbb R^n\times (0,T))} \leq C ||f||_{L^p_{x,t}(\mathbb R^n\times (0,T))},
    $$
    where $D^2_x u$ is a distributional (hence weak) derivative.
\end{theorem}

\begin{proof}
 This uses the generalized theory of Calderon-Zygmund operators. See \cite[Chapter 7]{LemariRieusset2002RecentDI}.
\end{proof}

\begin{lemma}[Interpolation Inequality] \label{interpolation lemma, sobolev version}
    For $1 \leq p < +\infty$, there exists $K = K(n,p)$ for which
    \begin{align}
        ||\nabla_x u||_{L^{p}_{x}(\mathbb R^n)} &\leq K \left[ \varepsilon  ||D^2_x u||_{L^{p}_{x}(\mathbb R^n)}  + \varepsilon^{-1} ||u||_{L^{p}_{x}(\mathbb R^n)} \right],\label{L^p norm of nabla_x u, interpolation}
    \end{align}
    for any $\varepsilon > 0$ and $u \in W^{2,p}(\mathbb R^n)$. 
    
    Similarly, assume that $u \in L^p(\mathbb R^n \times I)$, $I \subset \mathbb R$ open, has distributional (hence weak) derivatives $\nabla_x u, D^2_x u \in L^p(\mathbb R^n \times I)$. Then,
    \begin{align}
        ||\nabla_x u||_{L^{p}_{x}(\mathbb R^n \times I)} &\leq 2K \left[ \varepsilon  ||D^2_x u||_{L^{p}_{x}(\mathbb R^n \times I)}  + \varepsilon^{-1} ||u||_{L^{p}_{x}(\mathbb R^n \times I)} \right] \label{L^p norm of nabla_x u, interpolation with time}
    \end{align}
    for any $\varepsilon > 0$.

    The Lemma still holds if one replaces $\mathbb R^n$ by the half-space $\mathbb R^n \cap \{x_n > 0\}$.
\end{lemma}

\begin{proof}
    See \cite[Lemma 5.5 and Remark 5.7]{adams2003sobolev} for the first part (\ref{L^p norm of nabla_x u, interpolation}) of the theorem. 

    As for the second part, we do the half-space $\mathbb  R^{n,+} = \mathbb R_n \cap \{x_n > 0\}$ case (and the full space one is similar). If $u \in C^{\infty}(\mathbb R^{n,+} \times I)$, one can apply (\ref{L^p norm of nabla_x u, interpolation}) for fixed $t \in I$, integrate and deduce
\begin{align*}
        \int_{I}||\nabla_x u(\cdot,t)||_{L^{p}_{x}(\mathbb  R^{n,+})}^pdt &\leq K^p \int_{I} \left[ \varepsilon  ||D^2_x u(\cdot,t)||_{L^{p}_{x}(\mathbb  R^{n,+})}  + \varepsilon^{-1} ||u(\cdot,t)||_{L^{p}_{x}(\mathbb  R^{n,+})} \right]^pdt \\
         &\leq 2^p K^p \int_{I} \left[ \varepsilon  ||D^2_x u(\cdot,t)||_{L^{p}_{x}(\mathbb  R^{n,+})}  \right]^p + \left[ \varepsilon^{-1} ||u(\cdot,t)||_{L^{p}_{x}(\mathbb  R^{n,+})} \right]^pdt,
\end{align*}
which leads to (\ref{L^p norm of nabla_x u, interpolation with time}) after taking the $p$-th root. 

Write $I = (a,b)$, $a,b \in \mathbb R \cup \{\pm\infty\}$. In the general case, it suffices to approximate $u$ using a mollification. Let $u_{\varepsilon}(x,t) = \phi_{\varepsilon} *_{x,t} u(x,t)$, $(x,t) \in \mathbb R^n \cap \{x_n > \varepsilon\} \times (a+\varepsilon, b-\varepsilon)$ (note that the value of $u(x,t)$ outside $\mathbb R^{n,+} \times I$ does not matter to compute this integral) and $\phi_{\varepsilon} = \varepsilon^{-n-1} \phi(\varepsilon^{-1}x, \varepsilon^{-1}t)$ is an approximate identity defined by some $\phi \in C^{\infty}_c(B_{1/2}(0)) \subset C^{\infty}_c(\mathbb R^{n+1})$, $\phi \geq 0$, $\int_{\mathbb R^{n+1}} \phi dxdt = 1$. 

One verifies that $u_{\varepsilon} \in C^{\infty}(\mathbb R^n \cap \{x_n \geq \varepsilon\} \times [a+\varepsilon, b-\varepsilon])$ and
$$
\nabla_x u_{\varepsilon} = \phi_{\varepsilon} * \nabla_x u, \quad D^2_x u_{\varepsilon} = \phi_{\varepsilon} * D^2_x u, \quad (x,t) \in \mathbb R^n \cap \{x_n \geq \varepsilon\}  \times [a+\varepsilon, b-\varepsilon],
$$
as well as $u_{\varepsilon} \to u$, $\nabla_x u_{\varepsilon} \to \nabla_x u$, $D^2_xu_{\varepsilon} \to D^2_x u$ in $L^p_{loc}(\mathbb R^{n,+} \times I)$. Moreover, they are uniformly bounded in $L^p(\mathbb R^n \cap \{x_n \geq \varepsilon\}  \times [a+\varepsilon, b-\varepsilon])$. For any fixed compact set $S \subset \mathbb R^{n,+} \times (a,b)$, the smooth case leads to
\begin{align*}
    ||\nabla_x u||_{L^{p}_{x}(S)} &\leq  ||\nabla_x u_{\varepsilon}||_{L^{p}_{x}(S)} \leq \lim_{\varepsilon \to 0} ||\nabla_x u_{\varepsilon}||_{L^{p}_{x}(\mathbb R^{n} \cap \{x_n > \varepsilon\} \times (a+\varepsilon, b- \varepsilon))} \\
&\leq \lim_{\varepsilon \to 0} 2K \delta  ||D^2_x u_{\varepsilon}||_{L^{p}_{x}(\mathbb R^{n} \cap \{x_n > \varepsilon\} \times (a+\varepsilon, b- \varepsilon))} \\
&+ 2K\delta^{-1} ||u_{\varepsilon}||_{L^{p}_{x}(\mathbb R^{n} \cap \{x_n > \varepsilon\} \times (a+\varepsilon, b- \varepsilon))} \\
&\leq \lim_{\varepsilon \to 0} 2K \left[ \delta ||D^2_x u||_{L^{p}_{x}(\mathbb R^{n,+} \times I)}  + \delta^{-1} ||u||_{L^{p}_{x}(\mathbb R^{n,+} \times I)} \right] \\
&\leq 2K \left[ \delta  ||D^2_x u||_{L^{p}_{x}(\mathbb R^{n,+} \times I)}  +\delta^{-1} ||u||_{L^{p}_{x}(\mathbb R^{n,+} \times I)} \right], \quad \delta > 0. 
\end{align*}
It suffices to take increasing compact sets to conclude.
\end{proof}

\begin{theorem}[Interior $L^p$ Estimates for the Heat Equation] \label{thm:Interior L^p Estimates for the Heat Equation}
    For $x,t \in \mathbb R^n \times \mathbb R$ and $r > 0$, let $C(x,t;r)$ denote the open cylinder
    $$
    C(x,t;r) := \{(y,s) \in \mathbb R^n \times \mathbb R: |x-y| < r, t-r^2 < s < t\}.
    $$
    Assume that $u \in L^p(C(x,t;r))$, $1 < p < +\infty$, $\nabla_x u \in L^p(C(x,t;r))$, is a distributional solution to the linear heat equation
    \begin{align}
    u_t - \Delta u &= f(x,t), \quad (x,t) \in C(x,t;r)
    \end{align}
    with $f \in  L^p(C(x,t;r))$.
    
    There exists $C = C(n,p)$ (independent of $x$, $t$, $r$) for which
    $$
    ||D^2_x u||_{L^p_{x,t}(C(x,t;r/2))} \leq C r^{-2} \left( ||f||_{L^p_{x,t}(C(x,t;r))} + ||u ||_{L^p_{x,t}(C(x,t;r))} \right).
    $$ 
\end{theorem}

\begin{proof}
We follow the localization argument from \cite{krylov}, which is done for the Schauder estimates (instead of the $L^p$-estimates) but also works in that setting.

After translation and parabolic scaling, one can assume without loss of generality that $x = t = 0$ and $r = 1$. 
    
    We cover $C_{1} = C(0,0;1)$ with an increasing sequence of cylinders
    $$
    C_{1} = \bigcup_{l=1}^{\infty} P_l, \quad P_l = C(0,0;R_l), \quad R_l = \sum_{i=1}^l 2^{-i},
    $$
    and consider smooth cut-off functions $\chi_l(x,t) \in C^{\infty}(\mathbb R^n \times \mathbb R;[0,1])$ for which $\chi_l = 1$ on $P_l$,
    $$
    \supp(\chi_l) \subset \{(y,s) \in \mathbb R^{n+1}: |y| \leq R_{l+1}, -R_{l+1}^2 \leq s \leq R_{l+1}^2\} = -P_{l+1} \cup P_{l+1}
    $$
    and
    $$
    ||\chi_l||_{C^{2,1}_{x,t}} \lesssim 2^{2l} =: \rho^{-l}.
    $$
    Then $v_l = u\chi_l \in L^p(\mathbb R^n \times [-1,0))$ is a solution to the initial value
\begin{align*}
        \partial_t v_l - \Delta v_l &= f\chi_l +u \partial_t \chi_l + u \Delta \chi_l + 2 \nabla_x u \cdot \nabla_x \chi_l, \quad (x,t) \in \mathbb R^n \times [-1,0) \\
        v_l(x,-1) &= 0, \quad x \in \mathbb R^n
\end{align*}
which is identically zero near $t = -1$ and has a $L^p(\mathbb R^n \times [-1,0))$-forcing term. Such a solution is unique and given by convolution with the (time-translated) heat kernel (Lemma \ref{lemma:uniqueness heat equation}). Then Global $L^p$-estimates (Theorem \ref{thm: global lp estimate}) show that $D^2_x (u \chi_l) \in L^{p}$ and
\begin{align}
   ||D^2_x(u\chi_l)||_{L^p_{x,t}(P_{l+1})} &=  ||D^2_x(u\chi_l)||_{L^p_{x,t}(\mathbb R^n \times [-1,0])} \notag  \\
   &\leq C ||f\chi_l +u \partial_t \chi_l + u \Delta \chi_l + 2 \nabla_x u \cdot \nabla_x \chi_l||_{L^p_{x,t}(\mathbb R^n \times [-1,0])} \notag \\
    &= C||f\chi_l +u \partial_t \chi_l + u \Delta \chi_l + 2 \nabla_x u \cdot \nabla_x \chi_l||_{L^p_{x,t}(P_{l+1})} \notag \\
    &\leq \rho^{-l} C\left(  ||f||_{L^p_{x,t}(P_{l+1})}  + ||u||_{L^p_{x,t}(P_{l+1})}  + ||\nabla_x u||_{L^p_{x,t}(P_{l+1})}  \right) \label{lp estimates proof, upper bound on D^2 u chi}
\end{align}
for all $l \geq 1$. We introduce back $\chi_{l+1}$
    \begin{align*}
    ||D^2_x(u\chi_l)||_{L^p_{x,t}(P_{l+1})} &\leq \rho^{-l} C\left(  ||f||_{L^p_{x,t}(P_{l+1})}  + ||u||_{L^p_{x,t}(P_{l+1})}  + ||\nabla_x u||_{L^p_{x,t}(P_{l+1})}  \right)\\
    &=\rho^{-l} C\left(  ||f||_{L^p_{x,t}(P_{l+1})}  + ||u \chi_{l+1}||_{L^p_{x,t}(P_{l+1})}  + ||\nabla_x (u\chi_{l+1})||_{L^p_{x,t}(P_{l+1})}  \right) \\
    &\leq \rho^{-l} C\left(  ||f||_{L^p_{x,t}(C_1)}  + ||u \chi_{l+1}||_{L^p_{x,t}(\mathbb R^n \times (-\infty,0])}  \right. \\
    &\left. \phantom{\leq}+ ||\nabla_x (u\chi_{l+1})||_{L^p_{x,t}(\mathbb R^n \times (-\infty,0])}  \right) 
    \end{align*}
    and from now on, $C = C(n,p) > 1$ is fixed and $\tilde{C} = \tilde{C}(n,p) > 1$ denotes the constant from the interpolation Lemma \ref{interpolation lemma, sobolev version}. One can apply Lemma \ref{interpolation lemma, sobolev version} with $\varepsilon = (2C\tilde{C})^{-1}\rho^2 \rho^l$ and obtain
   \begin{align*}
||D^2_x(u\chi_l)||_{L^p_{x,t}(P_{l+1})}  &\leq \rho^{-l} 4C\tilde{C}\left( ||f||_{L^p_{x,t}(C_1)}   + \varepsilon ||D^2_x (u\chi_{l+1})||_{L^p_{x,t}(\mathbb R^n \times (-\infty,0])}  \right. \\
& \left.\phantom{\leq}+ \varepsilon^{-1} ||u \chi_{l+1}||_{L^p_{x,t}(\mathbb R^n \times (-\infty,0])}  \right) \\
    &\leq \rho^{-l} 4C \tilde{C} \left( ||f||_{L^p_{x,t}(C_1)}  + \varepsilon ||D^2_x (u\chi_{l+1})||_{L^p_{x,t}(P_{l+2})} + \varepsilon^{-1} ||u||_{L^{p}_{x,t}(C_1)} \right) \\
    &\leq \rho^{-2l} \underbrace{(4C \tilde{C})^2}_{=: \hat{C}} \left( ||f||_{L^p_{x,t}(C_1)}  + ||u ||_{L^p_{x,t}(C_1)} \right) + \frac{\rho^2}{2}|| D^2_x (u\chi_{l+1})||_{L^p_{x,t}(P_{l+2})}.
  \end{align*}
  Starting with $l = 1$ and iterating, we obtain
\begin{align}
||D^2_x(u\chi_1)||_{L^p_{x,t}(P_{2})} &\leq  \rho^{-2}\hat{C} \left( ||f||_{L^p_{x,t}(C_1)}  + ||u ||_{L^p_{x,t}(C_1)} \right)  + \frac{\rho^2}{2}|| D^2_x (u\chi_{2})||_{L^p_{x,t}(P_{3})} \notag \\
&\leq \left( \rho^{-2}\hat{C} + \frac{\rho^2}{2} \cdot \rho^{-4}\hat{C} \right) \left( ||f||_{L^p_{x,t}(C_1)}  + ||u ||_{L^p_{x,t}(C_1)} \right)  \notag \\
&\phantom{\leq}+ \left(\frac{\rho^2}{2}\right)^2 || D^2_x (u\chi_{3})||_{L^p_{x,t}(P_{4})} \notag \\
&\leq \rho^{-2} \hat{C} \left( 1 + \frac{1}{2} + ... + \frac{1}{2^{k-2}} \right) \left( ||f||_{L^p_{x,t}(C_1)}  + ||u ||_{L^p_{x,t}(C_1)} \right) \notag \\
&\phantom{\leq}+ \left(\frac{\rho^2}{2}\right)^{k-1} || D^2_x (u\chi_{k})||_{L^p_{x,t}(P_{k+1})} \label{lp estimate proof: final inequality}
  \end{align}
  for any $k > 1$. Since $\rho < 1$ and
\begin{align*}
      || D^2_x (u\chi_{k})||_{L^p_{x,t}(P_{k+1})}  &\leq \rho^{-k} C \left(  ||f||_{L^p_{x,t}(C_{1})}  + ||u||_{L^p_{x,t}(C_{1})}  + ||\nabla_x u||_{L^p_{x,t}(C_{1})}  \right)  
\end{align*}
by  (\ref{lp estimates proof, upper bound on D^2 u chi}), we can let $k \to +\infty$ in (\ref{lp estimate proof: final inequality}). Observing that $u = u \chi_1$ on $C(0,0;1/2) \subset P_2$, this concludes the proof. 

We note that we only need an a priori estimate on $u, \nabla_x u$ but not on $D^2_x u$ to take $k \to +\infty$. The same will be true for the Schauder estimates (Theorem \ref{interior schauder estimates}).
\end{proof}

We prove a last lemma which will be useful in the context of the Parabolic Sobolev Embedding.

\begin{lemma}[Global $L^p$ Estimates for compactly supported solutions of parabolic equations]\label{lemma:global l^p estimates for parabolic operator, compact support}
    Let $\alpha \in (0,1)$, $0 < T < +\infty$ and $1 < p < +\infty$. Consider a parabolic operator
    $$
    L = c(x,t) + \sum_{i=1}^n b^i(x,t) D_{x_i} + \sum_{i,j=1}^n a^{i,j}(x) D_{x_i}D_{x_j}
    $$
    defined on $\overline{\Omega}$, $\Omega = B_{\kappa}(x_1) \times (0,T)$, satisfying
    \begin{align*}
    ||a,b,c||_{\alpha,\alpha/2;\overline{\Omega}} &\leq \Lambda, \quad \Lambda > 0, \\
        \Lambda^{-1}|\xi|^2 \leq \sum_{i,j=1}^n a^{i,j}(x,t)\xi_i\xi_j &\leq \Lambda |\xi|^2, \quad (x,t) \in \overline{\Omega}, \xi \in \mathbb R^n,
    \end{align*}
    and where $a^{i,j}(x)$ does not depend on time.
    
    There exists $\kappa_0 = \kappa_0(n,\alpha,\Lambda) > 0$ and $C = C(n,p,\alpha,\Lambda) > 0$ such that if $\kappa \leq \kappa_0$, for any $u \in L^{\infty}_{t,x}((0,T);L^p(\mathbb R^n))$ with weak derivatives $\nabla_x u, (\partial_t - L)u, D^2_xu \in L^p(B_{\kappa}(x_1) \times (0,T))$, that is compactly supported in the sense that $\overline{\supp(u(\cdot,t))} \subset B_{\kappa}(x_1)$ for all $t \in (0,T)$ and which is weakly continuous at zero in the sense that
    $$
    \lim \limits_{t \to 0^+} \langle u(\cdot,t), \phi \rangle = 0 \quad \forall \phi \in C^{\infty}_c(\mathbb R^n),
    $$
    then one has
    $$
    ||D^2_x u||_{L^p(\mathbb R^n \times (0,T))} \leq C\left( ||(\partial_t - L)u||_{L^p(\mathbb R^n \times (0,T))} + ||u||_{L^p(\mathbb R^n \times (0,T))} \right).
$$
\end{lemma}

\begin{proof}
    Without loss of generality, we translate the problem and assume that $x_1 = 0$. We also assume that $b^i = c = 0$ and only deal with the second order terms of $L$, as one can use Interpolation (Lemma \ref{interpolation lemma, sobolev version}) with a sufficiently small $\varepsilon$ for the lower order terms. 
    
    Write
    $$
\partial_t u - \sum_{i,j=1}^n a^{i,j}(x) D_{x_i}D_{x_j}u = \partial_t u - \sum_{i,j=1}^n a^{i,j}(x_1) D_{x_i}D_{x_j}u + \sum_{i,j=1}^n \left[ a^{i,j}(x_1) - a^{i,j}(x) \right] D_{x_i}D_{x_j}u.
    $$
    Let $P$ be an orthogonal matrix for which $P^T (a_{i,j}(x_1)) P = I_n$. Let $v(x,t) = u(Px,t)$ (in other words, we are rotating $\mathbb R^n$). Then 
    $$(\partial_t - \Delta)v = \partial_t u - \sum_{i,j=1}^n a^{i,j}(x_1) D_{x_i}D_{x_j}u = (\partial_t - L_0)u.
    $$
    As
\begin{align*}
    ||(\partial_t - L_0)u||_{L^{p}_{x,t}(\mathbb R^n \times (-\infty,0])} &= ||(\partial_t - L_0)u||_{L^{p}_{x,t}(\Omega)} \leq ||(\partial_t - L)u||_{L^{p}_{x,t}(\Omega)}  \\
    &\phantom{=}+ \bigg|\bigg| \sum_{i,j=1}^n \left[ a^{i,j}(x_1) - a^{i,j}(x) \right] D_{x_i}D_{x_j}u \bigg|\bigg|_{L^{p}_{x,t}(\Omega)}  \\
    % &\lesssim  ||(\partial_t - L)u||_{L^{p}_{x,t}(\Omega)}  + \sum_{i,j=1}^n ||a^{i,j}(x_1) - a^{i,j}(x)||_{L_{x,t}^{\infty}(\Omega)} \cdot ||D^2_x u||_{L^{p}_{x,t}(\Omega)} \\ 
    &\lesssim  ||(\partial_t - L)u||_{L^p(\mathbb R^n \times (-\infty,0])} + \kappa^{\alpha} \Lambda ||D^2_x u||_{L^p(\mathbb R^n \times (-\infty,0])},
 \end{align*}
 the function $v \in L^{\infty}((0,T);L^p_x(\mathbb R^n))$ must be the unique solution to the heat equation
     \begin{align*}
    v_t - \Delta v &= (\partial_t - L_0)u \in L^p(\mathbb R^n \times (0,T)),  \\
    v(x,0) &= 0, \quad x \in \mathbb R^n,
    \end{align*}
    given by convolution with the heat kernel (Lemma \ref{lemma:uniqueness heat equation}). Note that $D^2_xu $, $(\partial_t - L_0) u$ and $(\partial_t - L)u$ were naturally extended as zero for $x \in B_{1}(0)^c$. Hence one can estimate
\begin{align*}
        ||D^2_x u||_{L^p(\mathbb R^n \times (-\infty,0])} 
        &\lesssim ||D^2_x v||_{L^p(\mathbb R^n \times (-\infty,0])} \\
        &\leq C\left( ||(\partial_t - \Delta)v||_{L^{p}_{x,t}(\mathbb R^n \times (-\infty,0])} + ||v||_{L^{p}_{x,t}(\mathbb R^n \times (-\infty,0])} \right) \\
    &\lesssim C\left( ||(\partial_t - L_0)u||_{L^{p}_{x,t}(\mathbb R^n \times (-\infty,0])} + ||u||_{L^{p}_{x,t}(\mathbb R^n \times (-\infty,0])} \right) 
\end{align*}
    using Theorem \ref{thm: global lp estimate}. Choose $\kappa$ small enough and the theorem is proved.
\end{proof}

% \begin{corollary}[Interior $L^p$ Estimates for  parabolic equations with a priori estimates] \label{interior L^p estimates for parabolic eq with a priori estimates}
% Assume that $u \in L^p(C(x,t;r))$, $1 < p < +\infty$, $\nabla_x u, D^2_x u\in L^p(C(x,t;r))$, is a distributional solution to the linear heat equation
% \begin{equation}
%     u_t - L u = f(x,t), \quad (x,t) \in C(x,t;r) 
% \end{equation}
% where $f \in L^p_{x,t}(C(x,t;r))$ and $L$ is a parabolic operator satisfying the hypotheses from Lemma \ref{lemma:global l^p estimates for parabolic operator, compact support} on $\overline{C(x,t;r)}$ for some constant $\alpha \in (0,1)$ and $\Lambda > 0$.

% There exists $C = C(n, p, \alpha, \Lambda)$ (independent of $x,t,r$) for which
% \begin{align*}
%     ||D^2_x u||_{L^p(C(x,t;r/2))} &\leq C\left(1+\frac{1}{r^{2}}\right)\left( ||f||_{L^p(C(x,t;r))} + ||u||_{L^{p}_{x,t}(C(x,t;r))} \right) 
% \end{align*}
% \end{corollary}

% \begin{proof}
%     Using a parabolic rescaling, it suffices to prove the result with $r = \kappa_0$, where $\kappa_0$ is an in Lemma \ref{lemma:global l^p estimates for parabolic operator, compact support} (similarly the proof of Corollary \ref{cor: local parabolic sobolev embedding}). Then, one proceeds as in the proof of Corollary \ref{thm:Interior L^p Estimates for the Heat Equation} using cut-off functions.
% \end{proof}

\section{Parabolic Sobolev Embedding}
In this section, we prove a version of the Sobolev-Morrey embedding into a $C^{1+\alpha, \alpha/2}_{x,t}$ Hölder space from a parabolic Sobolev space where $u$, $\nabla_x u$, $D^2_x u$, $\partial_t u \in L^p(\Omega \times [0,T])$ and $p$ is large enough (Theorem \ref{thm:parabolic sobolev embedding}). The classical reference for this result is \cite[Lemma 3.3, Chapter 2]{ladyzhenskaia1968linear}. However, after verification, the proof is not in this book and could not be found elsewhere. I thank my advisor, Pr. Joachim Krieger, for the proof of Theorem \ref{thm:parabolic sobolev embedding} which uses a simple Littlewood-Paley decomposition. We then generalize this estimate to the half-space setting with Dirichlet boundary condition by reflection, we localize the estimate (Corollary \ref{cor: local parabolic sobolev embedding}) and finally obtain it for general bounded domains with smooth boundary by straightening the boundary and  freezing the coefficients of the newly obtained parabolic operator (Theorem \ref{thm:boundary parabolic sobolev embedding}).

\begin{lemma}[\cite{ladyzhenskaia1968linear}, Chapter 2, Lemma 3.1] \label{lemma:ladyzenskaia reduction Holder regularity} 
    Let $\Omega \subset \mathbb R^{n}$ be an open set satisfying the interior cone condition with a cone $C$ of height $h$. Let $I \subset \mathbb R$ be any open interval.

    Assume that $u \in C^0_b(\Omega \times I)$ is $\alpha$-Hölder-continuous with respect to $t$, meaning that 
    $$
    |u(x,t)-u(x,t')| \leq \mu_1 |t-t'|^{\alpha} \quad \forall x \in \Omega, \forall t,t' \in I.
    $$
    Further assume that $\nabla_x u \in C^0_b(\Omega \times I)$ is $\beta$-Hölder continuous with respect to $x$, meaning that 
    $$
    |\nabla_x u(x,t)- \nabla_x u(y,t)| \leq \mu_2 |x-y|^{\beta} \quad \forall x,y \in \Omega, \forall t \in I.
    $$
    
    Then $\nabla_x u$ is $\frac{\alpha \beta}{1+\beta}$-Hölder continuous with respect to $t$ with Hölder constant
    $$
    |\nabla_x u(x,t)- \nabla_x u(x,s)| \leq 2\left( \mu_1+\mu_2+ h^{-\beta}||\nabla_x u||_{L^{\infty}(\Omega \times I)} \right) |t-s|^{\frac{\alpha \beta}{1+\beta}} \quad \forall t,s \in I, \forall x \in \Omega.
    $$
    In particular, if $\Omega = \mathbb R^n$ or $\Omega = \mathbb R^n \cap \{x_n > 0\}$, $h$ can be taken arbitrarily large so that 
    $$
    |\nabla_x u(x,t)- \nabla_x u(x,s)| \leq 2\left( \mu_1+\mu_2 \right) |t-s|^{\frac{\alpha \beta}{1+\beta}} \quad \forall t,s \in I, \forall x \in \Omega.
    $$
\end{lemma}

 \begin{proposition}[Extension by reflection]\label{prop:extension by reflection}
     Let 
     $$
     u \in C^{2,1}(\mathbb R^n \cap \{x_n > 0\} \times I) \cap L^p(\mathbb R^n \cap \{x_n > 0\} \times I),$$
     $I \subset \mathbb R$ open, $1 \leq p \leq +\infty$. Assume that $\lim \limits_{\varepsilon \to 0^+}u(x,\varepsilon,t) \to 0$ in $L^p_{loc}(\mathbb R^{n-1} \times I)$ and $\lim \limits_{\varepsilon \to 0^+} \nabla_x u (\tilde{x},\varepsilon,t)$ converges to a limit in $L^p_{loc}(\mathbb R^{n-1} \times I)$. Denote by $\nabla_x u(\tilde{x},0,t)$ its limit. Let 
    $$
\tilde{u}(\tilde{x},x_n,t) = \begin{cases} u(\tilde{x},x_n,t), \quad &x_n > 0, \\
- u(\tilde{x},- x_n,t), \quad &x_n < 0, \\
\end{cases}
$$
be the odd reflection of $u$. 

If the distribution $v \in \{\partial_{x_i} u, \partial_{x_i}\partial_{x_j}u, \partial_t u, (\partial_t - \Delta)u, \partial_{x_n}^2 u: i,j \neq n\}$ belongs to  $L^p(\mathbb R^n \cap \{x_n > 0\} \times I)$, then the corresponding distribution $\tilde{v} \in \{\partial_{x_i} \tilde{u}, \partial_{x_i}\partial_{x_j}\tilde{u}, \partial_t \tilde{u}, (\partial_t - \Delta)\tilde{u}, \partial_{x_n}^2 \tilde{u}: i,j \neq n\}$ is given by the odd reflection of $v$. If $v \in \{\partial_{x_n} u, \partial_{x_i}\partial_{x_n} u: i \neq n\}$ belongs to $ L^p(\mathbb R^n \cap \{x_n > 0\} \times I)$, then the corresponding distribution $\tilde{v}$ is given by the even reflection of $v$ (i.e., $\tilde{v}(\tilde{x},x_n,t) = v(\tilde{x},-x_n,t)$ for $x_n < 0$).

Additionally, if $u \in C^{2+\alpha,1+\alpha/2}(\mathbb R^n \cap \{x_n \geq 0\} \times I)$ for $\alpha \in [0,1)$, then $\tilde{u} \in  C^{2+\alpha,1+\alpha/2}(\mathbb R^n \cap \{x_n \geq 0\} \times I)$ as well.

The same result holds if one replaces $\mathbb R^n \cap \{x_n > 0\}$ by a half-ball $B_r^+(x_0) = B_r(x_0) \cap \{x_n > 0\}$ with $(x_0)_n = 0$ and one assumes $\lim \limits_{\varepsilon \to 0^+} u(\tilde{x},\varepsilon,t) \to 0$ in $L^p_{loc}(B_{r}(\tilde{x}_0) \times I)$, where $x_0 = (\tilde{x}_0, (x_0)_n)$, and $\lim \limits_{\varepsilon \to 0^+}  \nabla_x u (\tilde{x},\varepsilon,t)$ converges in $L^p_{loc}(B_r(\tilde{x}_0) \times I)$.
 \end{proposition}

 \begin{proof}
      First, $\tilde{u} \in L^p(\mathbb R^n \times I)$. Next, let $\phi \in C^{2,1}_c(\mathbb R^{n} \times I)$. Write $(\partial_t - \Delta) u = f \in L^p(\mathbb R^n \cap \{x_n > 0\} \times I)$ and $\tilde{f}$ its odd reflection to the whole plane. For $\varepsilon > 0$, integrating by parts shows
      \begin{align*}
   \int_{\mathbb R^n \cap \{x_n > \varepsilon\} \times I} u \partial_{x_i}^{\alpha_i}\partial_{x_j}^{\alpha_j} \phi dxdt &= (-1)^{\alpha_i + \alpha_j}\int_{\mathbb R^n \cap \{x_n > \varepsilon\} \times I} \partial_{x_i}^{\alpha_i}\partial_{x_j}^{\alpha_j}u \phi dxdt, \quad i,j \neq n, \alpha_k \in \{0,1\},  \\
       \int_{\mathbb R^n \cap \{x_n > \varepsilon\} \times I} u \partial_{t} \phi dxdt &= -\int_{\mathbb R^n \cap \{x_n > \varepsilon\} \times I} \partial_{t}u \phi dxdt, \\
         \int_{\mathbb R^n \cap \{x_n > \varepsilon\} \times I} u \partial_{x_i}^{\alpha_i}\partial_{x_n} \phi dxdt &= (-1)^{\alpha_i}\int_{\mathbb R^{n-1} \times I} \partial_{x_i}^{\alpha_i} u(\tilde{x},\varepsilon,t) \phi(\tilde{x},\varepsilon,t) d\tilde{x}dt \\
         &+ (-1)^{1+\alpha_i} \int_{\mathbb R^n \cap \{x_n > \varepsilon\} \times I} \partial_{x_i}^{\alpha_i} \partial_{x_n}u \phi dxdt, \quad i \neq n, \alpha_i \in \{0,1\},
\end{align*}
as well as
\begin{align*}
     \int_{\mathbb R^n \cap \{x_n > \varepsilon\} \times I} u \partial_{x_n}^2 \phi dxdt &= \int_{\mathbb R^{n-1} \times I} u(\tilde{x},\varepsilon,t) \partial_{x_n}\phi(\tilde{x},\varepsilon,t) d\tilde{x}dt \\
     &- \int_{\mathbb R^{n-1} \times I} \partial_{x_n} u(\tilde{x},\varepsilon,t) \phi(\tilde{x},\varepsilon,t) d\tilde{x}dt +\int_{\mathbb R^n \cap \{x_n > \varepsilon\} \times I} \partial_{x_n}^2u \phi dxdt. \\
\end{align*}
As $u(\tilde{x},\varepsilon,t) \to 0$ in $L^p(\mathbb R^{n-1} \times I)$ and $\phi$ is $C^{2,1}$ with compact support, taking the limit $\varepsilon \to 0$ shows
\begin{align*}
    \int_{\mathbb R^n \cap \{x_n > 0\} \times I} u (-\partial_t - \Delta) \phi dxdt &= \int_{\mathbb R^{n-1} \times I} \partial_{x_n} u(\tilde{x}, 0, t)  \phi(\tilde{x},0,t) d\tilde{x}dt \\
    &+ \int_{\mathbb R^n \cap \{x_n > 0\} \times I } f \phi dx dt
\end{align*}
if $f = (\partial_t - \Delta)u \in L^p(\mathbb R^n \cap \{x_n > 0\} \times I)$.

Denoting $\tilde{\phi}(\tilde{x},x_n,t) = -\phi(\tilde{x},-x_n,t)$, then
\begin{align*}
    \langle (\partial_t - \Delta) \tilde{u}, \phi \rangle &=  \langle \tilde{u}, (-\partial_t - \Delta) \phi \rangle \\
    &= \int_{\mathbb R^n \cap \{x_n > 0\} \times I} u (-\partial_t - \Delta) \phi dxdt + \int_{\mathbb R^n \cap \{x_n > 0\} \times I} u (-\partial_t - \Delta) \tilde{\phi} dxdt \\
    &= \int_{\mathbb R^n \cap \{x_n > 0\} \times I } f \cdot (\phi + \tilde{\phi}) dx dt = \int_{\mathbb R^n \times I } \tilde{f} \cdot \phi dx dt,
 \end{align*}
where $\tilde{f} \in L^p(\mathbb R^n \times I)$ is the odd reflection of $f$.  We proceed similarly for the other derivatives.

If $u, \nabla_x u, \partial_tu, D^2_x u$ are continuous up to the boundary $x_n = 0$, then differentiating the boundary condition yields $\partial_t u = \partial_{x_i} u = \partial_{x_i}\partial_{x_j} u = 0$ on $\{x_n = 0\}$ for $i,j \in \{1,...,n-1\}$. It also follows from $(\partial_t - \Delta) u = 0$ that $\partial_{x_n}^2 u = 0$ on $\{x_n = 0\}$. We also have continuity of $\partial_{x_n} u$, $\partial_{x_i}\partial_{x_n} u$, $i \in \{1,...,n-1\}$, on $\{x_n = 0\}$ and it does not matter whether the value is zero or not as they are extended using an even reflection. Hence, $\partial_t \tilde{u}, \nabla_x \tilde{u}, D^2_x \tilde{u}$, which are of class $C^{0}(\mathbb R^n \setminus \{x_n = 0\} \times (-\infty,0])$, extend continuously on $\{x_n = 0\}$. If $D^2_x u \in C^{\alpha}(\mathbb R^n \cap \{x_n > 0\} \times (-\infty,0])$, then the only issue for the Hölder regularity of $\partial_{x_i} \partial_{x_j} \tilde{u}$ (and similarly for $\partial_t \tilde{u}$) is when one has two points $x, y \in \mathbb R^n$ on opposite half-planes $y_n < 0 < x_n$. 

If $\partial_{x_i} \partial_{x_j} \tilde{u}$ is an even reflection, then
$$
|\partial_{x_i} \partial_{x_j} \tilde{u}(\tilde{x},x_n,t) - \partial_{x_i} \partial_{x_j} \tilde{u}(\tilde{y},y_n,t)| = |\partial_{x_i} \partial_{x_j} \tilde{u}(\tilde{x},x_n,t) - \partial_{x_i} \partial_{x_j} \tilde{u}(\tilde{y},-y_n,t)|
$$
and $(\tilde{x},x_n),(\tilde{y},-y_n)$ are on the same half-plane. If it is an odd reflection, then $\partial_{x_i} \partial_{x_j} \tilde{u} = 0$ on $\{x_n = 0\}$ and
\begin{align*}
|\partial_{x_i} \partial_{x_j} \tilde{u}(\tilde{x},x_n,t) - \partial_{x_i} \partial_{x_j} \tilde{u}(\tilde{y},y_n,t)| &\leq |\partial_{x_i} \partial_{x_j} \tilde{u}(\tilde{x},x_n,t) - \partial_{x_i} \partial_{x_j} \tilde{u}(\tilde{x},0,t)| \\
&+ |\partial_{x_i} \partial_{x_j} \tilde{u}(\tilde{y},y_n,t) - \partial_{x_i} \partial_{x_j} \tilde{u}(\tilde{y},0,t)| \\
&\leq  [D^2_xu]_{C^{\alpha}} \cdot \left( |x_n|^{\alpha} + |y_n|^{\alpha} \right) \\
&\lesssim |x_n - y_n|^{\alpha} \lesssim |x-y|^{\alpha}
\end{align*}
as $y_n < 0 < x_n$.

The same argument works for a ball by symmetry of the domain. The integration by parts have no terms on $\partial B_r(x_0)$ because the test function must be compactly supported inside the ball.
 \end{proof}
 
\begin{theorem}[Global Parabolic Sobolev Embedding]\label{thm:parabolic sobolev embedding}
   Let $n+2 < p < +\infty$ and $I \subset \mathbb R$ be any open interval. Assume that $u(x,t) \in L^{p}(\mathbb R^{n} \times I)$ has weak derivatives $\nabla_x u, (\partial_t - \Delta) u \in L^p(\mathbb R^{n} \times I)$. Then $u \in C^{1+2\gamma,\gamma}_{x,t}(\mathbb R^{n}\times I)$ where
    $$
    0 < \gamma < \frac{(p-n-1)(p-n-2)}{p(2p-n-2)}.
    $$
    Moreover, there exists $C = C(n,p,\gamma)$ for which
\begin{align}
    ||u||_{C^{1+2\gamma,\gamma}_{x,t}(\mathbb R^{n}\times I)}
&\leq C \left( 1 + |I|^{-1-2\gamma-\frac{n+2}{p}}\right) \left( ||u||_{L^p(\mathbb R^n \times I)} + ||(\partial_t - \Delta)u||_{L^p(\mathbb R^n \times I)} \right).  \label{ineq:global parabolic sobolev}
\end{align}
The theorem stills holds if one replaces $\mathbb R^n$ by the half-space $\mathbb R^n \cap \{x_n > 0\}$ and one a priori assumes that $u \in C^{2,1}(\mathbb R^n \cap \{x_n > 0\} \times I)$, $\lim \limits_{\varepsilon \to 0^+} u(x,\varepsilon,t) \to 0$ in $L^p_{loc}(\mathbb R^{n-1} \times I)$ and $\lim \limits_{\varepsilon \to 0^+} \nabla_x u (\tilde{x},\varepsilon,t)$ converges in $L^p_{loc}(\mathbb R^{n-1} \times I)$.
\end{theorem}

\begin{notation}
    In the following, $u \in W^{1,p}(\mathbb R^n \times I)$ means that $u, \nabla_x u$ and $\partial_t u$ are in $L^p(\mathbb R^n \times I)$.
\end{notation}

\begin{remark}
    The convergence of $\lim \limits_{\varepsilon \to 0^+}u(x,\varepsilon,t) \to 0$ and $\lim \limits_{\varepsilon \to 0^+} \nabla_x u (\tilde{x},\varepsilon,t)$ in $L^p_{loc}(\mathbb R^{n-1} \times I)$ is a technicality needed for the boundary estimate. As in Theorem \ref{thm:boundary parabolic sobolev embedding}, one can assume that $u \in C^0([0,T),C^1(\mathbb R^n \cap \{x_n \geq 0\}))$ and $u(\tilde{x},0,t) = 0$ instead, which is the regularity obtained by doing a fixed point argument in Theorem \ref{thm:taylor local existence} if one replaces the half-space by a bounded domain with smooth boundary.
\end{remark}

\begin{proof}
First, observe that $D^2_x u, \partial_t u \in L^p(\mathbb R^n \times I)$ thanks to Global $L^p$ estimates (Theorem \ref{thm: global lp estimate}). Then $u(x,t) \in W^{1,p}(\mathbb R^n \times I)$ extends as a function $u(x,t) \in W^{1,p}(\mathbb R^{n+1}) \hookrightarrow C^{0,1-\frac{n+1}{p}}(\mathbb R^{n+1})$ (see \cite[Theorem 5.28]{adams2003sobolev}). In particular, $u$ is bounded on $\mathbb R^{n} \times I$ and Hölder-continuous with respect to $t$ on $I$, uniformly with respect to $x \in \mathbb R^n$. The Hölder constant $\mu_1$ satisfies
$$
\mu_1 \leq C(n,p,I) ||u||_{W^{1,p}(\mathbb R^n \times I)}.
$$

Similarly, for fixed $t \in I$ and $0 < \varepsilon < 1 - \frac{n}{p}$, $u(x,t) \in W^{2,p}(\mathbb R^n) \hookrightarrow W^{2-\varepsilon,p}(\mathbb R^n) \hookrightarrow C^{1,1-\frac{n}{p}-\varepsilon}(\mathbb R^n)$. The $C^{0,1-\frac{n}{p}-\varepsilon}(\mathbb R^n)$-Hölder-norm of $\nabla_x u(x,t)$ is uniform with respect to $t \in \mathbb R$ provided that $u \in L^{\infty}_{t}(W^{2-\varepsilon,p}_x)$, which we will be proving now for an appropriate $\varepsilon$ to be chosen later. The $L^{\infty}$-norm and the Hölder constant $\mu_2$ will satisfy
$$
||\nabla_x u||_{L^{\infty}(\mathbb R^n \times I)} + \mu_2 \leq C ||u||_{L^{\infty}_{t}(W^{2-\varepsilon,p}_x)} \leq C(n,p,\varepsilon) \left( ||u||_{W^{1,p}(\mathbb R^n \times I)} + ||\Delta u||_{L^{p}(\mathbb R^n \times I)} \right).
$$

First, we recall some facts about Sobolev spaces. For $s \in \mathbb R_{\geq 0}$ and $1  < p <\infty$, recall that
$$
W^{s,p}(\mathbb R^n) = H^{s,p}(\mathbb R^n) = \{u \in S'(\mathbb R^n): (1-\Delta)^{s/2} u \in L^p(\mathbb R^n) \} \ (= F^{s}_{p,2}(\mathbb R^n)),
$$
where $S'(\mathbb R^n)$ denotes the set of Schwartz distributions, $F^{s}_{p,2}$ is a Triebel-Lizorkin space and $(1-\Delta)^{s/2}$ is the (spatial) Fourier multiplier $(1-\Delta)^{s/2} = (1+|D|^2)^{s/2}$. For a reference, see \cite[Corollary 3.1, Chapter 3.2.1.2]{sawano}, for integer $s$ and then use complex interpolation to get the spaces inbetween (Theorem 4.29, Chapter 4.2.3.3 from \cite{sawano}). 

We apply a nonhomogeneous spatial Littlewood-Paley decomposition and write
$$
u = \chi(D)u + \sum_{j \geq 0} \phi(2^{-j}D)u,
$$
where $\chi(x)$, resp. $\phi(x)$,  are smooth and supported on a ball, resp. an annulus, centered at the origin (for fixed $t \in \mathbb R$, one has convergence in the sense of Schwartz distribution). 

For $j \geq 0$, one can write 
$$
(1 - \Delta)^{s/2}  \phi(2^{-j}D) u = (1 - \Delta)^{s/2} \theta(D)  \phi(2^{-j}D) u, \quad s \in \mathbb R,
$$
where $\theta \in C^{\infty}$ is a smooth function which is zero in a neighborhood of the origin. Then
\begin{align}
    \big\| (1 - \Delta)^{(2-\varepsilon)/2} \phi(2^{-j}D) u\big\|_{L^p(\mathbb R^n\times I)} &= \big\| \big\| (1 - \Delta)^{(2-\varepsilon)/2} \phi(2^{-j}D) u \big\|_{L^p_x(\mathbb R^n)}^p \big\|_{L^1_t(I)}^{1/p} \notag \\
    &= \big\| \big\| (1 - \Delta)^{-\varepsilon/2} \phi(2^{-j}D)  (1 - \Delta)  u \big\|_{L^p_x(\mathbb R^n)}^p \big\|_{L^1_t(I)}^{1/p} \notag  \\
    &\leq C(\varepsilon)  \left( 2^j \right)^{-\varepsilon} \big\| \big\| \phi(2^{-j}D)  (1 - \Delta)  u\big\|_{L^p_x(\mathbb R^n)}^p \big\|_{L^1_t(I)}^{1/p} \notag \\
    &\leq  C(\varepsilon) \left( 2^j \right)^{-\varepsilon} \big\| \big\|(1 - \Delta) u\big\|_{L^p_x(\mathbb R^n)}^p \big\|_{L^1_t(I)}^{1/p} \notag  \\
    &=  C(\varepsilon) \left( 2^j \right)^{-\varepsilon}\big\|(1 - \Delta) u\big\|_{L^p(\mathbb R^n\times I)}  \label{application of bernstein lemma}
\end{align}
using Bernstein lemma (\cite[Chapter 2.1, Lemma 2.2]{bahouri}) applied with the multiplier $(1-\Delta)^{-\varepsilon/2} \theta(D)$ for the first inequality and Young Convolution Inequality in the second ($\phi(2^{-j}D)u = \mathcal{F}^{-1}(\phi(2^{-j}\cdot ))(x) * u(x,t)$ and $\mathcal{F}^{-1}(\phi(2^{-j}\cdot))(x)$ has $L^1$-norm independent of $j$). Similarly, it holds that
\begin{align*}
     \big\| (1 - \Delta)^{(2-\varepsilon)/2} \chi(D) u\big\|_{L^p(\mathbb R^n\times I)} &= \big\| \big\| (1 - \Delta)^{(2-\varepsilon)/2} \chi(D) u \big\|_{L^p_x(\mathbb R^n)}^p \big\|_{L^1_t(I)}^{1/p} \\
     &\leq  C(\varepsilon)  \big\| \big\| u \big\|_{L^p_x(\mathbb R^n)}^p \big\|_{L^1_t(I)}^{1/p} =  C(\varepsilon)  ||u||_{L^p(\mathbb R^n \times I)} 
\end{align*}
using Young Convolution Inequality with the function $\mathcal{F}^{-1} \left[ (1 + |\cdot |^2)^{(2-\varepsilon)/2} \chi(\cdot) \right](x)$ which is Schwartz (in particular $L^1(\mathbb R^n)$) since it is the inverse Fourier transform of a smooth, compactly supported function.

At the same time, if $J \subset I$ is bounded, then
\begin{align*}
    \left(2^j \right)^{-p \varepsilon} \big\|(1 - \Delta) u\big\|_{L^p(\mathbb R^n\times I)}  &\gtrsim_{\varepsilon} \big\|(1 - \Delta)^{(2-\varepsilon)/2} \phi(2^{-j}D)u \big\|_{L^p( \mathbb R^n\times J)}^p \\
    &\geq |J|\cdot \inf_{t\in J}\int_{\mathbb R^n}|(1 - \Delta)^{(2-\varepsilon)/2} \phi(2^{-j}D)u(\cdot, t)|^p\,dx.
\end{align*}
Hence, if $J \subset I$ is an interval of length $\leq \left( 2^j \right)^{-p \tilde{\varepsilon}}$ with $0 < \tilde{\varepsilon} < \varepsilon$, there exists $t \in J$ with
$$
 \int_{\mathbb R^n}|(1 - \Delta)^{(2-\varepsilon)/2}  \phi(2^{-j}D)u(\cdot, t)|^p\,dx \leq 1 + \inf_{t\in J}\int_{\mathbb R^n}|(1 - \Delta)^{(2-\varepsilon)/2} \phi(2^{-j}D)u(\cdot, t)|^p\,dx,
$$
so that
$$
\big\|(1 - \Delta)^{(2-\varepsilon)/2} \phi(2^{-j}D) u(\cdot, t)\big\|_{L^p(\mathbb R^n)}\lesssim_{\varepsilon} \left( 2^j \right)^{-p(\varepsilon - \tilde{\varepsilon})}\big\|(1 - \Delta) u\big\|_{L^p(\mathbb R^n\times I)}.  
$$
Similarly, one gets
$$
\big\|(1 - \Delta)^{(2-\varepsilon)/2} \chi(D) u(\cdot, t)\big\|_{L^p(\mathbb R^n)}\lesssim_{\varepsilon} ||u||_{L^p(\mathbb R^n \times I)} 
$$
for an interval of length $|J| \leq 1$ and some $t \in J$.

If $t_1 \in J$ is another time in this interval $J$, we show that the control of $\big\|(1 - \Delta)^{(2-\varepsilon)/2} \phi(2^{-j}D) u(\cdot, t)\big\|_{L^p(\mathbb R^n)}$ can be extended to $t_1$. For this, we write
$$
\phi(2^{-j}D)u(x,t_1) = \mathcal{F}^{-1}(\phi(2^{-j}\cdot ))(x) * \left(u(x,t) + \int_{t}^{t_1} \partial_t u(x,s)ds \right)
$$
and use Minkowski's inequality,
\begin{align*}
    \big\|(1 - \Delta)^{(2-\varepsilon)/2} \phi(2^{-j}D) u(\cdot, t_1) \big\|_{L^p(\mathbb R^n)}
&\leq \big\|(1 - \Delta)^{(2-\varepsilon)/2} \phi(2^{-j}D) u(\cdot, t) 
 \big\|_{L^p(\mathbb R^n)} \\
 &+ \int_t^{t_1} \big\|(1 - \Delta)^{(2-\varepsilon)/2} \phi(2^{-j}D) \partial_t u(\cdot, s) \big\|_{L^p(\mathbb R^n)}\,ds.
\end{align*}

The term at time $t$ in the right-hand side is $\lesssim_{\varepsilon} \left( 2^j \right)^{-p(\varepsilon-\tilde{\varepsilon})}||(1 - \Delta) u||_{L^p(\mathbb R^n\times I)}$ and we only need to estimate the double integral, which is
\begin{align*}
    \int_t^{t_1} \big\|(1 - \Delta)^{(2-\varepsilon)/2} \phi(2^{-j}D) \partial_t &u(\cdot, s) \big\|_{L^p(\mathbb R^n)}\,dt \\
    &\leq |t-t_1|^{1-p^{-1}} \cdot \|(1 - \Delta)^{(2-\varepsilon)/2} \phi(2^{-j}D) \partial_t u(\cdot, s) \|_{L^p(\mathbb R^n\times J)}
\end{align*}
by Hölder's inequality.

Since $t_1\in J$, we have $|t-t_1|^{1-p^{-1}}\leq |J|^{1-p^{-1}} \leq \left( 2^j \right)^{-p\tilde{\varepsilon}(1-p^{-1})} = \left( 2^j \right)^{-\tilde{\varepsilon}(p-1)}$. Similarly to (\ref{application of bernstein lemma}), we have the bound 
$$
\|(1 - \Delta)^{(2-\varepsilon)/2}   \phi(2^{-j}D)  \partial_t u\|_{L^p(\mathbb R^n\times J)}\lesssim_{\varepsilon}  \left( 2^j \right)^{2-\varepsilon} \cdot \|\partial_t u\|_{L^p(\mathbb R^n\times J)}
$$
using Bernstein lemma with the multiplier $(1 - \Delta)^{(2-\varepsilon)/2} \theta(D)$ and Young convolution's inequality. It follows that
\begin{align*}
    |t-t_1|^{1-p^{-1}}\cdot \|(1 - \Delta)^{(2-\varepsilon)/2} \phi(2^{-j}D) &\partial_t u(\cdot, s) \|_{L^p(\mathbb R^n\times J)} \\
    &\lesssim_{\varepsilon}  \left( 2^j \right)^{2- \tilde{\varepsilon} p - (\varepsilon - \tilde{\varepsilon}) }\|\partial_t u\|_{L^p(\mathbb R^n\times J)} \\
    &= C(\varepsilon) \left( 2^j \right)^{2- \varepsilon p + (\varepsilon - \tilde{\varepsilon})(p-1) } \|\partial_t u\|_{L^p(\mathbb R^n\times J)}.
\end{align*}
Recall that $0 < \varepsilon < 1 - \frac{n}{p}$. In addition, if $\varepsilon > \frac{2}{p}$ and $0 < \tilde{\varepsilon} < \varepsilon$ is close enough to $\varepsilon$, then
$$
 |t-t_1|^{1-p^{-1}}\cdot \|(1 - \Delta)^{(2-\varepsilon)/2} \phi(2^{-j}D) \partial_t u(\cdot, s) \|_{L^p(\mathbb R^n\times J)} \lesssim_{\varepsilon}  \left( 2^j \right)^{-\delta} \|\partial_t u\|_{L^p(\mathbb R^n\times J)} 
$$
for some small $\delta > 0$.

Such an $\varepsilon > 0$ is guaranteed to exist if
$$
\frac{2}{p} < 1 - \frac{n}{p} \iff p > n+2
$$
As we can split $I$ into intervals of length $\leq \left( 2^j \right)^{-p \tilde{\varepsilon}}$ and $t_1$ was arbitrary, we have proved that
\begin{align*}
    \sup_{t\in I}\|\phi(2^{-j}D) (1 - \Delta)^{(2-\varepsilon)/2}  u(t)  \|_{L^p(\mathbb R^n)} &\lesssim_{\varepsilon}  \left[  \left( 2^j \right)^{-p(\varepsilon - \tilde{\varepsilon})} + \left( 2^j \right)^{-\delta} \right] \\
    &\phantom{\leq} \cdot \left[ \|(1-\Delta) u\|_{L^p(\mathbb R^n\times J)} 
 +\|\partial_t u\|_{L^p(\mathbb R^n\times J)} 
 \right] \\
 &\lesssim_{\varepsilon}  \left( 2^j \right)^{-\tilde{\delta}} \left[ \|(1-\Delta) u\|_{L^p(\mathbb R^n\times J)} 
 + \|\partial_t u\|_{L^p(\mathbb R^n\times J)} \right].
\end{align*}
Similarly, one can prove
$$
\sup_{t\in  I}\|\chi(D) (1 - \Delta)^{(2-\varepsilon)/2}  u(\cdot, t)  \|_{L^p(\mathbb R^n)} \lesssim_{\varepsilon}  \big\|u\big\|_{L^p(\mathbb R^n\times J)} 
 + \big\|\partial_t u\big\|_{L^p(\mathbb R^n\times J)}.
$$
Finally, going back to the Littlewood-Paley decomposition, we find that
\begin{align*}
    \sup_{t\in  I}\| u(\cdot, t)  \|_{W^{2-\varepsilon,p}(\mathbb R^n)} &= \sup_{t\in I}\|(1 - \Delta)^{(2-\varepsilon)/2}  u(\cdot, t)  \|_{L^p(\mathbb R^n)} \\
    &\lesssim_{\varepsilon}  \left[ 1 + \sum_{j \geq 0} \left( 2^j \right)^{-\tilde{\delta}} \right] \cdot \left[ ||u||_{L^p(\mathbb R^n \times I)} + ||\partial_t u||_{L^p(\mathbb R^n \times I)} + ||\Delta u||_{L^p} \right] \\
    &\leq C(\varepsilon) \left[ ||u||_{L^p(\mathbb R^n \times I)} + ||\partial_t u||_{L^p(\mathbb R^n \times I)} + ||\Delta u||_{L^p(\mathbb R^n \times I)} \right],
\end{align*}
which concludes the proof that $\nabla_x u \in L^{\infty}_t(W^{2-\varepsilon,p}_x)$ for $\frac{2}{p} < \varepsilon < 1 - \frac{n}{p}$.

Finally, we choose $\varepsilon = \frac{2}{p} + \delta$, $0 < \delta \ll 1$ arbitrarily small.  By applying Lemma \ref{lemma:ladyzenskaia reduction Holder regularity}, we find that $\nabla_x u(x,t)$ is $\gamma$-Hölder continuous with respect to $t$, uniformly with respect to $x \in \mathbb R^n$, where 
$$
    \gamma = \frac{(p-n-1)(p-n-2)}{p(2p-n-2)} - \hat{\delta} = \frac{\alpha \beta}{1+\beta}, \quad \alpha = 1 - \frac{n+1}{p}, \ \beta = 1 - \frac{n+2}{p} - \delta.
$$
As $2\gamma \leq \beta$, it follows from (\ref{equivalence Holder norm}) that $\nabla_x u(x,t) \in C^{2\gamma, \gamma}_{x,t}(\mathbb R^{n} \times I )$.

To obtain a final inequality where $||\nabla_x u||_{L^p}$ is absent from the right-hand side, one uses the Interpolation Lemma \ref{interpolation lemma, sobolev version}. Similarly, to replace $||D^2_xu||_{L^p}$ and $||\partial_t u||_{L^p}$ by $||(\partial_t - \Delta)u||_{L^p}$, one uses the Global $L^p$ Estimates (Theorem \ref{thm: global lp estimate}).

We finally note that using translations and parabolic rescaling, it follows that the constant $C =
C(n, p, \gamma, I)$ from the theorem depends only on $|I|$ with $$
C(n,p,\gamma,I) \leq
C(n, p, \gamma,[0,1]) \left( 1 + |I|^{-1-2\gamma-\frac{n+2}{p}} \right).
$$

In the half-space case, if $u \in C^{2,1}(\mathbb R^n \cap \{x_n > 0\} \times I\}$, it suffices to extend $u$ by odd reflection (Proposition \ref{prop:extension by reflection}).
\end{proof}

\begin{corollary}[Local Parabolic Sobolev Embedding]\label{cor: local parabolic sobolev embedding}
   Let $n+2 < p < +\infty$, $\Omega \subset \mathbb R^n$ be open and $I \subset \mathbb R$ be open as well. Assume that $u(x,t) \in L^{p}( \Omega \times I)$ has weak derivatives $\nabla_x u, (\partial_t-\Delta) u \in L^p(\Omega \times I)$. For any ball $\overline{B_r(x_0)} \subset \Omega$, one has $u \in C^{1+2\gamma,\gamma}_{x,t}(B_{r/2}(x_0) \times I)$ where
    $$
    0 < \gamma < \frac{(p-n-1)(p-n-2)}{p(2p-n-2)}.
    $$
    Moreover, there exists $C = C(n,p,\gamma)$ for which
    \begin{align*}
            ||u||_{C^{1+2\gamma,\gamma}_{x,t}(B_{r/2}(x_0) \times I)} &\leq C \left( 1 + (|I|/r^2)^{-1-2\gamma-\frac{n+2}{p}} + r^{-1-2\gamma-\frac{n+2}{p}} \right) \\
            &\phantom{\leq}\cdot\left( ||u||_{L^p(\Omega \times I)} + ||(\partial_t - \Delta) u||_{L^p(\Omega \times I)} \right).
        \end{align*}
        The same result holds if one replaces $B_{r}(x_0)$, resp. $B_{r/2}(x_0)$, by some half-ball $B_r^+(x_0) = B_r(x_0) \cap \{x_n > (x_0)_n\}$ (up to a rotation), where the flat part $T_r^+(x_0) = B_r(x_0) \cap \{x_n = (x_0)_n\} \subset \overline{\Omega}$ can touch the boundary while $\overline{B_r^+(x_0)} \setminus T_r^+(x_0) \subset \Omega$. Moreover, one a priori assumes that $u \in C^{2,1}(B_r^+(x_0) \times I)$, $\lim \limits_{\varepsilon \to 0^+} u(\tilde{x},(x_0)_n+\varepsilon,t) \to 0$ in $L^p_{loc}(B_{r}(\tilde{x}_0) \times I)$, where $x_0 = (\tilde{x}_0, (x_0)_n)$, $\lim \limits_{\varepsilon \to 0^+} \nabla_x u(\tilde{x},(x_0)_n+\varepsilon,t)$ converges in $L^p_{loc}(B_{r}(\tilde{x}_0) \times I)$.
        
    % More generally, if $K \subset \Omega$ is any compact subset and $r = \dist(K,\partial \Omega)$, one has
    %  \begin{align*}
    %         ||u||_{C^{1+2\gamma,\gamma}_{x,t}(K \times I)} &\leq C \left( 1 + (|I|/r^2)^{-1-2\gamma-\frac{n+2}{p}} + r^{-1-2\gamma-\frac{n+2}{p}} \right) \left( ||u||_{L^p(\Omega \times I)}  + ||(\partial_t - \Delta) u||_{L^p(\Omega \times I)} \right)
    %     \end{align*}
\end{corollary}

\begin{proof}
     We cover $B_{1} = B_1(0)$ with an increasing sequence of balls
    $$
    B_{1} = \bigcup_{l=1}^{\infty} B_{R_l}(0), \quad R_l = \sum_{i=1}^l 2^{-i},
    $$
    and consider smooth cut-off functions $\chi_l(x) \in C^{\infty}(\mathbb R^n \times \mathbb R;[0,1])$ for which $\chi_l = 1$ on $B_{R_l}$,
    $$
    \supp(\chi_l) \subset B_{R_{l+1}}
    $$
    and
    $$
    ||\chi_l||_{C^{2}_{x}} \lesssim 2^{2l} =: \rho^{-l}.
    $$
    Set $v_l = u\chi_l \in L^p(\mathbb R^n \times I)$ (resp. $\in L^p(\mathbb R^{n} \cap \{x_n > 0\} \times I)$ in the half-ball case) and proceed as in the proof of Theorem \ref{thm:Interior L^p Estimates for the Heat Equation} when $D^2_x u \in L^p(\Omega \times I)$.

If one omits the hypothesis $D^2_x u \in L^p(\Omega \times I)$, it suffices to use a mollification. Write $I = (a,b)$ with $-\infty \leq a < b \leq +\infty$, $\Omega_{\varepsilon} = \{x \in \Omega: \dist(x,\partial \Omega) > \varepsilon\}$ and let $u_{\varepsilon}(x,t) = \phi_{\varepsilon} *_{x,t} u(x,t)$, $(x,t) \in \overline{\Omega_{\varepsilon}} \times [a+\varepsilon, b-\varepsilon]$ and $\phi_{\varepsilon} = \varepsilon^{-n-1} \phi(\varepsilon^{-1}x, \varepsilon^{-1}t)$ is an approximate identity defined by some $\phi \in C^{\infty}_c(B_{1/2}(0)) \subset C^{\infty}_c(\mathbb R^{n+1})$, $\phi \geq 0$, $\int_{\mathbb R^{n+1}} \phi dxdt = 1$. 

One verifies that $u_{\varepsilon} \in C^{\infty}(\overline{\Omega_{\varepsilon}} \times [a+\varepsilon, b-\varepsilon])$ and
$$
\nabla_x u_{\varepsilon} = \phi_{\varepsilon} * \nabla_x u, \quad (\partial_t - \Delta)  u_{\varepsilon} = \phi_{\varepsilon} * (\partial_t - \Delta)  u, \quad (x,t) \in \Omega_{\varepsilon} \times (a+\varepsilon, b-\varepsilon),
$$
as well as $u_{\varepsilon} \to u$, $\nabla_x u_{\varepsilon} \to \nabla_x u$, $ (\partial_t - \Delta) u_{\varepsilon} \to  (\partial_t - \Delta) u$ in $L^p_{loc}(\Omega \times I)$. Moreover, they are uniformly bounded in $L^p(\overline{\Omega_{\varepsilon}} \times [a+\varepsilon, b-\varepsilon])$. 

For each $\varepsilon$, $D^2_x u_{\varepsilon} = D^2_x \phi_{\varepsilon} * u \in L^p(\Omega_{\varepsilon} \times (a+\varepsilon,b-\varepsilon))$ as a convolution of a $L^{\infty}$ and $L^p$ function. Let $B_{\delta}$ be a smaller ball $B_{\delta}(0) \subset B_1(0)$. Then one can apply the theorem to $u_{\varepsilon}$ on $B_{\delta}$, leading to
\begin{align*}
        ||u_{\varepsilon}||_{C^{1+2\gamma,\gamma}_{x,t}(B_{\delta/2} \times  (a+\varepsilon,b-\varepsilon))}
&\leq C \left( 1 + |I|^{-1-2\gamma-\frac{n+2}{p}}\right) \\
&\phantom{\leq}\cdot\left( ||u_{\varepsilon} ||_{L^p(\Omega_{\varepsilon} \times  (a+\varepsilon,b-\varepsilon))} + ||(\partial_t - \Delta)u_{\varepsilon}||_{L^p(\Omega_{\varepsilon} \times  (a+\varepsilon,b-\varepsilon)} \right) \\
&\leq C \left( 1 + |I|^{-1-2\gamma-\frac{n+2}{p}}\right) \left( ||u ||_{L^p(\Omega \times I)} + ||(\partial_t - \Delta)u||_{L^p(\Omega \times I)} \right),
\end{align*}
where we used Young Convolution's inequality once again. As $u_{\varepsilon} \to u, \nabla_x u_{\varepsilon} \to \nabla_x u$ almost everywhere on $\Omega \times I$ along a sequence $\varepsilon \to 0^+$, one gets the same estimate for $||u_{\varepsilon}||_{C^{1+2\gamma,\gamma}_{x,t}(K_{\delta/2} \times  (a,b))}$. It then suffices to take $\delta \to 1$.

In the half-space case, it suffices to replace $\Omega$ by $B_r(x_0)$ and consider the odd reflection $\tilde{u}$ of $u$ on this ball instead of $u$ defined only on the half-ball (see Proposition \ref{prop:extension by reflection}). We also note that the Laplacian is rotation invariant.
\end{proof}

The method of freezing coefficients allows to deduce $L^p$ estimates for more general parabolic operators.

\begin{lemma}[Global Parabolic Sobolev Embedding for compactly supported solutions of parabolic equations]\label{lemma:global parabolic estimates for parabolic operator, compact support}
    Let $n+2 < p < +\infty$ and
    $$
    0 < \gamma < \frac{(p-n-1)(p-n-2)}{p(2p-n-2)}.
    $$
    Consider a parabolic operator
    $$
    L = c(x,t) + \sum_{i=1}^n b^i(x,t) D_{x_i} + \sum_{i,j=1}^n a^{i,j}(x) D_{x_i}D_{x_j}
    $$
    defined on $\overline{\Omega}$, $\Omega = B_{\kappa}(x_1) \times (0,T)$, satisfying
    \begin{align*}
    ||a,b,c||_{2\gamma,\gamma;\overline{\Omega}} &\leq \Lambda, \quad \Lambda > 0, \\
        \Lambda^{-1}|\xi|^2 \leq \sum_{i,j=1}^n a^{i,j}(x,t)\xi_i\xi_j &\leq \Lambda |\xi|^2, \quad (x,t) \in \overline{\Omega}, \xi \in \mathbb R^n,
    \end{align*}
    and where $a^{i,j}(x)$ does not depend on time.
    
    There exists $\kappa_0 = \kappa_0(n,\alpha,\Lambda) > 0$ and $C = C(n,p,\alpha,\Lambda) > 0$ such that if $\kappa \leq \kappa_0$, for any $u \in L^{\infty}((0,T);L^p_x(\mathbb R^n))$ with weak derivatives $\nabla_x u, \partial_t u, D^2_xu \in L^p(B_{\kappa}(x_1) \times (0,T))$, that is compactly supported in the sense that $\overline{\supp(u(\cdot,t))} \subset B_{\kappa}(x_1)$ for all $t \in (0,T)$ and which is weakly continuous at zero in the sense that
    $$
    \lim \limits_{t \to 0^+} \langle u(\cdot,t), \phi \rangle = 0 \quad \forall \phi \in C^{\infty}_c(\mathbb R^n),
    $$
    then one has
    $$
    ||u||_{C^{1+2\gamma,\gamma}_{x,t}(\mathbb R^{n}\times  (-\infty,0] )} \leq C\left( ||(\partial_t - L)u||_{L^p(\mathbb R^n \times (-\infty,0])} + ||u||_{L^p(\mathbb R^n \times (-\infty,0])} \right).
$$
\end{lemma}

\begin{proof}
    It follows from Lemma \ref{lemma:global l^p estimates for parabolic operator, compact support} that $D^2_xu \in L^p(\mathbb R^n \times (0,T))$, hence $(\partial_t - \Delta) u = (\partial_t - L)u + (L-\Delta)u \in L^p(\mathbb R^n \times (0,T))$ as well. Then
\begin{align*}
        ||u||_{C^{1+2\gamma,\gamma}_{x,t}(\mathbb R^{n}\times  (-\infty,0] )} &\leq C\left( ||(\partial_t - \Delta)u||_{L^p(\mathbb R^n \times (-\infty,0])} + ||u||_{L^p(\mathbb R^n \times (-\infty,0])} \right) \\
        &\leq  C\left( ||(\partial_t - L)u||_{L^p(\mathbb R^n \times (-\infty,0])} + ||D^2_x u||_{L^p(\mathbb R^n \times (-\infty,0])} + ||u||_{L^p} \right) \\
        &\leq C\left( ||(\partial_t - L)u||_{L^p(\mathbb R^n \times (-\infty,0])} + ||u||_{L^p(\mathbb R^n \times (-\infty,0])} \right) 
\end{align*}
    using Lemma \ref{lemma:global l^p estimates for parabolic operator, compact support} and Theorem \ref{thm:parabolic sobolev embedding}.
\end{proof}

\begin{corollary}[Parabolic Sobolev Embedding for parabolic equations with a priori estimates] \label{cor:interior parabolic estimates for parabolic eq with a priori estimates}
   Let $0 < T < +\infty$, $\Omega \subset \mathbb R^n$ be open, $n+2 < p < +\infty$ and
    $$
    0 < \gamma < \frac{(p-n-1)(p-n-2)}{p(2p-n-2)}.
    $$
    Let $L$ be a parabolic operator on $(x,t) \in \overline{\Omega} \times [0,T]$ as in Lemma \ref{lemma:global parabolic estimates for parabolic operator, compact support} with coefficients $a^{i,j}(x)$ independent of $t$. Assume that $u \in L^{\infty}((0,T);L^p_x(\mathbb R^n))$ has weak derivatives $\nabla_x u, \partial_t u, D^2_xu \in L^p(B_{\kappa}(x_0) \times (0,T))$ and is weakly continuous at zero in the sense that
    $$
    \lim \limits_{t \to 0^+} \langle u(\cdot,t), \phi \rangle = 0 \quad \forall \phi \in C^{\infty}_c(\mathbb R^n).
    $$
    For any ball $\overline{B_r(x_0)} \subset \Omega$, one has $u \in C^{1+2\gamma,\gamma}_{x,t}(B_{r/2}(x_0) \times I)$. Moreover, there exists $C = C(n,p,\gamma,\Lambda)$ for which
    \begin{align*}
            ||u||_{C^{1+2\gamma,\gamma}_{x,t}(B_{r/2}(x_0) \times I)} &\leq C \left( 1 + (|I|/r^2)^{-1-2\gamma-\frac{n+2}{p}} + r^{-1-2\gamma-\frac{n+2}{p}} \right) \\
            &\phantom{\leq}\cdot \left( ||u||_{L^p(\Omega \times I)} + ||(\partial_t - L) u||_{L^p(\Omega \times I)} \right).
        \end{align*}
        The same result holds if one replaces $B_{r}(x_0)$, resp. $B_{r/2}(x_0)$, by some half-ball $B_r^+(x_0) = B_r(x_0) \cap \{x_n > (x_0)_n\}$, where the flat part $T_r^+(x_0) = B_r(x_0) \cap \{x_n = (x_0)_n\} \subset \overline{\Omega}$ can touch the boundary while $\overline{B_r^+(x_0)} \setminus T_r^+(x_0) \subset \Omega$. Moreover, one assumes $u \in C^{2,1}(B_r^+(x_0) \times I)$, $\lim \limits_{\varepsilon \to 0^+} u(\tilde{x},(x_0)_n+\varepsilon,t) \to 0$ in $L^p_{loc}(B_{r}(\tilde{x}_0) \times I)$, where $x_0 = (\tilde{x}_0, (x_0)_n)$, $\lim \limits_{\varepsilon \to 0^+}  \nabla_x u(\tilde{x},(x_0)_n+\varepsilon,t)$ converges in $L^p_{loc}(B_{r}(\tilde{x}_0) \times I)$.
\end{corollary}

\begin{proof}
    After translation and parabolic rescaling, one can assume that $x_0 = 0$ and $r = 1$. Then, for any small ball $B_{\kappa}(x_1) \subset B_1(0)$, one can proceed as in the proof of Corollary \ref{cor: local parabolic sobolev embedding} using cut-off functions to get an estimate on $B_{\kappa/2}(x_1)$, but one needs to use Lemma \ref{lemma:global parabolic estimates for parabolic operator, compact support} instead of Theorem \ref{thm:parabolic sobolev embedding}. Then it suffices to cover $B_{1/2}(0)$ with finitely many such estimates.

    In the half-ball case, restrict $\Omega$ to $B_r^+(x_0)$, extend $u$ by odd reflection (See Proposition \ref{prop:extension by reflection}), extend the coefficients $c$, $b^i$, $a^{i,j}$, $a^{n,n}$, $i,j \neq n$, by odd reflection and the coefficients $b^n, a^{i,n}, a^{n,i}$, $i \neq n$, by even reflection and apply the theorem with the ball $B_{r'}(x_0)$ for $r' \to r$.
\end{proof}

   Finally, one obtains a boundary parabolic estimate for bounded domains with smooth boundary by flattening and covering the boundary.
   
\begin{theorem}[Boundary Parabolic Sobolev Embedding on smooth domains]\label{thm:boundary parabolic sobolev embedding}
 Let $0 < T < +\infty$, $\gamma \in (0,1/2)$ and $\Omega \subset \mathbb R^n$ be open, bounded, with $C^{\infty}$-boundary.

    Assume that $u \in C^0([0,T);C^1_0(\overline{\Omega}))$ is a distribution solution to the Dirichlet problem
\begin{align*}
    u_t - \Delta u &= f(x,t), \quad (x,t) \in \Omega \times (0,T), \\
    u &= 0, \quad (x,t) \in \partial \Omega \times (0,T) \cap \Omega \times \{0\},
\end{align*}
with $f \in L^{\infty}(\Omega \times (0,T))$.

There exists $C =  C(n,\gamma, \Omega)$ and $\varepsilon = \varepsilon(n,\gamma,\Omega)$ such that the estimate
    \begin{align*}
         ||u||_{1+2\gamma,\gamma;\Omega_{\varepsilon}  \times (0,T)} 
         &\leq C \bigg( ||f||_{L^{\infty}_{x,t}(\Omega \times (0,T))} + ||u||_{L^{\infty}_{x,t}(\Omega \times (0,T)) } \bigg)
     \end{align*}
     holds on the $\varepsilon$-neighborhood
     $$
    \Omega_{\varepsilon} = \{x \in \Omega: \dist(x,\partial \Omega) < \varepsilon\}
    $$
    of the boundary. 
\end{theorem}

\begin{proof}
One first notes that a $C^0([0,T) \times \overline{\Omega})$ solution is unique thanks to hypoellipticity and the Maximum Principle (Theorem \ref{maximum principle}). The solution $u$ is given by integration with a kernel $k_t(x,y)$ given by
 $$
k_t(x,y) = \sum_{k=1}^{\infty} e^{-t\lambda_k}w_k(x)\overline{w_k}(y),
 $$
 where, as shown in \cite[Chapter 6.5.1]{evans10}, the $(w_k)_{k \geq 1}$ form a countable orthonormal basis of $L^2(\Omega)$ made of eigenfunctions $w_k(x) \in C^{\infty}(\overline{\Omega}), w_k(\partial \Omega) = 0$, $-\Delta w_k = \lambda_k w_k$ on $\Omega$, $\lambda_k > 0$. The sum converges in $L^2(\Omega^2)$ for all $t > 0$ and the kernel is $C^{\infty}((0,+\infty) \times \Omega^2)$ by hypoellipticity because $k_t(x/2,y/2)$ is a distribution solution to the heat equation
 $$
    (\partial_t - \Delta_{x,y}) u(x,y,t) = 0, \quad (x,y) \in 2\Omega \times 2\Omega, t > 0.
 $$ 
 The Maximum Principle (\cite[Chapter 2.3, Theorem 6]{protter2012maximum}) implies $w_k(x) > 0$ on $\Omega$ for all $k \in \mathbb N$, hence $k_t(x,y) > 0$ as well. By Monotone Convergence, one deduces that $k_t$ extends continuously as a function $k_t(x,y) \in C^0((0,+\infty) \times \overline{\Omega}^2)$ with $k_t(x,y) = 0$ when $x \in \partial \Omega$ or $y \in \partial \Omega$. For fixed $y_0 \in \Omega$, $u(x,t) = G(x-y_0,t) - k_t(x,y_0)$ solves the heat equation on $(0,+\infty)$ with $u(\partial \Omega,t) > 0$. The Maximum Principle shows that $k_t(x,y_0) < G(x-y_0,t)$ on $(0,+\infty) \times \Omega$. 
 
 Hence, the corresponding heat semi-group maps $L^{\infty}(\Omega)$ to $C^0(\overline{\Omega})$ for $t > 0$ with 
     \begin{align*}
||e^{t \Delta} g||_{L^{\infty}(\Omega)} &\leq C ||g||_{L^{\infty}(\Omega)} \quad \forall t > 0.
\end{align*}
Thus the solution operator maps $f \in L^{\infty}(\Omega \times (0,T))$ to $u \in C^0(\overline{\Omega} \times [0,T])$ by Dominated Convergence, and more generally $L^p(\Omega \times (0,T))$ to $L^p(\Omega \times (0,T))$ by comparison with the Gaussian Heat Kernel, with
\begin{equation}
    ||u||_{L^{p}(\Omega \times (0,T))} \leq C T ||f||_{L^{p}(\Omega \times (0,T))}, \quad p \in [1,+\infty]. \label{eq: L infty global estimate}
\end{equation}

Next, observe that the theorem is a direct consequence of Corollary \ref{cor:interior parabolic estimates for parabolic eq with a priori estimates} when $u \in C^{2,1}_{x,t}(\Omega \times (0,T))$, $\partial_t u, D^2_x u \in L^{p}(\Omega \times (0,T))$, for $p = p(\gamma)$ large enough by locally straightening the boundary. Otherwise, let $f_n \to f$ in $L^{p}(\mathbb R^n \times (0,T))$, $(f_n) \in C_c^{\infty}(\Omega \times (0,T))$. The corresponding sequence of solutions $u_n$ to the Dirichlet problem with zero initial data has regularity $u_n \in C^{\infty}(\overline{\Omega} \times [0,T])$ (see \cite[Chapter 7.1.3, Theorem 7]{evans10}). Moreover, the sequence $u_n$ converges to $u$ in $L^{p}(\Omega \times (0,T))$ thanks to (\ref{eq: L infty global estimate}).

One obtains the estimate for each $u_n$ and $u_n - u_m$, which concludes the proof upon passing to the limit.
\end{proof}

\section{Schauder Estimates for the heat equation}
In this appendix, we present the so-called Schauder estimates (both interior and boundary) for linear parabolic equations. Those estimates provide an upper bound on the Hölder semi-norm $[D^2_x u]_{\alpha,\alpha/2}$ in terms of the semi-norm $[f]_{\alpha,\alpha/2}$ of the forcing term and $||u||_{L^{\infty}_{x,t}}$. If the parabolic operator has constant coefficient, the proof goes as follows: use an argument by contradiction and the derivative estimates to prove a global estimate, which we then localize (Corollary \ref{interior schauder estimates, ver.1}). Then, freeze the coefficients to deduce the estimate for parabolic operators with varying coefficients, so that one can also obtain a boundary Schauder estimates (Theorem \ref{boundary schauder smooth boundary}). We note that the estimates from \cite{Simon1997} or \cite{krylov} require some a priori smoothness on the solution $u$. For the heat operator and thanks to the Parabolic Sobolev Embedding proved in the previous appendix, we show that this smoothness can be weakened.

    In the following, $L_0$ denotes a parabolic operator
    \begin{equation}
            L_0 =\sum_{i,j=1}^n a^{i,j}D_{x_i}D_{x_j} \label{eq: L_0 constant coefficient parabolic}
    \end{equation}
     with constant coefficients satisfying $a^{i,j} = a^{j,i}$ and 
    \begin{align*}
        \Lambda^{-1}|\xi|^2 \leq \sum_{i,j=1}^n a^{i,j}\xi_i\xi_j &\leq \Lambda |\xi|^2, \quad \xi \in \mathbb R^n,
    \end{align*}
for some $\Lambda > 0$ (i.e., $A = (a^{i,j})$ is a symmetric, positive definite matrix).

One should note that if $u \in C^{2,1}_{x,t}(\mathbb R^n \times (-\infty,0])$ is a solution to the free parabolic equation
$$
    u_t - L_0 u = 0, \quad (x,t) \in \mathbb R^n \times (-\infty,0],
$$
and if $P$ is a matrix for which $P^T (a^{i,j}) P = I_n$ (the matrix $P$ is obtained through a orthogonal diagonalization $Q^T A Q = Diag(\lambda_1, ..., \lambda_n)$ and then $P = QDiag(\lambda_1^{-1/2}, ...,\lambda_n^{-1/2})$), then $v(x,t) =  u(Px,t)$ solves
$$
    v_t - \Delta v = 0, \quad (x,t) \in \mathbb R^n \times (-\infty,0].
$$
Similarly, if $u \in C^{2,1}_{x,t}(\mathbb R^n \cap \{\langle d, x \rangle \geq 0\} \times (-\infty,0])$ is a solution to the free parabolic equation
\begin{align*}
        u_t - L_0 u &= 0, \quad (x,t) \in \mathbb R^n \cap \{\langle d,x \rangle > 0\} \times (-\infty,0], \\
        u(x,t) &= 0, \quad (x,t) \in \mathbb R^{n} \cap \{ \langle d, x \rangle = 0\} \times (-\infty,0],
\end{align*}
 and if $P$ satisfies $P^T (a^{i,j}) P = I_n$, $R = (b_1,...,b_{n-1},P^T d/||P^Td||) \in \mathbb R^{n \times n}$ is an orthogonal matrix, then $v(x,t) =  u(PRx,t)$ solves
\begin{align*}
        v_t - \Delta v &= 0, \quad (x,t) \in \mathbb R^n \cap \{x_n > 0\} \times (-\infty,0], \\
        v(x,t) &= 0, \quad (x,t) \in \mathbb R^{n} \cap \{ x_n  = 0\} \times (-\infty,0].
\end{align*}

\begin{corollary}[Interior Schauder Estimates for the Heat Equation with a priori estimates] \label{interior schauder estimates, ver.1}
Assume that $u \in C^{2+\alpha,1+\alpha/2}(C(x,t;r))$ is a classical solution to the linear parabolic equation
\begin{equation}
    u_t - L_0 u = f(x,t), \quad (x,t) \in C(x,t;r), 
\end{equation}
where $f \in C^{\alpha,\alpha/2}_{x,t}(C(x,t;r))$ and $L_0$ is as in (\ref{eq: L_0 constant coefficient parabolic}).

There exists $C = C(n, \alpha, L_0)$ (independent of $x,t,r$) for which
\begin{align*}
    % [D^2_xu]_{\alpha,\alpha/2;C(x,t;r/2)} + [\partial_t u]_{\alpha,\alpha/2;C(x,t;r/2)}&\leq C\left(1+\frac{1}{r^{2+\alpha}}\right)\left(||f||_{\alpha,\alpha/2;C(x,t;r)} + ||u||_{L^{\infty}_{x,t}(C(x,t;r))} \right) \\
    ||u||_{2+\alpha,1+\alpha/2;C(x,t;r/2)} &\leq C\left(1+\frac{1}{r^{2+\alpha}}\right)\left( ||f||_{\alpha,\alpha/2;C(x,t;r)} + ||u||_{L^{\infty}_{x,t}(C(x,t;r))} \right).
\end{align*}
The result still holds if one replaces the cylinder $C(x,t;r)$ by a half-cylinder
$$
    C^+(x,t;r) := \{(y,s) \in \mathbb R^n \times \mathbb R: |x-y| < r, \langle d, y-x \rangle \geq 0, t-r^2 < s < t\}, \quad |d| = 1,
    $$
and one assumes a boundary condition $u = 0$ on $C(x,t;r) \cap \{(y,t):  \langle d, y-x \rangle = 0\}$.
\end{corollary}

\begin{remark}
     It is a priori assumed that $D^2_x u$ is globally Hölder-continuous on $C(x,t;r)$. This assumption will be removed later.
\end{remark}

\begin{proof}
See the argument by contradiction of Simon \cite{Simon1997} for the global and half-space estimates, which we then localize following the proof of Theorem \ref{thm:Interior L^p Estimates for the Heat Equation} or the localization argument from Krylov (\cite[Theorem 8.11.1]{krylov}).
\end{proof}

\begin{corollary}[Interior Schauder Estimates for the Heat Equation with smooth forcing term] \label{cor:interior schauder estimates with smooth forcing}
Assume that $u \in L^{\infty}(C(x,t;r))$ is a distributional solution to the linear parabolic equation
\begin{equation}
    u_t - L_0 u = f(x,t), \quad (x,t) \in C(x,t;r),
\end{equation}
where $f \in C^{\alpha,\alpha/2}_{x,t}(C(x,t;r)) \cap C^{\infty}_{x,t}(C(x,t;r))$ and $L_0$ is as in (\ref{eq: L_0 constant coefficient parabolic}).

There exists $C = C(n, \alpha, L_0)$ (independent of $x,t,r$) for which
\begin{align*}
    % [D^2_xu]_{\alpha,\alpha/2;C(x,t;r/2)} + [\partial_t u]_{\alpha,\alpha/2;C(x,t;r/2)}&\leq C\left(1+\frac{1}{r^{2+\alpha}}\right)\left(||f||_{\alpha,\alpha/2;C(x,t;r)} + ||u||_{L^{\infty}_{x,t}(C(x,t;r))} \right) \\
    ||u||_{2+\alpha,1+\alpha/2;C(x,t;r/2)} &\leq C\left(1+\frac{1}{r^{2+\alpha}}\right)\left( ||f||_{\alpha,\alpha/2;C(x,t;r)} + ||u||_{L^{\infty}_{x,t}(C(x,t;r))} \right).
\end{align*}
\end{corollary}

\begin{proof}
    By hypoellipticity, $u \in C^{\infty}(C(x,t;r))$. Take $r_n = r - 1/n$ and start at $n \geq n_0$ large enough so that $r > 1/n$. Take also a decreasing sequence $0 < \varepsilon_n < 2r/n - 1/n^2$ with $\varepsilon_n \to 0$. One obtains a sequence of increasing cylinders $(C(x,t-\varepsilon_n;r_n))_{n \geq n_0}$ with the inclusion
    $$
    \overline{C(x,t-\varepsilon_n;r_n)} \subset C(x,t;r)
    $$
    It then suffices to apply Corollary \ref{interior schauder estimates, ver.1} on these cylinders and let $n \to +\infty$.
    
\end{proof}

\begin{corollary}[Global Schauder Estimates for compactly supported solutions]\label{cor:global schauder estimates for compactly supported solutions}
    Assume that $u \in C^{0}_{b}(\mathbb R^n \times (-\infty,0])$, $\alpha \in (0,1)$, is a distributional solution to the linear heat equation
\begin{align*}
    u_t - L_0 u &= f(x,t), \quad (x,t) \in \mathbb R^n \times (-\infty,0),
\end{align*}
where $L_0$ is as in (\ref{eq: L_0 constant coefficient parabolic}).

Further assume that $u$ is compactly supported in the sense that there is some $x_1 \in \mathbb R^n$ and $r > 0$ for which $\overline{\supp(u(\cdot,t))} \subset B_r(x_1)$ for all $t \leq 0$, $u(\cdot,t) = 0$ for $t \leq -r^2$ and $f \in C^{\alpha,\alpha/2}_{x,t}(C(x,0;r))$. 

There exists $C = C(n, \alpha, L_0)$ for which
$$
    ||u||_{2+\alpha,1+\alpha/2;\mathbb R^n \times (-\infty,0]} \leq C\left(1+\frac{1}{r^{2+\alpha}}\right)\left( ||f||_{\alpha,\alpha/2;\mathbb R^n \times (-\infty,0]} + ||u||_{L^{\infty}_{x,t}(\mathbb R^n \times (-\infty,0])} \right).
$$
\end{corollary}

\begin{proof}
After translation, one can assume $x_1 = 0$. After applying a linear transformation (and up to taking a larger $r$), one can assume that $L_0 = \Delta$. After parabolic rescaling, one can assume without loss of generality that $r = 1$.

On $\mathbb R^n \times [-4,0)$, $u \in C^0_b(\mathbb R^n \times [-4,0))$ is the unique solution for the problem 
    \begin{align*}
     \partial_t u - \Delta u &= f(x,t) \in C^0_b(\mathbb R^n \times [-4,0)), \quad (x,t) \in \mathbb R^n \times (-4,0), \\
u(x,-4) &= 0, \quad x \in \mathbb R^n,
 \end{align*}
 which is given by convolution with the Gaussian heat kernel (up to a time translation), as shown in Lemma \ref{lemma:uniqueness heat equation}.
 
 Note that given the compact support assumption of $u(\cdot,t)$, the same holds for $f$. Setting $f(x,t) = 0$ for $x \in B_r(x_1)^c$, $f(x,t) = f(x,-t)$ for $t > 0$, $f(x,t) = 0$ for $t \leq -1$ yields a $C^{\alpha,\alpha/2}(\mathbb R^{n+1})$ extension of $f$.  Approximate $f$ by a sequence of $C^{\infty}_c(\mathbb R^{n+1})$-functions $(f_n)_{n \geq 0}$ with $||f_n-f||_{\alpha,\alpha/2;\mathbb R^n \times (-\infty,0)} \to 0$. Let $(u_n) \in C^0_b(\mathbb R^n \times [-4,0))$ be the corresponding sequence of solutions for
    \begin{align*}
     \partial_t u_n - \Delta u_n &= f_n(x,t), \quad (x,t) \in \mathbb R^n \times (-4,0) \\
u_n(x,-4) &= 0, \quad x \in \mathbb R^n
 \end{align*}
 given by convolution with the heat kernel. Thanks to (\ref{Gaussian heat kernel convolution, young inequality}), one has 
   $$
   ||u_n-u||_{L^{\infty}(\mathbb R^n \times (-4,0))} \to 0.
   $$
    By Corollary \ref{cor:interior schauder estimates with smooth forcing},
    \begin{align*}
            ||u_n-u_m||_{2+\alpha,1+\alpha/2;C(0,0;1)} &\leq C\left(1+\frac{1}{r^{2+\alpha}}\right)\left( ||f_n-f_m||_{\alpha,\alpha/2;C(0,0;2)} + ||u_n-u_m||_{L^{\infty}_{x,t}} \right),
    \end{align*}
    meaning that $(u_n)$ is a Cauchy sequence in the Banach space $C^{2+\alpha,1+\alpha/2}(C(0,0;1))$, where $C(0,0;1) \subset C(0,0;2) \subset \mathbb R^n \times (-4,0)$. Hence, $u \in C^{2+\alpha,1+\alpha/2}(C(0,0;1))$ and $u_n \to u$ in $C^{2+\alpha,1+\alpha/2}(C(0,0;1))$. Once more, by Corollary \ref{cor:interior schauder estimates with smooth forcing},
    \begin{align*}
            ||u_n||_{2+\alpha,1+\alpha/2;C(0,0;1)} &\leq C\left(1+\frac{1}{r^{2+\alpha}}\right)\left( ||f_n||_{\alpha,\alpha/2;C(0,0;2)} + ||u_n||_{L^{\infty}_{x,t}(C(0,0;2))} \right) \\
            &\leq 2 C\left(1+\frac{1}{r^{2+\alpha}}\right)\left( ||f||_{\alpha,\alpha/2;C(0,0;2)} + ||u||_{L^{\infty}_{x,t}(C(0,0;2))} \right) \\
            &= 2C\left(1+\frac{1}{r^{2+\alpha}}\right)\left( ||f||_{\alpha,\alpha/2;C(0,0;1)} + ||u||_{L^{\infty}_{x,t}(C(0,0;1))} \right)
    \end{align*}
    for all $n$ large enough. Take the limit $n \to +\infty$ to conclude.
\end{proof}

\begin{theorem}[Interior Schauder Estimates for the Heat Equation] \label{interior schauder estimates}
Assume that $u \in C^{1+\alpha,\alpha/2}(C(x,t;r))$ is a solution to the linear heat equation
\begin{equation}
    u_t - L_0 u = f(x,t), \quad (x,t) \in C(x,t;r), 
\end{equation}
where $f \in C^{\alpha,\alpha/2}_{x,t}(C(x,t;r))$ and $L_0$ is as in (\ref{eq: L_0 constant coefficient parabolic}).

There exists $C = C(n, \alpha, L_0)$ (independent of $x,t,r$) for which
\begin{align*}
    % [D^2_xu]_{\alpha,\alpha/2;C(x,t;r/2)} + [\partial_t u]_{\alpha,\alpha/2;C(x,t;r/2)}&\leq C\left(1+\frac{1}{r^{2+\alpha}}\right)\left(||f||_{\alpha,\alpha/2;C(x,t;r)} + ||u||_{L^{\infty}_{x,t}(C(x,t;r))} \right) \\
    ||u||_{2+\alpha,1+\alpha/2;C(x,t;r/2)} &\leq C\left(1+\frac{1}{r^{2+\alpha}}\right)\left( ||f||_{\alpha,\alpha/2;C(x,t;r)} + ||u||_{L^{\infty}_{x,t}(C(x,t;r))} \right).
\end{align*}
\end{theorem}

\begin{remark}\label{weakening of interior Schauder estimates}
    If $L_0 = \Delta$, the a priori regularity on $u, \nabla_x u$ can be weakened. Thanks to Interior $L^p$-estimates (Theorem \ref{thm:Interior L^p Estimates for the Heat Equation}) and Parabolic Sobolev Embedding (Corollary \ref{cor: local parabolic sobolev embedding}), it suffices that $u \in L^p(C(x,t;r)), \nabla_x u \in L^{p}(C(x,t;r))$ for $p$ large enough (depending on $\alpha$) to get $C^{1+\alpha,\alpha/2}$-regularity on a smaller cylinder $C(x,t;r/4)$. Then one can apply the interior Schauder estimates to get $C^{2+\alpha,1+\alpha/2}$-regularity on $C(x,t;r/8)$.
\end{remark}

\begin{proof}
    It is a localization of the Global Schauder Estimates for compactly supported solutions (Corollary \ref{cor:global schauder estimates for compactly supported solutions}), done exactly using the localization argument in Krylov's book (\cite[Theorem 8.11.1]{krylov}). 
\end{proof}

The method of 'freezing' the coefficients allows to generalize Corollary \ref{interior schauder estimates, ver.1} to more general parabolic operators.

\begin{theorem}[Boundary Schauder Estimates for parabolic equations on half-cylinders] \label{boundary schauder estimates, ver.1}
Let $\mu \in (0,1), \rho  > 0,x_0 \in \mathbb R^n, t_0 \in \mathbb R$. Consider 
$$
    C^+(x_0,t_0;\rho) := \{(y,s) \in \mathbb R^n \times \mathbb R: |x_0-y| < \mu, y_n > (x_0)_n , t-\rho^2 < s < t\}
    $$
a half-cylinder and a parabolic operator
    $$
    L = c(x,t) + \sum_{i=1}^n b^i(x,t) D_{x_i} + \sum_{i,j=1}^n a^{i,j}(x,t) D_{x_i}D_{x_j}
    $$
    defined on $\overline{C^+(x_0,t_0;\rho) }$ and satisfying
    \begin{align*}
        ||a,b,c||_{\mu,\mu/2;\overline{C^+(x_0,t_0;\rho) }} &\leq \Lambda, \quad \Lambda > 0, \\
        \Lambda^{-1}|\xi|^2 \leq \sum_{i,j=1}^n a^{i,j}(x,t)\xi_i\xi_j &\leq \Lambda |\xi|^2, \quad (x,t) \in C^+(x_0,t_0;\rho) , \xi \in \mathbb R^n.
    \end{align*}

Assume that $u \in C^{2+\mu,1+\mu/2}(C^+(x_0,t_0;\rho))$ is a classical solution to the linear parabolic equation
\begin{align*}
    u_t - L  u &= f(x,t), \quad (x,t) \in C^+(x_0,t_0;\rho),  \\
    u &= u_0(x,t), \quad (x,t) \in C^+(x_0,t_0;\rho) \cap \{x_n = (x_0)_n \},
\end{align*}
where $f \in C^{\mu,\mu/2}_{x,t}(C^+(x_0,t_0;\rho) )$, $u_0 \in C^{2+\mu,1+\mu/2}_{x,t}(C^+(x_0,t_0;\rho) )$.

There exists $C = C(n, \mu, L)$ (independent of $x_0,t_0,\rho$) and $\kappa = \kappa(n,p,\mu,L)$ for which
\begin{align}
||D^2_xu||_{\mu,\mu/2;C^{+}_{\rho/2}} + ||\partial_t u||_{\mu,\mu/2;C^{+}_{\rho/2}} &\leq C \bigg( ||f||_{\mu,\mu/2;C^{+}_{\rho}} + ||\partial_t u_0||_{\mu,\mu/2;C^{+}_{\rho}} \notag \\
&+\sum_{i,j \neq n} ||D_{x_i}D_{x_j} u_0||_{\mu,\mu/2;C^{+}_{\rho}}  + \rho^{-2-\mu}||u||_{L^{\infty}_{x,t}(C^{+}_{\rho}) } \bigg) \label{boundary schauder, norm}
\end{align}
whenever $\rho \leq \kappa$.
\end{theorem}

\begin{proof}
Without loss of generality, one can assume $u_0 = 0$ by replacing $u$ by $u + u_0$ and $f$ by $f + \partial_t u_0 - L u_0$. One can also assume that $b^i = c = 0$. Indeed, extend the coefficients $c$, $b^i$, $a^{i,j}$, $a^{n,n}$, $i,j \neq n$, by odd reflection, the coefficients $b^n, a^{i,n}, a^{n,i}$, $i \neq n$, by even reflection and $u$ by odd reflection to the whole cylinder $C(x_0,t_0;\rho)$ (Proposition \ref{prop:extension by reflection}), and use interpolation to obtain
\begin{align}
    ||D^2_xu||_{L^{\infty}(C^+)} + ||\partial_t u||_{L^{\infty}(C^+)} 
    &\lesssim \varepsilon [D^2_x u]_{\mu,\mu/2;C^+} + \varepsilon [\partial_t u]_{\mu,\mu/2;C^+}  + ||u||_{L_{x,t}^{\infty}(C^+)} \label{interpolation and cut-off, lower order terms}
\end{align}
for $\varepsilon > 0$ arbitrarily small. We can proceed similarly to estimate terms $c(x,t) u$ and $b^i(x,t) \partial_{x_i} u$ coming from $Lu$.

First, we observe that (\ref{boundary schauder, norm}) is true if $L$ has constant coefficients by Corollary \ref{interior schauder estimates, ver.1}. For general parabolic operators $L$, assume without loss of generality that $x = t = 0$, $\rho = 1$. Fix $(x_0,t_0) \in C^+_{1/2}$. Write
    $$
L_0 = \sum_{i,j=1}^n a^{i,j}(x_0,t_0) D_{x_i}D_{x_j}.
$$
    One estimates 
\begin{align*}
        ||u||_{2+\mu,1+\mu/2;C^+_{1/2}} 
        &\lesssim C\left( ||(\partial_t - L_0)u||_{\mu,\mu/2;C^+_{1}} + ||u||_{L^{\infty}_{x,t}(C^+_{1})} \right) 
\end{align*}
    using the constant-coefficient case for $L = L_0$. Then
\begin{align*}
    ||(\partial_t - L_0)u||_{\mu,\mu/2;C^+_{1}} &\leq ||(\partial_t - L)u||_{\mu,\mu/2;C^+_{1}} \\
    &+ \bigg|\bigg| \sum_{i,j=1}^n \left[ a^{i,j}(x_0,t_0) - a^{i,j}(x,t) \right] D_{x_i}D_{x_j}u \bigg|\bigg|_{\mu,\mu/2;C^+_{1}}  \\
    &\lesssim  ||(\partial_t - L)u||_{\mu,\mu/2;C^+_{1}}  \\
    &+ \sum_{i,j=1}^n ||a^{i,j}(x_0,t_0) - a^{i,j}(x,t)||_{L_{x,t}^{\infty}(C^+_{1})} \cdot ||D^2_x u||_{\mu,\mu/2;C^+_{1}} \\ 
    &+ \sum_{i,j=1}^n ||a^{i,j}(x_0,t_0) - a^{i,j}(x,t)||_{\mu,\mu/2;C^+_{1}} \cdot ||D^2_x u||_{L_{x,t}^{\infty}(C^+_{1})} \\
    &\lesssim  ||(\partial_t - L)u||_{\mu,\mu/2;C^+_{1}} + \Lambda  \kappa^{\alpha} ||D^2_x u||_{\mu,\mu/2;C^+_{1}}  +  \Lambda ||D^2_x u||_{L^{\infty}_{x,t}(C^+_{1})}
\end{align*}
and we estimate $||D^2_x u||_{L^{\infty}(C^+_{1})}$ as in (\ref{interpolation and cut-off, lower order terms}). If $\kappa$ is small enough, the theorem is proved.
\end{proof}

 Finally, one obtains a boundary Schauder estimate for bounded domains with smooth boundary by flattening and covering the boundary.

\begin{theorem}[Boundary Schauder Estimates for parabolic equations on smooth domains] \label{boundary schauder smooth boundary}
    Assume that $u \in C^{2+\mu,1+\mu/2}_{x,t}(\Omega \times (0,T))$ is a classical solution to the linear parabolic equation
\begin{align*}
    u_t - L u &= f(x,t), \quad (x,t) \in \Omega \times (0,T), \\
    u &= u_0(x,t), \quad (x,t) \in \partial \Omega \times (0,T),
\end{align*}
where $f \in C^{\mu,\mu/2}_{x,t}(\Omega \times (0,T))$, $u_0 \in C^{2+\mu,1+\mu/2}_{x,t}(\Omega \times (0,T))$, $\Omega$ is open, bounded with $C^{2+\mu}$-boundary and $L$ is a parabolic operator satisfying the hypotheses from Theorem \ref{boundary schauder estimates, ver.1} on $\overline{\Omega} \times [0,T]$ for some constant $\Lambda > 0$.

For any interval $I = (\delta,T)$, $\delta > 0$, there exists $C =  C(n,\mu,\Lambda, \Omega, I)$ and $\varepsilon = \varepsilon(n,\mu,\Lambda, \Omega, I)$ such that the estimate
    \begin{align*}
         ||(|D^2_xu| + |\partial_t u|)||_{\mu,\mu/2;\Omega_{\varepsilon}  \times I} 
         &\leq C \bigg( ||f||_{\mu,\mu/2;\Omega \times (0,T)} + ||\partial_t u_0||_{\mu,\mu/2;\Omega \times (0,T)}\\
         &+ \sum_{i,j \neq n} ||D_{y_i}D_{y_j} (u_0)||_{\mu,\mu/2; \Omega \times (0,T)}  + ||u||_{L^{\infty}_{x,t}(\Omega \times (0,T)) } \bigg)
     \end{align*}
     holds on the $\varepsilon$-neighborhood
     $$
    \Omega_{\varepsilon} = \{x \in \Omega: \dist(x,\partial \Omega) < \varepsilon\}
    $$
    of the boundary. 

    For the heat operator $L = \Delta$, if $\Omega$ has $C^{\infty}$-boundary, then the assumption $u \in C^{2+\mu,1+\mu/2}_{x,t}(\Omega \times (0,T))$ can we weakened to $u \in C^0([0,T);C^1(\overline{\Omega}))$. 
\end{theorem}

\begin{remark}
    To obtain smoothness at time $t = 0$, one needs to assume some additional compatibility condition between the initial and boundary data $u_0$ and the forcing term $f$. See for example \cite[Chapter 7.1.3, Theorem 7]{evans10}.
\end{remark}

\begin{proof}
    The theorem is a direct consequence of Theorem \ref{boundary schauder estimates, ver.1} by locally straightening the boundary. 

    For the heat operator $L = \Delta$, replacing $u$ with $u+u_0$ and $f$ by $f - (\partial_t u_0 - \Delta u_0)$, one can assume that $u_0 = 0$. 
    
    Letting $\chi(t) \in C^{\infty}_c(\mathbb R;[0,1])$ be a cut-off which is $1$ on $I = (\delta,2T)$, $0$ on $(-\infty,\delta/2)$, it follows that $v = u \cdot \chi(t)$ solves
\begin{align*}
    v_t - L v &= f \cdot \chi(t) + u \cdot \chi'(t), \quad (x,t) \in \Omega \times (0,T), \\
    v &= 0, \quad (x,t) \in \partial \Omega \times (0,T) \cup \Omega \times \{0\},
\end{align*}
where $u \cdot \chi'(t) \in C^0([0,T) \times \overline{\Omega})$. Applying Corollary \ref{cor: local parabolic sobolev embedding} and Theorem \ref{thm:boundary parabolic sobolev embedding} for $0 < T' < T$, it follows that $v \in C^{1+\mu,\mu/2}(\overline{\Omega} \times [0,T])$ with 
$$
||u||_{C^{1+\mu,\mu/2}(\overline{\Omega} \times [\delta,T])} 
 \leq ||v||_{C^{1+\mu,\mu/2}(\overline{\Omega} \times [0,T])} \leq C(n,\mu,\Omega,T,\delta) \left( ||f||_{L^{\infty}(\overline{\Omega} \times [0,T])} + ||u||_{L^{\infty}} \right)
$$
for any $\delta \in (0,T)$. Hence, $F = f \cdot \chi(t) + u \cdot \chi'(t) \in C^{\mu,\mu/2}_{x,t}(\overline{\Omega} \times [0,T])$ with
$$
||F||_{C^{1+\mu,\mu/2}(\overline{\Omega} \times [0,T])} 
 \leq  C(n,\mu,\Omega,T,\delta) \left( ||f||_{C^{\mu,\mu/2}_{x,t}(\overline{\Omega} \times [0,T])} + ||u||_{L^{\infty}(\overline{\Omega} \times [0,T])} \right)
$$
as well. Let $F_n \in C^{\infty}(\overline{\Omega} \times [0,T])$, $F_n(x,t) = 0$ for $t \leq \delta/4$ converge to $F$ in $C^{\mu,\mu/2}_{x,t}(\overline{\Omega} \times [0,T])$. Such an approximation can be constructed by first extending $F_n$ as a global $C^{\mu,\mu/2}(\mathbb R^{n+1})$ function
$$
F_n(x,t) = \inf\{F_n(y,s) + [F_n]_{C^{\mu,\mu/2}_{x,t}(\overline{\Omega} \times [0,T])} \cdot |(x,t) - (y,s)|^{\mu}: (y,s) \in \overline{\Omega} \times [0,T]\},
$$
and then using a convolution with an approximate identity and a smooth cut-off to remove small times. Here, $|(x,t) - (y,s)| = |x-y| + |s-t|^{1/2}$ denotes the parabolic distance.

Let $v_n$ solve $(\partial_t  - \Delta)v_n= F_n$ on $\Omega \times (0,T)$, $v_n = 0$ on $\partial \Omega \times (0,T) \cup \Omega \times \{0\}$. Then $v_n \to v$ uniformly (see e.g. the proof Theorem \ref{thm:boundary parabolic sobolev embedding}) and $v_n \in C^{\infty}(\overline{\Omega} \times [0,T])$ (\cite[Chapter 7.1.3, Theorem 7]{evans10}). It follows from the boundary case with a priori estimates and the interior case (Theorem \ref{interior schauder estimates}) that each $v_n$, $v_n-v_m$ is $C^{2+\mu,1+\mu/2}$ with norms
\begin{align}
    ||v_n||_{C^{2+\mu,1+\mu/2}(\overline{\Omega} \times [\delta,T])} &\leq  
  C(n,\mu,\Omega,T,\delta) \left( ||F_n||_{C^{\mu,\mu/2}_{x,t}(\overline{\Omega} \times [0,T])} + ||v_n||_{L^{\infty}(\overline{\Omega} \times [0,T])} \right) \notag \\
  &\leq C(n,\mu,\Omega,T,\delta) \left( ||F||_{C^{\mu,\mu/2}_{x,t}(\overline{\Omega} \times [0,T])} + ||v||_{L^{\infty}(\overline{\Omega} \times [0,T])} \right) \notag \\
  &\leq C(n,\mu,\Omega,T,\delta) \left( ||f||_{C^{\mu,\mu/2}_{x,t}(\overline{\Omega} \times [0,T])} + ||u||_{L^{\infty}(\overline{\Omega} \times [0,T])} \right) \label{eq:boundary schauder estimate last inequality}
\end{align}
for all $n$ large enough, and similarly for $v_n-v_m$. In the limit, we obtain that $v \in C^{2+\mu,1+\mu/2}(\overline{\Omega} \times [\delta,T])$ with
$$
||u||_{C^{2+\mu,1+\mu/2}(\overline{\Omega} \times [\delta,T])} = ||v||_{C^{2+\mu,1+\mu/2}(\overline{\Omega} \times [\delta,T])}
$$
satisfying (\ref{eq:boundary schauder estimate last inequality}).
\end{proof}

\section{Energy Concentration and Bubbles}

In this section, we follow the argument of Struwe (\cite[Chapter 3, Theorem 5.6]{struwe2008variational}), which holds for a wider class of evolution of harmonic maps, and proves that a finite-energy, smooth solution $v \in C^{\infty}(B^2 \times (0,T))$ of 
\begin{align*}
    v_t = \Delta v + |\nabla v|^2v, \quad (x,t) \in B^2  \times (0,T),
\end{align*}
blows-up in finite-time (in the sense that one cannot extend $v$ to a $C^{\infty}(B^2 \times (0,T])$ function) if and only if there is a concentration of energy at a finite number of points from $B^2$ (Theorem \ref{thm:smoothness around all but finitely many points}). In the $k$-equivariant setting, such a concentration can only happen at the origin (and $v$ is actually smooth at the boundary at the blow-up time $T$, Proposition \ref{prop:smoothness at the boundary}). Moreover, along some appropriate sequences converging to a point of concentration $(x,T)$, one can extract a non-constant harmonic map (called 'bubble') from $v$ and even grab all the bubbles at the same time (Theorem \ref{thm: full bubble decomposition}), yielding the so-called bubble-tree decomposition.

We note that Struwe's work concerns only manifolds without boundary, which is not the case here. As his arguments are essentially local, we show that one can still deduce a similar result in the interior of the ball, without assuming any kind of regularity at the boundary (the analog result with a boundary condition which does not depend on time has been proved in \cite{ChangStruweAnalog}) and only assuming finiteness of the energy. The main difference in our setting is that the energy $E(v(\cdot,t))$ is not necessarily decreasing in time, so Struwe's argument needs to be adapted slightly. 

We will still deduce that under $k$-equivariance symmetry and smooth boundary data, no singularity can occur at the boundary and the solution is actually smooth at the boundary at the time of blow-up (Proposition \ref{prop:smoothness at the boundary}). 

For $U \subset  B^2$, let
    $$
 E(v(t);U) = \frac{1}{2}\int_{U} |\nabla v|^2 dx
    $$
be the local energy of $v$ on $U$. 

\begin{lemma}\label{lemma:no concentration of energy away origin, k equivariant}
    Let $v(x,t) \in C^{2,1}_{x,t}(B^2 \times (0,T))$ be a finite-energy and $k$-equivariant solution of (\ref{heat map flow}) on $[0,T)$.
    For any $x_0 \in \overline{B^2} \setminus \{0\}$, for any $\varepsilon > 0$, there is $\delta > 0$ for which $0 < R < \delta$ implies
    $$
    \limsup_{t \to T} E(v(t);B_R(x_0) \cap B^2) \leq \varepsilon.
    $$
    In other words, around any $x_0 \neq 0$, there is no concentration of energy as $t \to T$.
\end{lemma}

\begin{proof}
    Let $E_0 = \sup_{t \in [0,T]} E(v(t);B^2) < +\infty$ be the maximal energy of $v(x,t)$.

    Assume for a contradiction that there is $x_0 \neq 0$ for which the lemma does not hold, i.e.,  there is $\varepsilon > 0$ and a decreasing sequence $(R_n)_{n > 0} \subset (0,1]$, $R_n \to 0$, for which
    $$
\limsup_{n \to +\infty} \limsup_{t \to T} E(v(t);B_{R_n}(x_0) \cap B^2) > \varepsilon.
    $$
    As $|\nabla v|^2 = h_r^2 + \frac{k^2}{r^2}\sin(h)^2$ is radial, $E(v(t);B_R(x_0) \cap B^2) = E(v(t);B_R(x_1) \cap B^2)$ whenever $|x_0| = |x_1|$.

    Choose $M$ so that $M\varepsilon > E_0$. Choose $N$ large enough so that $R_N$ is small enough and one can choose $M$ distinct points $y_1,...,y_M$ on the circle $|y| = |x_0|$ satisfying $B_{R_N}(y_i) \cap B_{R_N}(y_j) = \emptyset$ for all $i \neq j$.

    Then
    \begin{align*}
    E_0 = \sup_{t \in [0,T]} E(v(t);B^2) \geq E(v(t);B^2) &\geq  \sum_{k = 1}^M E(v(t);B_{R_N}(y_k) \cap B^2) \\
    &=  M \cdot E(v(t);B_{R_N}(y_1) \cap B^2) \\
    &\geq M \cdot E(v(t);B_{R_n}(y_1) \cap B^2), \quad t \in [0,T],
\end{align*}
holds for all $n \geq N$, since $(R_n)_{n \geq 0}$ is decreasing. Applying $\limsup \limits_{n \to +\infty} \limsup\limits_{t \to T}$ shows that $E_0 > M\varepsilon > E_0$, a contradiction.
\end{proof}

In the following Lemmas, unless specified otherwise, the solution $v(x,t)$ need not to be equivariant or defined at the boundary $\partial B_R(x_0)$ or at time $t \in \{0,T\}$. 

\begin{corollary} \label{cor:extension of v to time T}
     Let $v \in C^2(B_R(x_0) \times (0,T), S^2)$, $0 < R < 1-|x_0|$, solve
    \begin{align}
    v_t = \Delta v + |\nabla v|^2v, \quad (x,t) \in B_R(x_0)  \times (0,T), \label{local heat map flow}
    \end{align} with
    $$
E_0 := \sup_{0 < t < T}E(v(t), B_R(x_0)) < +\infty, \quad \nabla^2 v \in L^2(B_R(x_0) \times (0,T)).
    $$
    Then $v$ extends to $v \in C^{\infty}(B_{R}(x_0) \times (0,T], S^2)$.
\end{corollary}

\begin{proof}
See \cite[Chapter 3, Lemma 5.11]{struwe2008variational}. We note that Lemma 5.11 requires a local estimate on the $L^2$-norm of $|v_t|$, as well as the energy at a later time, i.e., 
    \begin{align}
             \int_{t_1}^{t_2} \int_{B_{R/2}(x_0) } \frac{1}{2}  |v_t|^2dxdt  &+ E(v(t_2);B_{R/2}(x_0) )  \notag \\
             &\leq E(v(t_1);B_{R/2}(x_0)) + \frac{128(t_2-t_1)}{R^2}\sup_{t_1 < t < t_2} E(v(t);B_{R}(x_0)). \label{estimate on v_t L^2 norm}
    \end{align}
    This inequality is proved in \cite[Chapter 3, Lemma 5.9]{struwe2008variational} (the estimate for $|v_t|$ does not appear in the statement, but in the proof of the lemma). In both lemmas, the arguments are completely local and does not depend on the monotonicity of the energy, only on the finiteness. The idea of Lemma 5.11 is to use local energy inequalities to obtain sufficient $L^p$-integrability of $|\nabla v|$ around any point. As $v$ is bounded, this then allows to apply the local Schauder estimates (Theorem \ref{interior schauder estimates}, Remark \ref{weakening of interior Schauder estimates}).
\end{proof}

\begin{lemma} \label{lemma: struwe analog estimate}
    Let $\phi(x) \in C^{\infty}_c(B_{R}(x_0) )$, $0 \leq \phi \leq 1$, be a smooth cut-off function which is $1$ on $B_{R/2}(x_0) $ and $|\nabla \phi| \leq 4/R$ on $B_R(x_0)$. There exists $0 < \varepsilon_1 < 1$ and $C > 0$ (independent of $x_0$, $R$, $t_1$, $t_2$, $T$ and $\phi$) such that if $v \in C^2(B_R(x_0) \times (0,T), S^2)$, $0 < R < 1-|x_0|$, solves (\ref{local heat map flow}) with
    $$
    \sup_{t_1 < t < t_2} E(v(t);B_R(x_0) ) \leq \varepsilon_1
    $$
    for some $0 < t_1 < t_2 < T$, then for all $t \in [t_1,t_2]$, one has
    \begin{align*}
        \int_{B_{R/2}(x_0)} |D^2_x v(x,t)|^{2}dx &\leq  \frac{C}{R^2} \int_{B_R(x_0) } | \nabla v(x,t)|^2 dx + C \int_{B_R(x_0) }|\partial_t v(x,t)|^2 \phi^2 dx
    \end{align*}
    and
    \begin{align*}
        \int_{t_1}^{t_2}\int_{B_{R/4}(x_0)} |D^2_x v(x,t)|^{2}dxdt &\leq C\sup_{t_1 < t < t_2} E(v(t);B_R(x_0) )\left(  1+\frac{(t_2-t_1)}{R^2}  \right).
    \end{align*}
\end{lemma}

\begin{proof}
The proof is similar to \cite[Chapter 3, Lemma 5.10]{struwe2008variational}. For a fixed time $0 < t_1 < t < t_2 < T$, $\partial_t v = \Delta v + |\nabla v|^2 v$ and $|v| = 1$ implies
    \begin{align*}
     I = \int_{B_R(x_0) } \phi^2  |\Delta v|^2 dx &\leq   C \int_{B_R(x_0) } |\nabla v|^4 \phi^2  dx + C \int_{B_R(x_0) }|\partial_t v|^2 \phi^2 dx.
    \end{align*}
    Apply Lemma 5.7 from \cite{struwe2008variational} to get
     \begin{align*}
     I \leq C \varepsilon_1 \left( \int_{B_R(x_0) } | D^2_x v|^2 \phi^2 dx + R^{-2}  \int_{B_R(x_0) } |\nabla u|^2 dx  \right) +  C \int_{B_R(x_0) }|\partial_t v|^2 \phi^2 dx.
    \end{align*}
    Choose $0 < \varepsilon_1 < 1$ small enough so that $C\varepsilon_1 \leq 1/4$ to obtain 
    $$
    I \leq \frac{1}{4} \int_{B_R(x_0) }  \phi^2 |D^2_x v|^2 dx + \frac{C}{R^2} \int_{B_R(x_0) } | \nabla v|^2 dx + C \int_{B_R(x_0) }|\partial_t v|^2 \phi^2dx.
    $$
    A simple integration by parts and Young's inequality for products shows that if $u \in H^2(B_R(x_0),\mathbb R)$, $B_R(x_0) \subset \mathbb R^2$, and $\phi(x) \in C^{\infty}_c(B_{R}(x_0))$, $0 \leq \phi \leq 1$, is any cut-off function with $|\nabla \phi|^2 \leq C/R^2$, then
    $$
\int_{B_R(x_0) } \phi^2  (\Delta u)^2 dx \geq \frac{1}{2} \int_{B_R(x_0) } \phi^2  |D^2_x u|^2 dx - 12C^2R^{-2} \int_{B_R(x_0) } |\nabla u|^2 dx
    $$
    Hence, we also obtain
    $$
 \frac{1}{2} \int_{B_R(x_0) } \phi^2  |D^2_x u|^2 dx - \frac{12C^2}{R^2} \int_{B_R(x_0) } |\nabla u|^2 dx \leq I,
    $$
    which finishes the proof for fixed time. Integrating on $[t_1,t_2]$ and using (\ref{estimate on v_t L^2 norm}) shows the other inequality.
\end{proof}

\begin{lemma}\label{lemma:no concentration of energy away origin, general case}
    Let $v(x,t) \in C^{2,1}_{x,t}(B^2 \times (0,T))$ be a finite-energy solution to
    \begin{align*}
    v_t = \Delta v + |\nabla v|^2v, \quad (x,t) \in B^2  \times (0,T).
    \end{align*}
   For all $\varepsilon > 0$, for all but finitely many $x_0 \in B^2$ (this number might depend on $\varepsilon$), there is $\delta > 0$ for which $0 < R < \delta$ implies
    $$
    \limsup_{t \to T} E(v(t);B_R(x_0)) \leq \varepsilon.
    $$
    In other words, there is no concentration of energy as $t \to T$ around all but finitely many points in $B^2$. If $v(x,t)$ is $k$-equivariant, the only point where concentration of energy could happen is the origin.
\end{lemma}

\begin{proof}
    The proof is as in Lemma \ref{lemma:no concentration of energy away origin, k equivariant}. If there were too many points $x_0 \in B^2$ where the energy concentrates with a value $> \varepsilon$, then the sum of their local contributions to the energy would be larger than
    $$
    \sup_{0 < t < T} E(v(t);B^2),
    $$
    which is a contradiction.
\end{proof}

\begin{theorem}[\cite{struwe2008variational}, Chapter 3, Theorem 5.6]\label{thm:smoothness around all but finitely many points}
    Let $v(x,t) \in C^2(B^2 \times (0,T), S^2)$ be a finite-energy solution to
    \begin{align*}
    v_t = \Delta v + |\nabla v|^2v, \quad (x,t) \in B^2  \times (0,T),
    \end{align*}
    with $T < +\infty$. There exists finitely many points $x_1, ..., x_m \in B^2$ such that $$v(x,t) \in C^{\infty}(B^2 \times (0,T)) \cap C^{\infty}(B^2 \setminus \{x_1,...,x_m\} \times (0,T]),$$
    and if $v(x,t)$ is $k$-equivariant, then $\{x_1,...,x_m\} = \emptyset$ or $\{0\}$.

    Moreover, around any singular points $(x_0,T)$, $x_0 \in \{x_1,..,x_m\}$, for any sequence $r_n \to 0^+$, there are sequences $R_n \to 0^+$, $0 < R_n < r_n$, $T_n \to T^-$, $x_n \to x_0$ such that for almost every $- \varepsilon_1/(64E_0) < s  < 0$, up to passing to a subsequence (which might depend on $s$), $v_n(x) = v(x_n + R_nx, T_n + R_n^2s)$  converges weakly in $H^{2,2}_{loc}(\mathbb R^2,S^2)$, strongly in $H^{1,2}_{loc}(\mathbb R^2 ,S^2)$ and uniformly on any compact set to some smooth non-constant harmonic map $v_{\infty}(x)$, i.e.,  $v_{\infty}(x)$ solves $-\Delta v_{\infty} = |\nabla v_{\infty}|^2 v_{\infty}$.

    Finally, if $v$ solves (\ref{heat map flow}) with smooth and $k$-equivariant initial data and if $(0,T)$ is a singular point, then for almost every $- \varepsilon_1/(64E_0) < s  < 0$, up to passing to another subsequence, $V_n(x) = v(R_nx, T_n + R_n^2s)$ also converges weakly in $H^{2,2}_{loc}(\mathbb R^2,S^2)$, strongly in $H^{1,2}_{loc}(\mathbb R^2 ,S^2)$ and uniformly on any compact set to $V_{\infty}(x) = v_{\infty}(x+x^*)$ for some $x^* \in \mathbb R^2$. Moreover, if $h(r,t)$, $H(r)$ are respectively the inclination coordinates (\ref{heat map formulation for h(r,t)}) of $v$, $V_{\infty}$, then  $h(R_n r, T_n + R_n^2s)$ converges uniformly to $H(r)$ on all compact subsets of $[0,+\infty)$ and $H(r)$ must be of the form:
\begin{align}
    H(r) = m\pi \pm 2\arctan \left( (\alpha r)^k \right), \alpha > 0, \label{inclination coordinate of harmonic map}
\end{align}
where $h(0,t) = m\pi$ on $[0,T)$.
\end{theorem}

\begin{remark}
    In particular, $\lim \limits_{\substack{x \to x_0 \\ t \to T^-}} v(x,t)$ does not exist for any singular point $(x_0,T)$.
\end{remark}

\begin{proof}
    First, $u \in C^{\infty}(B^2 \times (0,T))$ because $u \in C^2(B_R(x_0) \times (t_0,t_1))$ satisfies the hypotheses of Corollary \ref{cor:extension of v to time T} (after applying a space and time translation) for any fixed $\overline{B_R(x_0)} \subset B^2$ and $0 < t_0 < t_1 < T$. 
    
    Around any point $x_0 \in B^2 \setminus \{x_1,...,x_m\}$ where there is no concentration of energy, one can take $0 < R < 1 - |x_0|$ small enough so that
    $$
    \sup_{0 < t < T} E(v(t);B_R(x_0)) < \varepsilon_1,
    $$
    where $\varepsilon_1$ is as in Lemma \ref{lemma: struwe analog estimate}.
    
    Then this lemma yields $\nabla^2 v \in L^2(B_{R/2}(x_0) \times (0,T))$ and we can apply Corollary \ref{cor:extension of v to time T} on $B_{R/2}(x_0) \times (0,T)$ to deduce that $v(x.t)$ can be extended to $B_{R/2}(x_0) \times (0,T]$. Hence, $v(x,t) \in C^{\infty}(B^2 \setminus \{x_1,...,x_m\} \times (0,T])$, where $\{x_1, ..., x_m\} = \emptyset$ or $\{0\}$ is $v$ has smooth and $k$-equivariant initial data (Lemma \ref{lemma:no concentration of energy away origin, k equivariant}).

Now fix a singular point $(x_s,T) \in B^2 \times \{T\}$. By concentration, for all decreasing sequences $r_n \to 0$,
\begin{align}
    \lim_{n \to +\infty} \limsup_{t \to T} E(v(t);B_{r_n}(x_s)) > \varepsilon_1, \label{singularity, concentration}
\end{align}
where the limit exists because $n \mapsto \limsup \limits_{t \to T} E(v(t);B_{r_n}(x_s))$ is bounded by $E_0$ and decreasing, hence convergent. 

Let $0 < \delta < (1-|x_s|)/8$ be small enough so that there is no other singular point in $\overline{B_{2\delta}(x_s)}$. One can prove that there are sequences $0 < R_n < r_n$, $T_n \to T^-$, $x_n \to x_s$ and an index $N \in \mathbb N$ such that one has $T_n \geq T/2, x_n \in B_{\delta}(x_s)$ and
\begin{align}
    E(v(T_n);B_{R_n}(x_n)) = \sup_{x \in B_{2\delta}(x_s)} \sup_{T/2 < t < T_n} E(v(t);B_{R_n}(x)) = \varepsilon_1  \label{energy equation leading to harmonic map extraction}
\end{align}
for all $n \geq N$. 

Let $\tau(r)$ be defined via 
 $$
\tau(r) = t_0 r^2, \quad t_0 =  \frac{\varepsilon_1}{64 E_0},
$$
and set $\tau_n = \tau(R_n)$.

Given our sequences and starting at a larger index for which $T_n - \tau(R_n) \geq T/2$, one has 
$$
E(v(T_n);B_{R_n}(x_n)) = \sup_{x \in B_{2\delta}(x_s)} \sup_{T_n - \tau(R_n) < t < T_n} E(v(t);B_{R_n}(x)) = \varepsilon_1 
$$
for all $n$.

We now define
$$
v_n(x,t) = v(x_n + R_nx, T_n + R_n^2 t): B_{\delta/R_n}(0) \times [-t_0, 0] \to S^2,
$$
which satisfies 
$$
\sup_{-t_0 < t < 0} E(v_n(t); B_{\delta/R_n}(0)) \leq E_0 = \sup_{0 < t < T} E(v(t); B^2).
$$
First, observe that
\begin{align*}
   || \partial_t v_n||_{L^2([-t_0,0] \times B_{\delta/R_n}(0))}^2 &= \int_{T_n-\tau_n}^{T_n} \int_{B_{\delta}(x_n)} |\partial_t v(y, t) |^2 dydt \\
   &\leq  \int_{T_n-\tau_n}^{T_n} \int_{B_{2\delta}(x_s)} |\partial_t v(y, t) |^2 dydt.
\end{align*}
Yet, by \cite[Lemma 5.9, Chapter 3,]{struwe2008variational} (the estimate for $|v_t|$ does not appear in the statement but in the proof of the lemma),
$$
 \int_{0}^{T} \int_{B_{2\delta}(x_s)} |\partial_t v(y, t) |^2 dydt \leq C(\delta, T, E_0) < +\infty,
$$
meaning that 
\begin{align}
    \lim_{n \to +\infty} || \partial_t v_n||_{L^2([-t_0,0] \times B_{\delta/R_n}(0))}^2 \leq \lim_{n \to +\infty} \int_{T_n-\tau_n}^{T_n} \int_{B_{2\delta}(x_s)} |\partial_t v(y, t) |^2 dydt = 0 \label{l^2 norm of partial_t v_n goes to zero}
\end{align}
since the measure of the set $[T_n - \tau_n,T_n] \times B_{2\delta}(x_s)$ goes to zero. 

In particular, $t \mapsto || \partial_t v_n(\cdot,t)||_{L^2(B_{\delta/R_n}(0))}^2$ is a sequence converging to zero in $L^1([-t_0,0])$. Up to passing to a subsequence (which we still denote with the index $n$), for all $t \in [-t_0,0] \setminus N$, $|N| = 0$, one has pointwise convergence 
\begin{align}
     \lim_{n \to +\infty} ||\partial_t v_n(\cdot, t)||_{L^2(B_{\delta/R_n}(0))}^2 = 0. \label{convergence of partial_t v_n(,t)}
\end{align}

Furthermore, if $x \in B_{\delta R_n^{-1}-1}(0)$ is fixed, then $B_1(x) \subset B_{\delta/R_n}(0)$, $x_n + R_n y \in B_{2\delta}(x_s)$ for all $y \in B_{\delta R_n^{-1}-1}(0)$, and, for all $t \in [-t_0,0]$,
\begin{align}
    E(v_n(t);B_1(x)) &= \int_{x_n + B_{R_n}(R_n x)} |\nabla v(y, T_n + R_n^2 t) |^2dy \notag \\
    &\leq \sup_{T_n - \tau_n < t < T_n}E(v(t);B_{R_n}(x_n+R_nx))  \notag \\
    &\leq \sup_{y \in B_{2\delta}(x_s)}  \sup_{T_n - \tau_n < t < T_n}E(v(t);B_{R_n}(y)) = \varepsilon_1. \label{norm less than espilon on ball of radius 1}
\end{align}
Hence, for all $x \in B_{\delta R_n^{-1}-1}(0)$, for all $t \in (-t_0,0) \setminus N$,
\begin{align}
     \int_{B_{1/2}(x)} |D^2_x v_n|^{2}dy \leq C\int_{B_1(x_0) } | \nabla v_n|^2 dx + C \int_{B_1(x_0) }|\partial_t v_n|^2 dx \quad \forall n \geq 0 \label{local hessian bound}
\end{align}
by Lemma \ref{lemma: struwe analog estimate}.

Using a covering argument (in the same spirit of those of Vitali and Besikovitch), one can extend this estimate to $B_{\delta R_n^{-1}-1}(0)$, i.e., 
\begin{align}
     \int_{B_{\delta R_n^{-1}-1}(0)} |D^2_x v_n(x,t)|^{2}dx  &\leq C_2 \left(   C \int_{B_{\delta /R_n}(0) } | \nabla v_n|^2 dx + C \int_{B_{\delta /R_n}(0) } |\partial_t v_n|^2 dx   \right) \notag \\
    &\leq C \left( E_0 + \int_{B_{\delta /R_n}(0) } |\partial_t v_n(x,t)|^2 dx   \right) \leq C(t) < +\infty \label{unif bounded hessian}
    \end{align}
for all $t \in (-t_0,0)  \setminus N$ and $n \geq 0$, where $C_2$ is a dimensional constant. If such a $t$ is fixed, the upper bound is uniform with respect to $n$ thanks to (\ref{convergence of partial_t v_n(,t)}).

Finally, fix any $s \in (-t_0,0)  \setminus N$ and let
$$
\tilde{v}_n(x) = v(x_n + R_nx, T_n + R_n^2 s): B_{\delta/R_n}(0) \to S^2.
$$
By construction, on any compact set $K \subset \mathbb R^2$, $|\tilde{v}_n| = 1$, $\partial_t v_n(\cdot,s) \to 0$ in $L^2(K)$, $D^2_x \tilde{v}_n$ is uniformly bounded in $L^2(K)$ as $n \to +\infty$, meaning that $\tilde{v}_n \to \tilde{v} \in H^{2,2}_{loc}(\mathbb R^2,S^2)$ weakly along some subsequence (which we still denote with the index $n$) and $\tilde{v}_n \to \tilde{v} \in H^{1,2}_{loc}(\mathbb R^2,S^2)$ strongly by Rellich–Kondrachov Theorem. The convergence is also uniform on any compact set because of the embeddings $H^{2,2}(K) \hookrightarrow C^{0,\alpha}(K) \subset \subset C^0(K)$, $K \subset \mathbb R^2$. Since
$$
\partial_t v_n = \Delta v_n + |\nabla v_n|^2 v_n, \quad (x,t) \in B_{\delta/R_n}(0) \times [-t_0, 0],
$$
fixing $t = s$ and passing to the limit (in the sense of distributions), we find that $\tilde{v}$ is harmonic. Finally, $\tilde{v}$ is smooth by \cite[Theorem 3.6]{sacks-uhlenbeck}, and non-constant because 
\begin{align}
    E(\tilde{v}; B_1(0)) &= \lim_{n \to +\infty} E(\tilde{v}_n; B_1(0)) = \lim_{n \to +\infty} E(v_n(s), B_1(0)) \notag \\
    &\geq \lim_{n \to +\infty} \left( E(v_n(0); B_1(0)) - 32 s E_0  \right) \notag \\
    &\geq \lim_{n \to +\infty} \left( E(v(T_n); B_{R_n}(x_n)) - 32 t_0 E_0  \right) \notag \\
    &\geq \varepsilon_1 - \varepsilon_1/2 > 0 \label{non constant limit}
\end{align} 
using (\ref{estimate on v_t L^2 norm}) (on the translated time-interval $[-t_0,0]$ instead of $[0,T]$). This concludes the proof in the general setting.

If $v$ solves (\ref{heat map flow}) with smooth and $k$-equivariant initial data, then the only possible singular point in $B^2$ is $(x_s,T) = (0,T)$. In that case, we show that $x_nR_n^{-1}$ is a bounded sequence and converges to some $x^*$ up to taking a subsequence. Then if $\tilde{v}_n(x) = v(x_n + R_nx, T_n + R_n^2 s) \to v_{\infty}(x)$ uniformly on compact sets as above, 
$$\tilde{V}_n(x) =  v(R_nx, T_n + R_n^2 s) = v_n(x-R_n^{-1}x_n) \to v_{\infty}(x-x^*) = V_{\infty}(x)$$
uniformly on compact sets as well. Similarly, $\tilde{V}_n(x) \to V_{\infty}(x)$ converges weakly in $H^{2,2}_{loc}(\mathbb R^2)$ by continuity of the translation operator in $L^2$ and the limit $V_{\infty}$ is smooth and harmonic.

Thanks to radial symmetry, one has
\begin{equation}
    E(v(T_n); B_{R_n}(Ox_n)) = E(v(T_n); B_{R_n}(x_n)) > \varepsilon_1 \label{eq:radial symmetry, eneryg rotation matrices}
\end{equation}
for all $n \geq 0$ and all rotation matrices $O$.

One can fit $m \sim |x_n|R_n^{-1}$ disjoint balls $B_{R_n}(O_1x_n)$, ..., $B_{R_n}(O_mx_n)$ in the annulus
$$
C_n = \{x \in \mathbb R^2: |x_n| - R_n < |x| < |x_n| + R_n\},
$$
meaning that
$$
E_0 \geq E(v(T_n); C_n) > \varepsilon_1 |x_n|R_n^{-1}.
$$
Hence, $x_nR_n^{-1}$ is bounded and, up to passing to a subsequence, converges in $\mathbb R^2$ to a limit $x^*$. This finishes the proof of the first part of the statement in the equivariant setting.

% Fix a compact set $K \subset \mathbb R^n$. Let also be a constant $D$ such that $|x_nR_n^{-1}| \leq D$ for all $n$ and consider the larger compact set $\tilde{K} = \{x + y: x \in K, |y| \leq 3D\}$. Then
% \begin{align*}
%     \int_{K} \left| \tilde{V}_n(x + x_nR_n^{-1}) - \tilde{V}(x+x^*) \right|^2 dx &\lesssim \int_{K} \left| \tilde{V}_n(x + x_nR_n^{-1}) - \tilde{V}_n(x + x^*) \right|^2  \\
%     &+\int_{K} \left| \tilde{V}_n(x + x^*)  - \tilde{V}(x+x^*) \right|^2 dx
% \end{align*}
% The second term goes to zero as $n \to +\infty$ since $\tilde{V}_n \to \tilde{V}$ in $L^2(x^*+K)$. As for the first one, using Minkowski's inequality for integrals,
% \begin{align*}
%      \big\|  \tilde{V}_n(\cdot + x_nR_n^{-1}) - \tilde{V}_n(\cdot+x^*)\big\|_{L^2(K)} &= \bigg\| (x_nR_n^{-1} -x^*) \cdot \int_0^1 \nabla_x \tilde{V}_n(\cdot+x^*+s(x_nR_n^{-1}-x^*))ds \bigg\|_{L^2(K)} \\
%      &\leq |x_nR_n^{-1} -x^*| \int_0^1 \bigg\| \nabla_x \tilde{V}_n(\cdot+x^*+s(x_nR_n^{-1}-x^*))\bigg\|_{L^2(K)} ds \\
%      &\leq |x_nR_n^{-1} -x^*| \int_0^1 \big\| \nabla_x \tilde{V}_n(\cdot) \big\|_{L^2(\tilde{K})} ds
% \end{align*}
% As $\nabla_x \tilde{V}_n \to \nabla_x \tilde{V}$ in $L^2(\tilde{K})$, 
% $$
% \big\| \nabla_x \tilde{V}_n(\cdot) \big\|_{L^2(\tilde{K})} \leq C(\tilde{K}) < +\infty
% $$
% is uniformly bounded with respect to $n$. Letting $n \to +\infty$ finishes the first part of the statement in the equivariant setting.

Finally, let $\tilde{H}_n(r) = h(R_n r, T_n + R_n^2s)$ be the inclination coordinate (\ref{heat map formulation for h(r,t)}) of $\tilde{V}_n$, where $h$ denotes the smooth inclination coordinate of $v$, $h(0,t) = m \pi$ for all $t \in [0,T)$ for some $m \in \mathbb Z$ (Proposition \ref{value of v and h at origin}, which does not require Proposition \ref{prop:smoothness at the boundary} when omitting $t = T$). Observe that $\tilde{H}_n$ is uniformly bounded in $L^{\infty}(K)$ (Proposition \ref{prop:equivalent fomrulation heat map flow}) and $|\partial_r \tilde{H}_n|$ is uniformly bounded in $L^2(K)$ for any compact set $K \subset (0,+\infty)$ because the solution $v$ has finite energy (\ref{Energy of v in terms of h}). Up to passing to a subsequence, $\tilde{H}_n$ converges weakly in $H^{1,2}_{loc}((0,+\infty)) \hookrightarrow C^{0,1/4}_{loc}((0,+\infty)) \subset \subset C^0_{loc}((0,+\infty))$ and uniformly on any compact set to a limit $H_{\infty} \in C^{0}((0,+\infty))$. One must have
\begin{equation}
    V_{\infty}(re^{i\theta}) = \left(e^{ik \theta} \sin H_{\infty}(r), \cos H_{\infty}(r) \right)  \label{eq: H vs H_infty, 1}
\end{equation}
 by pointwise convergence $\tilde{V}_n \to V_{\infty}$, $\tilde{H}_n \to H_{\infty}$. But Lemma \ref{k-equivariance, smoothness of h} shows that 
\begin{equation}
    V_{\infty}(re^{i\theta}) = \left(e^{ik \theta} \sin H(r), \cos H(r) \right) \label{eq: H vs H_infty, 2}
\end{equation}
for some lifting $H \in C^{\infty}([0,+\infty))$. Hence, $H_{\infty}(r)$ and $H(r)$ must differ by a fixed multiple of $2\pi$ and $H_{\infty}$ is smooth as well. It follows that $H_{\infty} \in C^{\infty}([0,+\infty))$ is a smooth non-constant solution for the ODE
\begin{align}
     0 = H_{rr} + \frac{H_r}{r} - k^2 \frac{\sin(2H)}{2r^2}, \quad r \in (0,+\infty), \label{time independent ode for H v2}
\end{align}
 with initial condition $H_{\infty}(0) = \tilde{m}\pi$, where $\tilde{m}$ depends on the $k$-equivariant harmonic map $V_{\infty}$ (it is shown later that $\tilde{m} = m = h(0,t)$).  All the solutions to this problem are of the form:
\begin{equation} \label{all possible form for limiting H}
    \tilde{m}\pi \pm 2\arctan \left( (\alpha r)^k \right), \alpha > 0.
\end{equation}
 Indeed, fix $r_0 > 0$ with $H_{\infty}(r_0) \in ((\tilde{m}-1)\pi,(\tilde{m}+1)\pi) \setminus \{\tilde{m}\pi\}$ (this exists as $H_{\infty}$ is non-constant). Let $\varepsilon = 1$ if $H_{\infty}(r_0) > \tilde{m}\pi$ and $\varepsilon = -1$ otherwise. By selecting $\alpha$ appropriately, one can find a solution
 $$
\chi_{\alpha}(r) = \tilde{m}\pi + \varepsilon \cdot 2\arctan \left( (\alpha r)^k \right)
 $$
 of (\ref{time independent ode for H v2}), $\chi_{\alpha}(0) =\tilde{m}\pi$, $\chi_{\alpha}(r_0) = H_{\infty}(r_0)$. The Comparison Principle (Theorem \ref{comparison principle}) applied to $\theta = H_{\infty}$ and $\psi = \chi_{\alpha}$ (and then reversing the roles of $\theta, \psi$) shows that $H_{\infty} = \chi_{\alpha}$ on $[0,r_0]$. As $H_{\infty}(r_0/2) = \chi_{\alpha}(r_0/2)$ and $\partial_r H_{\infty}(r_0/2) = \partial_r \chi_{\alpha}(r_0/2)$, standard existence and uniqueness theory for regular initial-value problems shows that $H_{\infty} = \chi_{\alpha}$ on all $[0,+\infty)$.

 Finally, observe that $\tilde{H}_n(r)$ converges uniformly near $r = 0$ as well (up to taking a final further subsequence). First note that $\tilde{H}_n(0) = h(0,T_n + R_n^2s) = m\pi$ and $\tilde{V}_n(0) = (0,0,1) \in S^2$ for all $n$. As $\tilde{V}_n \to V_{\infty}$ uniformly around $x = 0$, the first coordinate belongs to $\tilde{V}_{n,1}(x) \in [-\pi/4,\pi/4]$ for all $n \geq n_0$ and $x \in B_{K}(0)$ for some $K,n_0 > 0$. On $[0,K]$, one must have
 $$
 \tilde{H}_n(r) = m_n \pi + \varepsilon_n \arcsin \left( \tilde{V}_{n,1}(re^{i0}) \right).
$$
As $\tilde{V}_{n,1}(0) = 0$, $\tilde{H}_n(r) = m \pi$, one must have $m_n = m$, which does not depend on $n$. Up to taking a subsequence, one can assume that $\varepsilon_n \in \{-1,1\}$ is constant as well. Hence, $\tilde{H}_n(r)$ converges uniformly on $[0,K]$.
 \end{proof}

As a Corollary from Struwe's result, one deduces:

\begin{proposition}[Smoothness at $t = T$] \label{prop:smoothness at the boundary}
Let
$$
     v(t,x) \in C^0([0,T),C^1_0(\overline{B^2})) \cap C^{\infty}((0,T) \times \overline{B^2}) 
$$
solve (\ref{heat map flow}) with smooth and $k$-equivariant initial data. Let 
$$
h(t,r) \in C^0([0,T),C^1([0,1])) \cap C^{\infty}((0,T) \times [0,1])
$$
be the corresponding inclination coordinate (\ref{heat map formulation for h(r,t)}). Assume that $v$ blows-up at time $T < +\infty$.

Then 
$$
v(t,x) \in C^{\infty}((0,T] \times \overline{B^2} \setminus \{0\} ), \quad h(t,r) \in  C^{\infty}((0,T] \times (0,1] ).
$$

\end{proposition}

\begin{proof}
Recall that $v$ has finite energy (Remark \ref{k-equivariance implies finite energy}).

    If $T < +\infty$, the energy concentration argument from Struwe (Theorem \ref{thm:smoothness around all but finitely many points}) shows that $v(x,t) \in C^{\infty}(B^2 \setminus \{0\} \times (0,T])$. Hence, writing $v(x,t)$ in its symmetric form (\ref{k-equivariant form of v(x,t)}), the corresponding 
    $$
    h(t,r) \in C^0([0,T),C^1([0,1])) \cap C^{\infty}((0,T) \times [0,1])
    $$
    has additional regularity $h(t,r) \in C^{\infty}((0,T] \times (0,1))$ at time $T$ (same argument as in the beginning of the proof of Lemma \ref{k-equivariance, smoothness of h}). We prove that $h$ is smooth at $r =1, t = T$ as well.

    The function defined as
$$
    \tilde{h}_0(t,r) = \begin{cases}
        h(t,r) &\text{ if } (t,r) \in  [T/2,T] \times  \{1/2\} \\
        h(t,r) &\text{ if } (t,r) \in \{T/2\} \times [1/2,1] \\
        h_0(t,r) &\text{ if } (t,r) \in [T/2,2T] \times \{1\}
    \end{cases}
 $$
is smooth on the connected, closed set where it is defined and can be extended to a smooth function $\tilde{h}_0(t,r) \in C^{\infty}_c(\mathbb R^2)$ (e.g. by Whitney's Extension Theorem and multiplication by a cut-off). Then $h$ locally solves
\begin{align}
   h_t &= h_{rr} + \frac{h_r}{r} - k^2 \frac{\sin(2h)}{2r^2}, \quad (r,t) \in (1/2,1) \times (T/2,+\infty),  \notag \\
h(r,t) &= \tilde{h}_0(r,t), \quad (r,t) \in \{1/2,1\} \times [T/2,+\infty) \cup (1/2,1) \times \{T/2\} ,  \label{equation extending h}
\end{align}
on $[T/2,T)$. 

We rewrite (\ref{extension of h using a tilde h_0}) as a Dirichlet problem $\tilde{h} = h + \tilde{h}_0$,
\begin{align}
   \tilde{h}_t &= \tilde{h}_{rr} + \frac{\tilde{h}_r}{r} - k^2 \frac{\sin(2\tilde{h}-2\tilde{h}_0)}{2r^2} + H_0, \quad (r,t) \in (1/2,1) \times (T/2,+\infty),  \notag \\
\tilde{h}(r,t) &= 0, \quad (r,t) \in \{1/2,1\} \times [T/2,+\infty) \cup (1/2,1) \times \{T/2\}, \label{extension of h using a tilde h_0}
\end{align}
where $H_0 = \tilde{h}_{0,t} - \tilde{h}_{0,rr} - \frac{\tilde{h}_{0,r}}{r}$ and shift the time to $0$. This corresponds to a 2D nonlinear heat equation on an annulus $C$ with a smooth, globally bounded and globally Lipschitz (w.r.t. to $P,Q$) nonlinearity
\begin{align*}
    (x,t,P,Q) \mapsto F(x,t,P,Q) &= - k^2 \frac{\sin(2P-2\tilde{h}_0)}{2|x|^2} + H_0(x,t) \\
    &= - k^2 \frac{\sin(2P-2\tilde{h}_0)}{2|x|^2} \chi_{(1/4,2)}(|x|) + H_0(x,t),
\end{align*}
where $\chi_{(1/4,2)} \in C^{\infty}_c(\mathbb R;[0,1])$ is a smooth cut-off which is $1$ on $(1/2,1)$ and zero on $(1/4,2)^c$. We can then apply Theorem \ref{thm:taylor local existence} and Theorem \ref{thm:resgularty solution nonlinear pb} to deduce that $\tilde{h} \in C^{\infty}((T/2,+\infty),\overline{C})$ must be global. In particular, $h$ extends as a function $h \in C^{\infty}((T/2,T] \times [0,1])$.
\end{proof}

An argument similar to Theorem \ref{thm:smoothness around all but finitely many points} allows to extract all the bubbles at the same time, leading to the following well-known bubble decomposition of blowing-up solutions along an appropriate time sequence. One should note that the decomposition holds only in the energy space and not uniformly. The harmonic maps appearing in the decomposition are called \textit{bubbles} of energy, as they account for all the energy that $v(t)$ loses at the blow-up time (see (\ref{pythagorean decomposition of energy})).

\begin{theorem}[Bubble decomposition along a sequence of times]\label{thm: full bubble decomposition}
    Let $v(x,t) \in C^{2}(B^2 \times (0,T), S^2)$ be a finite-energy solution to
    \begin{align*}
    v_t = \Delta v + |\nabla v|^2v, \quad (x,t) \in B^2  \times (0,T)
    \end{align*}
    with $T < +\infty$. 

    Around any singular points $(x_s,T)$, $x_s \in B^2$, there exists a sequence $T_n \to T^-$, there exists finitely many smooth non-constant harmonic maps $\{\omega_1, ..., \omega_p\}$ from $\mathbb R^2 \cup \{\infty\} \simeq S^2$ to $S^2$, there exists sequences $(x_n^{(i)})_{n \geq 0}$, $i \in \{1,...,p\}$, all converging to $x_s$, there exists positive sequences $(R_n^{(i)})_{n \geq 0}$, $i \in \{1,...,p\}$, all converging to zero, satisfying for all $i \neq j$,
    \begin{align}
        \frac{R_n^{(i)}}{R_n^{(j)}} + \frac{R_n^{(j)}}{R_n^{(i)}} + \frac{|x_n^{(i)}-x_n^{(j)}|^2}{R_n^{(i)}R_n^{(j)}} \to +\infty, \quad n \to +\infty \label{different rate of convergence of scaling of bubbles}
    \end{align}
    and
 \begin{align}
     v(x,T_n) - \sum_{i=1}^p \left[ \omega_i \left(  \frac{x-x_n^{(i)}}{R_n^{(i)}}\right) - \omega_i(\infty) \right] \to v(x,T), \quad n \to +\infty \label{multi bubble decomposition along a sequence of times}
 \end{align}
    strongly in $H^{1,2}(B_{\delta}(x_s))$ for all balls $B_{\delta}(x_s)$ such that $\overline{B_{2\delta}(x_s)} \subset B^2$ contains no other singular points from $\{x_1,...,x_m\}$. Moreover,
\begin{equation}
        \lim_{t \to T^-} E(v(t);B_{\delta}(x_s)) = E(v(T);B_{\delta}(x_s)) + \sum_{i=1}^p E(\omega_i;\mathbb R^2). \label{pythagorean decomposition of energy}
\end{equation}
\end{theorem}

\begin{proof}
    See \cite[Theorem 1 and Remark 6]{Qing1995OnSO}. The argument is essentially local around the singularity and does not depend on the boundary.
\end{proof}

\begin{corollary}[Recovering the bubbles and decomposition in the $k$-equivariant setting]\label{rmk:recovering the bubbles}
    Let $v$ be as in Theorem \ref{thm: full bubble decomposition}. Order the indices $\{1,...,p\}$ so that $R_n^{(i)}$ has faster or comparable decay to $R_n^{(j)}$, $j > i$, meaning that $R_n^{(i)}/R_n^{(j)}$ is bounded. For $j \in \{1,...,p\}$, up to taking a subsequence, there is $M = M(j)$ large enough and some constant $c_j \in \mathbb R^3$ for which
    \begin{align*}
    v_n^{(j)}(x) &= \left(v(y,T_n) - \sum_{i=1}^{j-1} \left[ \omega_i \left(  \frac{y-x_n^{(i)}}{R_n^{(i)}}\right) - \omega_i(\infty) \right] \right)_{| y = x_n^{(j)} + R_n^{(j)}x} \\
    &=: V_n^{(j-1)}(x_n^{(j)} + R_n^{(j)}x)
\end{align*}
converges strongly to $\omega_{j}(x) + c_{j}$ in $H^{1,2}_{loc}(\mathbb R^2,B_{M}(0))$. When $j = 1$, the convergence and the limit $\omega_1(x) + c_1$ are actually in $H^{1,2}_{loc}(\mathbb R^2,S^2)$.

In the $k$-equivariant setting, keeping the same sequences $T_n$ and $R_n$, (\ref{multi bubble decomposition along a sequence of times}) can be rewritten as:
$$
v(x,T_n) - \sum_{i=1}^p \left[ \tilde{\omega}_i \left(  \frac{x}{R_n^{(i)}}\right) - \tilde{\omega}_i(\infty) \right] \to v(x,T), \quad n \to +\infty.
$$
where $\tilde{\omega}_i(x) = \omega_i(x+x_i^*)$, $x_i^* \in \mathbb R^2$, is a smooth, non-constant, harmonic map from $\mathbb R^2 \cup \{\infty\} \simeq S^2$ to $S^2$ for which $|\nabla \tilde{\omega}_i|$ is radial. Moreover, for $i \in \{1,...,p\}$, there exists $\tilde{c}_i \in \mathbb R^3$ for which
$$
v(R_n^{(i)}x,T_n) \to \tilde{\omega}_i(x) + \tilde{c}_i
$$
in $H^{1,2}_{loc}(\mathbb R^2,S^2)$, where $\tilde{\omega}_i(x) + \tilde{c}_i$ is a $k$-equivariant harmonic map from $\mathbb R^2 \cup \{\infty\} \simeq S^2$ to $S^2$ with inclination coordinate of the form (\ref{inclination coordinate of harmonic map}).
\end{corollary}

\begin{remark}
    On the level of the inclination coordinate $h$ of $v$, we have uniform convergence of $h(R_n^{(i)}r,T_n)$ to the inclination coordinate of $\tilde{\omega}_i(x) + \tilde{c}_i$ on any compact subset $K \subset (0,+\infty)$ thanks to the one-dimensional Sobolev embedding  $H^{1,2}_{loc}((0,+\infty)) \hookrightarrow C^{0,1/4}_{loc}((0,+\infty))$.
\end{remark}

\begin{proof}
Up to taking a subsequence, $R_n^{(i)}/R_n^{(j)}$ converges to a finite value in $\mathbb R_{\geq 0}$. Up to passing to further subsequences, we will also assume that for all $i \neq j$,
$$
\frac{x_n^{(i)}-x_n^{(j)}}{R_n^{(i)}}
$$
either converges in $\mathbb R^2$ or is unbounded.

Let
\begin{align*}
    A_n^{(i)} = \int_{B_{KR_n^{(1)}}(x_n^{(1)})}\left| \left(R_n^{(i)}\right)^{-1} \nabla \omega_i \left(\frac{x-x_n^{(i)}}{R_n^{(i)}} \right) \right|^2 dx = \int_{B_{KR_n^{(1)}/R_n^{(i)}}\left( \frac{x_n^{(1)}-x_n^{(i)}}{R_n^{(i)}}\right)}\left| 
    \nabla \omega_i \left(x \right) \right|^2 dx.
\end{align*}
If $i = 1$, one has $A_n^{(1)} = \int_{B_K(0)} |\nabla \omega_1(x)|^2 dx > 0$. Assume now that $i \neq 1$. If $R_n^{(1)}/R_n^{(i)} \to 0$, then $A_n^{(i)} \to 0$ because the domain of integration is shrinking to a zero measure set. Else, $R_n^{(1)}/R_n^{(i)} \to c > 0$ and one has
$$
\left( \frac{|x_n^{(1)}-x_n^{(i)}|}{R_n^{(i)}} \right)^2 \gtrsim \frac{|x_n^{(1)}-x_n^{(i)}|^2}{R_n^{(1)}R_n^{(i)}} \to +\infty,
$$
i.e., the domain $B_{KR_n^{(1)}/R_n^{(i)}}\left( \frac{x_n^{(1)}-x_n^{(i)}}{R_n^{(i)}}\right)$ is drifting to infinity, while keeping a measure which is uniformly bounded with respect to $n$. In other words, the characteristic function of the domain $B_{KR_n^{(1)}/R_n^{(i)}}\left( \frac{x_n^{(1)}-x_n^{(i)}}{R_n^{(i)}}\right)$ converges, in a pointwise sense, to zero as $n \to +\infty$ and it follows from Dominated Convergence Theorem that $A_n^{(i)} \to 0$ as well.

Similarly, the mixed term
\begin{align*}
    B_n^{(i,j)} = \int_{B_{KR_n^{(1)}}(x_n^{(1)})} \left| \left(R_n^{(i)}\right)^{-1} \nabla \omega_i \left(\frac{x-x_n^{(i)}}{R_n^{(i)}} \right) \right| \cdot \left| \left(R_n^{(j)}\right)^{-1} \nabla \omega_j \left(\frac{x-x_n^{(j)}}{R_n^{(j)}} \right) \right| dx
\end{align*}
converges to zero as $n \to +\infty$ by Cauchy-Schwarz inequality and the convergence of $A_n^{(i)}$.

One also has 
\begin{align*}
    C_n = \int_{B_{KR_n^{(1)}}(x_n^{(1)})}\left| \nabla v(x,T) \right|^2 dx \to 0, \quad n \to +\infty,
\end{align*}
as the domain is shrinking to a zero measure set and
\begin{align*}
    D_n^{(i)} = \int_{KB_{R_n^{(1)}}(x_n^{(1)})}\left| \nabla v(x,T) \right| \cdot  \left| \left(R_n^{(i)}\right)^{-1} \nabla \omega_i \left(\frac{x-x_n^{(i)}}{R_n^{(i)}} \right) \right|  dx \to 0, \quad n \to +\infty,
\end{align*}
for any $i \in \{1,...,p\}$ by Cauchy-Schwarz inequality.

It follows from the decomposition that 
$$
\int_{B_{KR_n^{(1)}}(x_n^{(1)})} \left| \nabla v(x,T_n) - \left(R_n^{(1)}\right)^{-1} \nabla \omega_1 \left( \frac{x-x_n^{(1)}}{R_n^{(1)}} \right) \right|^2 dx \to 0,
$$
i.e., if $v_n^{(1)}(x) = v(x_n^{(1)} + R_n^{(1)}x, T_n)$, then $\nabla v_n^{(1)}(x) \to \nabla \omega_1$ in $L^2(B_K(0))$. By Poincaré's inequality, up to taking a subsequence, $v_n^{(1)}$ converges strongly in $H^{1,2}_{loc}$ to $\omega_1(x) + c_1 \in H^{1,2}_{loc}$ for some constant $c_1$. As $|v| = 1$, convergence is in $H^{1,2}_{loc}(\mathbb R^2,S^2)$.

 Similarly, one proves that 
\begin{align*}
    v_n^{(j+1)}(x) &= \left(v(y,T_n) - \sum_{i=1}^j \left[ \omega_i \left(  \frac{y-x_n^{(i)}}{R_n^{(i)}}\right) - \omega_i(\infty) \right] \right)_{| y = x_n^{(j+1)} + R_n^{(j+1)}x} \\
    &=: V_n^{(j)}(x_n^{(j+1)} + R_n^{(j+1)}x)
\end{align*}
converges strongly to $\omega_{j+1} + c_{j+1}$ in $H^{1,2}_{loc}(\mathbb R^2,B_{M}(0))$ for some fixed $M = M(j)$ large enough (the maps $v$ and $\omega_i$ are bounded, but we can no longer ensure $|v_n^{(j+1)}| = 1$).

% Moreover, on compact sets of $\mathbb R^n$, it follows from the $H^{1,2}$-convergence it follows from $L^2$-convergence that $v_n^{(1)}(x) \to \omega_1 + c_1 \in L^2(B_K(0),S^2)$ for some constant $c_1$ because $u(x_n^{(1)} + R_n^{(1)}x,T) \to u(0,T)$ (Proposition \ref{value of v and h at origin}) and for $i > 1$,
% either
% $$
% \frac{R_n^{(1)}}{R_n^{(i)}} \to 0, \quad \frac{x_n^{(1)} - x_n^{(i)}}{R_n^{(i)}} \to x^* \in \mathbb R^2 \cup \{+\infty\}, \quad \omega_i\left( \frac{R_n^{(1)}x + x_n^{(1)} - x_n^{(i)}}{R_n^{(i)}} \right)  \to \omega_i(x^*) 
% $$
% or
% $$
% \frac{R_n^{(1)}}{R_n^{(i)}} \to c > 0, \quad \frac{|x_n^{(1)} - x_n^{(i)}|}{R_n^{(i)}} \to +\infty, \quad \omega_i\left( \frac{R_n^{(1)}x + x_n^{(1)} - x_n^{(i)}}{R_n^{(i)}} \right)  \to \omega_i(\infty) 
% $$
% uniformly in $B_K(0)$. In other words, $v_n^{(1)}(x) \to \omega_1 + c_1$ strongly in $H^{1,2}_{loc}(\mathbb R^2,S^2)$. 

In the $k$-equivariant setting, one can proceed as in the end of the proof of Theorem \ref{thm:smoothness around all but finitely many points} (see the energy argument below equation (\ref{eq:radial symmetry, eneryg rotation matrices})) to show that $x_n^{(1)}/R_n^{(1)}$ is bounded thanks to radial symmetry and, up to taking a subsequence, converges. Hence, one can always replace $(x_n^{(1)})$ by $(0)$ in the decomposition and $\omega_1(x)$ by a translate $\omega_1(x + x_1^*)$ without changing the conclusion of Theorem \ref{thm: full bubble decomposition}. As $|\nabla v|$ is radial, so is the case for the limiting harmonic map, i.e., $x \mapsto |\nabla \omega_1(x+x_1^*)|$ is radial. Assume then by induction that $(x_n^{(i)}) = (0)$ and that $|\nabla \omega_i|$ is radial for $i \leq j$. In particular, $|\nabla V_n^{(j)}|$ is radial.

Then 
$$
E(V_n^{(j)};B_{KR_n^{(j+1)}}(x_n^{(j+1)})) = E(v_n^{(j+1)};B_{K}(0)) \to E(\omega_{j+1};B_K(0))  > 0, \quad n \to +\infty,
$$
for some fixed $K = K(j)$ large enough and 
$$
E(V_n^{(j)};B_{KR_n^{(j+1)}}(x_n^{(j+1)})) \leq E(V_n^{(j)};B_{1}(0)) \leq E_0 + \sum_{i=1}^j E(\omega_i;\mathbb R^2),
$$
where $E_0$ is our usual upper bound on the energy of $v(x,t)$ on $B_1(0) \times (0,T)$. As in the end of the proof of Theorem \ref{thm:smoothness around all but finitely many points}, radial symmetry shows that $x_n^{(j+1)}/R_n^{(j+1)}$ is bounded and, up to taking a subsequence, converges. Then one can always replace $(x_n^{(j+1)})$ by $(0)$ in the decomposition and $\omega_{j+1}(x)$ by a translate $\omega_{j+1}(x + x_{j+1}^*)$. As $|\nabla V_n^{(j)}|$ is radial, the translate of the harmonic map is such that $x \mapsto |\nabla \omega_{j+1}(x+x_{j+1}^*)|$ is radial, which finishes the induction proof.

We have now proved that (\ref{multi bubble decomposition along a sequence of times}) can be rewritten as:
$$
v(x,T_n) - \sum_{i=1}^p \left[ \tilde{\omega}_i \left(  \frac{x}{R_n^{(i)}}\right) - \tilde{\omega}_i(\infty) \right] \to v(x,T), \quad n \to +\infty.
$$
in $H^{1,2}_{loc}(B_{\delta}(0))$, where $\tilde{\omega}_i(x) = \omega_i(x+x_i^*)$ is a non-constant harmonic map for which $|\nabla \tilde{\omega}_i|$ is radial. It is then easier to recover each $\tilde{\omega}_i$. Simply observe that 
$$
v(R_n^{(i)}x,T_n) \to \tilde{\omega}_i(x)
$$
in $\dot{H}^{1,2}_{loc}(\mathbb R^2, B_{M}(0))$. By Poincaré's inequality, up to taking subsequences, there exists $c_i$ so that $v(R_n^{(i)}x,T_n) \to \tilde{\omega}_i(x) + c_i$ strongly in $H^{1,2}_{loc}(\mathbb R^2,S^2)$. As the image of $v(R_n^{(i)}x,T_n)$ is in $S^2$, it follows that $\tilde{\omega}_i + c_i$ is also a non-constant harmonic map from $\mathbb R^2 \cup \{\infty\}$ to $S^2$.

Repeating the ODE argument (see (\ref{time independent ode for H v2})) at the end of Theorem \ref{thm:smoothness around all but finitely many points} shows that $\tilde{\omega}_i(x) + c_i$ is $k$-equivariant with inclination coordinate of the form (\ref{inclination coordinate of harmonic map}).
\end{proof}

\section*{Declaration}
\textbf{Data Availability:} Data sharing is not applicable as no datasets were generated or analyzed during the current study.

\textbf{Conflict of interest:} The authors declare that they have no conflict of interest.

\vfill

\nocite{*}
\bibliography{sn-bibliography}

\textbf{Dylan Samuelian} \\
Ecole Polytechnique Fédérale de Lausanne (EPFL) \\
dylan.samuelian@epfl.ch

\end{document}